\documentclass[aap]{imsart}

\setpkgattr{journal}{name}{\empty}
\setpkgattr{author}{prefix}{}

\RequirePackage{amsthm,amsmath,amsfonts,amssymb}
\RequirePackage[numbers,sort]{natbib}
\RequirePackage{nccmath}

\usepackage[english]{babel}
\usepackage{commath}
\usepackage[toc,page]{appendix}
\usepackage{verbatim}
\usepackage{graphicx}
\usepackage[colorlinks=true, allcolors=blue]{hyperref}
\usepackage{enumerate}
\usepackage{graphicx}
\usepackage[capitalize]{cleveref}
\usepackage{booktabs}

\usepackage{color}
\usepackage[dvipsnames]{xcolor}

\startlocaldefs

\theoremstyle{definition}
\newtheorem{theorem}{Theorem}
\newtheorem{definition}[theorem]{Definition}
\newtheorem{corollary}[theorem]{Corollary}
\newtheorem{example}[theorem]{Example}
\newtheorem{lemma}[theorem]{Lemma}

\newtheorem*{remark}{Remark}

\DeclareMathOperator{\E}{\mathbb{E}}
\DeclareMathOperator{\G}{\mathbb{G}}
\DeclareMathOperator{\prob}{\mathbb{P}}
\DeclareMathOperator{\ind}{\mathbb{I}}
\DeclareMathOperator{\B}{\mathbf{B}}
\DeclareMathOperator{\A}{\mathbf{A}}
\DeclareMathOperator{\N}{\mathbb{N}}
\DeclareMathOperator{\R}{\mathbb{R}}
\DeclareMathOperator{\Z}{\mathbb{Z}}
\DeclareMathOperator{\Jn}{\mathcal{J}_n}
\DeclareMathOperator{\Jjn}{\mathcal{J}^{(j)}_n}
\DeclareMathOperator{\Ln}{\mathcal{L}_n}
\DeclareMathOperator{\Jsn}{\mathcal{J}_{n1}}
\DeclareMathOperator{\msum}{\medmath\sum}
\DeclareMathOperator{\tsum}{{\textstyle\sum}}

\newcommand{\TV}{\text{\tiny\rm TV}}
\newcommand{\dTV}{d_{\TV}}

\newcommand{\mycomment}[1]{}

\newcounter{noteHctr} 

\newcounter{noteMctr} 

\newcounter{notePctr} 

\newcommand{\equdist}{\stackrel{\tiny{\rm d}}{=}}

\endlocaldefs

\begin{document}

\begin{frontmatter}
\title{Poisson approximation for stochastic processes summed over amenable groups}
%\title{A sample article title with some additional note\thanksref{T1}}
\runtitle{Poisson approximation for amenable groups}
%\thankstext{T1}{A sample of additional note to the title.}

\begin{aug}
%%%%%%%%%%%%%%%%%%%%%%%%%%%%%%%%%%%%%%%%%%%%%%
%%Only one address is permitted per author. %%
%%Only division, organization and e-mail is %%
%%included in the address.                  %%
%%Additional information can be included in %%
%%the Acknowledgments section if necessary. %%
%%%%%%%%%%%%%%%%%%%%%%%%%%%%%%%%%%%%%%%%%%%%%%
\author{\fnms{Haoyu} \snm{Ye}\ead[label=e1]{haoyu\_ye@g.harvard.edu}},
\author{\fnms{Peter} \snm{Orbanz}\ead[label=e2]{porbanz@gatsby.ucl.ac.uk}}
\and
\author{\fnms{Morgane} \snm{Austern}\ead[label=e3]{maustern@fas.harvard.edu}}
\vspace{.5em}
\begin{center}
  \normalfont\footnotesize\it
  Harvard University, University College London \and Harvard University
\end{center}
%%%%%%%%%%%%%%%%%%%%%%%%%%%%%%%%%%%%%%%%%%%%%%
%% Addresses                                %%
%%%%%%%%%%%%%%%%%%%%%%%%%%%%%%%%%%%%%%%%%%%%%%
%\address[A]{Harvard University}%, \printead{e1,e3}}
%\address[B]{University College London}%, \printead{e2}}
\end{aug}

\begin{abstract}
  We generalize the Poisson limit theorem to binary functions of
  random objects whose law is
  invariant under the action of an amenable group.
  Examples include
  stationary random fields, exchangeable sequences, and exchangeable graphs.
  A celebrated result of E.\ Lindenstrauss shows that normalized sums over
  certain increasing subsets of such groups approximate expectations.
  Our results clarify that the corresponding unnormalized sums of binary statistics are asymptotically Poisson,
  provided suitable mixing conditions hold.
  They extend further to randomly subsampled sums and also show that strict invariance of the distribution is not needed
  if the requisite mixing condition defined by the group holds.
  We illustrate the results with applications to random fields, Cayley graphs, and Poisson processes on groups.
\end{abstract}

% \begin{abstract}
%   The basic Poisson limit theorem specifies conditions under which the sum of independent
%   Bernoulli variables is asymptotically Poisson-distributed.
%   We prove a similar result for binary functions of
%   random structures and stochastic processes whose law is
%   invariant under the action of an amenable group.
%   Examples of such invariant structures include
%   stationary random fields, exchangeable sequences, and exchangeable graphs.
%   It is known, by results of ergodic theory, that normalized sums over
%   certain increasing subsets of such groups approximate expectations.
%   Our results clarify that unnormalized sums are asymptotically Poisson
%   if suitable mixing conditions hold.
%   Our results also show that strict invariance of the distribution is not required,
%   provided the mixing condition defined by the group still holds, and we extend the results
%   further to randomly subsampled sums.
%   With a series of examples, we demonstrate the applicability of our results, notably to exchangeable sequences, stationary sequences, stationary random fields, and homogeneous Markov random fields.
%   Our proofs adapt Stein's method for compound Poisson approximation.  
% \end{abstract}

%% \begin{keyword}[class=MSC2020]
%% \kwd[Primary ]{???}
%% \kwd{???}
%% \kwd[; secondary ]{???}
%% \end{keyword}

%% \begin{keyword}
%% \kwd{???}
%% \kwd{???}
%% \end{keyword}

\end{frontmatter}

\maketitle

\section{Introduction}
We consider random elements $X_n$ of a standard Borel space $\mathcal{X}$, for ${n\in\mathbb{N}}$. Informally, we think of each $X_n$ as a large random structure (such as a large random graph or the path of a process).
The basic Poisson limit theorem states that, if each $X_n$ is an i.i.d.\ sequence of Bernoulli variables $X_{ni}$, and if ${\prob(X_{n1}=1)=\lambda/n}$ for some ${\lambda>0}$, then
\begin{equation}
  \label{intro:poisson:limit}
  \msum_{i\le n}X_{ni}\quad\xrightarrow{\;\tiny\rm d\;}\quad\text{\rm Poisson}(\lambda)
  \qquad\text{ as }n\rightarrow\infty\;.
\end{equation}
This result has a number of generalizations, where sequence entries may be exchangeable \cite{PLTweaklyexchangeable,POISSON_EXCHANGEABLETRIALS}, locally dependent
\cite{Barbour1992CPStein}, or each $X_n$ may be a stationary random field rather than a sequence \cite{Berman1987,Berman1984-rf}.

We take a different approach to Poisson limit theorems and consider an amenable group $\G$ that acts measurably on $\mathcal{X}$. (We assume for the moment that $\G$ is discrete, but also consider the uncountable case later.)
Amenability means $\G$ can be approximated, in a precise sense, by a sequence of finite subsets ${\A_1,\A_2,\ldots\subset\G}$ (see \cref{sec:amenability}).
If the law of $X_n$ is $\G$-ergodic,
\begin{equation}
  \label{eq:lindenstrauss}
  \mfrac{1}{|\A_k|}\msum_{\phi\in\A_k}\,f(\phi X_n)\;\xrightarrow{k\rightarrow\infty}\;\E[f(X_n)]\qquad\text{almost surely for any }f\in\mathcal{L}_1(X_n)
\end{equation}
holds under suitable conditions on $(\A_n)$,
by Lindenstrauss' ergodic theorem for amenable groups \cite{Lindenstrauss2001}. This result subsumes the strong law of large numbers
(if $X_n$ is an i.i.d.\ sequence, $f$ a function of the first entry, ${\A_k}$ is the group of permutations of $\{1,\ldots,k\}$,
and ${\G=\cup_k\A_k}$ the group of finitely supported permutations of $\mathbb{N}$).
%It can also be shown that, if an additional mixing condition holds, the normalized sums in \eqref{eq:lindenstrauss} satisfy a central limit theorem
%and bounds of Berry-Esseen type \citep{AO18}.
The purpose of this work is to establish Poisson approximation results for Lindenstrauss' theorem.
In analogy to \eqref{intro:poisson:limit}, we replace the normalized sums in \eqref{eq:lindenstrauss}
by the unnormalized sums
\begin{equation*}
  W_n\;:=\;\msum_{\phi\in\A_n}\,f_n(\phi X_n)\;,
\end{equation*}
where we now consider the ``triangular array'' case ${k=n}$, and additionally permit the function ${f_n:\mathcal{X}\rightarrow\{0,1\}}$ to depend on $n$.
We first assume for simplicity that $X_n$ is ergodic and ${\prob(f_n( X_n)=1)=O(|\A_n|^{-1})}$. In this case,
under suitable mixing conditions,
\begin{equation*}
  W_n\;\xrightarrow{\;\tiny\rm d\;}\;Z(\lambda)\qquad\text{ as }n\rightarrow\infty\;,
\end{equation*}
where $Z$ is a compound Poisson variable and the parameter is now a function ${\lambda:\mathbb{N}\rightarrow[0,\infty)}$.
To characterize this function, it is helpful to think of
\begin{equation}
  \label{eq:process}
  \phi\;\mapsto\; f_n(\phi X_n)\qquad\text{ for }\phi\in\G\text{ and each fixed }n\in\mathbb{N}
\end{equation}
as a stochastic process indexed by $\G$. The parameter $\lambda$ depends on the local dependence within this process, i.e.,\ between nearby indices $\phi$, as ${n\rightarrow\infty}$.
Since the law of $X_n$ is ergodic, it is invariant under $\G$, so ``everything shows
at the origin'', and it suffices to control dependence on a neighborhood of the unit element ${e\in\G}$.
If mixing holds, the limit
\begin{equation*}
  \lambda(k)\;=\;\lim_{b\rightarrow\infty}\lim_{\phantom{b}n\rightarrow\infty\phantom{b}}
    |\A_n|\;\mathbb{P}\big(f_n(X_n)=1 \textrm{ and } \msum_{\phi'\in\B_b}f_n(\phi'\cdot X_n)=k\big)
\end{equation*}
exists for each ${k\in\mathbb{N}}$, and determines the law of $Z(\lambda)$.
More generally, if the $X_n$ are not ergodic or the limits do not exist, $\lambda$ becomes a random variable.
These are \cref{result1} (the general case) and
\cref{result1_cor} (the simple ergodic case sketched here).

Lindenstrauss' theorem \citep{Lindenstrauss2001} extends from discrete groups to groups that are locally compact, second countable, and
Hausdorff; in this case, the sets $\A_n$ are compact and the sum in \eqref{eq:lindenstrauss} becomes an integral.
The Poisson approximation problem has an inherently discrete nature, but we can nonetheless consider a lcscH group $\G$ and discretize
it by subsampling: Let $\nu_n$ be a probability measure on $\A_n$, and sample a (possibly random) number $J_n$ of random transformations ${\phi_{n1},\dots,\phi_{nJ_n}}$ i.i.d.\ from $\nu_n$, and consider the subsampled sum $W_n:=\sum_{i\le J_n}f_n(\phi_{ni}X_n)$.
Provided $\nu_n$ does not concentrate on a set that is ``too small'',
and similar mixing conditions as above hold, $W_n$
is asymptotically compound Poisson distributed, where the parameter may again be random. See \cref{MtKatahdin}.

Two aspects of our results are interesting to note in the broader context:
\\[.5em]
{\em Universality}. It is a common and well-understood theme that suitable aggregate statistics of i.i.d.\ variables
exhibit universality, i.e.\ their asymptotic behavior is insensitive to the input distribution.
The central limit theorem and the Poisson limit theorem are key examples of such results.
Our results show that the Poisson limit theorem extends to aggregate statistics of
random objects with amenable symmetries, and so does the central limit theorem as shown in
\citep{AO18,bjorklund2020central}. This suggests a broader theme,
where universality derives from amenable symmetries rather than an i.i.d.\ assumption.
\\[.5em]
{\em Mixing is essential, but invariance is not}. A perhaps surprising aspect of \cref{result1} is that---unlike the pointwise ergodic theorem and mean ergodic theorems for amenable groups---
it does \emph{not} require the law of $X_n$ to be invariant under $\G$. Rather, mixing with respect to $\G$ suffices.
Observe that the law is invariant if and only if 
\begin{equation*}
  \E[f(\phi X_n)]\;=\;\E[f(\psi X_n)]\qquad\text{ for all }\phi,\psi\in\G\text{ and all }f\in\mathcal{L}_1(X_n)\;.
\end{equation*}
Phrased in terms of the stochastic process
\eqref{eq:process} on $\G$, invariance hence means that expectations are constant over the index set, whereas
mixing means, loosely speaking, that two values are approximately independent if the distance between their indices is large.
%% In terms of the stochastic process \eqref{eq:process} on $\G$, invariance hence means that the expectation of is constant over $\G$.
%% Mixing, on the other hand, means loosely speaking that the values of this process at indices that are far from each other are
%% approximately independent.
In this sense, the relevant notion of symmetry is not invariance of the law,
but rather that the process decouples with respect to its (highly structured) index set $\G$.

\subsection{Related literature}
We briefly summarize further references on the main themes of this work.
   \\[.2em]
   {\em Poisson approximation}.
The Poisson limit theorem goes back to Poisson's work on Bernoulli sequences \citep[e.g.][]{kallenberg2002foundations}.
Convergence rates in total variation were first obtained by Prokhorov, and then improved by a number of authors
\cite{prokhorov1953asymptotic,vervaat1969upper,kerstan1964verallgemeinerung}.
They were generalized beyond identically distributed variables by Hodges and Le Cam \cite{LeCam},
and from Bernoulli to integer-valued random variables by \citet{marcinkiewicz1938fonctions}.
The independence assumption has been relaxed to a form of local dependence \cite{Barbour1992CPStein}, 
and also to dependent sequences $(X_i)$ for which the probability of two consecutive successes
is negligible \cite{freedman1974poisson}. This was extended to certain mixing sequences \cite{PoissonapproximationtoDependentTrials},
and to dissociated random variables \cite{barbour1984poisson,barbour1987improved}. Excluding consecutive successes is, however, a strong requirement,
see \cite{aldous2013probability}. The survey \cite{poissonapproxsurvey} provides an overview.
Other extensions include exchangeable sequences \cite{PLTweaklyexchangeable,POISSON_EXCHANGEABLETRIALS},
stationary sequences \cite{Berman1987} and stationary random fields \cite{Berman1987,Berman1984-rf} of Bernoulli random variables,
which from our perspective are all examples of distributions invariant under discrete amenable groups,
and like our results exhibit compound Poisson limits. Finally in a recent paper \cite{bjorklund2023poisson}, the authors looked at random lattices in $\mathbb{R}^d$ that are invariant under the action of $SL_d(\mathbb{R})$ and showed that the number of points at the intersection of this lattice and a shrinking set is asymptotically Poisson distributed. 
\\[.2em]
{\em Stein's method}.
The technical tool of choice for handling dependence is Stein's method \cite{Rossmonograph}, and specifically its adaptation to the Poisson approximation problem by Chen \cite{chen1974convergence,PoissonapproximationtoDependentTrials}.
The compound Poisson case is due to \cite{Barbour1992CPStein}. Our proofs draw on the Chen-Stein method and interleave it with a suitable notion of mixing
with respect to the group $\G$. 
\\[.2em]
  {\em Ergodic theorem and amenability}.
  Amenability plays a pivotal role in modern ergodic theory, graph theory, and a number of related fields.
  An introduction to amenable groups and their properties can be found in
  \citep{Bekka:delaHarpe:Valette:2008}. \citet{weiss2003actions} and \citet{Einsiedler:Ward} give accounts of mean and pointwise ergodic theorems for such groups. 
  The pointwise ergodic theorem \eqref{eq:lindenstrauss} extends both the basic law of large numbers and Birkhoff's ergodic theorem for stationary sequences. Increasingly more general versions were obtained over much of the 20th century, by authors including Ornstein, Tempelman, Weiss, and others \citep[e.g.][]{weiss2003actions}. The general version was then established by \citet{Lindenstrauss2001}.
\\[.2em]
  {\em Mixing}.
  Bradley \cite{bradley2005basic} surveys definitions of mixing, and common applications such as stationary random fields.
  Decker \cite{dedecker1998central}, for example, uses $\alpha$-mixing to establish central limit theorems for random fields and
  sequences. 
  More recently, \citet{bjorklund2020quantitative} proposed a type of mixing coefficient for semisimple Lie group and $\mathcal{S}$ semi-simple algebraic group actions, and give conditions under which these coefficients decay exponentially fast. This notion of mixing is similar to $\alpha$-mixing, and has been used to establish asymptotic normality of some F{\o}lner averages \cite{bjorklund2020central}. For $\alpha$-mixing and central limit theorems for amenable groups, see \cite{AO18}.
  
%%   Mixing coefficients play a crucial role in deriving limit theorems for stationary sequences and random fields. For a comprehensive overview, we recommend reading \cite{bradley2005basic}. The central limit theorem, in particular, has been derived based on conditions involving $\alpha$-mixing coefficients for random fields and sequences \cite{dedecker1998central}. Beyond those, \cite{bjorklund2020quantitative} proposed new mixing coefficients that are akin to $\alpha$-mixing coefficients for semisimple Lie group and $\mathcal{S}$ semi-simple algebraic group actions. They showed that under general conditions, those will decay exponentially fast. This has been used to establish the asymptotic normality of some F{\o}lner averages \cite{bjorklund2020central}. For deriving the Poisson limit theorem, stronger mixing coefficients are considered \cite{PoissonapproximationtoDependentTrials} namely $\phi$-mixing coefficients. In this paper, we consider similar coefficients and adapt them to group actions.

\section{Preliminaries}

This section briefly reviews invariant measures, amenable groups, and the adaptation of Stein's for the compound Poisson approximation.

 \subsection{Groups and amenability}
\label{sec:amenability}

  Let $\G$ be a group that is locally compact, second countable and Hausdorff (lcscH), let $e$ be its identity element, and write $|\,\cdot\,|$ for left Haar measure on $\G$.
  Such a group admits both a left- and a right-invariant compatible metric \citep{Becker:Kechris:1996}, and we choose a metric
  $d$ that is left-invariant, that is, ${d(\phi^{-1}\psi,\phi^{-1}\pi)=d(\psi,\pi)}$ for all ${\phi,\psi,\pi\in\G}$.
  We denote closed metric balls in $\G$ by ${\B_t(\phi):=\{\psi\in\G|\,d(\phi,\psi)\leq t\}}$, and abbreviate ${\B_t:=\B_t(e)}$.
  
  The group is \textbf{amenable} if it contains a compact subsets ${\A_1,\A_2,\ldots\subset\G}$ that satisfy
  \begin{equation}\label{eq:amenable:lcscH}
    \frac{|\A_n\vartriangle K\A_n|}{|\A_n|}\;\xrightarrow{n\rightarrow\infty}\;0\qquad\text{ for every compact }K\subset\G\;.
  \end{equation}
  Such a sequence $(\A_n)$ is called a \textbf{F{\o}lner sequence}. If $\G$ is countable, it suffices to require ${|\A_n\vartriangle \phi\A_n|/|\A_n|\rightarrow 0}$ for each ${\phi\in\G}$.
  \\[.5em]
  \textbf{Assumptions}. We assume throughout that $\G$ is lcscH and amenable, that $\G$ acts measurably on a standard Borel space $\mathcal{X}$, and that the sets
  ${\A_n^{-1}:=\{\phi^{-1}|\phi\in\A_n\}}$ form a F{\o}lner sequence. If $\G$ is discrete, we always choose the Haar measure as the counting measure.
  \\[.5em]  
  For each $\phi$, the action defines a Borel automorphism of $\mathcal{X}$, which we again denote by $\phi$. The measurable sets that are invariant under the action form a $\sigma$-algebra, which we denote
  \begin{equation*}
    \sigma(\G)\;:=\;\{A\subset\mathcal{X}\text{ Borel }|\,\phi^{-1}A=A\text{ for all }\phi\in\G\}.
  \end{equation*}
  A probability measure $P$ on $\mathcal{X}$ is \textbf{$\G$-invariant} if ${P\circ\phi^{-1}=P}$ for all ${\phi\in\G}$, and \textbf{$\G$-ergodic} if it is $\G$-invariant and trivial on $\sigma(\G)$.
  If $X$ is a random element of $\mathcal{X}$ and $f$ integrable, we denote conditioning on the invariant sets as ${\E[X|\G]:=\E[X|\sigma(\G)]}$, and note that
  \begin{equation*}
    \E[f(X)|\G]\;=\;\E[f(X)]\qquad\text{ almost surely if $X$ is $\G$-ergodic,}
  \end{equation*}
  by the ergodic decomposition theorem \citep[e.g.][]{Einsiedler:Ward}.

  \begin{example}
    (i) Let $\mathbb{S}_n$ be the group of permutations of $\{1,\ldots,n\}$, and ${\mathbb{S}_\infty:=\cup_{n\in\mathbb{N}}\mathbb{S}_n}$ the group of finitely supported permutations of $\mathbb{N}$.
    A right-invariant metric is given by the inverse prefix metric ${d(\phi,\psi):=\min\{n\in\mathbb{N}|\phi^{-1}(n,n+1,\ldots)=\psi^{-1}(n,n+1,\ldots)\}}$.
    Since each ${\phi\in\mathbb{S}_\infty}$ satisfies ${\phi\in\mathbb{S}_n}$ for $n$ sufficiently large, the sequence ${(\mathbb{S}_n)}$ is F{\o}lner.
    \\[.2em]
(ii) Suppose that $\mathbb{S}_\infty$ acts on the product space ${\mathcal{X}=\mathbb{R}^\mathbb{N}}$ of the sequences by permuting the sequence indices. Then the $\mathbb{S}_\infty$-invariant random elements are exchangeable sequences. The $\mathbb{S}_\infty$-ergodic sequences are the i.i.d.\ sequences, by the Hewitt-Savage theorem.
    \\[.2em]
    (iii) Fix ${r\in\mathbb{N}}$, and equip the group ${(\mathbb{Z}^r,+)}$ with the metric ${d(\boldsymbol{i},\boldsymbol{j}) = \max_{k\le r}|i_k-j_k|}$, which is left-invariant. The group is amenable, since the metric balls
    ${\B_n=\{-n,\ldots,n\}^r}$ form a F{\o}lner sequence. Similarly, a F{\o}lner sequence in ${(\mathbb{R}^r,+)}$ is given by the sets ${[-n,n]^r}$.
  \end{example}
  Further examples of amenable groups include the scale group ${(\mathbb{R}_{>0},\cdot)}$; the isometry groups of Euclidean space $\mathbb{R}^d$;
  the direct limit ${\cup_{n\in\mathbb{N}}\,\mathbb{SO}(n)}$ of $n$-dimensional rotation groups, whose invariant distributions
  are rotatable Gaussian processes \citep{Kallenberg:2005}; the crystallographic groups \citep{Ratcliffe:2006}; all abelian, all nilpotent, and all solvable groups \citep{Bekka:delaHarpe:Valette:2008};
  any set of triangular matrices with diagonal entries 1 that forms a group with respect to matrix multiplication (since such groups are nilpotent);
  as a special case of the latter, the Heisenberg groups \citep{Bekka:delaHarpe:Valette:2008}; and the lamplighter groups \citep{Lindenstrauss2001}.
  We refer to \citep{Bekka:delaHarpe:Valette:2008} for futher background on amenable groups, to \citep{Einsiedler:Ward} for their use in ergodic theory, and to \citep{weiss2003actions} for background on their ergodic theorems.

\subsection{Stein's method for compound Poisson variables}\label{test_functions}

Let ${\lambda:\mathbb{N}\rightarrow[0,\infty)}$ be a function with ${\sum_k\lambda(k)<\infty}$. A random variable $Z(\lambda)$
has \textbf{compound Poisson} distribution if it is distributed as
\begin{equation*}
  Z(\lambda)\;\equdist\;\msum_{i\leq N}M_i
  \qquad\text{ for }N\sim\text{Poisson}(\tsum_k\lambda(k))\;,
\end{equation*}
where ${M_1,M_2,\ldots}$ are independent random variables with distribution
\begin{equation*}
  \mathbb{P}(M_i=k)\;=\;\mfrac{\lambda(k)}{\sum_j\lambda(j)}\qquad\text{ for }k\in\mathbb{N}\;.
\end{equation*}
A variant of Stein's method, due to \citet*{Barbour1992CPStein}, can be used to approximate a random variable $W$ with values in $\mathbb{N}$
by a compound Poisson variable $Z$. We briefly review the main ingredients of this approach.
Proximity between discrete variables is quantified using the \textbf{total variation distance}
\begin{equation*} 
  \dTV(W,Z)\;:=\;\sup\nolimits_{h\in\mathcal{H}}|\E[h(W)]-\E[h(Z)]|\;,
\end{equation*}
where ${\mathcal{H}:= \{\ind(\cdot \in A): A \subseteq \N \}}$ is the class of indicator functions. Also, we denote the conditional version of the metric by
\begin{equation*} 
  \dTV(W,Z|\G)\;:=\;\sup\nolimits_{h\in\mathcal{H}}|\E[h(W)|\G]-\E[h(Z)|\G]|.
\end{equation*}
\begin{lemma}[\citet*{Barbour1992CPStein}]\label{CPlemma1}
  Require ${\sum_j\lambda(j)<\infty}$. For every bounded function ${g: \mathbb{N} \rightarrow \mathbb{R}}$,
  there is a bounded function $f: \N \rightarrow \R$ that solves the equation
  \begin{equation}\label{CPProp3.1}
    w f(w)-\msum_{k=1}^{\infty} k \lambda(k) f(w+k)=g(w) \qquad\text{ for all }w \in \N
  \end{equation}
  if and only if $\E[\,g(Z(\lambda))]=0$.
\end{lemma}
This holds in particular for ${g=h-\E[h(Z)]}$ if $h$ is a bounded function. It follows that, for each
${h\in\mathcal{H}}$, the equation 
\begin{equation}\label{mangogelato}
w f(w)-\msum_{k=1}^{\infty} k \lambda(k) f(w+k) = h(w) - \E[h(Z)].
\end{equation}
has a 
unique solution $f=f^{\lambda}_h$. Substituting into the definition of $\dTV$ gives
\begin{equation}\label{Steineqwithsup}
  \dTV(W,Z)\;=\;
    \sup _{h \in \mathcal{H}}\big|\mathbb{E} h(W)-\mathbb{E} h(Z)\big|=\sup _{h \in \mathcal{H}}\Big|\mathbb{E}\Big[W f^{\lambda}_h(W) - \sum_{k=0}^{\infty} k \lambda(k) f^{\lambda}_h(W+k)\Big]\Big|.
\end{equation}
Thus, we can bound $\dTV$ by bounding the right-hand side of \eqref{Steineqwithsup}.
To this end, we define the constants 
\begin{equation}\label{H_defn}
  H_0({\lambda}) := \sup_{h\in\mathcal{H}}\sup_{w\in\N}\big|f^{\lambda}_h(w)\big|
  \quad\text{ and }\quad
  H_1({\lambda}) := \sup_{h\in\mathcal{H}}\sup_{w\in\N}\big|f^{\lambda}_h(w+1)-f^{\lambda}_h(w)\big|.
\end{equation}
Our proofs rely crucially on the following result.
\begin{lemma}[\citet*{Barbour1992CPStein}]\label{SteinCPboundTV}
  Require ${\sum_j\lambda(j)<\infty}$, and define the class of functions 
  \begin{equation*}
    \mathcal{F}_{\TV}(\lambda): =
    \{ f:\mathbb{N}\rightarrow\mathbb{R} \;\text{ such that }\;  \|f\| \leq H_0(\lambda)\;\text{ and }\; \|\Delta f\| \leq H_1(\lambda)\}\;.
  \end{equation*}
  Then any random variable $W$ with value in ${\mathbb{N}\cup\lbrace 0\rbrace}$ satisfies
\begin{equation*}
\dTV(W, Z)\;\leq\; \sup_{g\in \mathcal{F}_{\TV}(\lambda)}\big|\E[Wg(W) - \msum_{k\geq 1}k\lambda(k) g(W+k)]\big|.  
\end{equation*}
\end{lemma}
A number of authors have given bounds for the constants $H$ \cite{Barbour1992CPStein,magicfactor,CPMarkov,BarbourXia2000}.
The following result provides a bound in terms of $\lambda$.
\begin{lemma}[\citet*{Barbour1992CPStein}]\label{HTVboundnomono}
  If ${\lambda:\mathbb{N}\rightarrow[0,\infty)}$ satisfies ${\sum_k\lambda(k)<\infty}$, then
\begin{equation}\label{CPProp3.2i}
  H_{0}(\lambda),\; H_{1}(\lambda)\;\leq\;\min\lbrace 1/\lambda(1), 1\rbrace\, e^{\sum_k\lambda(k)}\;.
\end{equation}
If $k \lambda(k)\rightarrow 0$ as ${k\to\infty}$, we even have
\begin{equation}\label{CPProp3.2ii}
  H_{1}(\lambda) \;\leq\;
  \frac{1}{\lambda(1)-2 \lambda(2)}\left(\frac{1}{4\left(\lambda(1)-2 \lambda(2)\right)}+\log ^{+} 2\left(\lambda(1)-2 \lambda(2)\right)\right) \wedge 1.
\end{equation}
\end{lemma}

\section{Mixing}
\label{sec:mixing}

Our results rely on two notions of mixing, $\Psi$- and $\xi$-mixing, which we introduce next.
Recall that mixing coefficients measure long-range dependence within a random structure.
The basic example is a discrete time process, say ${(Y_1,Y_2,\ldots)}$. Here, one might fix a time $k$,
and consider the tail ${(Y_{k+t},Y_{k+t+1},\ldots)}$ following some time gap $t$. A mixing coefficient then
measures the dependence between $Y_k$ and the tail as a function of $t$. To construct the coefficient,
one introduces a class $\mathcal{C}(t)$ of pairs of events $A$, formulated in terms of $Y_k$, and $B$,
formulated in terms of the tail. 
Different definitions use different ways to express independence. One common type of condition,
$\alpha$-mixing, expresses independence as a factorial distribution, and defines a mixing coefficient of the form
\begin{equation*}
  \alpha(t)\;=\;\sup\nolimits_{(A,B)}\,|\mathbb{P}(A,B)-\mathbb{P}(A)\mathbb{P}(B)|\;,
\end{equation*}
where the supremum runs over all events $B$ in the $\sigma$-algebra ${\sigma(Y_{k+t},Y_{k+t+1},\ldots)}$ generated by the tail,
and all ${A\in\sigma(Y_k)}$.
One then says that mixing holds if $\alpha(t)\rightarrow 0$ as
${t\rightarrow\infty}$. Such $\alpha$-mixing coefficients arise in particular when central limit theorems are generalized from i.i.d.\ processes
to dependent processes. Poisson approximation results tend to require a stronger condition,
and often draw on $\Phi$-mixing. This is defined by a coefficient of the form
\begin{equation*}
  \Phi(t)\;=\;\sup\nolimits_{(A,B)}\,|\mathbb{P}(B|A)-\mathbb{P}(B)|\;,
\end{equation*}
where $A$ now runs through the $\sigma$-algebra ${\sigma(\ldots,Y_{k-1},Y_{k})}$, and $B$ through that generated by
the tail as above. In contrast to $\alpha$-mixing, this
compares the entire past of $Y_k$ to to the tail, rather than just $Y_k$ itself, and measures independence in terms
of a conditional rather than a joint. That makes $\Phi(t)\rightarrow 0$ a stronger requirement, since clearly
${|\mathbb{P}(B|A)-\mathbb{P}(B)|\geq|\mathbb{P}(A,B)-\mathbb{P}(A)\mathbb{P}(B)|}$.
We refer to Bradley's survey \citep{bradley2005basic} for further background on mixing.
Under sufficiently strong mixing assumptions, Poisson limit theorems have been extended to a number
of problems that involve dependent variables, see for example \cite{PoissonapproximationtoDependentTrials,Berman1984-rf,Berman1987, mixingtriangulararray}.
%% \begin{remark}
%%   Under conditions involving stronger mixing coefficients, the Poisson limit theorem has also been extended to the dependent case \cite{PoissonapproximationtoDependentTrials,Berman1984-rf,Berman1987, mixingtriangulararray}. 
%% \end{remark}

\subsection{$\Psi$-mixing}

To define mixing for our problems, it is useful to regard ${(f_n(\phi X_n))_{\phi\in\mathbb{G}}}$ as a binary-valued stochastic process indexed by $\mathbb{G}$. We then consider an index ${\phi\in\mathbb{G}}$, which substitutes for the time $k$ above, and a subset ${G\subset\mathbb{G}}$, which
substitute for the index set of the tail. Informally, we hence measure the dependence between
\begin{equation*}
  f_n(\phi X_n)\quad\text{ and }\quad (f_n(\psi X_n))_{\psi\in G}\qquad\text{ if }\phi\in\mathbb{G}\text{ and }G\subset\mathbb{G}\text{ are far apart.}
\end{equation*}
To make that precise, we extend the distance $d$ on $\mathbb{G}$ to sets, as
\begin{equation*}
  \bar{d}(\phi,G)\;:=\;\inf\nolimits_{\psi\in G}d(\phi,\psi)
\end{equation*}
For any set ${G\subset\mathbb{G}}$, denote by $\sigma_n(G)$ the $\sigma$-algebra generated by the random variables
${\lbrace f_n(\phi X_n)\,|\,\phi\in G\rbrace}$. The set of all events that can be formulated in terms of a process value
and a segment that have at least distance $t$ is then
\begin{equation*}
  \mathcal{C}_n(t)=\{(A,B)\in \sigma_n(\phi)\otimes \sigma_n(G)\;|\; \phi \in \mathbb{G},\, G \subset \mathbb{G}
  \;\text{ with }\;\bar{d}(\phi,G)\ge t\}\;.
\end{equation*}
We then define a conditional form of $\Phi$-mixing, as
\begin{equation*}
  \Psi_{n,p}(t|\mathbb{G})\;:=\;\big\|\,\sup\nolimits_{(A,B)\in\mathcal{C}_n(t)}\,|\mathbb{P}(B|A,\mathbb{G})-\mathbb{P}(B|\mathbb{G})|\,\big\|_{p}
  \qquad\text{ for }p\geq 1\;.
\end{equation*}
If $X_n$ is $\mathbb{G}$-ergodic, and hence ${\E[f(\phi X_n)|\mathbb{G}]=\E[f(\phi X_n)]}$ almost surely, this becomes
\begin{equation*}
  \Psi_{n,p}(t)
  \;=\;
  \,\sup\nolimits_{(A,B)\in\mathcal{C}_n(t)}\,|\mathbb{P}(B|A)-\mathbb{P}(B)|\;.
\end{equation*}

\subsection{$\xi$-mixing}
We also need a refinement of this definition which additionally considers the sum of $f_n$ over a neighborhood of $\phi$, given by a ball ${\B_b(\phi)=\lbrace{d(\psi,\phi)\leq b}\rbrace}$. 
The goal is now to measure the dependence between
\begin{equation*}
  \bigl(f_n(\phi X_n),\tsum_{\psi\in\B_b(\phi)}f_n(\psi X_n)\bigr)\quad\text{and}\quad (f_n(\psi X_n))_{\psi\in G}\quad\text{ if }\phi\in\mathbb{G}\text{ and }G\subset\mathbb{G}\text{ are far apart.}
\end{equation*}
We therefore substitute $\sigma_n(\phi)$ by the $\sigma$-algebra
\begin{equation*}
  \sigma^b_n(\phi)\;:=\;\sigma\bigl(f_n(\phi X_n),\tsum_{\psi\in\B_b(\phi)}f_n(\psi X_n)\bigr)\quad\text{ for }b>0\;.
\end{equation*}
We note that this is always a sub-$\sigma$-algebra of ${\sigma_n(\B_b(\phi))}$, but can in general be much smaller. We then consider the set of events
\begin{equation*}
  \mathcal{D}_n^b(t)\;:=\;
  \{(A,B)\in \sigma_n^b(\phi)\otimes \sigma_n(G)\;|\; \phi \in \mathbb{G},\, G \subset \mathbb{G}
  \;\text{ with }\;\bar{d}(\phi,G)\ge t\}\;.
\end{equation*}
The \textbf{$\xi$-mixing coefficient} is the quantity
\begin{equation*}
  \xi_{n,p}^{b}(t|\mathbb{G}) \;:=\;
  \E\Big[\sup \Big|\tsum_{k\in \N}\,\frac{\prob(A_k|\G)(\prob(B_k|A_k,\G) - \prob(B_k|\G))}{\prob(\cup_k A_k|\G)}\Big|^p~\Big]^{\frac{1}{p}},
\end{equation*}
where the supremum runs over all sequence ${(A_1,B_1),(A_2,B_2),\ldots}$ in $\mathcal{D}_n^b(t)$ for which the sets ${A_1,A_2,\ldots}$ are pairwise disjoint.
If $X_n$ is $\mathbb{G}$-ergodic, the definition again simplifies to
\begin{equation*}
\xi_{n}^{b}(t) = \sup_{(A,B)\in \mathcal{D}_n^{b}(t)}\big|\prob(B|A) - \prob(B)\big|\;.
\end{equation*}

\begin{example}
  (i) Suppose the group ${\G=\Z}$ acts on the set of $\mathbb{R}^{\Z}$ of sequences by shifts. Choose $f_n$, independently of $n$, as the $0$th coordinate function ${f_n(\ldots,x_{-1},x_0,x_1,\ldots)=x_0}$.
  For each $n$, let $X_n:=(X^n_t)_{t\in\Z}$ be sequences of random variables whose distribution is invariant under $\Z$, and ergodic. Thus, each $X_n$ is an ergodic stationary sequence.
  Then 
  \begin{equation}
    \label{ct}
    \Psi_{n}(t)\;=\;\sup_{\substack{A\in \sigma(X^n_0)\\B\in \sigma(\dots,X^n_{-t-1},X^n_{-t},X^n_t,X^n_{t+1}\dots)}}\big|P(B|A)-P(B)\big|
    \qquad
    \text{ for all }p\geq 1\;.
  \end{equation}
  (ii) For each $n$, let $X_n$ be an exchangeable random sequence (not necessarily ergodic). Then
  \begin{equation*}
    \Psi_{n,p}(t|\mathbb{G})\;=\;0
    \quad\text{ and }\quad
    \xi_{n,p}^b(t|\mathbb{G})\;=\;0
    \qquad\text{ for all }t>0\text{ and }b>0\;.
  \end{equation*}
\end{example}

\section{Results}
With the requisite definitions of mixing in place, we can state our main results, \cref{result1} and \cref{MtKatahdin} below.

\subsection{Moment conditions}\label{momentcond}

\newcommand{\Q}{Q}

For $p\ge 1$, we define the aggregate moment
\begin{equation*}
  \mu_{n,p}\;:=\;\msum_{\phi\in \A_n} \|Q_n(\phi)\|_{\frac{p}{p-1}}\;.
\end{equation*}
How $Q_n(\phi)$ varies as $\phi$ runs through $\A_n$ is summarized by 
\begin{equation*}
  \eta_{n,p}:=\Big(|\A_n|^{p-1}\msum_{\phi\in\A_n}\|Q_n(\phi)\|^p_1\Big)^{1/p}\;.
\end{equation*}
If the distribution of $X_n$ is $\mathbb{G}$-invariant, these terms simplify to 
$\mu_{n,p}=|\A_n|\|Q_n(e)\|_{\frac{p}{p-1}}$ and
$\smash{\eta_{n,p}=\|Q_n(e)\|_{1}}$.
We must also account for interactions between pairs of group elements, and define 
\begin{equation*}
  \gamma_n(b)\;:=\;\mfrac{|\A_n|}{|\B_{b}|}\msum_{\phi\in\A_n}\msum_{\psi\in\B_b(\phi)} \|Q_n(\phi)Q_n(\psi)\|_1\;.
\end{equation*}
If $X_n$ is $\mathbb{G}$-ergodic, this simplifies to ${\E[f_n(X_n)]\E[f_n(\phi X_n)]=O(|\A_n|^{-2})}$.

\subsection{Main result}\label{sec:mainresult}

Recall that the elementary Poisson approximation theorem holds for an increasing number $n$ Bernoulli variables with expectations
${q_i=\lambda/n}$, for some ${\lambda>0}$, and hence ${\sup_n\sum_{i\leq n} q_i\leq\infty}$. For the dependent variables
$f_n(\phi X_n)$, the corresponding condition turns out to be 
\begin{equation}
  \tag{A1}\label{A1}
  \sup_{n}\big\|\msum_{\phi\in \A_n}\Q_n(\phi)\big\|_\infty
  \;<\;\infty\;,
\end{equation}
so conditional expectations substitute for expectations in the dependent case. 
\begin{equation}
  \tag{A2}\label{A2}
  \frac{|\B_{2b_n}|}{|\A_n|}\, \gamma_n(2b_n)\;\xrightarrow{n\to\infty}\; 0.
\end{equation}
If ${b>0}$ is small, the ball $\B_b$ contains only group elements similar to the identity. One can then loosely think of
${\A_n\B_b}$ as a small perturbation of $\A_n$, and of ${\A_n\B_b\vartriangle\A_n}$ as a set of elements near the boundary of $\A_n$.
The condition
\begin{equation}
  \tag{A3}\label{A3}
  \inf_{p\ge 1}\eta_{n,p}\Big(\frac{\big|\A_n\B_{b_n}\triangle\A_n\big|}{|\A_n|}\Big)^{(p-1)/p}\xrightarrow{n\to\infty} 0
\end{equation}
thus requires $(b_n)$ to be chosen such that the relative mass of elements near the boundary shrinks sufficiently quickly
relative to the moment term $\eta_{n,p}$. Moreover, define
\begin{equation*}
  \mathcal{R}_\Psi(n,p,b)\;:=\;\msum_{i \ge b}|\mathbf{B}_{i+1}\backslash\mathbf{B}_{i}|\; \Psi_{n,p}(\,i\,| \G)
\end{equation*}
for the $\Psi$-mixing coefficient, and similarly 
\begin{equation*}
  \mathcal{R}_\xi(n,p,b)\;:=\;\msum_{i \ge 2b}|\mathbf{B}_{i+1}\backslash\mathbf{B}_{i}|\; \xi_{n,p}^b(\,i\,| \G).
\end{equation*}
The sequence ${(|\mathbf{B}_{j+1}\backslash\mathbf{B}_j|)_{j\in\mathbb{N}}}$ that appears here as a weight sequence is
known as the \textbf{growth rate} of the metric group $\mathbb{G}$. The growth rate provides rich information
about certain properties of the group; for example, all finitely generated groups whose growth rate is subexponential are amenable \citep{Einsiedler:Ward}.
\begin{equation}
  \tag{A4}\label{A4}
  \inf_{p\ge 1}\mu_{n,p}\,\mathcal{R}_\Psi(n,p,b_n)\;\xrightarrow{n\rightarrow\infty}\;0
  \quad\text{ and }\quad
  \inf_{p\ge 1}\mu_{n,p}\,\mathcal{R}_\xi(n,p,b_n)\;\xrightarrow{n\rightarrow\infty}\;0.
\end{equation}
\begin{theorem}\label{result1}
  Let ${f_1,f_2,\ldots:\mathcal{X}\rightarrow\{0,1\}}$ be measurable functions, and let $(X_n)$ be a sequence of random elements of $\mathcal{X}$. Fix a sequence $(b_n)$ of positive integers, and for each ${n\in\mathbb{N}}$, let
  $Z(\lambda_n)$ be a compound Poisson variable with parameter ${\lambda_n:\mathbb{N}\rightarrow[0,\infty)}$, where
  \begin{equation*}
    \lambda_n(k)\;:=\;
    \mfrac{1}{k}\msum_{\phi\in \A_n}
    \mathbb{P}\big(f_n(\phi X_n)=1 \textrm{ and } \msum_{\psi\in\B_{b_n}(\phi)}f_n(\psi X_n)=k\big|\G\big)
  \end{equation*}
  for each ${k\in\mathbb{N}}$. Then
%\label{bigineq}
\begin{align*}
  &\E\big[\dTV(W_n,Z(\lambda_n)| \G)\big]\;\leq\;
  \big\|H_0(\lambda_n)\big\|_\infty\;\inf_{p\ge 1}\eta_{n,p}
  \Big(\mfrac{\big|\A_n\B_{b_n}\triangle\A_n\big|}{|\A_n|}\Big)^{(p-1)/p}\\
  &\qquad+\;\big\|H_1(\lambda_n)\big\|_{\infty}\;
  \Big(
  2\mfrac{|\B_{2b_n}|}{|\A_n|}\,\gamma_n(2b_n)
  +\inf_{p\geq 1}\mu_{n,p}\mathcal{R}_\Psi(n,p,b_n)
  +2\inf_{p\geq 1}\mu_{n,p}\mathcal{R}_\xi(n,p,b_n)\Big)
  %% \\
%%   &\hspace{8em} +\;\inf_{p\ge 1}\mu_{n,p}\msum_{j \ge b_n}\big|{\mathbf{B}_{j+1} \backslash \mathbf{B}_{j}}\big|\, \psi_{n,p}(j|\G)\\
%%   &\hspace{8em} + 2\inf_{p\ge 1}\mu_{n,p}\msum_{j\geq 2b_n}\big|{\mathbf{B}_{j+1}\backslash\mathbf{B}_j}\big|\,\xi_{n,p}^{b_n}(j|\mathbb{G})\Big)\;.
\end{align*}
In particular, if conditions \eqref{A1}--\eqref{A4} hold, then ${\E[\dTV(W_n,Z_n|\G)]\rightarrow 0}$ as ${n\rightarrow\infty}$.
\end{theorem}
\begin{remark}
  Relaxing \eqref{A1} by assuming $\sup_{n}\sum_{\phi\in\A_n}\E[f_n(\phi X_n)|\G]=O_p(1)$ does not affect the convergence result, as demonstrated in the %\hyperref[appn]{Appendix}
  \cref{appn} \cref{lavenderhaze}.
\end{remark}

\subsection{A simple case: Ergodicity}

The result simplifies considerably if
each $X_n$ is $\mathbb{G}$-ergodic. In this case, conditional expectations match expectations, and the random variables $Q_n(\phi)$ take non-random
values ${Q_n(\phi)=\E[f_n(\phi X)]}$ almost surely. The moment conditions can then be substituted by
\begin{equation}
  \tag{B1}\label{B1}
  \textstyle{\sup_n}\msum_{\phi\in\A_n}\E[f(\phi X_n)]\;<\;\infty
  \quad\text{ and }\quad
  \sup_\phi\E[f(\phi X_n)]\;=\;o(1)\;.
\end{equation}
The mixing coefficients likewise simplify, to
$\Psi_{n}(t)$ and $\xi_{n}^{b}(t)$ as defined in \cref{sec:mixing}.
We require
\begin{equation}
  \tag{B2}\label{B2}
  {\textstyle\sup_n}\,\Psi_n(t)\;\longrightarrow\;0\qquad\text{ as }t\rightarrow\infty\;,
\end{equation}
and condition \eqref{A4} becomes 
\begin{equation}
  \tag{B3}\label{B3}
    \sup_n\msum_{j\geq b}|\B_{j+1}\backslash\B_j|\Psi_n(j)\;\rightarrow\;0
    \quad\text{ and }\quad
    \sup_n\msum_{j\geq 2b}|\B_{j+1}\backslash\B_j|\xi_n^b(j)\;\rightarrow\;0
\end{equation}
as ${b\rightarrow\infty}$. \cref{result1} can then be stated as follows.
\begin{corollary}\label{result1_cor}
  For each ${n\in\mathbb{N}}$, let $X_n$ be a $\G$-ergodic random element of $\mathcal{X}$, and ${f_n:\mathcal{X}\rightarrow\{0,1\}}$
  a measurable function. Suppose that the limit
  \begin{equation*}
    \lambda^b(k)\;:=\;\lim_{n\rightarrow\infty}
    \msum_{\phi\in \A_n}\mathbb{P}\Big(f_n(\phi X_n)=1 \textrm{ and } \msum_{\phi'\in\B_b(\phi)}f_n(\phi'X_n)=k\Big)
%    \lambda_n^b(k)
  \end{equation*}
  exists for all ${b>0}$ and ${k\in\mathbb{N}}$.
  If conditions \eqref{B1} and \eqref{B2} are satisfied, the limit
  \begin{equation*}
    \lambda(k)\;:=\;\lim_{b\rightarrow\infty}\lambda^b(k)
  \end{equation*}
  exists for each ${k\in\mathbb{N}}$. If \eqref{B3} also holds,
  we have ${\dTV(W_n,Z(\lambda))\rightarrow 0}$ as ${n\rightarrow\infty}$.
\end{corollary}

\subsection{Generalizations}\label{sec:generalization}

We now randomize the unnormalized averages $W_n$,
by summing over a random subset of the set $\A_n$ of transformations. In this context,
we do not require $\mathbb{G}$ to be countable---it may now be any amenable lcscH group. The sets $\A_n$ are therefore compact subsets
satisfying \eqref{eq:amenable:lcscH}.
For each $n$, let $\nu_n$ be a probability measure on $\A_n$, and generate
\begin{equation*}
  \phi_{n1},\phi_{n2},\ldots\;\sim_{\text\tiny iid}\;\nu_n\qquad\text{ independently of }X_n
\end{equation*}
and an independent random element $J_n$ of $\mathbb{N}$. The randomized sum is then
\begin{equation}
  \label{randomized:average}
  W_n\;:=\;\msum_{i\leq J_n}f_n(\phi_{ni}X_n)\;.
\end{equation}
Clearly, convergence may fail if the random transformations $\phi_{ni}$ concentrate on too small a subset,
and we must exclude such pathologies. To this end, 
call the sequence $(\nu_n)$ \textbf{well-spread} if there is a constant $\mathcal{S}>0$ such that
\begin{equation*}
  \nu_n(\B_b(\phi)\setminus\B_a(\phi))\;\leq\;\mathcal{S}\mfrac{|\B_b\backslash\B_a|}{|\A_n|}
  \qquad\text{ for all }b>a\geq 1,\phi\in\A_n,\text{ and }n\in\mathbb{N}\;.
\end{equation*}
The uniform distribution on $\A_n$, for example, is well-spread with ${\mathcal{S}=1}$.
We impose a somewhat stronger condition than in \cref{result1}, namely
\begin{equation*}
  \tag{C0}\label{C0}
  \sup_n\frac{\|J_n\|_3}{|\A_n|}\;<\;\infty
  \quad\text{ and }\quad
  \sup_n\sup_{\phi\in \A_n}|\A_n|\|Q_n(\phi)\|_2\;<\;\infty\;.
\end{equation*}
That is not required by our proof techniques, but lets us state the result without more
bespoke moment conditions as those in \cref{result1} and therefore simplifies matters considerably.
Condition \eqref{A1} translates directly to
\begin{equation}
  \tag{C1}\label{C1}
  \sup_{n}\big\|\msum_{i\leq J_n}\Q_n(\phi_{ni})\big\|_\infty
  \;<\;\infty\;,
\end{equation}
and \eqref{A2} and \eqref{A3} become
\begin{equation}
  \tag{C2}\label{C2}
  \frac{|\B_{2b_n}|}{|\A_n|}\;\xrightarrow{n\to\infty}\; 0
  \quad\text{ and }\quad
    \frac{|\B_{b_n}|}{|\A_n|}\,\max\lbrace |\A_n|^{1/2},\text{Var}(J_n)^{1/2}\rbrace\;\xrightarrow{n\to\infty}\; 0
\end{equation}
Since \eqref{C0} ensures square-integrability, it suffices to impose mixing conditions for ${p=2}$, namely 
\begin{equation}
  \tag{C3}\label{C3}
  \mathcal{R}_\Psi(n,2,b_n)\longrightarrow 0
  \qquad\text{ and }\qquad
  \mathcal{R}_\xi(n,2,b_n)\longrightarrow 0
  \qquad\text{ as }n\rightarrow\infty
\end{equation}
and
\begin{equation}
  \tag{C4}\label{C4}
  \sup_n\;\Psi_{n,2}(t|\G)\;\longrightarrow\; 0 \qquad\text{ as }t\rightarrow\infty\;.
\end{equation}
From the choice of the F{\o}lner sequence $(\A_n)$, the distributions $\nu_n$, and the constants $b_n$, we can compute a sequence $(\epsilon_n)$ of constants
%% \begin{equation*}
%%   \epsilon_n:=\frac{2}{|\A_n|}\Big\{\E[J_n]\sup_{\phi\in \A_n}\|\mu_n^\phi\|_2^2+{2\mathcal{S}^2\|J_n\|_3^3}\frac{|\B_{b_n}|^2}{|\A_n|^2}\sup_{\phi\in \A_n}\|\mu_n^\phi\|_1^2+ 4\mathcal{S}\|J_n\|^2_2\frac{|\B_{b_n}|}{|\A_n|}\sup_{\phi\in\A_n}\|\mu_n^\phi\|_2^2\Big\}^{1/2}
%% \end{equation*}
\begin{equation*}
  \epsilon_n:=2\Big(\frac{2\mathcal{S}^2\|J_n\|_3^3\,|\B_{b_n}|^2}{|\A_n|^2}\sup_{\phi\in \A_n}\|\Q_n(\phi)\|_1^2+\big(\|J_n\|_1+\frac{4\mathcal{S}\|J_n\|^2_2\,|\B_{b_n}|}{|\A_n|}\big)\sup_{\phi\in\A_n}\|\Q_n(\phi)\|_2^2\Big)^{\frac{1}{2}}
\end{equation*}
These in turn determine a sequence of radii
\begin{equation}
  \label{eq:radii}
  c_n:=\max\{r:|\B_{r}| \le \lfloor {\epsilon_n}^{\alpha -1}\rfloor^{1-\beta}
  \}\quad\text{ for given constants }\alpha,\beta\in(0,1)\;.
\end{equation}
\begin{theorem}\label{MtKatahdin}
  For each ${n\in\mathbb{N}}$, let $X_n$ be a random element of $\mathcal{X}$, let ${f_n:\mathcal{X}\rightarrow\lbrace 0,1\rbrace}$ be a measurable function,
  and $b_n$ a positive integer. Let $(\nu_n)$ be a sequence of probability measures that is well-spread with some constant $\mathcal{S}$, and define
  $\lambda_n:\mathbb{N}\rightarrow[0,\infty)$ by
  \begin{equation*}
    \lambda_n(k)\;:=\;
    \mathbb{E}\big[ \tsum_{i\le J_n}f_n(\phi_{ni} X_n)\cdot\ind\big(\tsum_{j\le J_n,\,d(\phi_{ni},\phi_{nj})\le b_n}f_n(\phi_{nj}X_n)=k\big)\big|\G\big].
  \end{equation*}
  Suppose \eqref{C0} holds. 
  Then for all ${\alpha,\beta\in(0,1)}$, there are constants $\kappa_1,\kappa_2>0$ such that the randomized sum \eqref{randomized:average}
  satisfies
    \begin{align*}
      &\mathbb{E}\big[d_{\TV}(W_n,Z(\lambda_n)|\G)\big]
      \;\leq\;
      \kappa_1\,\|H_1(\lambda_n)\|_\infty
      \,\Big(\mathcal{R}_{\Psi}(n,2,b_n)+ \mathcal{R}_{\xi}(n,2,b_n)+
        \frac{|\B_{2b_n}|}{|\A_n|}\Big)\\
        &\qquad+\; \kappa_2\,\|H_0(\lambda_n)\|_\infty\,
        \Big(\Psi_{n,2}(c_n|\G)
        +
        \frac{\abs{\B_{b_n}}}
             {|\A_n|}\sqrt{\textrm{Var}(J_n)}
             +
        \Big(\frac{|\B_{b_n}|}{\sqrt{|\A_n|}}\Big)^{\min\{\alpha,(1-\alpha)\beta\}}
        \Big),
    \end{align*}
    where $c_n$ are the radii defined in \eqref{eq:radii}. In particular, if \eqref{C1}--\eqref{C4} hold, then
    \begin{equation*}
      \E[\dTV(W_n,Z(\lambda_n)|\G)]\;\longrightarrow\; 0\qquad\text{ as }n\rightarrow\infty\;.
    \end{equation*}
\end{theorem}
\begin{remark}
  An explicit bound for $\E[\dTV(W_n,Z(\lambda_n|\G)]$ that does not require assumption \eqref{C0} and gives explicit constants
  can be found \cref{Jnbound} in the appendix.
\end{remark}

\section{Applications}\label{example_subsection}
To illustrate our results, we first use them to recover known result on random fields.
We then obtain some new results on bond percolation in Cayley graphs and on subsampled sums.

\subsection{Random fields}\label{app1}
We first show how an application of \cref{result1} to random fields on a grid recovers a result of \citep{Berman1987}, under slightly different conditions.
Fix a dimension ${m\in\mathbb{N}}$. Let ${X}$ be a binary random field on the grid $\mathbb{Z}^m$, i.e.,\ a random element of the
space ${\lbrace 0,1\rbrace^{\mathbb{Z}^m}}$ endowed with the product topology.
If we choose ${\mathbb{G}}$ as another copy of ${\mathbb{Z}^m}$, then $\mathbb{G}$ acts on the grid
by shifts. A random field whose law is invariant under this action is called
\textbf{stationary} or \textbf{homogeneous}. The group $\mathbb{Z}^m$ is amenable, and the metric balls
\begin{equation*}
  \A_n\;:=\;\B_n\;=\lbrace -n,\ldots,n\rbrace^m
\end{equation*}
around the origin constistute a valid F{\o}lner sequence.
We can hence consider a sequence of stationary random fields $X_1,X_2,\ldots$, and
define the partial sums
\begin{equation*}
  W_n\;=\;\msum_{\phi\in\A_n}X_n(\phi)\;=\;\msum_{\mathbf{t}\in\lbrace{-n,\ldots,n}\rbrace^m}X_n(\mathbf{t})\;.
\end{equation*}
\cref{result1} then implies the following result.
\begin{corollary}
  \label{cor_rf}
  For each ${n\in\mathbb{N}}$, let $X_n$ be a binary random field on $\mathbb{Z}^m$ that is stationary. 
  Require that:\\[.2em]
  (i) There is a ${c>0}$ such that ${(2n+1)^m\E[X_n(0)]\leq c}$ for all ${n\in\mathbb{N}}$.\\
  (ii) For every ${b>0}$ and each ${k\in\mathbb{N}}$, the limit 
  \begin{equation*}
    \lambda^b(k)\;:=\;\lim_{n\to\infty}\mfrac{(2n+1)^m}{k}\,\prob\bigl(X_n(0)=1\text{ and }\tsum_{\mathbf{t}\in\lbrace -b,\ldots,b\rbrace^m}X_n(\mathbf{t})=k\bigr)
    \quad\text{ exists.}
  \end{equation*}
  (iii) The random fields are mixing uniformly over $n$, in the sense that, for ${b\rightarrow\infty}$,
  \begin{equation*}
    \sup\nolimits_n\msum_{j\geq b}j^{m-1}\Psi_n(j)\;\rightarrow\;0
    \quad\text{ and }\quad
    \sup\nolimits_n\msum_{j\geq 2b}j^{m-1}\xi^b_n(j)\;\rightarrow\;0\;.
  \end{equation*}
  Then the limit ${\lambda(k):=\lim_{b\rightarrow\infty}\lambda^b(k)}$ exists for each ${k\in\mathbb{N}}$, and 
    \begin{equation*}
      \dTV(W_n,Z(\lambda))\;=\;
      \dTV\Bigl(
      \msum_{\mathbf{t}\in\lbrace{-n,\ldots,n}\rbrace^m}X_n(\mathbf{t}),\,Z(\lambda)
      \Bigr)\;\longrightarrow\;0\qquad\text{ as }n\rightarrow\infty\;.
    \end{equation*}
\end{corollary}

\subsection{Random fields with a Markov property}\label{app2}

Consider stationary random fields as above,  
We choose a left-invariant metric on the group as 
${d(\mathbf{s},\mathbf{t}):=\max_{i\leq m}|s_i-t_i|}$
for ${\mathbf{s},\mathbf{t}\in\mathbb{Z}^m}$.
For a subset ${I\subset\mathbb{Z}^m}$ of indices, 
we denote the restriction of $X_n$ to $I$ by ${X_n(I)=(X_n(\mathbf{t}))_{\mathbf{t}\in I}}$, and define
the boundary of $I$ as ${\partial I=\lbrace{\mathbf{s}\in\mathbb{Z}^m|d(\mathbf{s},I)=1\rbrace}}$.
The random field is \textbf{Markov} if its law satisfies
\begin{equation*}
  \mathcal{L}(X_n(I)\,|\,X_n(\mathbb{Z}^m\backslash I))
  \;=\;
  \mathcal{L}(X_n(I)\,|\,X_n(\partial I))
  \quad\text{ for all finite }I\subset\mathbb{Z}^m\;.
\end{equation*}
The mixing properties for stationary Markov fields have been studied thoroughly.
For example, if $X$ satisfies the so-called Dobrushin condition \cite{Berman1987},
there exists a constant ${0<\rho<1}$ such that 
\begin{equation*}
  \sup_{x\in\lbrace 0,1\rbrace^{\partial I}}
  \dTV\bigl(\mathcal{L}(X_n(I)\,|\,X_n(\partial I)=x),\,\mathcal{L}(X_n(I))\bigr)
  \;\leq\;
  \rho^{d(I,\partial J)}
\end{equation*}
for all pairs of finite sets ${I\subset J\subset\mathbb{Z}^m}$. If we assume the Markov property
in addition to stationarity, we can substitute a similar mixing condition for that
in \cref{cor_rf} above.
\begin{corollary}\label{cor_hmrf}  
  For each ${n\in\mathbb{N}}$, let $X_n$ be a binary random field on $\mathbb{Z}^m$ that is stationary and Markov.
  Require that the random fields satisfy conditions (i) and (ii) of \cref{cor_rf}, and the mixing property
  \begin{equation*}
    \max\lbrace\textstyle{\sup_n}\Psi_n(t),\,\sup_n\xi_n^b(t)\rbrace\;\leq\;c_1 e^{-c_2t}\quad\text{ for all }b,t>0\;.
  \end{equation*}
  Then the limit ${\lambda(k):=\lim_{b\rightarrow\infty}\lambda^b(k)}$ exists for each ${k\in\mathbb{N}}$, and 
    \begin{equation*}
      \dTV(W_n,Z(\lambda))\;=\;
      \dTV\Bigl(
      \msum_{\mathbf{t}\in\lbrace{-n,\ldots,n}\rbrace^m}X_n(\mathbf{t}),\,Z(\lambda)
      \Bigr)\;\longrightarrow\;0\qquad\text{ as }n\rightarrow\infty\;.
    \end{equation*}
\end{corollary}

\subsection{Exchangeable sequences}\label{app3}

Denote by $\mathbb{S}_k$ the group of permutations of the set $[k]$, and choose $\G$ as the group
${\mathbb{S}=\cup_{k\in\mathbb{N}}\mathbb{S}_k}$ of finitely supported permutations of $\mathbb{N}^{+}$.
For each $n$, let ${X_n=(X_n(1),X_n(2),\ldots)}$ be a random sequence, where each ${X_n(t)}$ is binary.
A sequence is \textbf{exchangeable} if its law is invariant under $\mathbb{S}$ acting on the indices in the natural way,
i.e.,\ if
\begin{equation*}
  \phi X_n\;\equdist\;X_n
  \qquad\text{ where }
  \phi X_n\;:=\;(X_n(\phi(1)),X_n(\phi(2)),\ldots)
\end{equation*}
holds for all ${\phi\in\mathbb{S}}$.
The group $\mathbb{S}$ is amenable, and ${\A_n:=\mathbb{S}_n}$ is a F{\o}lner sequence. Given functions $f_n(X_n)=X_n(1)$,
we have ${\sum_{\phi\in\mathbb{S}n}f_n(\phi X_n) = (n-1)!\sum_{t\leq n}X_n(t)}$. We drop the scaling factor, and define
\begin{equation*}
  W_n = \msum_{t\le n}X_n(t).
\end{equation*}
For ${\mathbb{G}=\mathbb{S}}$, the $\sigma$-algebra of $\mathbb{G}$-invariant sets is known as the exchangeable
$\sigma$-algebra, and by the well-known theorems of de Finetti and of Hewitt and Savage, 
every exchangeable sequence is conditionally independent given this $\sigma$-algebra. That implies
${\Psi(t|\mathbb{S})=0}$ and ${\xi_n^b(t|\mathbb{S})=0}$ almost surely for all ${t,b\geq 0}$.
Thus, exchangeable sequences are always mixing.

%% \begin{remark}
%% $W_n$ is typically defined as $W_n = \sum_{\phi\in\mathbb{S}n}f_n(\phi X_n) = (n-1)!\sum_{t\leq n}X_n(t)$. However, for our analysis, we simplify it to $\sum_{t\leq n}X_n(t)$ without the constant factor.
%% \end{remark}

\begin{corollary}
  \label{cor_exchangeableseq}
  Require that $\sup_n\|n\prob(X_n(1)=1| \mathbb{S}_\infty)\|_\infty<\infty$. If ${n\prob(X_n(1)=1| \mathbb{S}_\infty)}$
  converges in $L_1$ to a random variable $\Theta$ as ${n\rightarrow\infty}$, then 
  \begin{equation*}
    \E\big[\dTV(W_n, \mathrm{Poisson}(\Theta)|\mathbb{S}_\infty)\big]\;\longrightarrow\;0\qquad\text{ as }n\rightarrow\infty\;.
  \end{equation*}
\end{corollary}

\begin{remark}
We can relate this result to \cref{cor_rf} above. 
If we set ${m=1}$ in \cref{cor_rf}, each random field is a stationary sequence $(X_n(t))_{t\in\mathbb{Z}}$.
By deleting all indices below 1, we obtain the marginal sequence ${(X_n(t))_{t\in\mathbb{N}}}$, whose law is invariant under
shifts to the left,
\begin{equation*}
  (X_n(j),X_n(j+1),\ldots)\;\equdist\;(X_n(1),X_n(2),\ldots)\qquad\text{ for all }j\in\mathbb{N}\;.
\end{equation*}
Each such shift is a bijection of $\mathbb{N}$, and the law of a sequence is invariant under all bijections of the index
set if and only if is invariant under all finitely supported bijections; the latter is just exchangeability as defined above.
Thus, every exchangeable sequence is stationary. The (semi)group of shifts is of linear growth, whereas $\mathbb{S}$ is of exponential
growth. If we think of each transformation that leaves the law invariant as a constraint, exchangeability is hence a strictly stronger,
and indeed substantially stronger, hypothesis than stationarity, which illustrates why \cref{cor_exchangeableseq} holds under
significantly weaker conditions than \cref{cor_rf}.
\end{remark}

\subsection{Cayley graphs}
\label{app3.5}
Let $\mathbb{G}$ be a finitely generated amenable group with generating set $\mathcal{S}$. Denote by $d$ the word metric, and let $\mathcal{C}=(\mathcal{V},\mathcal{E})$ be the associated Cayley graph with vertex set indexed by $\mathbb{G}$. 
Write $\cdot$ for the natural action of $\mathbb{G}$ on $\mathcal{C}$.
We generate a random subgraph $\mathcal{C}'=(\mathcal{V},\mathcal{E}')$ from $\mathcal{C}$ by removing each edge with a probability of $1-p$. To this end, define a sequence $(d_n)$ as
\begin{equation*}
d_n:=\text{\rm argmin}\Big\{ k\in\mathbb{N}_+\,\Big|\,
\frac{1}{2}\big(|\mathcal{S}| |\mathbf{B}_{k}|-(|\mathcal{S}|-1)|\mathbf{B}_{k}\setminus\mathbf{B}_{k-1}|\big)\;\ge\; -\frac{\log(|\A_n|)}{\log(p)}\Big\}\;.
\end{equation*}
Let $\mathcal{G}_n=\mathcal{C}\big|_{\mathbf{B}_{d_n}}$ be the induced subgraph of $\mathcal{C}$ with vertex set $\mathbf{B}_{d_n}$. Note this is a clique in $\mathcal{C}$.
Consider the statistic
\begin{equation*}
  W_n\;:=\;\msum_{\phi\in A_n}Y_\phi
  \quad\text{ where }\quad
  Y_\phi\;:=\;\mathbb{I}\big(\mathcal{G}_n~\text{\rm is~a~subgraph~of~}\phi^{-1}\cdot \mathcal{C}'\big)~\mathbb{I}(\B(\phi,d_n)\subset \mathbf{A}_n)\;.
\end{equation*}
In words, $W_n$ is the number of subgraphs of $\mathcal{C'}$ that are translated versions of $\mathcal{G}_n$ with index sets in $\mathbf{A}_n$.
For the result, we assume that $\G$ is chosen such that
\begin{equation}
  \label{eq:condition:cayley}
  \min_{\phi\ne e}|\mathbf{B}(\phi,b)\setminus \mathbf{B}(e,b)|=\omega_{b\rightarrow\infty}\big(\log(|\B_{b+1}|)\big)
\end{equation}
That holds, for example, for all torsion-free groups of rank $k$.
\begin{corollary}
  \label{cor_cayley}
  If the group $\G$ satisfies \eqref{eq:condition:cayley}, then $W_n$ is asymptotically Poisson,
  \begin{equation*}
    d_{\TV}\big(W_n,\mathrm{Poisson}(\lambda_n)\big)\;\xrightarrow{n\rightarrow\infty}\;0
    \qquad\text{ for }\lambda_n:=p^{|\mathcal{G}_n|}\;,
  \end{equation*}
  where $|\mathcal{G}_n|$ is the number of edges in $\mathcal{G}_n$.     
\end{corollary}

\subsection{Randomized sums: Fixed number of samples and Poisson process}\label{app4} 

We now consider the randomized sum $W_n$ as defined in \eqref{randomized:average}, and
discuss some special cases of \cref{MtKatahdin}.
Recall that $J_n$ denotes the random number of summands in $W_n$, and that, among the
conditions \eqref{C0}---\eqref{C4}, those involving $J_n$ are
\begin{equation}
  \label{condition:Jn}
  \sup\nolimits_n\mfrac{\|J_n\|_3}{|\A_n|}\;<\;\infty
  \quad\text{ and }\quad 
  \mfrac{|\B_{b_n}|}{|\A_n|}\,\max\lbrace |\A_n|^{1/2},\text{Var}(J_n)^{1/2}\rbrace\;\xrightarrow{n\to\infty}\; 0\;.
\end{equation}
We first consider specific choices of distributions for $J_n$.
\begin{corollary}\label{lobsterroll}
  Assume the definitions of \cref{MtKatahdin}.\\
  (i) Let ${J_n=j_n}$ be a non-random positive integer
  for each $n$, with ${j_n=O(|\A_n|)}$. Then \eqref{condition:Jn} holds, and
  if the remaining conditions in \eqref{C0}---\eqref{C4} are satisfied,
  \begin{equation*}
    \E[\dTV(W_n,Z(j_n\lambda_n)|\G)]\;\longrightarrow\;0\qquad\text{ as }n\rightarrow\infty\;.
  \end{equation*}
  (ii) Let ${J_n\sim\text{Poisson}(\theta_n)}$, where ${\theta_n=O(|\A_n|)}$.
  Then \eqref{condition:Jn} holds, and
  if the remaining conditions in \eqref{C1}---\eqref{C4} are satisfied,
  \begin{equation*}
    \E[\dTV(W_n,Z(\lfloor\theta_n\rfloor\lambda_n)|\G)]\;\longrightarrow\;0\qquad\text{ as }n\rightarrow\infty\;.
  \end{equation*}
\end{corollary}
Now consider a more specific setup: Let $\G$ be the lcscH group $(\mathbb{R}^2,+)$. This group is amenable,
and ${\A_n=[0,n]^2}$ is a F{\o}lner sequence. Fix a $\sigma$-finite measure $\nu$ on $\mathbb{R}^2$, and for each $n$,
draw
\begin{equation*}
  \phi_{n1},\phi_{n2},\ldots\sim_{\textrm\tiny iid}\frac{\nu(\,\cdot\,\cap\A_n)}{\nu(\A_n)}
  \quad\text{ and independently }\quad
  J_n\sim\text{Poisson}(\nu(\A_n))\;.
\end{equation*}
The random measure ${\zeta_n:=\sum_{i\leq J_n}\delta_{\phi_{ni}}}$ is a Poisson process on $\A_n$ with ${\E[\zeta_n]=\nu(\,\cdot\,\cap\A_n)}$.
\begin{corollary}\label{acadia}
  For each $n$, let ${X_n:\mathbb{R}^2\rightarrow\{0,1\}}$ be a random function, and
  \begin{equation*}
    W_n\;:=\;\zeta_n(X_n)\;=\;\msum_{i\leq J_n}X_n(\phi_{ni})
  \end{equation*}
  its integral under $\zeta_n$.
  Then \eqref{condition:Jn} holds. 
  Choose a sequence $(b_n)$ of positive integers, and set
  \begin{equation*}
    \lambda_n(k)\;=\;\prob\Bigl(X_n(\phi_{n1})=1\text{ and }\msum_{i\leq\nu(\A_n),\|\phi_{n1}-\phi_{ni}\|\leq b_n}X_n(\phi_{ni})=k\Bigr)\;.
  \end{equation*}
  If the measures ${{\nu(\,\cdot\,\cap\A_n)}/{\nu(\A_n)}}$ are well-spread, and if the remaining conditions in
  \eqref{C0}---\eqref{C4} hold, then
  \begin{equation*}
    \dTV(\zeta_n(X_n),Z(\lambda_n))\;\longrightarrow\;0\qquad\text{ as }n\rightarrow\infty\;.
  \end{equation*}
\end{corollary}

\begin{acks}[Acknowledgments]
PO was supported by the Gatsby Charitable Foundation.
\end{acks}

\bibliography{references}
\bibliographystyle{abbrvnat}

\newpage
%%%%%%%%%%%%%%%%%%%%%%%%%%%%%%%%%%%%%%%%%%%%%%
%% Example with single Appendix:            %%
%%%%%%%%%%%%%%%%%%%%%%%%%%%%%%%%%%%%%%%%%%%%%%
\appendix
\begin{appendix}
\section*{Proofs}\label{appn} %% if no title is needed, leave empty \section*{}.

The proofs are organized as follows: In \cref{proof_main_thm}, we present the proof for the main result \cref{result1}, followed by the proofs of \cref{result1_cor} and the corresponding corollaries discussed in \cref{app1,app2,app3,app3.5}. In \cref{proof_gen_thm}, we provide the proof for the generalization result \cref{MtKatahdin}, followed by the proofs of the corollaries presented in \cref{app4}.

\subsection{Proof of Theorem \ref{result1}}\label{proof_main_thm}

\subsubsection{Notation}
In the subsequent subsection, we introduce several notations that we consistently use throughout the remainder of the section. \begin{itemize}\item 
Denote $Q_n(\phi) = \prob(f_n(\phi X_n)= 1| \mathbb{G})$. 
\item Recall that in \cref{momentcond}, we define for all $p\ge 1$ and all $b>0$, 
\begin{equation*}
  \mu_{n,p}\;:=\;\msum_{\phi\in \A_n} \|Q_n(\phi)\|_{\frac{p}{p-1}}\;;
\end{equation*}
\begin{equation*}
  \eta_{n,p}:=\Big({|\A_n|^{p-1}}\msum_{\phi\in\A_n}\|Q_n(\phi)\|^p_1\Big)^{1/p}\;;
\end{equation*}
\begin{equation*}
  \gamma_n(b)\;:=\;\mfrac{|\A_n|}{|\B_{b}|}\msum_{\phi\in\A_n}\msum_{\psi\in\B_b(\phi)} \|Q_n(\phi)Q_n(\psi)\|_1\;.
\end{equation*}
\item Recall that in \cref{sec:mainresult}, we define the residue terms
\begin{equation*}
  \mathcal{R}_\Psi(n,p,b)\;:=\;\msum_{i \ge b}|\mathbf{B}_{i+1}\backslash\mathbf{B}_{i}|\; \Psi_{n,p}(\,i\,| \G);
\end{equation*}
\begin{equation*}
  \mathcal{R}_\xi(n,p,b)\;:=\;\msum_{i \ge 2b}|\mathbf{B}_{i+1}\backslash\mathbf{B}_{i}|\; \xi_{n,p}^b(\,i\,| \G).
\end{equation*}

\item For all $b>0$ and $\phi\in \G$, we denote the events $$E^{n,k}_{b,\phi}:=\Big\{\sum_{\substack{\phi^{\prime}\in \A_n \\ d(\phi, \phi^{\prime}) \leq b}} f_{n}(\phi^\prime X_{n})=k\Big\};$$$$\tilde E^{n,k}_{b,\phi}:=\Big\{\sum_{\substack{\phi^{\prime}\in \G \\ d(\phi, \phi^{\prime}) \leq b}} f_{n}(\phi^\prime X_{n})=k\Big\}.$$
\item For all $b>0$, we write ${\lambda^b_n:\mathbb{N}\rightarrow[0,\infty),}$ where
$$\lambda^{b}_{n}(k) := \frac{1}{k}\sum_{\phi\in \A_n}\mathbb{E}\Big[f_{n}(\phi X_{n}) \mathbb{I}\big(\tilde E^{n,k}_{b,\phi}\big)\Big|\G\Big].$$
Note that in the article, we always choose a sequence of positive integers $(b_n)$ and abbreviate $\lambda_n := \lambda^{b_n}_n$. In the proofs, we keep the superscript $b$ to differentiate $\lambda^{b}_n$'s with different $b$'s. %In the proof, we sometimes choose difference sequences $(b_n)$ and $(c_n)$
%\item We shorthand $\lambda^{b,n}:=(\lambda^{b,n}_k)_{k\in\mathbb{N}}$
  \item For all $\phi\in \G$ and all integer $j\in \mathbb{N}$, we denote $W_{j}^{\phi,n}:=\sum\limits_{\substack{\phi'\in \A_n\\d(\phi,\phi')>j}}f_n(\phi' X_n)$. For clarity, we reiterate the definition $W_n:=\sum\limits_{\substack{\phi'\in \A_n}}f_n(\phi' X_n)$.

  \item  
  For all $b>0$ and for each $\phi \in \G$, we denote $\B_b^{\A_n}(\phi):=\B_b(\phi)\bigcap \A_n$. 
\end{itemize}

\subsubsection{Auxiliary results}
In this subsection, we present some preliminary results that will be useful to the proof of \cref{result1}. 

\begin{comment}

The first lemma proves that a group with a left-invariant metric also has a right-invariant metric, as we remarked in \cref{sec:amenability}. An immediate consequence of the lemma is that any lcscH group, which possesses a left-invariant metric, also has a right-invariant metric.

\begin{lemma}\label{rightinv}
Suppose the group $\G$ has a left-invariant metric $d_l(\cdot,\cdot)$. Then, there exists a right-invariant metric on $\G$.
\end{lemma}

\begin{proof}
Define the function $d_r(\cdot,\cdot):\G\times\G\to\R$ as follows:
$$d_r(\phi_1,\phi_2)=d_l(\phi^{-1}_1, \phi^{-1}_2), \quad \text{for all $\phi_1,\phi_2\in\G$.}$$
We observe that $d_r(\cdot,\cdot)$ satisfies the axioms for a metric. Furthermore, for any $\phi, \phi_1, \phi_2\in\G$, we have
$$d_r(\phi_1\phi^{-1},\phi_2\phi^{-1}) = d_l(\phi\phi_1^{-1},\phi\phi_2^{-1}) = d_l(\phi_1^{-1},\phi_2^{-1}) = d_r(\phi_1,\phi_2).$$
Hence, $d_r(\cdot,\cdot)$ is a right-invariant metric on $\G$.
\end{proof}  
\end{comment}

The following lemma states that if $\Big\|\sum_{\phi\in \A_n}\mathbb{E}\big[f_{n}(\phi X_{n})|\G\big]\Big\|_\infty$ is uniformly bounded for all $n\in\N$, then $\|H_{0}(\lambda^{b_n}_{n})\|_\infty$ and $ \|H_{1}(\lambda^{b_n}_{n})\|_\infty$ are also uniformly bounded for all $n\in\N$. 
\begin{lemma}\label{cruelsummer}
Suppose that there exists $\mu\in\mathbb{R}$ such that  
$\sup_{n\in\N}\big\|\sum_{\phi\in \A_n}Q_n(\phi)\big\|_\infty\le \mu. $%$\sup_{n\in\N}\Big\|\sum_{\phi\in \A_n}\mathbb{E}\big[f_{n}(\phi X_{n})|\G\big]\Big\|_\infty\le \mu.$
Then the following holds: 
    $$\|H_{0}(\lambda^{b_n}_{n})\|_\infty, \|H_{1}(\lambda^{b_n}_{n})\|_\infty \le e^\mu, \text{ for all $n\in\N$}.$$
\end{lemma}
\begin{proof}
By \cref{HTVboundnomono}, we have that 
\begin{equation*}
H_{0}(\lambda^{b_n}_{n}), H_{1}(\lambda^{b_n}_{n})\leq \left(\frac{1}{\lambda^{b_n}_{n}(1)} \wedge 1\right) e^{\sum_{k}\lambda^{b_n}_{n}(k)}.
\end{equation*}
We note that 
$$
\begin{aligned}
\sum_{k\ge1}\lambda^{b_n}_{n}(k) \leq \sum_{k \ge 1}k\lambda^{b_n}_n(k) = & \sum_{k\ge 1}\sum_{\phi\in \A_n}\mathbb{E}\Big[f_{n}(\phi X_{n}) \mathbb{I}(\sum_{\substack{\phi^{\prime}\in \G \\ d(\phi, \phi^{\prime}) \leq b_n}} f_{n}(\phi^\prime X_{n})=k)\Big|\G\Big]\\
= & \sum_{\phi\in \A_n}Q_n(\phi).
%= & \sum_{\phi\in \A_n}\mathbb{E}\big[f_{n}(\phi X_{n})|\G\big].
\end{aligned}
$$
Since $\sup_{n\in\N}\big\|\sum_{\phi\in \A_n}Q_n(\phi)\big\|_\infty\le \mu$, we get   $\|H_{0}(\lambda^{b_n}_{n})\|_\infty, \|H_{1}(\lambda^{b_n}_{n})\|_\infty \le e^\mu \text{ for all $n\in\N$}.$
\end{proof}

The next lemma provides a simplification of the $\xi$-mixing coefficient in the ergodic case. 
\begin{lemma}[Ergodic $\xi$-mixing coefficient] When $X_n$ is $\G$-ergodic, the $\xi$-mixing coefficient can be simplified to 
$$
\xi_{n}^{b}(t) = \sup_{(A,B)\in \mathcal{D}_n^{b}(t)}\Big|\prob(B|A) - \prob(B)\Big|.
$$  
\end{lemma}
\begin{proof}
    Recall that the $\xi$-mixing coefficient is defined by
$$
\xi_{n,p}^{b}(t|\G) = \E\Big[\sup_{\substack{(A_k,B_k)\in \mathcal{D}_n^{b}(t)\\(A_k)\text{ are disjoints}}}\quad \Big|\sum_{k\in \N}\frac{\prob(A_k|\G)(\prob(B_k|A_k,\G) - \prob(B_k|\G))}{\prob(\cup_k A_k|\G)}\Big|^p\Big]^{\frac{1}{p}}.
$$
For ergodic $X_n$, we have 
$$
\begin{aligned}
\xi_{n}^{b}(t) & = \sup_{\substack{(A_k,B_k)\in \mathcal{D}_n^{b}(t)\\(A_k)\text{ are disjoints}}}\quad \Big|\sum_{k\in \N}\frac{\prob(A_k)(\prob(B_k|A_k) - \prob(B_k))}{\prob(\cup_k A_k)}\Big|\\
&\leq \sup_{\substack{(A_k,B_k)\in \mathcal{D}_n^{b}(t)\\(A_k)\text{ are disjoints}}}\quad \bigg|\frac{\sum_{k\in \N}\prob(A_k)\sup_{(A,B)\in \mathcal{D}_n^{b}(t)}|\prob(B|A) - \prob(B)|}{\prob(\cup_k A_k)}\bigg|\\
&= \sup_{\substack{(A_k,B_k)\in \mathcal{D}_n^{b}(t)\\(A_k)\text{ are disjoints}}}\quad \bigg|\frac{\prob(\cup_kA_k)\sup_{(A,B)\in \mathcal{D}_n^{b}(t)}|\prob(B|A) - \prob(B)|}{\prob(\cup_k A_k)}\bigg|\\
& =  \sup_{(A,B)\in \mathcal{D}_n^{b}(t)}\Big|\prob(B|A) - \prob(B)\Big|.
\end{aligned}
$$
Moreover, it becomes apparent that
\begin{align*}
\sup_{(A,B)\in \mathcal{D}_n^{b}(t)}\Big|\prob(B|A) - \prob(B)\Big| =& \sup_{(A,B)\in \mathcal{D}_n^{b}(t)}\Big|\frac{\prob(A)(\prob(B|A) - \prob(B))}{\prob(A)}\Big|\\
\leq& \sup_{\substack{(A_k,B_k)\in \mathcal{D}_n^{b}(t)\\(A_k)\text{ are disjoints}}}\quad \Big|\sum_{k\in \N}\frac{\prob(A_k)(\prob(B_k|A_k) - \prob(B_k))}{\prob(\cup_k A_k)}\Big|\\
=&~\xi_{n}^{b}(t).
\end{align*}
Hence it follows that: 
$$
\xi_{n}^{b}(t) = \sup_{(A,B)\in \mathcal{D}_n^{b}(t)}\Big|\prob(B|A) - \prob(B)\Big|.
$$  
\end{proof}

The following lemma shows how $\xi$-mixing  can help control the correlation between key terms.

\begin{lemma}[Conditional $\xi$-mixing bound]\label{lemma26general}
Let $(X_n)_{n\in\N}$ be a sequence of random elements taking values in $\mathcal{X}$. Let $(f_n)_{n\in\N}$ be a sequence of measurable functions, where $f_n: \mathcal{X} \to\{0,1\}$. Let $b>0$ be an integer. Denote $\mathcal{F}:=\mathcal{F}_{\TV}(\lambda^{b}_{n})$. For a fixed $\A_n \subseteq \G$ and $\phi \in \A_n$, we write $$h_{n,k}^\phi= f_n(\phi X_n)\ind(E^{n,k}_{b,\phi}).$$ Then for all $j\ge b$,  the following holds for all $p>1$: 
$$
\begin{aligned}
& \bigg\|\sup_{g\in \mathcal{F}}\Big|\sum_{k=1}^{\infty}\E\Big[\big(h_{n,k}^\phi-\E[h_{n,k}^\phi|\mathbb{G}]\big)\big(g(W^{\phi,n}_{j}+k)-g(W^{\phi,n}_{j+1}+k)\big)\Big|\mathbb{G}\Big]\Big|\bigg\|_1\\
\leq
&2 \Big\|\sup_{g\in \mathcal{F}}|\Delta g|\Big\|_{\infty}{|\B^{\A_n}_{j+1}(\phi)\setminus\B^{\A_n}_{j}(\phi)|}~\|Q_n(\phi)\|_{\frac{p}{p-1}}\xi_{n,p}^{b}(j|\G).
\end{aligned}
$$
\end{lemma}

\begin{proof} Throughout the proof, for the ease of notation, we will denote for all $g\in \mathcal{F}$,  $\Delta_j^k(g):=g(W^{\phi,n}_{j}+k)-g(W^{\phi,n}_{j+1}+k)$. It is worth noting that since $f_n$ takes value in $\{0,1\}$, it follows that $\sum\limits_{\substack{\phi'\in \A_n\\d(\phi,\phi')\le b}}f_n(\phi'X_n)\overset{a.s.}{\le }|\B_b|$. %Hence we know that $h_{n,k}^{\phi}=0$ for all $k>|\B_b|$.

We note that, by the definition of $\Delta_j^k(g)$, we have
$$
\begin{aligned}
&\bigg\|\sup_{g\in \mathcal{F}}\Big|\E\Big[\sum_{k=1}^{\infty}\big(h_{n,k}^\phi-\E[h_{n,k}^\phi|\mathbb{G}]\big)\big(g(W^{\phi,n}_{j}+k)-g(W^{\phi,n}_{j+1}+k)\big)\Big|\mathbb{G}\Big]\Big|\bigg\|_1\\
= &\bigg\|\sup_{g\in \mathcal{F}}\Big|\E\Big[\sum_{k=1}^{\infty}(h_{n,k}^\phi-\E[h_{n,k}^\phi|\mathbb{G}])\Delta_j^k(g)\Big|\mathbb{G}\Big]\Big|\bigg\|_1\\
% \le  &\norm{\sup_{g\in \mathcal{F}}\Big|\E \Big[ \sum_{k=1}^{\infty}
% h_{n,k}^\phi\Delta_j^k|\mathbb{G}\Big]\Big|
% +\Big\sum_{k=1}^{\infty}\E[h_{n,k}^\phi|\mathbb{G}]\E[\Delta_j^k|\mathbb{G}\big]}\\
= &\bigg\|\sup_{g\in \mathcal{F}}\Big|\E\Big[\sum_{k=1}^{\infty}h_{n,k}^\phi(\Delta_j^k(g)-\E[\Delta_j^k(g)|\mathbb{G}])\Big|\mathbb{G}\Big]\Big|\bigg\|_1
\end{aligned}
$$
Let $g\in \mathcal{F}$. We note that as $g$ is a measurable function, so is $\Delta_j^k(g)$. Therefore for, every $\epsilon>0$, there are an integer $N_{\epsilon}^g$, disjoint measurable sets $(A_{i,g}^k)_{i\le N^{g}_{\epsilon}}$ and scalars $(a_{i,k}^g)_{i\le N^{g}_{\epsilon}}$ such that  $|a_{i,k}^g|\leq \norm{\Delta_j^k(g)}_\infty  $ and such that if we define
$$\Delta_{j,k}^{*} (g)= \sum_{i=1}^{N_{\epsilon}^g}a_{i,k}^g\ind\big((W^{\phi,n}_j,W^{\phi,n}_{j+1})\in A_{i,g}^k\big),$$ then the following holds: $$\Big\|\sup_{g\in \mathcal{F}}|\Delta_j^k(g)-\Delta_{j,k}^{*} (g)|\Big\|_1\le \epsilon.$$

\noindent Note that $h_{n,k}^{\phi}=0$ for all $k>|\B_b|$ and $\|h^{\phi}_{n,k}\|_{\infty}\le 1$. Therefore, we have the following:
\begin{align}
    \bigg\|\sup_{g\in \mathcal{F}}\Big|\sum_{k=1}^{\infty} h_{n,k}^{\phi}\Big[\Delta_j^k(g)-\mathbb{E}[\Delta_j^k(g)|\G]-\Big(\Delta_{j,k}^*(g)-\mathbb{E}[\Delta_{j,k}^*(g)|\G]\Big)\Big]\Big|\bigg\|_1\le 2|\B_b|\epsilon.
\end{align}

%and wrote $E^k_i = \E[I^k_i|\mathbb{G}]$. Denote 
 
\noindent Using this, we note that the following holds 
\begin{align*}
& \bigg\|\sup_{g\in \mathcal{F}}\Big|\sum_{k=1}^{\infty}\E\Big[h_{n,k}^\phi(\Delta_{j}^k(g)-\E[\Delta_{j}^k(g)|\mathbb{G}])\Big|\mathbb{G}\Big]\Big|\bigg\|_1
\\\le& \bigg\|\sup_{g\in \mathcal{F}}\Big|\sum_{k=1}^{\infty}\E\Big[h_{n,k}^\phi(\Delta_{j,k}^*(g)-\E[\Delta_{j,k}^*(g)|\mathbb{G}])\Big|\mathbb{G}\Big]\Big|\bigg\|_1+2|\B_b|\epsilon.
\end{align*}
Therefore we focus on bounding the first term on the right-hand side. Finally, we abbreviate
%$$I^k_i:=\ind(C_n)\left(\ind\big((W^{\phi,n}_j,W^{\phi,n}_{j+1})\in A_{i,g}^k\big)-\prob[(W^{\phi,n}_j,W^{\phi,n}_{j+1})\in A_{i,g}^k | \G]\right)$$
{}
$$I^k_{i,g}:=h_{n,k}^{\phi}\Big(\ind\big((W^{\phi,n}_j,W^{\phi,n}_{j+1})\in A_{i,g}^k\big)-\prob\Big((W^{\phi,n}_j,W^{\phi,n}_{j+1})\in A_{i,g}^k \Big| \G\Big)\Big).$$
By exploiting the definition of $\Delta^{*}_{j,k}(\cdot)$ we obtain that 
\begin{align*}
&  \bigg\|\sup_{g\in \mathcal{F}}\Big|\sum_{k=1}^{\infty}\E\Big[h_{n,k}^\phi(\Delta_{j,k}^*(g)-\E[\Delta_{j,k}^*(g)|\mathbb{G}])\Big|\mathbb{G}\Big]\Big|\bigg\|_1\\
%= & \Bigg\|\sup_{g\in \mathcal{F}}\bigg|\E\bigg[\sum_{k=1}^{\infty}h_{n,k}^{\phi}\bigg(\sum_{i=1}^{N_\epsilon^g} a_{i,k}^g\ind\big((W^{\phi,n}_j,W^{\phi,n}_{j+1})\in A_{i,g}^k\big)-\E\Big[\sum_{i=1}^{N_{\epsilon}^g} a_{i,k}^g\ind\big((W^{\phi,n}_j,W^{\phi,n}_{j+1})\in A_{i,g}^k\big)\Big|\mathbb{G}\Big]\bigg)\bigg|\mathbb{G}\bigg]\bigg|\Bigg\|_1\\
\leq & \Bigg\|\sup_{g\in \mathcal{F}}\bigg|\sum_{i=1}^{N_{\epsilon}^g} \sum_{k=1}^{\infty}a_{i,k}^g\E\Big[h_{n,k}^\phi\Big(\ind((W^{\phi,n}_j,W^{\phi,n}_{j+1})\in A_{i,g}^k)\\
&\qquad\qquad\qquad\qquad\qquad\qquad-\prob\Big((W^{\phi,n}_j,W^{\phi,n}_{j+1})\in A_{i,g}^k \Big|\G\Big)\Big)\Big|\mathbb{G}\Big]\bigg|\Bigg\|_1\\
% \leq & \Bigg\|\sup_{g\in \mathcal{F}}\bigg|\sum_{i=1}^{N_{\epsilon}^g} \sum_{k=1}^{\infty}a_{i,k}^g\E\Big[h_{n,k}^\phi\Big(\ind((W^{\phi,n}_j,W^{\phi,n}_{j+1})\in A_{i,g}^k)-\prob\Big((W^{\phi,n}_j,W^{\phi,n}_{j+1})\in A_{i,g}^k\Big| \G\Big)\Big)\Big|\mathbb{G}\Big]\bigg|\Bigg\|_1+2|\B_b|\epsilon
% \\
\leq & \bigg\|\sup_{g\in \mathcal{F}}\Big|\sum_{i=1}^{N_{\epsilon}^g} \sum_{k=1}^{\infty}a_{i,k}^g\E\Big[I_{i,g}^k\Big|\mathbb{G}\Big]\Big|\bigg\|_1
\\\le &\bigg\|\sup_{g\in \mathcal{F}}\sum_{\substack{i\le N_{\epsilon}^g, k\in \mathbb{N}\\\mathbb{E}[I^k_{i,g}|\G]\ge 0}}|a_{i,k}^g|\E\Big[I^k_{i,g}\Big|\mathbb{G}\Big]-\sum_{\substack{i\le N_{\epsilon}^g, k\in \mathbb{N}\\\mathbb{E}[I^k_{i,g}|\G]<0}}|a_{i,k}^g|\E\Big[I^k_{i,g}\Big|\mathbb{G}\Big]\bigg\|_1
\end{align*}Moreover we know that the following fact holds for all  
 $g\in \mathcal{F}$, $$\|\Delta_j^k(g)\|_\infty=\Big\|g(W^{\phi,n}_{j}+k)-g(W^{\phi,n}_{j+1}+k)\Big\|_\infty\leq\norm{\Delta g}_\infty\Big|\mathbf{B}^{\A_n}_{j+1}(\phi)\backslash \mathbf{B}^{\A_n}_j(\phi)\Big|,$$ which directly implies that 
for all $i\le N_{\epsilon}^g$,  we have  $|a_{i,k}^g|\leq \norm{\Delta g}_\infty\Big|\mathbf{B}^{\A_n}_{j+1}(\phi)\backslash \mathbf{B}^{\A _n}_j(\phi)\Big|$.
%  Moreover, (b) is a consequence of the definition of the definition of $\xi^{b_n}_{n,\indty}$.
% Moreover note that 
% The last inequality above holds because
% $$
% \begin{aligned}
% &\sum_{i|F_i>0}\E[\sum_{k=1}^{\abs{B_{b_n}(\phi)}}E_i^k|\mathbb{G}]\\
% &=\sum_{i|F_i>0}\E[\sum_{k=1}^{\abs{B_{b_n}(\phi)}}\ind(C^k_n)\left(\ind((W^{\phi,n}_j,W^{\phi,n}_{j+1})\in A_i)-\prob[(W^{\phi,n}_j,W^{\phi,n}_{j+1})\in A_i|\mathbb{G}]\right)|\mathbb{G}]\\
% & =\E[\sum_{k=1}^{\abs{B_{b_n}(\phi)}}\ind(C^k_n)\left(\ind((W^{\phi,n}_j,W^{\phi,n}_{j+1})\in \bigcup_{i|F_i>0}A_i)-\prob(W^{\phi,n}_j,W^{\phi,n}_{j+1})\in \bigcup_{i|F_i>0}A_i|\mathbb{G})\right)|\mathbb{G}]
% \\
% & = \sum_{k=1}^{\abs{B_{b_n}(\phi)}}\prob(C^k_n,(W^{\phi,n}_j,W^{\phi,n}_{j+1})\in\bigcup_{i|F_i>0}A_i|\mathbb{G})-\prob(C^k_n|\mathbb{G})\prob((W^{\phi,n}_j,W^{\phi,n}_{j+1})\in\bigcup_{i|F_i>0}A_i)|\mathbb{G})\\
% & = \sum_{k=1}^{\abs{B_{b_n}(\phi)}}\prob(C^k_n|\mathbb{G})\left(\prob((W^{\phi,n}_j,W^{\phi,n}_{j+1})\in\bigcup_{i|F_i>0}A_i | C^k_n,\mathbb{G})-\prob((W^{\phi,n}_j,W^{\phi,n}_{j+1})\in\bigcup_{i|F_i>0}A_i|\mathbb{G}) \right)
% \end{aligned}
% $$ 
Hence we obtain that 
\begin{align*}\label{lemma26equigeneraleqn}
&  \bigg\|\sup_{g\in \mathcal{F}}\Big|\sum_{k=1}^{\infty}\E\Big[h_{n,k}^\phi(\Delta_{j,k}^*(g)-\E[\Delta_{j,k}^*(g)|\mathbb{G}])\Big|\mathbb{G}\Big]\Big|\bigg\|_1\\
% %= & \Bigg\|\sup_{g\in \mathcal{F}}\bigg|\E\bigg[\sum_{k=1}^{\infty}h_{n,k}^{\phi}\bigg(\sum_{i=1}^{N_\epsilon^g} a_{i,k}^g\ind\big((W^{\phi,n}_j,W^{\phi,n}_{j+1})\in A_{i,g}^k\big)-\E\Big[\sum_{i=1}^{N_{\epsilon}^g} a_{i,k}^g\ind\big((W^{\phi,n}_j,W^{\phi,n}_{j+1})\in A_{i,g}^k\big)\Big|\mathbb{G}\Big]\bigg)\bigg|\mathbb{G}\bigg]\bigg|\Bigg\|_1\\
% \leq & \Bigg\|\sup_{g\in \mathcal{F}}\bigg|\sum_{i=1}^{N_{\epsilon}^g} \sum_{k=1}^{\infty}a_{i,k}^g\E\Big[h_{n,k}^\phi\Big(\ind((W^{\phi,n}_j,W^{\phi,n}_{j+1})\in A_{i,g}^k)-\prob\Big((W^{\phi,n}_j,W^{\phi,n}_{j+1})\in A_{i,g}^k \Big|\G\Big)\Big)\Big|\mathbb{G}\Big]\bigg|\Bigg\|_1\\
% % \leq & \Bigg\|\sup_{g\in \mathcal{F}}\bigg|\sum_{i=1}^{N_{\epsilon}^g} \sum_{k=1}^{\infty}a_{i,k}^g\E\Big[h_{n,k}^\phi\Big(\ind((W^{\phi,n}_j,W^{\phi,n}_{j+1})\in A_{i,g}^k)-\prob\Big((W^{\phi,n}_j,W^{\phi,n}_{j+1})\in A_{i,g}^k\Big| \G\Big)\Big)\Big|\mathbb{G}\Big]\bigg|\Bigg\|_1+2|\B_b|\epsilon
% % \\
% \leq & \bigg\|\sup_{g\in \mathcal{F}}\Big|\sum_{i=1}^{N_{\epsilon}^g} \sum_{k=1}^{\infty}a_{i,k}^g\E\Big[I_{i,g}^k|\mathbb{G}\Big]\Big|\bigg\|_1+2|\B_b|\epsilon
% \\\le &\bigg\|\sup_{g\in \mathcal{F}}\sum_{\substack{i\le N_{\epsilon}^g, k\in \mathbb{N}\\\mathbb{E}[I^k_{i,g}|\G]\ge 0}}|a_{i,k}^g|\E\Big[I^k_{i,g}|\mathbb{G}\Big]-\sum_{\substack{i\le N_{\epsilon}^g, k\in \mathbb{N}\\\mathbb{E}[I^k_{i,g}|\G]<0}}|a_{i,k}^g|\E\Big[I^k_{i,g}|\mathbb{G}\Big]\bigg\|_1
\\\overset{}{\le} &\Big\|\sup_{g\in \mathcal{F}}|\Delta g|\Big\|_{\infty}\Big|\B_{j+1}^{\A_n}(\phi)\setminus \B_{j}^{\A_n}(\phi)\Big|\\
&\times\Big(\Big\|\sup_{g\in \mathcal{F}}\sum_{\substack{i\le N_{\epsilon}^g, k\in \mathbb{N}\\\mathbb{E}[I^k_{i,g}|\G]\ge 0}}\E\Big[I^k_{i,g}\Big|\mathbb{G}\Big]\Big\|_1+\Big\|\sup_{g\in \mathcal{F}}\sum_{\substack{i\le N_{\epsilon}^g, k\in \mathbb{N}\\\mathbb{E}[I^k_{i,g}|\G]<0}}\E\Big[I^k_{i,g}\Big|\mathbb{G}\Big]\Big\|_1\Big)
\\\overset{}{\le} &\Big\|\sup_{g\in \mathcal{F}}|\Delta g|\Big\|_{\infty}\Big|\B_{j+1}^{\A_n}(\phi)\setminus \B_{j}^{\A_n}(\phi)\Big|\\
&\times\Big(\Big\|\sup_{g\in \mathcal{F}}\sum_{\substack{ k\in \mathbb{N}}}\E\Big[h_{n,k}^\phi\Big(\ind(\bigcup_{\substack{ i\le N_\epsilon^g\\\mathbb{E}[I^k_{i,g}|\G]\ge 0}}A_{i,g}^k)-\mathbb{E}\Big[\ind(\bigcup_{\substack{ i\le N_\epsilon^g\\\mathbb{E}[I^k_{i,g}|\G]\ge 0}}A_{i,g}^k)\Big|\G\Big]\Big)\Big|\mathbb{G}\Big]\Big\|_1
\\&~~~ +\Big\|\sup_{g\in \mathcal{F}}\sum_{\substack{ k\in \mathbb{N}}}\E\Big[h_{n,k}^\phi\Big(\ind(\bigcup_{\substack{ i\le N_\epsilon^g\\\mathbb{E}[I^k_{i,g}|\G]< 0}}A_{i,g}^k)-\mathbb{E}\Big[\ind(\bigcup_{\substack{ i\le N_\epsilon^g\\\mathbb{E}[I^k_{i,g}|\G]< 0}}A_{i,g}^k)\Big|\G\Big]\Big)\Big|\mathbb{G}\Big]\Big\|_1\Big).
%\leq & \sum_{i=1}^N |a_i|{ }\norm{\sum_{k=1}^{\abs{B_{b_n}(\phi)}}E_i^k}_1\\
% = & \E[\sum_{i|F_i>0}|a_i|{ }\abs{\sum_{k=1}^{\abs{B_{b_n}(\phi)}}E_i^k}+\sum_{i|F_i<0}|a_i|{ }\abs{\sum_{k=1}^{\abs{B_{b_n}(\phi)}}E_i^k}]\\
% \leq & \sum_{i|F_i>0}|a_{i,k}^g|\E[\sum_{k=1}^{\abs{B_{b_n}(\phi)}}E_i^k]-\sum_{i|F_i<0}|a_{i,k}^g|\E[\sum_{k=1}^{\abs{B_{b_n}(\phi)}}E_i^k]\\
% \leq & \max_{i,k}{|a_{i,k}^g|}\left(\norm{\sum_{i|F_i>0}\E[\sum_{k=1}^{\abs{B_{b_n}(\phi)}}E_i^k]}_1+\norm{\sum_{i|F_i<0}\E[\sum_{k=1}^{\abs{B_{b_n}(\phi)}}E_i^k]}_1\right)\\
% 
\end{align*}
Moreover by H\"{o}lder's inequality and the definition of the conditional $\xi$-mixing coefficient, we have
\begin{align*}&
\Big\| \sup_{g\in \mathcal{F}}   \sum_{\substack{ k\in \mathbb{N}}}\E\Big[h_{n,k}^\phi\Big(\ind(\bigcup_{\substack{ i\le N_\epsilon^g\\\mathbb{E}[I^k_{i,g}|\G]\ge 0}}A_{i,g}^k)-\mathbb{E}\Big[\ind(\bigcup_{\substack{ i\le N_\epsilon^g\\\mathbb{E}[I^k_{i,g}|\G]\ge 0}}A_{i,g}^k)\Big|\G\Big]\Big)\Big|\mathbb{G}\Big]\Big\|_1
%\\&\le \Big\|    \sup_{g\in \mathcal{F}}   \sum_{\substack{ k\in \mathbb{N}}} \E\Big[h_{n,k}^\phi\Big[\ind(\bigcup_{\substack{ i\le N_\epsilon^g\\\mathbb{E}(I^k_{i,g}|\G)\ge 0}}A_{i,g}^k)-\mathbb{E}\Big(\ind(\bigcup_{\substack{ i\le N_\epsilon^g\\\mathbb{E}(I^k_{i,g}|\G)\ge 0}}A_{i,g}^k)|\G\Big)\Big]|\mathbb{G}\Big]\Big\|_1
\\\le&\Big\|  \sup_{g\in \mathcal{F}}   \sum_{\substack{ k\in \mathbb{N}}} \mathbb{P}(h_{n,k}^\phi=1|\G)  \Big(\mathbb{P}\Big(\bigcup_{\substack{ i\le N_\epsilon^g\\\mathbb{E}[I^k_{i,g}|\G]\ge 0}}A_{i,g}^k\Big|\G,h_{n,k}^{\phi}=1\Big)-\mathbb{P}\Big(\bigcup_{\substack{ i\le N_\epsilon^g\\\mathbb{E}[I^k_{i,g}|\G]\ge 0}}A_{i,g}^k\Big|\G\Big)\Big)\Big\|_1
%\\&\le\sum_{\substack{ k\in \mathbb{N}}}\Big\|  \sup_{g\in \mathcal{F}}    \mathbb{P}[h_{n,k}^\phi=1|\G]  \Big[\mathbb{P}\Big[\bigcup_{\substack{ i\le N_\epsilon^g\\\mathbb{E}(I^k_{i,g}|\G)\ge 0}}A_{i,g}^k\Big|\G,h_{n,k}^{\phi}=1\Big]-\mathbb{P}\Big[\bigcup_{\substack{ i\le N_\epsilon^g\\\mathbb{E}(I^k_{i,g}|\G)\ge 0}}A_{i,g}^k|\G\Big]\Big]\Big\|_1
\\\overset{(b)}{\le}& \xi_{n,p}^{b}(j|\G)\|Q_n(\phi)\|_{\frac{p}{p-1}},
\end{align*} 
where the inequality $(b)$ holds because of the following: we note that the events $\{h_{n,k}^{\phi}=1\}$ are disjoints and that $\bigcup_{k\in\N}\{h_{n,k}^{\phi}=1\} = \{f_n(\phi X_n) = 1\}$, so we have
\begin{align*}
&\bigg\|  \sup_{g\in \mathcal{F}} \sum_{\substack{ k\in \mathbb{N}}}   \mathbb{P}(h_{n,k}^\phi=1|\G)  \Big(\mathbb{P}\Big(\bigcup_{\substack{ i\le N_\epsilon^g\\\mathbb{E}[I^k_{i,g}|\G]\ge 0}}A_{i,g}^k\Big|\G,h_{n,k}^{\phi}=1\Big)-\mathbb{P}\Big(\bigcup_{\substack{ i\le N_\epsilon^g\\\mathbb{E}[I^k_{i,g}|\G]\ge 0}}A_{i,g}^k\Big|\G\Big)\Big)\bigg\|_1\\
{=} & \bigg\| \sup_{g\in \mathcal{F}}\frac{\mathbb{P}(f_n(\phi X_n)=1|\G)}{\mathbb{P}(f_n(\phi X_n)=1|\G)}\sum_{\substack{ k\in \mathbb{N}}}   \mathbb{P}(h_{n,k}^\phi=1|\G)\\
& \times \Big(\mathbb{P}\Big(\bigcup_{\substack{ i\le N_\epsilon^g\\\mathbb{E}[I^k_{i,g}|\G]\ge 0}}A_{i,g}^k\Big|\G,h_{n,k}^{\phi}=1\Big) -\mathbb{P}\Big(\bigcup_{\substack{ i\le N_\epsilon^g\\\mathbb{E}[I^k_{i,g}|\G]\ge 0}}A_{i,g}^k|\G\Big)\Big)\bigg\|_1\\
\overset{(c)}{\leq} & \Big\|\prob(f_n(\phi X_n) = 1|\G)\Big\|_{\frac{p}{p-1}}\xi_{n,p}^{b}(j|\mathbb{G}),
\end{align*}
where to obtain (c) we used H\"{o}lder's inequality and the definition of $\xi^b_{n,p}$-mixing coefficient. 
Therefore, this implies that

\begin{equation}\label{lemma26equigeneraleqn}
\begin{aligned}
& \bigg\|\sup_{g\in \mathcal{F}}\Big|\sum_{k=1}^{\infty}\E\Big[h_{n,k}^\phi(\Delta_{j}^k(g)-\E[\Delta_{j}^k(g)|\mathbb{G}])\Big|\mathbb{G}\Big]\Big|\bigg\|_1
\\ \le&~ 2 \Big\|\sup_{g\in \mathcal{F}}|\Delta g|\Big\|_{\infty}|\B^{\A_n}_{j+1}(\phi)\setminus\B^{\A_n}_{j}(\phi)|\xi_{n,p}^{b}(j|\G)\|Q_n(\phi)\|_{\frac{p}{p-1}} +2|\B_b|\epsilon.
\end{aligned}
\end{equation}
As this result holds for arbitrary $\epsilon>0$, we obtain the desired result:
\begin{equation}\label{lemma26equigeneraleqn}
\begin{aligned}
& \bigg\|\sup_{g\in \mathcal{F}}\Big|\sum_{k=1}^{\infty}\E\Big[h_{n,k}^\phi(\Delta_{j}^k(g)-\E[\Delta_{j}^k(g)|\mathbb{G}])\Big|\mathbb{G}\Big]\Big|\bigg\|_1
\\\le&~ 2 \Big\|\sup_{g\in \mathcal{F}}|\Delta g|\Big\|_{\infty}|\B^{\A_n}_{j+1}(\phi)\setminus\B^{\A_n}_{j}(\phi)|\xi_{n,p}^{b}(j|\G)\|Q_n(\phi)\|_{\frac{p}{p-1}} .
\end{aligned}
\end{equation}

\end{proof}

For certain $\phi\in \A_n$ the set $\B^{\A_n}_b(\phi)$ will be a strict subset of  $\B_b(\phi)$. The following lemma allows us to handle this.

\begin{lemma} \label{folnerlemma}
Let $\G$ be an 
lcscH group. Let $(\A_n)_{n\in\N}$ be a F\o lner sequence for $\G$. Set $(c_{\phi})_{\phi\in \G}$ to be a sequence of positive reals indexed by $\G$. Then for all $b>0$ and all $p\ge 1$ we have
$$\frac{1}{\abs{\A_n}}\sum_{\phi \in \A_n}\abs{\B_{b}(\phi)\setminus\mathbf{B}^{\A_n}_{b}(\phi)}c_{\phi}\leq \Big( \frac{1}{|\A_n|}\sum_{\phi\in\A_n}c_{\phi}^{p}\Big)^{1/p}\Big(\frac{|\A_n\B_b\triangle\A_n|}{|\A_n|}\Big)^{(p-1)/p}.$$
\end{lemma}
\begin{proof}
Note that, since the metric $d$ is left-invariant, we know that for all $\phi'\in \G$, the distance $d(\phi,\phi')\le b$ if and only if $d(e,\phi^{-1}\phi')\le b$. Therefore, it follows that 
$$
\begin{aligned}
\B_{b}(\phi) & = \{\phi^\prime \in \mathbb{G}:d(\phi,\phi^\prime)\leq b\}\\
& = \{\phi^\prime \in \mathbb{G}:d(e,\phi^{-1}\phi^\prime)\leq b\}\\
& = \phi\{\phi^{-1}\phi^\prime \in \mathbb{G}:d(e,\phi^{-1}\phi^\prime)\leq b\}\\
& = \phi\B_{b},
\end{aligned}
$$
where, as a reminder, we have abbreviated $\B_{b}:= \B_{b}(e)$. This implies that 
$\mathbf{B}^{\A_n}_{b}(\phi)= \B_{b}(\phi)\cap \A_n = \phi\B_{b} \cap \A_n$. 
Therefore we obtain that
$$
\begin{aligned}
\sum_{\phi \in \A_n}\abs{\mathbf{B}^{\A_n}_{b}(\phi)}c_\phi &= \sum_{\phi \in \A_n}\abs{\phi\B_{b} \cap \A_n}c_{\phi}\\
&= \sum_{\phi \in \A_n}\sum_{\eta \in \B_{b}}\ind(\phi\eta \in \A_n)c_{\phi}.
\end{aligned}
$$
Therefore, by exploiting H\"{o}lder's inequality, we have
\begin{align*}
&\frac{1}{\abs{\A_n}}\sum_{\phi \in \A_n}\abs{\B_{b}(\phi)\backslash\mathbf{B}^{\A_n}_{b}(\phi)}c_{\phi} \\
=& \frac{1}{\abs{\A_n}}\sum_{\phi \in \A_n}\abs{\B_{b}(\phi)}c_{\phi}-\frac{1}{\abs{\A_n}}\sum_{\phi \in \A_n}\abs{\mathbf{B}^{\A_n}_{b}(\phi)}c_{\phi}
\\=&\frac{1}{|\A_n|}\sum_{\phi\in \A_n}\sum_{\eta\in \B_b}\mathbb{I}(\phi\eta\not \in \A_n)c_{\phi}
\\\le&\Big( \frac{1}{|\A_n|}\sum_{\phi\in\A_n}c_{\phi}^{p}\Big)^{1/p}\Big(\frac{1}{|\A_n|}\sum_{\phi\in \A_n}\sum_{\eta\in \B_b}\mathbb{I}(\phi\eta\not \in \A_n)\Big)^{(p-1)/p}
\\\le&\Big( \frac{1}{|\A_n|}\sum_{\phi\in\A_n}c_{\phi}^{p}\Big)^{1/p}\Big(\frac{|\A_n\B_b\triangle\A_n|}{|\A_n|}\Big)^{(p-1)/p}.
\end{align*}
Hence the result follows.
In addition, we remark that
\begin{align*}
|\A_n\B_b\triangle\A_n| =& \sum_{\phi\in \A_n}\sum_{\eta\in \B_b}\mathbb{I}(\phi\eta\not \in \A_n)+\sum_{\phi\in \A_n}\mathbb{I}(\phi \not \in \A_n\B_b)\\
=& \sum_{\phi\in \A_n}\sum_{\eta\in \B_b}\mathbb{I}(\eta^{-1}\phi^{-1}\not \in \A_n^{-1})+\sum_{\phi\in \A_n}\mathbb{I}(\phi^{-1} \not \in \B_b^{-1}\A_n^{-1})\\
=& \sum_{\phi\in \A_n^{-1}}\sum_{\eta\in \B_b^{-1}}\mathbb{I}(\eta\phi\not \in \A_n^{-1})\sum_{\phi\in \A_n}\mathbb{I}(\phi \not \in \B_b^{-1}\A_n^{-1})\\
\le& |\B_b^{-1}\A_n^{-1}\triangle\A_n^{-1}|.
\end{align*}
\end{proof}

In the following lemma, we demonstrate that for any sequence $(b_n)_{n\in\N}$ fulfilling the conditions \eqref{A1}--\eqref{A4}, the limit of $Z(\lambda_{n}^{b_n})$ remains invariant as $n$ approaches infinity.
\begin{lemma}[Well-definedness of the limit distribution] Suppose \eqref{A1}--\eqref{A4} are all satisfied. 
Define $$\mathcal{S} := \{(a_n)_{n\in\N}: (a_n)_{n\in\N} \text{ satisfies conditions \eqref{A1}--\eqref{A4}}\}.$$
For all $(b_n)_{n\in\N},(c_n)_{n\in\N}\in\mathcal{S}$, the following holds: 
$$\E\Big[d_{\TV}(Z(\lambda_{n}^{b_n}),Z(\lambda_{n}^{c_n})|\G)\Big]\xrightarrow{n\to\infty} 0.$$
\end{lemma}
\begin{proof}
For the ease of notation, we denote $Z_n\sim Z(\lambda_{n}^{b_n})$ and $Z_n^\prime \sim Z(\lambda_{n}^{c_n})$.
The Stein equation for compound Poisson random variables \cite{Barbour1992CPStein} implies that  
$$\E[Z_ng(Z_n)]=\sum_{k=1}^{\infty} k \lambda^{b_n}_n(k)\E[g(Z_n+k)]$$
Therefore, using \cref{SteinCPboundTV}, we have
\begin{align*}
& \Big\|d_{\TV}(Z_n,Z_n^\prime|\G)\Big\|_1 \\
\leq & \bigg\|\sup_{g\in\mathcal{F}_\TV(\lambda_{n}^{c_n})}\Big|\E\Big[Z_ng(Z_n)-\sum_{k=1}^{\infty} k \lambda^{c_n}_n(k) g(Z_n+k)\Big|\G\Big]\Big|\bigg\|_1\\
= & \bigg\|\sup_{g\in\mathcal{F}_\TV(\lambda_{n}^{c_n})}\Big|\sum_{k=1}^{\infty} k \lambda^{b_n}_n(k)\E[g(Z_n+k)|\G]-\sum_{k=1}^{\infty} k \lambda^{c_n}_n(k) \E[g(Z_n+k)|\G]\Big|\bigg\|_1\\
%& \leq \sup_{g\in\mathcal{F}_{\TV}(\lambda^{c_n})}\|g\|_\infty\Big\|\sum_{k=1}^{\infty} k (\lambda^{b_n}_n(k)- \lambda^{c_n}_n(k))\Big\|_1\\
\le &  \Big\|\sup_{g\in\mathcal{F}_{\TV}(\lambda_{n}^{c_n})}|g|\Big\|_\infty\Big\|\sum_{k=1}^{\infty} k (\lambda^{b_n}_n(k)- \lambda^{c_n}_n(k))\Big\|_1
\end{align*}
By \cref{cruelsummer}, we know that $\Big\|\sup_{g\in\mathcal{F}_{\TV}(\lambda_{n}^{c_n})}|g|\Big\|_\infty$ is bounded as \eqref{A1} holds. Hence it suffices to prove $\Big\|\sum_{k=1}^{\infty} k (\lambda^{b_n}_n(k)- \lambda^{c_n}_n(k))\Big\|_1\to0$ as $n\to 0$.
We note that by the definition of $k\lambda^{b_n}_{n}(k)$ and $k\lambda^{c_n}_{n}(k)$, we have 
\begin{equation*}
\begin{aligned}
\sum_{k=1}^{\infty}k(\lambda^{b_n}_n(k)- \lambda^{c_n}_n(k)) &= \sum_{\phi\in{\A_n}}\sum_{k=1}^{\infty}\bigg(\mathbb{E}\Big[f_{n}(\phi X_{n}) \mathbb{I}\Big(\sum_{\substack{\phi^{\prime}\in{\G} \\ d(\phi, \phi^{\prime}) \leq b_n}} f_{n}(\phi^\prime X_{n})=k\Big)\Big|\G\Big]
\\& \quad\quad\quad\quad\quad-\mathbb{E}\Big[f_{n}(\phi X_{n}) \mathbb{I}\Big(\sum_{\substack{\phi^{\prime}\in{\G} \\ d(\phi, \phi^{\prime}) \leq c_n}} f_{n}(\phi^\prime X_{n})=k\Big)\Big|\G\Big]\bigg)\\
& = \sum_{\phi\in{\A_n}}\mathbb{E}\Big[f_{n}(\phi X_{n})\sum_{k=1}^{\infty}\Big( \mathbb{I}\big(\tilde E_{b_n,\phi}^{n,k}\big)-\mathbb{I}\big(\tilde E_{c_n,\phi}^{n,k}\big)\Big)\bigg|\G\Big].
\end{aligned}
\end{equation*}
Assume, without loss of generality, that $b_n>c_n$. We obtain 
\begin{align*}
& \Big\|\sum_{k=1}^{\infty} k (\lambda^{b_n}_n(k)- \lambda^{c_n}_n(k))\Big\|_1\\
\leq %&\sum_{\phi\in{\A_n}}\bigg\|\mathbb{E}\bigg[f_{n}(\phi X_{n})\sum_{k=1}^{\infty}\bigg( \mathbb{I}\Big(\sum_{\substack{\phi^{\prime}\in{\G} \\ d(\phi, \phi^{\prime}) \leq b_n}} f_{n}(\phi^\prime X_{n})=k\Big)-\mathbb{I}\Big(\sum_{\substack{\phi^{\prime}\in{\G} \\ d(\phi, \phi^{\prime}) \leq c_n}} f_{n}(\phi^\prime X_{n})=k\Big)\bigg)\bigg|\G\bigg]\bigg\|_1\\
%\overset{(a)}{=}& \sum_{\phi\in{\A_n}}\mathbb{E}\bigg[f_{n}(\phi X_{n})\sum_{k=1}^{\infty}\bigg|\mathbb{I}\Big(\sum_{\substack{\phi^{\prime} \in \G\\ d(\phi, \phi^{\prime}) \leq b_n}} f_{n}(\phi^\prime X_{n})=k\Big)-\mathbb{I}\Big(\sum_{\substack{\phi^{\prime} \in \G\\ d(\phi, \phi^{\prime}) \leq c_n}} f_{n}(\phi^\prime X_{n})=k\Big)\bigg|\bigg]\\
& \sum_{\phi\in{\A_n}}\bigg\|\mathbb{E}\Big[f_{n}(\phi X_{n})\mathbb{I}\Big(\sum_{\substack{\phi^{\prime} \in \G\\ c_n< d(\phi, \phi^{\prime}) \leq b_n}} f_{n}(\phi^\prime X_{n})\neq 0\Big)\Big|\G\Big]\bigg\|_1\\
=& \sum_{\phi\in{\A_n}}\bigg\|\mathbb{E}\Big[f_{n}(\phi X_{n})\mathbb{I}\Big(\sum_{\substack{\phi^{\prime} \in \G\\ c_n< d(\phi, \phi^{\prime}) \leq b_n}} f_{n}(\phi^\prime X_{n})\neq 0\Big)\Big|\G\Big]\\
&-\mathbb{E}[f_n(\phi X_n)|\G]\E\Big[\mathbb{I}\Big(\sum_{\substack{\phi^{\prime} \in \G\\ c_n< d(\phi, \phi^{\prime}) \leq b_n}} f_{n}(\phi^\prime X_{n})\neq 0\Big)\Big|\G\Big]\bigg\|_1\\
&+\sum_{\phi\in{\A_n}}\bigg\|\mathbb{E}[f_n(\phi X_n)|\G]\E\Big[\mathbb{I}\Big(\sum_{\substack{\phi^{\prime} \in \G\\ c_n< d(\phi, \phi^{\prime}) \leq b_n}} f_{n}(\phi^\prime X_{n})\neq 0\Big)\Big|\G\Big]\bigg\|_1\\
\overset{(a)}{\leq} & \sum_{\phi\in{\A_n}}{\|Q_n(\phi)\|_{\frac{p}{p-1}}}\Psi_{n,p}(c_n|\G)+\sum_{\phi\in{\A_n}}\sum_{\phi^{\prime} \in \B_{b_n}(\phi)\backslash \B_{c_n}(\phi)} \|Q_n(\phi)Q_n(\phi^\prime)\|_1\\
\leq & \mu_{n,p}\Psi_{n,p}(c_n|\G)+\frac{\abs{\B_{b_n}}}{\abs{\A_n}} \gamma_n(b_n)
\overset{(b)}{\longrightarrow}0.
\end{align*}
The first term in $(a)$ follows from H\"{o}lder's inequality and the definition of the $\Psi$-mixing coefficient. The second term in $(a)$ arises from the union bound, which implies that
$$
\begin{aligned}
\prob\Big(\sum_{\substack{\phi^{\prime} \in \G\\ c_n< d(\phi, \phi^{\prime}) \leq b_n}} f_{n}(\phi^\prime X_{n})\neq 0\Big|\G\Big)\le \sum_{\phi^{\prime} \in \B_{b_n}(\phi)\backslash \B_{c_n}(\phi)} Q_n(\phi^\prime).
\end{aligned}
$$
The convergence $(b)$ is because of \eqref{A2} and \eqref{A4}.
\end{proof}

\subsubsection{Proof of Theorem \ref{result1}}
%test bold
Equipped with the lemmas introduced in the preceding section, we are now prepared to present the proof of \cref{result1}. For notational simplicity, we consistently adopt $\mathcal{F} := \mathcal{F}_{\TV}(\lambda^{b_n}_{n})$ in the rest of the proofs in this section. 
\begin{lemma}\label{snow}
Let $(X_n)_{n\in\N}$ be a sequence of random objects taking values in $\mathcal{X}$. Let $(f_n: \mathcal{X} \to\{0,1\})_{n\in\N}$ be a sequence of measurable functions. Let $(b_n)_{n\in\N}$ be a sequence of positive integers. Then the following holds: 
\begin{equation}
\begin{aligned}\label{new_england}
    & \E\Big[\dTV(W_n,Z(\lambda^{b_n}_{n})|\mathbb{G})\Big]\\
    \leq &\Bigg\|\sup_{g\in\mathcal{F}}\bigg|\mathbb{E}\bigg[\sum_{\phi \in \A_n}f_{n}(\phi X_{n}) \sum_{k=1}^{\infty} \mathbb{I}\big(E^{n,k}_{b_n,\phi}\big)\Big(g(W^{\phi,n}_{b_n}+k)-g(W^{\phi,n}_{2b_n}+k)\Big)\bigg|\mathbb{G}\bigg]\bigg|\Bigg\|_1\\ &+ \Bigg\|\sup_{g\in\mathcal{F}}\bigg|\mathbb{E}\bigg[\sum_{\phi \in \A_n}f_{n}(\phi X_{n}) \sum_{k=1}^{\infty} \mathbb{I}\big(E^{n,k}_{b_n,\phi}\big)g(W^{\phi,n}_{2b_n}+k)-\sum_{k=1}^{\infty} k \lambda^{b_n}_n(k) g(W_n+k)\bigg|\mathbb{G}\bigg]\bigg|\Bigg\|_1\\
    :=& (A)+(B).
\end{aligned}
\end{equation}
\end{lemma}
\begin{proof} 
Firstly using \cref{SteinCPboundTV}, we note that 
\begin{equation*}
\norm{\dTV\big(W_n,Z(\lambda^{b_n}_{n})|\mathbb{G}\big)}_1
\le\bigg\| \sup_{g\in \mathcal{F}}\Big|\E\Big[W_n g(W_n)-\sum_{k=1}^{\infty} k \lambda^{b_n}_n(k) g(W_n+k)\Big| \mathbb{G}\Big]\Big|\bigg\|_1.
\end{equation*}
%To simplify notation, we represent $\mathcal{F} := \mathcal{F}(\lambda^{b_n}_{n})$ in the remainder of the proof.
Recall that $W_n := \sum_{\phi \in \A_n}f_n(\phi X_n) $. Therefore, using the definition of $W_n$, we remark that for all $g\in \mathcal{F}$, the following holds:
\begin{align*}
&\E\Big[W_n g(W_n)-\sum_{k=1}^{\infty} k \lambda^{b_n}_n(k) g(W_n+k)\Big| \mathbb{G}\Big] \\
= &\E\Big[\sum_{\phi \in \A_n}f_n(\phi X_n)g(W_n)-\sum_{k=1}^{\infty} k \lambda^{b_n}_n(k) g(W_n+k)\Big|\mathbb{G}\Big]
\\
\overset{(a)}{=} & \E\bigg[\sum_{\phi \in \A_n}f_n(\phi X_n)\sum_{k=1}^{\infty}\mathbb{I}\big(E^{n,k}_{b_n,\phi}\big)g(W_n)-\sum_{k=1}^{\infty} k \lambda^{b_n}_n(k) g(W_n+k)\bigg| \mathbb{G}\bigg]\\
\overset{(b)}{=} & \E\bigg[\sum_{\phi \in \A_n}f_n(\phi X_n)\sum_{k=1}^{\infty}\mathbb{I}\big(E^{n,k}_{b_n,\phi}\big)g(W^{\phi,n}_{b_n}+k)-\sum_{k=1}^{\infty} k \lambda^{b_n}_n(k) g(W_n+k)\bigg| \mathbb{G}\bigg]\\
= & \mathbb{E}\bigg[\sum_{\phi \in \A_n}f_{n}(\phi X_{n}) \sum_{k=1}^{\infty} \mathbb{I}\big(E^{n,k}_{b_n,\phi}\big)\Big(g(W^{\phi,n}_{b_n}+k)-g(W^{\phi,n}_{2b_n}+k)\Big)\bigg|\mathbb{G}\bigg]\\ &+ \mathbb{E}\bigg[\sum_{\phi \in \A_n}f_{n}(\phi X_{n}) \sum_{k=1}^{\infty} \mathbb{I}\big(E^{n,k}_{b_n,\phi}\big)g(W^{\phi,n}_{2b_n}+k)-\sum_{k=1}^{\infty} k \lambda^{b_n}_n(k) g(W_n+k)\bigg|\mathbb{G}\bigg],
\end{align*}
where to obtain $(a)$ we used the fact that $$\bigcup_{k=1}^{\infty}E^{n,k}_{b_n,\phi}=\Big\{\sum\limits_{\substack{\phi^{\prime}\in \A_n \\ d(\phi, \phi^{\prime}) \leq b_n}} f_{n}(\phi^\prime X_{n})\ge 1\Big\}\supset\{f_n(\phi X_n)\ne 0\};$$ 
To get (b) we exploited the fact that under the event $E_{b_n,\phi}^{n,k}$, we have $W_n = W^{\phi,n}_{b_n}+k$.

\noindent Therefore, by the triangle inequality, we obtain that 
\begin{equation*}
\begin{aligned}
&\bigg\|\sup_{g\in \mathcal{F}}\Big|\E\Big[W_n g(W_n)-\sum_{k=1}^{\infty} k \lambda^{b_n}_n(k) g(W_n+k)\Big|\mathbb{G}\Big]\Big|\bigg\|_1 \\\leq& \Bigg\|\sup_{g\in \mathcal{F}}\bigg|\mathbb{E}\bigg[\sum_{\phi \in \A_n}f_{n}(\phi X_{n}) \sum_{k=1}^{\infty} \mathbb{I}\big(E^{n,k}_{b_n,\phi}\big)\Big(g(W^{\phi,n}_{b_n}+k)-g(W^{\phi,n}_{2b_n}+k)\Big)\bigg|\mathbb{G}\bigg]\bigg|\Bigg\|_1\\ +& \Bigg\|\sup_{g\in \mathcal{F}}\bigg|\mathbb{E}\bigg[\sum_{\phi \in \A_n}f_{n}(\phi X_{n}) \sum_{k=1}^{\infty} \mathbb{I}\big(E^{n,k}_{b_n,\phi}\big)g(W^{\phi,n}_{2b_n}+k)-\sum_{k=1}^{\infty} k \lambda^{b_n}_n(k) g(W_n+k)\bigg|\mathbb{G}\bigg]\bigg|\Bigg\|_1.\\
\end{aligned}
\end{equation*}
% Therefore, by Stein's equation, we obtain

% \begin{align*}
% &\norm{\E[d(W_n,Z(\lambda^{b_n}_{n})| \G)]}_1 \\
% \leq & \sup_{g\in\mathcal{F}}\bigg\|\E\Big[W_n g(W_n)-\sum_{k=1}^{\infty} k \lambda^{b_n}_n(k) g(W_n+k)\Big|\mathbb{G}\Big]\bigg\|_1 \tag{by Stein's equation for Compound Poisson}\\% \leq & \sup_{g\in\mathcal{F}}\Bigg\|\mathbb{E}\bigg[\sum_{\phi \in \A_n}f_{n}(\phi X_{n}) \sum_{k=1}^{\infty} \mathbb{I}(E^{n,k}_{b_n,\phi})\Big(g(W^{\phi,n}_{b_n}+k)-g(W^{\phi,n}_{2b_n}+k)\Big)\bigg| \mathbb{G}\bigg]\Bigg\|_1\\ +& \sup_{g\in\mathcal{F}}\Bigg\|\mathbb{E}\bigg[\sum_{\phi \in \A_n}f_{n}(\phi X_{n}) \sum_{k=1}^{\infty} \mathbb{I}(E^{n,k}_{b_n,\phi})g(W^{\phi,n}_{2b_n}+k)-\sum_{k=1}^{\infty} k \lambda^{b_n}_n(k) g(W_n+k)\bigg|\mathbb{G}\bigg]\Bigg\|_1.% \end{align*}

\end{proof}

In the next two lemmas, we successively bound the two terms on the right hand side of \cref{new_england}. 
\begin{lemma}\label{boundA}Suppose that all the conditions of \cref{snow} are satisfied. Then for all $s\ge 1$ the following holds
\begin{equation}
\begin{aligned}
(A)=&\Bigg\|\sup_{g\in\mathcal{F}}\bigg|\mathbb{E}\bigg[\sum_{\phi \in \A_n}f_{n}(\phi X_{n}) \sum_{k=1}^{\infty} \mathbb{I}(E^{n,k}_{b_n,\phi})\Big(g(W^{\phi,n}_{b_n}+k)-g(W^{\phi,n}_{2b_n}+k)\Big)\bigg|\mathbb{G}\bigg]\bigg|\Bigg\|_1\\\leq& \Big\|\sup_{g\in \mathcal{F}}|\Delta g|\Big\|_{\infty}\Big(\mu_{n,s}\mathcal{R}_\Psi(n,s,b_n)+\frac{|\B_{2b_n}|}{|\A_n|}\gamma_n(2b_n) \Big).
% \\&\leq \sup_{g\in \mathcal{F}}\|\Delta g\|_{\infty}\sum_{\phi \in \A_n}\Big\{\sum_{j = b_n}^{2b_n-1}\Big|\mathcal{B}^{\A_n}_{j+1}(\phi) \backslash \mathcal{B}^{\A_n}_{j}(\phi)\Big| \norm{Q_n(\phi)}_{\frac{s}{s-1}}\Psi_{n,s}(j)+\sum_{\phi^{\prime} \in \mathcal{B}^{\A_n}_{2b_n}(\phi) \backslash \mathcal{B}^{\A_n}_{b_n}(\phi)}\norm{Q_n(\phi)Q_n(\phi^\prime)}_1 \Big\}
\end{aligned}
\end{equation}
\end{lemma}
\begin{proof}
% Recall that 
% $$(A) = \sup_{g\in\mathcal{F}}\Bigg\|\mathbb{E}\bigg[\sum_{\phi \in \A_n}f_{n}(\phi X_{n}) \sum_{k=1}^{\infty} \mathbb{I}(E^{n,k}_{b_n,\phi})(g(W^{\phi,n}_{b_n}+k)-g(W^{\phi,n}_{2b_n}+k))\bigg|\mathbb{G}\bigg]\Bigg\|_1$$
For all $g\in\mathcal{F}$, we have
\begin{equation*}
\allowdisplaybreaks
\begin{aligned}
&\mathbb{E}\bigg[\sum_{\phi \in \A_n}f_{n}(\phi X_{n}) \sum_{k=1}^{\infty} \mathbb{I}(E^{n,k}_{b_n,\phi})\Big(g(W^{\phi,n}_{b_n}+k)-g(W^{\phi,n}_{2b_n}+k)\Big)\bigg| \mathbb{G}\bigg]\\
\overset{(a)}{\leq}&\sum_{\phi \in \A_n}\mathbb{E}\Big[f_{n}(\phi X_{n}) \sup _{k\in \mathbb{N}} ~\Big|g(W^{\phi,n}_{b_n}+k)-g(W^{\phi,n}_{ 2 b_n}+k)\Big|~\Big|\mathbb{G}\Big] \\
\leq &\sum_{\phi \in \A_n}\Big\|\sup_{g\in\mathcal{F}}|\Delta g|\Big\|_{\infty} \mathbb{E}\Big[f_n(\phi X_n) \sum_{\phi^{\prime} \in \mathbf{B}^{\A_n}_{2b_n}(\phi) \backslash \mathbf{B}^{\A_n}_{b_n}(\phi)} f_{n}(\phi^\prime X_{n})\Big|\mathbb{G}\Big],
% = & \|\Delta g\|_{\infty}\sum_{\phi \in \A_n} \bigg\|\mathbb{E}\Big[f_n(\phi X_n) \sum_{j=b_n}^{2b_n-1}\sum_{\phi^{\prime} \in \mathcal{B}^{\A_n}_{j+1}(\phi) \backslash \mathcal{B}^{\A_n}_{j}(\phi)} f_{n}(\phi^\prime X_{n}) \Big| \mathbb{G}\Big]\bigg\|_1\\
% %= & \abs{\A_n}\|\Delta g\|_{\infty} \sum_{j=b_n}^{2b_n-1}\sum_{\phi^{\prime} \in B(e, j+1) \backslash B(e, j)}\mathbb{E}[f_n(\phi X_n) f_{n}(\phi^\prime X_{n}) ]aaa\\
% \overset{(a)}{=} & \|\Delta g\|_{\infty}\sum_{\phi \in \A_n} \sum_{j=b_n}^{2b_n-1}\sum_{\phi^{\prime} \in \mathcal{B}^{\A_n}_{j+1}(\phi) \backslash \mathcal{B}^{\A_n}_{j}(\phi)}\Big(\Big\|\mathbb{E}[f_{n}(\phi X_{n}) f_{n}(\phi^\prime X_{n})| \mathbb{G} ]-\mathbb{E}[f_{n}(\phi X_{n}) | \mathbb{G}]\E[f_{n}(\phi^\prime X_{n}) | \mathbb{G}]\Big\|_1\\&+\Big\|\mathbb{E}[f_{n}(\phi X_{n}) | \mathbb{G}]\E[f_{n}(\phi^\prime X_{n}) | \mathbb{G}]\Big\|_1\Big)
% \\
% \overset{(b)}{=} & \|\Delta g\|_{\infty}\sum_{\phi \in \A_n} \sum_{j=b_n}^{2b_n-1}\sum_{\phi^{\prime} \in \mathcal{B}^{\A_n}_{j+1}(\phi) \backslash \mathcal{B}^{\A_n}_{j}(\phi)}\Big(\norm{Q_n(\phi)}_{\frac{s}{s-1}}\Psi_{n,s}(j)+\norm{Q_n(\phi)Q_n(\phi^\prime)}_1 \Big)\\
% \leq &\|\Delta g\|_{\infty}\sum_{\phi \in \A_n}(\sum_{j = b_n}^{2b_n-1}\Big|\mathcal{B}^{\A_n}_{j+1}(\phi) \backslash \mathcal{B}^{\A_n}_{j}(\phi)\Big| \norm{Q_n(\phi)}_{\frac{s}{s-1}}\Psi_{n,s}(j)+\sum_{\phi^{\prime} \in \mathcal{B}^{\A_n}_{2b_n}(\phi) \backslash \mathcal{B}^{\A_n}_{b_n}(\phi)}\norm{Q_n(\phi)Q_n(\phi^\prime)}_1 ).
\end{aligned}
\end{equation*} where to get (a) we used the fact that $\bigcup_{k=1}^{\infty}E^{n,k}_{b_n,\phi}\supset\{f_n(\phi X_n)\ne 0\}$. Therefore, it follows that 
\begin{equation*}
\allowdisplaybreaks
\begin{aligned}
&\Bigg\|\sup_{g\in \mathcal{F}}\bigg|\mathbb{E}\bigg[\sum_{\phi \in \A_n}f_{n}(\phi X_{n}) \sum_{k=1}^{\infty} \mathbb{I}(E^{n,k}_{b_n,\phi})\Big(g(W^{\phi,n}_{b_n}+k)-g(W^{\phi,n}_{2b_n}+k)\Big)\bigg| \mathbb{G}\bigg]\bigg|\Bigg\|_1
% \\
% \leq&\sum_{\phi \in \A_n}\mathbb{E}\Big[f_{n}(\phi X_{n}) \sup _{k} |g(W^{\phi,n}_{b_n}+k)-g(W^{\phi,n}_{ 2 b_n}+k)|\Big|\mathbb{G}\Big] \\
% \leq &\sum_{\phi \in \A_n}\|\Delta g\|_{\infty} \bigg|\mathbb{E}\Big[f_n(\phi X_n) \sum_{\phi^{\prime} \in \mathcal{B}^{\A_n}_{2b_n}(\phi) \backslash \mathcal{B}^{\A_n}_{b_n}(\phi)} f_{n}(\phi^\prime X_{n})\Big|\mathbb{G}\Big]\bigg| \\
\\\le & \Big\|\sup_{g\in \mathcal{F}}~|\Delta g|\Big\|_{\infty}\sum_{\phi \in \A_n} \bigg\|\mathbb{E}\Big[f_n(\phi X_n) \sum_{j=b_n}^{2b_n-1}\sum_{\phi^{\prime} \in \mathbf{B}^{\A_n}_{j+1}(\phi) \backslash \mathbf{B}^{\A_n}_{j}(\phi)} f_{n}(\phi^\prime X_{n}) \Big| \mathbb{G}\Big]\bigg\|_1\\
%= & \abs{\A_n}\|\Delta g\|_{\infty} \sum_{j=b_n}^{2b_n-1}\sum_{\phi^{\prime} \in B(e, j+1) \backslash B(e, j)}\mathbb{E}[f_n(\phi X_n) f_{n}(\phi^\prime X_{n}) ]aaa\\
\overset{(a)}{\le} & \Big\|\sup_{g\in \mathcal{F}}~|\Delta g|\Big\|_{\infty}\sum_{\phi \in \A_n} \sum_{j=b_n}^{2b_n-1}\sum_{\phi^{\prime} \in \mathbf{B}^{\A_n}_{j+1}(\phi) \backslash \mathbf{B}^{\A_n}_{j}(\phi)}\Big(\Big\|\mathbb{E}\big[f_{n}(\phi X_{n}) f_{n}(\phi^\prime X_{n})\big| \mathbb{G} \big]
\\
&\qquad -\mathbb{E}\big[f_{n}(\phi X_{n}) \big| \mathbb{G}\big]\E\big[f_{n}(\phi^\prime X_{n}) \big| \mathbb{G}\big]\Big\|_1+\Big\|\mathbb{E}\big[f_{n}(\phi X_{n}) \big| \mathbb{G}\big]\E\big[f_{n}(\phi^\prime X_{n}) \big| \mathbb{G}\big]\Big\|_1\Big),
% \\
% \overset{(b)}{\le} &\sup_{g\in \mathcal{F}} \|\Delta g\|_{\infty}\sum_{\phi \in \A_n} \sum_{j=b_n}^{2b_n-1}\sum_{\phi^{\prime} \in \mathcal{B}^{\A_n}_{j+1}(\phi) \backslash \mathcal{B}^{\A_n}_{j}(\phi)}\Big(\norm{Q_n(\phi)}_{\frac{s}{s-1}}\Psi_{n,s}(j)+\norm{Q_n(\phi)Q_n(\phi^\prime)}_1 \Big)\\
% \leq &\sup_{g\in \mathcal{F}} \|\Delta g\|_{\infty}\sum_{\phi \in \A_n}(\sum_{j = b_n}^{2b_n-1}\Big|\mathcal{B}^{\A_n}_{j+1}(\phi) \backslash \mathcal{B}^{\A_n}_{j}(\phi)\Big| \norm{Q_n(\phi)}_{\frac{s}{s-1}}\Psi_{n,s}(j)+\sum_{\phi^{\prime} \in \mathcal{B}^{\A_n}_{2b_n}(\phi) \backslash \mathcal{B}^{\A_n}_{b_n}(\phi)}\norm{Q_n(\phi)Q_n(\phi^\prime)}_1 ).
\end{aligned}
\end{equation*}
where $(a)$ is a result of the triangle equality. We note that by definition of $Q_n(\phi)$ we have \begin{align*}
    \Big\|\mathbb{E}\big[f_{n}(\phi X_{n}) \big| \mathbb{G}\big]\E\big[f_{n}(\phi^\prime X_{n}) \big| \mathbb{G}\big]\Big\|_1= \norm{Q_n(\phi)Q_n(\phi^\prime)}_1.
\end{align*}Morever, for all $\phi,\phi'\in \G$ such that $d(\phi,\phi')\ge j$, by using the definition of $\Psi$-mixing coefficients and H\"{o}lder's inequality, we obtain that 
\begin{align*}&
    \Big\|\mathbb{E}[f_{n}(\phi X_{n}) f_{n}(\phi^\prime X_{n})| \mathbb{G} ]-\mathbb{E}[f_{n}(\phi X_{n}) | \mathbb{G}]\E[f_{n}(\phi^\prime X_{n}) | \mathbb{G}]\Big\|_1
    \\\le& \Big\|\mathbb{P}\big(f_n(\phi X_n)=1\big|\G\big)\Big[\mathbb{P}\big(f_n(\phi' X_n)=1\big|\G,f_n(\phi X_n)=1\big)-\mathbb{P}\big(f_n(\phi'X_n)=1\big|\G\big)\Big]\Big\|_1\\
    \le&\norm{Q_n(\phi)}_{\frac{s}{s-1}}\Psi_{n,s}(j|\G).
\end{align*}
Therefore we obtain that
\begin{align*}
&\Bigg\|\sup_{g\in \mathcal{F}}\bigg|\mathbb{E}\bigg[\sum_{\phi \in \A_n}f_{n}(\phi X_{n}) \sum_{k=1}^{\infty} \mathbb{I}(E^{n,k}_{b_n,\phi})(g(W^{\phi,n}_{b_n}+k)-g(W^{\phi,n}_{2b_n}+k))\bigg| \mathbb{G}\bigg]\bigg|\Bigg\|_1
\\\overset{}{\le} & \Big\|\sup_{g\in \mathcal{F}}~|\Delta g|\Big\|_{\infty}\sum_{\phi \in \A_n} \sum_{j=b_n}^{2b_n-1}\sum_{\phi^{\prime} \in \mathbf{B}^{\A_n}_{j+1}(\phi) \backslash \mathbf{B}^{\A_n}_{j}(\phi)}\Big(\norm{Q_n(\phi)}_{\frac{s}{s-1}}\Psi_{n,s}(j|\G)+\norm{Q_n(\phi)Q_n(\phi^\prime)}_1\Big) \\
\leq &\Big\|\sup_{g\in \mathcal{F}} ~|\Delta g|\Big\|_{\infty}\sum_{\phi \in \A_n}\Big(\sum_{j = b_n}^{2b_n-1}\Big|\mathbf{B}^{\A_n}_{j+1}(\phi) \backslash \mathbf{B}^{\A_n}_{j}(\phi)\Big| \norm{Q_n(\phi)}_{\frac{s}{s-1}}\Psi_{n,s}(j|\G)\\
&\qquad\qquad\qquad\qquad\qquad+\sum_{\phi^{\prime} \in \mathbf{B}^{\A_n}_{2b_n}(\phi) \backslash \mathbf{B}^{\A_n}_{b_n}(\phi)}\norm{Q_n(\phi)Q_n(\phi^\prime)}_1 \Big)
\\\le& \Big\|\sup_{g\in \mathcal{F}}~|\Delta g|\Big\|_{\infty}\Big(\mu_{n,s}\mathcal{R}_\Psi(n,s,b_n)+\frac{|\B_{2b_n}|}{|\A_n|}\gamma_n(2b_n) \Big).
\end{align*}
\end{proof}

In order to bound the second term on the right-hand side of \cref{new_england}, we further split it into two terms.
\begin{lemma}\label{decomposeintoB1B2} Suppose that all the conditions of \cref{snow} are satisfied. Then the following holds
\begin{equation}
\begin{aligned}\label{connecticut}
(B)=&\Bigg\|\sup_{g\in\mathcal{F}}\bigg|\mathbb{E}\bigg[\sum_{\phi \in \A_n}f_{n}(\phi X_{n}) \sum_{k=1}^{\infty} \mathbb{I}(E^{n,k}_{b_n,\phi})g(W^{\phi,n}_{2b_n}+k)-\sum_{k=1}^{\infty} k \lambda^{b_n}_n(k) g(W_n+k)\bigg|\G\bigg]\bigg|\Bigg\|_1
\\\leq & \bigg\|\sup_{g\in\mathcal{F}}\Big|\sum_{\phi \in \A_n}\sum_{k=1}^{\infty}\E\Big[f_n(\phi X_n)\ind(\tilde E_{b_n,\phi}^{n,k})g(W^{\phi,n}_{ 2 b_n}+k)\Big|\G\Big]\\
&\qquad\qquad\qquad\qquad-\E\Big[f_n(\phi X_n)\ind(\tilde E_{b_n,\phi}^{n,k})\Big|\G\Big]\mathbb{E}\Big[g(W_n+k)\Big|\G\Big]\Big|\bigg\|_1\\
&+\bigg\|\sup_{g\in\mathcal{F}}\Big|\sum_{\phi \in \A_n}\sum_{k=1}^{\infty}\E\Big[f_n(\phi X_n)\ind(E_{b_n,\phi}^{n,k})g(W_n+k)\Big|\G\Big]\\
&\qquad\qquad\qquad\qquad- \E\Big[f_n(\phi X_n)\ind(\tilde E^{n,k}_{b_n,\phi})g(W_n+k)\Big|\G\Big]\Big|\bigg\|_1\\
:= & (B_1) + (B_2)
\end{aligned} 
\end{equation}
\end{lemma}

\begin{proof}
% Recall that $$
% (B) =\sup_{g\in\mathcal{F}}\Bigg\|\mathbb{E}\bigg[\sum_{\phi \in \A_n}f_{n}(\phi X_{n}) \sum_{k=1}^{\infty} \mathbb{I}(E^{n,k}_{b_n,\phi})g(W^{\phi,n}_{2b_n}+k)-\sum_{k=1}^{\infty} k \lambda^{b_n}_n(k) g(W_n+k)\bigg|\G\bigg]\Bigg\|_1\\
% $$
Recall that by definition, we have
$$k\lambda^{b_n}_{n}(k) := \sum_{\phi\in \A_n}\mathbb{E}\Big[f_{n}(\phi X_{n}) \mathbb{I}\big(\tilde E^{n,k}_{b_n,\phi}\big)\Big|\G\Big].$$
The result is a direct consequence of the triangle inequality, which implies that 
$$
\begin{aligned}
&\Bigg\|\sup_{g\in \mathcal{F}}\bigg|\mathbb{E}\bigg[\sum_{\phi \in \A_n}f_{n}(\phi X_{n}) \sum_{k=1}^{\infty} \mathbb{I}(E^{n,k}_{b_n,\phi})g(W^{\phi,n}_{2b_n}+k)-\sum_{k=1}^{\infty} k \lambda^{b_n}_n(k) g(W_n+k)\bigg|\G\bigg]\bigg|\Bigg\|_1\\
\leq &\bigg\|\sup_{g\in \mathcal{F}}\Big|\sum_{\phi \in \A_n}\sum_{k=1}^{\infty}\E\Big[f_n(\phi X_n)\ind(\tilde E_{b_n,\phi}^{n,k})g(W^{\phi,n}_{ 2 b_n}+k)\Big|\G\Big]\\
&\qquad\qquad\qquad\qquad-\E\Big[f_n(\phi X_n)\ind(\tilde E_{b_n,\phi}^{n,k})\Big|\G\Big]\mathbb{E}\Big[g(W_n+k)\Big|\G\Big]\Big|\bigg\|_1\\
&+\bigg\|\sup_{g\in \mathcal{F}}\Big|\sum_{\phi \in \A_n}\sum_{k=1}^{\infty}\E\Big[f_n(\phi X_n)\ind(E_{b_n,\phi}^{n,k})g(W_n+k)\Big|\G\Big]\\
&\qquad\qquad\qquad\qquad- \E\Big[f_n(\phi X_n)\ind(\tilde E^{n,k}_{b,\phi})g(W_n+k)\Big|\G\Big]\Big|\bigg\|_1.
\end{aligned}
$$
\end{proof}

We will successively bound each term on the right handside of \cref{connecticut}.

\begin{lemma}\label{boundB1} Suppose that all the conditions of \cref{snow} are satisfied. Then the following holds for all $q\ge 1$:
\begin{align*}(B_1)=&\bigg\|\sup_{g\in \mathcal{F}}\Big|\sum_{\phi \in \A_n}\sum_{k=1}^{\infty}\E\Big[f_n(\phi X_n)\ind(\tilde E_{b_n,\phi}^{n,k})g(W^{\phi,n}_{ 2 b_n}+k)\Big|\G\Big]\\
&\qquad\qquad\qquad\qquad-\E\Big[f_n(\phi X_n)\ind(\tilde E_{b_n,\phi}^{n,k})\Big|\G\Big]\mathbb{E}\Big[g(W_n+k)\Big|\G\Big]\Big|\bigg\|_1\\
\leq & \Big\|\sup_{g\in \mathcal{F}}~|\Delta g|\Big\|_\infty\Big(\frac{\abs{\B_{2b_n}}}{\abs{\A_n}}\gamma_n(2b_n) + 2\mu_{n,q}\mathcal{R}_\xi(n,q,b_n)\Big).
\end{align*}
\end{lemma}

\begin{proof}

% \begin{aligned}
% (b_1) =& \sup_{g\in\mathcal{F}}\left\|\sum_{\phi \in \A_n}\sum_{k=1}^{\infty}\E[f_n(\phi X_n)\ind(E_{b_n,\phi}^{n,k})g(W^{\phi,n}_{ 2 b_n}+k)|\G]\right.\\
% &\left. - \sum_{\phi \in \A_n}\sum_{k=1}^{\infty}\E[f_n(\phi X_n)\ind(E_{b_n,\phi}^{n,k})|\G]\mathbb{E}[g(W_n+k)|\G]\right\|_1
% \end{aligned}
% $$
For the ease of notation, for all $k\in \mathbb{N}$ and $\phi\in \G$, we write $ h_{n,k}^\phi= f_n(\phi X_n)\ind(\tilde E^{n,k}_{b_n,\phi}).$ Moreover, we also denote for all $k\in\N$ and $g\in \mathcal{F}$
$$
\begin{aligned}
D^g_k  :=&  \sum_{\phi \in \A_n}\Big(\E\Big[f_n(\phi X_n)\ind(\tilde E^{n,k}_{b_n,\phi})g(W^{\phi,n}_{ 2 b_n}+k)\Big| \mathbb{G}\Big]\\
&\qquad - \E\Big[f_n(\phi X_n)\ind(\tilde E_{b_n,\phi}^{n,k})\Big| \mathbb{G}\Big]\mathbb{E}\Big[g(W_n+k)\Big| \mathbb{G}\Big]\Big).
\end{aligned}$$
We will first rewrite $D^g_k$. In this goal, we remark that 
\begin{align*}
& \sum_{\phi \in \A_n}\Big(\E[f_n(\phi X_n)\ind(\tilde E^{n,k}_{b_n,\phi})g(W^{\phi,n}_{ 2 b_n}+k)| \mathbb{G}]
- \E[f_n(\phi X_n)\ind(\tilde E_{b_n,\phi}^{n,k})| \mathbb{G}]~\mathbb{E}[g(W_n+k)| \mathbb{G}]\Big)\\
 = & \sum_{\phi \in \A_n}\E[h_{n,k}^\phi g(W^{\phi,n}_{ 2 b_n}+k)| \mathbb{G}] - \sum_{\phi \in \A_n}~\E[h_{n,k}^\phi| \mathbb{G} ]~\mathbb{E}[g(W_n+k)| \mathbb{G}]\\
 = & \sum_{\phi \in \A_n}\E[h_{n,k}^\phi g(W^{\phi,n}_{ 2 b_n}+k)| \mathbb{G}] -\E[h_{n,k}^\phi| \mathbb{G}]~\mathbb{E}[g(W^{\phi,n}_{ 2 b_n}+k)| \mathbb{G}] \\
 &+ \sum_{\phi \in \A_n}\E[h_{n,k}^\phi| \mathbb{G}]\Big(\mathbb{E}[g(W^{\phi,n}_{ 2 b_n}+k)| \mathbb{G}]-\mathbb{E}[g(W_n+k)| \mathbb{G}]\Big).
 %= & \sum_{\phi \in \A_n}\E[h_{n,k}^\phi g(W^{\phi,n}_{ 2 b_n}+k)| \mathbb{G}] - \E[h_{n,k}^\phi| \mathbb{G}]\mathbb{E}[g(W^{\phi,n}_{ 2 b_n}+k)| \mathbb{G}]\\& \quad + \sum_{\phi \in \A_n}\E[h_{n,k}^\phi| \mathbb{G}](\mathbb{E}[g(W^{\phi,n}_{ 2 b_n}+k)| \mathbb{G}]-\mathbb{E}[g(W_n+k)| \mathbb{G}])
\end{align*}
Therefore, using the triangle inequality, we obtain that 
\begin{equation*}
\begin{aligned}
&\bigg\|\sup_{g\in  \mathcal{F}}\Big|\sum_{k=1}^{\infty}D^g_k\Big|\bigg\|_1\\
\leq &\sum_{\phi \in \A_n}\bigg\|\sup_{g\in \mathcal{F}}\Big|\sum_{k=1}^{\infty}\E[h_{n,k}^\phi g(W^{\phi,n}_{ 2 b_n}+k)| \mathbb{G}] - \E[h_{n,k}^\phi| \mathbb{G}]\times\mathbb{E}[g(W^{\phi,n}_{ 2 b_n}+k)| \mathbb{G}]\Big|\bigg\|_1\\
 & + \sum_{\phi \in \A_n}\bigg\|\sup_{g\in \mathcal{F}}\Big|\sum_{k=1}^{\infty}\E[h_{n,k}^\phi| \mathbb{G} ]\Big(\mathbb{E}[g(W^{\phi,n}_{ 2 b_n}+k)| \mathbb{G}]-\mathbb{E}[g(W_n+k)| \mathbb{G}]\Big)\Big|\bigg\|_1\\
  =: & (C_1) + (C_2).
\end{aligned}
\end{equation*} We will bound each term successively. Firstly we remark that the term $(C_2)$ can be bounded by using the fact that for all functions $g\in \mathcal{F}$, we know that $g$ is $\|\Delta g\|_{\infty}$-Lipschitz. Indeed, by exploiting H\"{o}lder's inequality, we obtain that 
\begin{align*}
(C_2) & =\sum_{\phi \in \A_n}\bigg\|\sup_{g\in \mathcal{F}}\Big|\sum_{k=1}^{\infty}\E[h_{n,k}^\phi| \mathbb{G} ]\mathbb{E}[g(W^{\phi,n}_{ 2 b_n}+k)-g(W_n+k)| \mathbb{G}]\Big|\bigg\|_1\\
& \leq \Big\|\sup_{g\in \mathcal{F}}~|\Delta g|\Big\|_\infty\sum_{\phi \in \A_n}\bigg\|\sum_{k=1}^{\infty}\E[h_{n,k}^\phi| \mathbb{G}]\Big|\mathbb{E}\Big[{W^{\phi,n}_{ 2 b_n}-W_n}\Big|\mathbb{G}\Big]\Big|\bigg\|_1\\
& \le \Big\|\sup_{g\in \mathcal{F}}~|\Delta g|\Big\|_\infty\sum_{\phi \in \A_n}\bigg\|\sum_{k=1}^{\infty}\E[h_{n,k}^\phi| \mathbb{G}]\sum_{\substack{\phi^{\prime} \in \A_n\\ d(\phi, \phi^{\prime})\le2b_n}}\mathbb{E}[f_n(\phi^\prime X)|\mathbb{G}]\bigg\|_1\\
& \overset{(a)}{=} \Big\|\sup_{g\in \mathcal{F}}~|\Delta g|\Big\|_\infty\sum_{\phi \in \A_n}\bigg\|\E[f_n(\phi X_n)| \mathbb{G}]\sum_{\substack{\phi^{\prime} \in \A_n\\ d(\phi, \phi^{\prime})\le2b_n}}\mathbb{E}[f_n(\phi^\prime X)|\mathbb{G}]\bigg\|_1\\
& = \Big\|\sup_{g\in \mathcal{F}}~|\Delta g|\Big\|_\infty\sum_{\phi \in \A_n}\bigg\|Q_n(\phi)\sum_{\substack{\phi^{\prime} \in \A_n\\ d(\phi, \phi^{\prime})\le2b_n}}Q_n(\phi^\prime)\bigg\|_1\\
&\leq\Big\|\sup_{g\in \mathcal{F}}~|\Delta g|\Big\|_\infty\frac{\abs{\B_{2b_n}}}{\abs{\A_n}}\gamma_n(2b_n).
\end{align*} where to obtain $(a)$, we used the fact that $
    \sum_{k=1}^{\infty}h_{n,k}^{\phi}=f_n(\phi X_n).$
We now bound $(C_1)$. In this goal, we remark that by
\cref{lemma26general} and a telescopic sum argument, we obtain that 
\begin{align*}
 (C_1)& = \sum_{\phi \in \A_n}\bigg\|\sup_{g\in \mathcal{F}}\Big|\sum_{k=1}^{\infty}\E[h_{n,k}^\phi g(W^{\phi,n}_{ 2 b_n}+k)| \mathbb{G}]- \E[h_{n,k}^\phi| \mathbb{G}]\times\mathbb{E}[g(W^{\phi,n}_{ 2 b_n}+k)| \mathbb{G}]\Big|\bigg\|_1\\
\\& = \sum_{\phi \in \A_n}\bigg\|\sup_{g\in \mathcal{F}}\Big|\sum_{k=1}^{\infty}\sum_{j\geq 2b_n}\E\Big[h_{n,k}^\phi\Big(g(W^{\phi,n}_{j}+k)-g(W^{\phi,n}_{j+1}+k)\Big)\Big|\mathbb{G}\Big]\\&\qquad \qquad \qquad\qquad\qquad - \E[h_{n,k}^\phi| \mathbb{G}]~\mathbb{E}\Big[g(W^{\phi,n}_{j}+k)-g(W^{\phi,n}_{j+1}+k)\Big| \mathbb{G}\Big]\Big|\bigg\|_1\\
& = \sum_{\phi \in \A_n}\sum_{j\geq 2b_n}\bigg\|\sup_{g\in \mathcal{F}}\Big|\sum_{k=1}^{\infty}\E\Big[\big(h_{n,k}^\phi-\E[h_{n,k}^\phi| \mathbb{G}]\big)\Big(g(W^{\phi,n}_{j}+k)-g(W^{\phi,n}_{j+1}+k)\Big)\Big| \mathbb{G}\Big]\Big|\bigg\|_1\\
%& \overset{(a)}{\leq} 2\sup_{g\in \mathcal{F}}\norm{\Delta g}_\infty\sum_{\phi \in \A_n}\norm{\E[f_n(\phi X_n)|\mathbb{G}]}_{\frac{p}{p-1}}\sum_{j\geq 2b_n}|\B_{j+1}\backslash \B_j|\xi_{n,p}^{b_n}(j-b_n|\mathbb{G})\\
& \leq 2\Big\|\sup_{g\in \mathcal{F}}|\Delta g|\Big\|_\infty\sum_{\phi \in \A_n}\norm{Q_n(\phi)}_{\frac{q}{q-1}}\sum_{j\geq 2b_n}|\B^{\A_n}_{j+1}(\phi)\setminus\B^{\A_n}_{j}(\phi)|\xi_{n,q}^{b_n}(j|\mathbb{G})\\
& \leq 2\Big\|\sup_{g\in \mathcal{F}}|\Delta g|\Big\|_\infty\mu_{n,q}\mathcal{R}_\xi(n,q,b_n).
\end{align*}
where the last step is by noting that $|\B^{\A_n}_{j+1}(\phi)\setminus\B^{\A_n}_{j}(\phi)|\le|\B_{j+1}(\phi)\setminus\B_{j}(\phi)|=|\B_{j+1}\setminus\B_{j}|$.

This directly implies the desired result.

% \begin{equation*}
% \begin{aligned}
% & \Big\|\sum_{k=1}^{\infty}\sum_{\phi \in \A_n}\E[f_n(\phi X_n)\ind(E^{n,k}_{b_n,\phi})g(W^{\phi,n}_{ 2 b_n}+k)| \mathbb{G}]-\E[f_n(\phi X_n)\ind(E_{b_n,\phi}^{n,k})| \mathbb{G}]\mathbb{E}[g(W_n+k)| \mathbb{G}]\big\|_1\\
% \leq & (C_{1}) + (C_{2})
% %\leq & (d_{1} ) +(d_{2} )
% % \\
% % \leq & \norm{\Delta g}_\infty\frac{\abs{\B_{2b_n}}}{\abs{\A_n}}\max_{\substack{\phi \in \A_n\\d(\phi, \phi^{\prime})<2b_n}}\norm{\mu^{\phi}_n\mu^{\phi^\prime}_n}_1 2\sup_{g\in \mathcal{F}}\norm{\Delta g}_\infty\sum_{\phi \in \A_n}\norm{Q_n(\phi)}_{\frac{q}{q-1}}\sum_{j\geq 2b_n}|\B_{j+1}\setminus\B_{j}|\xi_{n,q}^{b_n}(j-b_n|\mathbb{G})
% \end{aligned}
% \end{equation*}

% In particular, the right hand side can be simplied to
% $$
% \norm{\Delta g}_\infty\frac{\abs{\B_{2b_n}}}{\abs{\A_n}}\max_{\substack{\phi \in \A_n\\d(\phi, \phi^{\prime})<2b_n}}\norm{\mu^{\phi}_n\mu^{\phi^\prime}_n}_1 + 2\max_{\phi\in\A_n}\norm{\mu^{\phi}_n}_{\frac{q}{q-1}}\norm{\Delta g}_\infty\sum_{j\geq 2b_n}\abs{\B_{j+1}\backslash \B_j}\xi_{n,\infty}^{b_n}(j-b_n|\mathbb{G}).
% $$
\end{proof}

\begin{lemma}\label{boundB2} Suppose that all the conditions of \cref{snow} are satisfied. Then the following holds for all $p\ge 1$: 
\begin{align*}
(B_2)=&\bigg\|\sup_{g\in\mathcal{F}}\Big|\sum_{\phi \in \A_n}\sum_{k=1}^{\infty}\E\Big[f_n(\phi X_n)\ind(E_{b_n,\phi}^{n,k})g(W_n+k)\Big|\G\Big]\\
& \qquad \qquad \qquad- \E\Big[f_n(\phi X_n)\ind(\tilde E^{n,k}_{b,\phi})g(W_n+k)\Big|\G\Big]\Big|\bigg\|_1
\\&\leq\Big\|\sup_{g\in \mathcal{F}} |g|\Big\|_\infty\eta_{n,p}\Big(\frac{\big|\A_n\B_{b_n}\triangle\A_n\big|}{|\A_n|}\Big)^{(p-1)/p}
\end{align*}
\end{lemma}
\begin{proof}
% Recall that 
% $$\begin{aligned}
% (b_2)=&\sup_{g \in\mathcal{F}}\left\|\sum_{\phi \in \A_n}\sum_{k=1}^{\infty}\E[f_n(\phi X_n)\ind(E_{b_n,\phi}^{n,k})|\G]\mathbb{E}[g(W_n+k)|\G]\right.\\
% &\left. - \sum_{\phi \in \A_n}\sum_{k=1}^{\infty}\E[f_n(\phi X_n)\ind(\tilde E^{n,k}_{b_n,\phi}}f_n(\phi^\prime X_n)= k)|\G]\mathbb{E}[g(W_n+k)|\G]\right\|_1\\
% \end{aligned}$$
Using the fact that $E_{b_n,\phi}^{n,k}\ne \tilde E_{b_n,\phi}^{n,k}$ only if $\sum\limits_{\phi^{\prime} \in \B_{b_n}(\phi)\backslash \mathbf{B}^{\A_n}_{b_n}(\phi)}f_n(\phi^\prime X_n)$ is different from $0$, we obtain that 
$$
\begin{aligned}
(B_2):=& \bigg\|\sup_{g\in \mathcal{F}}\Big|\sum_{\phi\in \A_n}\sum_{k=1}^{\infty}\E\Big[f_n(\phi X_n)\Big(\ind(\tilde E^{n,k}_{b_n,\phi})-\ind(E_{b_n,\phi}^{n,k})\Big)g(W_n+k)\Big|\G\Big]\Big|\bigg\|_1
\\\le&\Big\|\sup_{g\in \mathcal{F}} |g|\Big\|_\infty \mathbb{E}\bigg[\sum_{\phi\in \A_n}\sum_{k=1}^{\infty}\E\Big[{f_n(\phi X_n)}\Big|\ind(\tilde E^{n,k}_{b_n,\phi})-\ind(E_{b_n,\phi}^{n,k})\Big|~\Big|\G\Big]\bigg]% \\
\\\leq & \Big\|\sup_{g\in \mathcal{F}} |g|\Big\|_\infty\sum_{\phi\in \A_n} \E\Big[f_n(\phi X_n)\ind\Big(\sum_{\phi^{\prime} \in \B_{b_n}(\phi)\backslash {\B}^{\A_n}_{b_n}(\phi)}f_n(\phi^\prime X_n)\neq 0\Big)\Big]
%\\= & \norm{\sum_{\phi^{\prime} \in \B_{b_n}(\phi)\backslash \mathcal{B}^{\A_n}_{b_n}(\phi)}\E[f_n(\phi X_n)\ind(f_n(\phi^\prime X_n)\neq 0)|\G]}_1\\
\\\overset{(a)}{\leq} & \Big\|\sup_{g\in \mathcal{F}} |g|\Big\|_\infty\sum_{\phi\in \A_n}\big| \B_{b_n}(\phi)\backslash {\B}^{\A_n}_{b_n}(\phi)\big|\norm{Q_n(\phi)}_1,
\end{aligned}
$$
where $(a)$ is a result of the union bound. 
% Hence,
% $$
% \begin{aligned}
% &\left\|\sum_{\phi \in \A_n}\sum_{k=1}^{\infty}\E[f_n(\phi X_n)\ind(E_{b_n,\phi}^{n,k})|\G]\mathbb{E}[g(W_n+k)|\G]\right.\\
% &\left. - \sum_{\phi \in \A_n}\sum_{k=1}^{\infty}\E[f_n(\phi X_n)\ind(\tilde E^{n,k}_{b_n,\phi}}f_n(\phi^\prime X_n)= k)|\G]\mathbb{E}[g(W_n+k)|\G]\right\|_1\\
% \leq & \norm{g}_\infty\sum_{\phi \in \A_n}\norm{\sum_{k=1}^{\infty}\E\abs{[f_n(\phi X_n)\left(\ind(\tilde E^{n,k}_{b_n,\phi}}f_n(\phi^\prime X_n)= k)-\ind(E_{b_n,\phi}^{n,k})\right)}]}_1\\
% \leq& \norm{g}_\infty\sum_{\phi \in \A_n}\sum_{\phi^{\prime} \in \B_{b_n}(\phi)\backslash \mathcal{B}^{\A_n}_{b_n}(\phi)}\norm{Q_n(\phi)}_1\\
% \leq& \norm{g}_\infty\sum_{\phi \in \A_n}\abs{ \B_{b_n}(\phi)\backslash \mathcal{B}^{\A_n}_{b_n}(\phi)}\norm{Q_n(\phi)}_1
% \end{aligned}
% $$ 
Therefore, to finish the proof, we need to upper-bound 
$
\sum_{\phi\in \A_n}\big| \B_{b_n}(\phi)\backslash {\B}^{\A_n}_{b_n}(\phi)\big|\|Q_n(\phi)\|_1.
$%$\big| \B_{b_n}(\phi)\backslash {\B}^{\A_n}_{b_n}(\phi)\big|$. 
In this goal, using \cref{folnerlemma}, we have 
\begin{align*}
\sum_{\phi\in \A_n}\big| \B_{b_n}(\phi)\backslash {\B}^{\A_n}_{b_n}(\phi)\big|\|Q_n(\phi)\|_1\le \eta_{n,p}\Big(\frac{\big|\A_n\B_{b_n}\triangle\A_n\big|}{|\A_n|}\Big)^{(p-1)/p},
\end{align*}
where, as a reminder, we have defined $\eta_{n,p}:=\Big( |\A_n|^{p-1}\sum_{\phi\in\A_n}\big\|Q_n(\phi)\big\|^p_1\Big)^{1/p}$. This directly gives us the desired result.
% $$
% \begin{aligned}
% (b_2)
% \leq &\norm{g}_\infty\sum_{\phi \in \A_n}\abs{ \B_{b_n}(\phi)\backslash \mathcal{B}^{\A_n}_{b_n}(\phi)}\norm{Q_n(\phi)}_1\\
% \leq& \norm{g}_\infty\frac{1}{\abs{\A_n}}\sum_{\phi \in \A_n}\abs{\B_{b_n}(\phi)\backslash \mathcal{B}^{\A_n}_{b_n}(\phi)}\max_{\phi \in \A_n}\norm{\mu^{\phi}_n}_1\\
% \leq& \max_{\phi \in \A_n}\norm{\mu^{\phi}_n}_1\norm{g}_\infty\frac{\abs{\B_{b_n} \A_n \triangle \A_n}}{\abs{\A_n}}
% \end{aligned}
% $$
\end{proof}

By combining \cref{snow}, \cref{boundA}, \cref{decomposeintoB1B2}, \cref{boundB1}, and \cref{boundB2}, 
we can finish the proof.

\begin{lemma}[Bound in \cref{result1}]\label{boundfortvdconditional}
Given all the conditions of \cref{snow}, the following inequality holds:
\begin{equation}\label{bigineq}
\begin{aligned}
&\E\Big[\dTV(W_n,Z(\lambda_{n}^{b_n})| \G)\Big]\\
\leq & \Big\|\sup_{g\in \mathcal{F}}|\Delta g|\Big\|_{\infty}\Big\{\inf_{s\ge 1}\mu_{n,s}\mathcal{R}_\Psi(n,s,b_n)+2\frac{|\B_{2b_n}|}{|\A_n|}\gamma_n(2b_n) \\
&\qquad\qquad\qquad + 2\inf_{s\ge 1}\mu_{n,s}\mathcal{R}_\xi(n,s,b_n)\Big\}\\
&+\Big\|\sup_{g\in \mathcal{F}}|g|\Big\|_\infty\inf_{s\ge 1}\eta_{n,s}\Big(\frac{\big|\A_n\B_{b_n}\triangle\A_n\big|}{|\A_n|}\Big)^{(s-1)/s}.
\end{aligned}
\end{equation}
\end{lemma}
\begin{proof}This is a direct consequence of \cref{snow}, \cref{boundA}, \cref{decomposeintoB1B2}, \cref{boundB1}, and \cref{boundB2}.
\end{proof}

\begin{theorem}[Convergence conditions in \cref{result1}]\label{maine}
Assume all the conditions of \cref{snow} holds. Suppose that $(b_n)_{n\in\N}$ is a sequence of positive integers, and the following conditions are satisfied:

\begin{enumerate}[(i)]
    \item $\sup_{n}\Big\|\sum_{\phi\in \A_n}\mathbb{E}\big[f_{n}(\phi X_{n})|\G\big]\Big\|_\infty\le \mu$ for some $\mu\in\R$;
    \item $\frac{\abs{\B_{2b_n}}}{\abs{\A_n}} \gamma_n(2b_n)  \xrightarrow{n\rightarrow\infty} 0$;
    \item $\inf_{s\ge 1}\eta_{n,s}\Big(\frac{\big|\A_n\B_{b_n}\triangle\A_n\big|}{|\A_n|}\Big)^{(s-1)/s}\xrightarrow{n\rightarrow\infty} 0;$
    \item $\inf_{s\ge 1}\mu_{n,s}\mathcal{R}_\Psi(n,s,b_n) \xrightarrow{n\rightarrow\infty}0$ and $\inf_{s\ge 1}\mu_{n,s}\mathcal{R}_\xi(n,s,b_n) \xrightarrow{n\rightarrow\infty} 0$.
\end{enumerate}
Then $\E\Big[\dTV(W_n,Z(\lambda_{n}^{b_n})|\G)\Big]\xrightarrow{n\rightarrow\infty} 0$.   
\end{theorem}
\begin{proof}    Provided condition (i) holds, we can establish $\|H_{0}(\lambda^{b_n}_{n})\|_\infty, \|H_{1}(\lambda^{b_n}_{n})\|_\infty \le e^\mu$, as a result of \cref{cruelsummer}. Therefore, under the conditions we have, we can obtain that all the terms on the right-hand side of \cref{bigineq} converge to 0. Hence we have $\E\Big[\dTV(W_n,Z(\lambda_{n}^{b_n})|\G)\Big]\xrightarrow{n\rightarrow\infty} 0$.
\end{proof}

\subsubsection{Proofs of corollaries}\label{proof_main_cor}
In this subsection, we provide the proof of the remark following \cref{result1}, as well as the proof of \cref{result1_cor}. Additionally, we include the proofs for all the examples discussed in \cref{app1,app2,app3}.

The subsequent corollary validates the remark that condition \eqref{A1} can be relaxed without compromising the convergence result in \cref{result1}.
\begin{corollary}\label{lavenderhaze}
    Suppose $\sup_{n}\sum_{\phi\in\A_n}\E[f_n(\phi X_n)|\G]=O_p(1)$. Further assume conditions (ii)-(iv) of \cref{maine}. 
    \begin{comment}
    \begin{enumerate}[(i)]
    \item $\inf_{s\ge 1}\mu_{n,s}\sum_{j = b_n}^{2b_n-1}|\B_{j+1}\backslash \B_{j}| \Psi_{n,s}(j| \G) \longrightarrow 0,$
    \item $\frac{\abs{\B_{2b_n}}}{\abs{\A_n}} \gamma_n(2b_n)  \longrightarrow 0$
    \item $\inf_{s\ge 1}\mu_{n,s}\sum_{j\geq 2b_n}\abs{\B_{j+1}\backslash \B_j}\xi_{n,s}^{b_n}(j|\mathbb{G}) \longrightarrow 0,$
    \item $\inf_{s\ge 1}\eta_{n,s}\Big(\frac{\big|\B_{b_n}\A_n\triangle\A_n\big|}{|\A_n|}\Big)^{(s-1)/s}\longrightarrow 0.$
    \end{enumerate}      
    \end{comment}
    Then $\E\Big[\dTV(W_n,Z(\lambda_{n}^{b_n})|\G)\Big]\longrightarrow 0$ as $n\to\infty$.
\end{corollary}
\begin{proof} Select a sequence $(\gamma_n)_{n\in\N}$ such that $\gamma_n\to\infty$ and such that:
    \begin{enumerate}[(i)]
    \item $e^{\gamma_n}\inf_{s\ge 1}\mu_{n,s}\sum_{j = b_n}^{2b_n-1}|\B_{j+1}\backslash \B_{j}| \Psi_{n,s}(j| \G) \longrightarrow 0$;
    \item $e^{\gamma_n}\frac{\abs{\B_{2b_n}}}{\abs{\A_n}} \gamma_n(2b_n)  \longrightarrow 0$;
    \item $e^{\gamma_n}\inf_{s\ge 1}\mu_{n,s}\sum_{j\geq 2b_n}\abs{\B_{j+1}\backslash \B_j}\xi_{n,s}^{b_n}(j|\mathbb{G}) \longrightarrow 0$;
    \item $e^{\gamma_n}\inf_{s\ge 1}\eta_{n,s}\Big(\frac{\big|\A_n\B_{b_n}\triangle\A_n\big|}{|\A_n|}\Big)^{(s-1)/s}\longrightarrow 0.$
    \end{enumerate} 
Such a sequence exists because conditions (ii)-(iv) of \cref{maine} holds. Define the events $F_n:=\{\sum_{\phi\in\A_n}\E[f_n(\phi X_n)|\G] \le \gamma_n\}$.
Note that by definition of the total variance metric we know that $\dTV(W_n,Z(\lambda_{n}^{b_n})|\G)\le 1$. This implies
$$\E\Big[\dTV(W_n,Z(\lambda_{n}^{b_n})|\G)\Big]\le\prob(F_n^c) + \E\Big[\dTV(W_n,Z(\lambda_{n}^{b_n})|\G)\Big|F_n\Big]\prob(F_n).
$$
Considering that 
$\sup_{n}\sum_{\phi\in\A_n}\E[f_n(\phi X_n)|\mathbb{G}]=O_p(1)$, it follows that $\prob(F_n^c)\to0$ as $n\to\infty$.
When $\sum_{\phi\in\A_n}\E[f_n(\phi X_n)|\mathbb{G}] \le \gamma_n$, the same argument as \cref{cruelsummer} provides us with  $H_{0}(\lambda^{b_n}_{n}),H_{1}(\lambda^{b_n}_{n})\overset{a.s}{\le }e^{\gamma_n}.$ By the choice of $\gamma_n$ and \cref{maine}, we have that $\E\Big[\dTV(W_n,Z(\lambda_{n}^{b_n})|\G)\Big|F_n\Big]\longrightarrow0$. Therefore, we conclude that
$$\E\Big[\dTV(W_n,Z(\lambda_{n}^{b_n})|\G)\Big]\longrightarrow0\text{ as $n\to\infty$}.
$$
\end{proof}

We now proceed to detail the proof of \cref{result1_cor}. It is important to note that \cref{result1_cor} presumes the $\G$-ergodicity of $X_n$.  Under ergodic conditions, the terms $k\lambda^{b_n}_{n}(k)$ and $Q_n(\phi)$ are no longer random variables, but rather deterministic quantities.
\begin{lemma}\label{chowder}
    Assume the conditions of \cref{snow}. Consider the case where $X_n$ is $\G$-ergodic. Suppose the following conditions hold: 
    \begin{enumerate}[(i)]
        \item For all $b>0$, we have $\sum_{\phi\in \A_n}\mathbb{E}\Big[f_{n}(\phi X_{n}) \mathbb{I}(\sum_{\substack{\phi^{\prime}\in \G \\ d(\phi, \phi^{\prime}) \leq b}} f_{n}(\phi^\prime X_{n})=k)\Big]\xrightarrow{n\rightarrow\infty}k\lambda^b(k)$, where $\lambda^b_k\in \mathbb{R}$ is a real number;
        \item $\sup_n\Psi_n(b) \rightarrow 0$ as $b\to\infty$;%Uniform $\Psi$-mixing conditions: $\sup_n~\sum_{j\ge b}\Psi_n(j) \rightarrow 0$ as b goes to infinity. 
        \item $\sum_{\phi\in \A_n}Q_n(\phi)$ is uniformly bounded by some $\mu\in\R$; 
        \item$\sup_{\phi}\mu_n^{\phi}=o(|\A_n|)$ .
    \end{enumerate}
    Then $\lambda^{b}(k) \longrightarrow \lambda(k)$ for some $\lambda(k)$  as $b\to\infty$. Denote ${\lambda}:\N\to[0,\infty),k\mapsto\lambda(k)$. Moreover, there exists a sequence $(b_n)_{n\in\N}$ satisfying both $\frac{|\B_{2b_n}|}{|\A_n|}\xrightarrow{n\to\infty} 0$
 and $\frac{|\A_n\B_{b_n}\triangle\A_n|}{|\A_n|}\xrightarrow{n\to\infty} 0$, such that the following holds:  
       \begin{equation}
    \begin{aligned}
\dTV(Z(\lambda),Z(\lambda_{n}^{b_n}))\longrightarrow 0 \text{ as $n\to\infty$}.
\end{aligned}       
    \end{equation}
\end{lemma}
\begin{proof}
    We first show that for all $k\in \mathbb{N}$ we have $k\lambda^{b}(k) \longrightarrow k\lambda(k)$ as $b\rightarrow\infty$. To establish this, it suffices to verify that $(k\lambda^{b}(k))_{b\in\N}$ is a Cauchy sequence. Recall that we define 
    $$\tilde E^{n,k}_{b,\phi}:=\Big\{\sum_{\substack{\phi^{\prime}\in \G \\ d(\phi, \phi^{\prime}) \leq b}} f_{n}(\phi^\prime X_{n})=k\Big\}.$$
    By condition (i), we have that, for any $c<d$, where $c,d\in \N$, we have
    $$
    \begin{aligned}
    \Big|k\lambda^c(k) - k\lambda^d(k)\Big| =& \lim_{n\to\infty}\bigg|\sum_{\phi\in \A_n}\mathbb{E}\Big[f_{n}(\phi X_{n}) \mathbb{I}(\tilde E^{n,k}_{c,\phi})\Big]-\sum_{\phi\in \A_n}\mathbb{E}\Big[f_{n}(\phi X_{n}) \mathbb{I}(\tilde E^{n,k}_{d,\phi})\Big]\bigg|.
    \end{aligned}
    $$
    Note that we have
    $$
    \begin{aligned}
        &\bigg|\sum_{\phi\in \A_n}\mathbb{E}\Big[f_{n}(\phi X_{n}) \mathbb{I}(\tilde E^{n,k}_{c,\phi})\Big]-\sum_{\phi\in \A_n}\mathbb{E}\Big[f_{n}(\phi X_{n}) \mathbb{I}(\tilde E^{n,k}_{d,\phi})\Big]\bigg|\\
       % =& \sum_{\phi\in \A_n}\E\Big[f_{n}(\phi X_{n})  \Big|\mathbb{I}(\tilde E^{n,k}_{c,\phi})-\mathbb{I}(\tilde E^{n,k}_{d,\phi})\Big|\Big]\\
        \le& \sum_{\phi\in \A_n}\E\Big[f_{n}(\phi X_{n}) \ind\big(\sum_{\substack{\phi^{\prime}\in \G \\ c< d(\phi, \phi^{\prime}) \leq d}} f_{n}(\phi^\prime X_{n})\neq 0\big)\Big]\\
        = & \sum_{\phi\in \A_n}\bigg(\E\Big[f_{n}(\phi X_{n}) \ind\big(\sum_{\substack{\phi^{\prime}\in \G \\ c< d(\phi, \phi^{\prime}) \leq d}} f_{n}(\phi^\prime X_{n})\neq 0\big)\Big] \\
        & \qquad \qquad - \E[f_{n}(\phi X_{n}) ]\E\Big[\ind\big(\sum_{\substack{\phi^{\prime}\in \G \\ c< d(\phi, \phi^{\prime}) \leq d}} f_{n}(\phi^\prime X_{n})\neq 0\big)\Big]\bigg) \\
        &+ \sum_{\phi\in \A_n}\E[f_{n}(\phi X_{n}) ]\E\Big[\ind\big(\sum_{\substack{\phi^{\prime}\in \G \\ c< d(\phi, \phi^{\prime}) \leq d}} f_{n}(\phi^\prime X_{n})\neq 0\big)\Big] \\
        \overset{(a)}{\leq}& \Psi_n(c)\sum_{\phi\in \A_n} Q_n(\phi) + \sum_{\phi\in \A_n}Q_n(\phi)\sum_{\phi^{\prime} \in \B_d(\phi)\setminus\B_c(\phi)}Q_n(\phi^\prime).
    \end{aligned}
    $$
    The inequality (a) is because of the definition of the $\Psi$-mixing coefficient and the union bound. % and the fact that 
    % $$
    % \begin{aligned}
    % \E\Big[\ind\big(\sum_{\substack{\phi^{\prime}\in \G \\ c< d(\phi, \phi^{\prime}) \leq d}} f_{n}(\phi^\prime X_{n})\neq 0\big)\Big]  &= \prob\Big(\bigcup_{\substack{\phi^{\prime}\in \G \\ c< d(\phi, \phi^{\prime}) \leq d}}\big\{f_{n}(\phi^\prime X_{n})\neq 0\big\}\Big) \\ 
    % &\leq \sum_{\phi^{\prime} \in \B_d(\phi)\setminus\B_c(\phi)}\prob\Big(f_{n}(\phi^\prime X_{n})=1\Big).
    % \end{aligned}
    % $$
    Hence we have
    $$
    \begin{aligned}
    \Big|k\lambda^c(k) - k\lambda^d(k)\Big| & = \lim_{n\to\infty}\bigg|\sum_{\phi\in \A_n}\mathbb{E}\Big[f_{n}(\phi X_{n}) \mathbb{I}(\tilde E^{n,k}_{c,\phi})\Big]-\sum_{\phi\in \A_n}\mathbb{E}\Big[f_{n}(\phi X_{n}) \mathbb{I}(\tilde E^{n,k}_{d,\phi})\Big] \bigg|\\
    & \le \lim_{n\to\infty}\bigg(\Psi_n(c)\sum_{\phi\in \A_n} Q_n(\phi) + \sum_{\phi\in \A_n}Q_n(\phi)\sum_{\phi^{\prime} \in \B_d(\phi)\setminus\B_c(\phi)}Q_n(\phi^\prime)\bigg)\\
    & \overset{(a)}{\le}  \mu\sup_n\Psi_n(c),
    \end{aligned}
    $$
    where $(a)$ is a result of conditions (iii) and (iv). Moreover note that $\sup_n \Psi_n(c)\xrightarrow{c\rightarrow\infty}0$.  
%Note that we have the condition that $\sup_n~\sum_{j\ge b}\Psi_n(j) \longrightarrow 0$ as $b\to\infty$.
    %Hence $\forall \epsilon >0$, $\exists N \in \mathbb{N}$ such that $\forall b>N$, $\sup_n~\sum_{j\ge b}\Psi_n(j) \leq \frac{\epsilon}{\mu}$, which implies that $\forall c,d \geq N$, $
    %\Big|k\lambda^c(k) - k\lambda^d(k)\Big| \leq \mu\sup_n\Psi_n(c)\leq \epsilon$. 
    Thus, $(k\lambda^b(k))_{b\in\N}$ is a Cauchy sequence, and hence it converges to some limit $k\lambda(k)$  as $b$ goes to infinity.

    Moreover, note that we have    \begin{equation}\label{uniformconvbound}
    \begin{aligned}
    \Big|k\lambda^c(k) - k\lambda(k)\Big| & = \Big|k\lambda^c(k) - \lim_{d\to\infty}k\lambda^d(k)\Big| = \lim_{d\to\infty}\Big|k\lambda^c(k) - k\lambda^d(k)\Big| \le \mu\sup_m\Psi_m(c).
    \end{aligned}
    \end{equation}
    We remark that the bound does not depend on $k$. 

    Denote ${\lambda}:\N\to[0,\infty),k\mapsto\lambda(k)$.
    %Denote $\lambda:= (\lambda_1,\lambda_2,\cdots)$. 
    Recall that by condition (i), for all $b>0$, we have $|k\lambda^{b}_{n}(k) - k\lambda^{b}(k)|\to0$ as $n\to\infty$. Therefore, we can choose a sequence $(b_n)_{n\in\N}$ that fulfills the following conditions:
\begin{enumerate}[(i)]
\item $b_n\xrightarrow{n\to\infty}\infty$;
\item $\frac{|\B_{2b_n}|}{|\A_n|}\xrightarrow{n\to\infty} 0$;
\item  $\frac{|\A_n\B_{b_n}\triangle\A_n|}{|\A_n|}\xrightarrow{n\to\infty} 0$;
\item $|\B_{b_n}|\sup_{\phi\in\G}Q_n(\phi)\xrightarrow{n\to\infty}0$;
\item For all $k\in\N$, $
\Big|k(\lambda^{b_n}_n(k)- \lambda^{b_n}(k))\Big|\xrightarrow{n\to\infty}0$.
\end{enumerate}

The rationale behind choosing this particular sequence of $(b_n)$ will become evident as we progress further in the proof.
    Now, we proceed to establish a bound for $d_{\TV}(Z(\lambda),Z(\lambda_{n}^{b_n}))$.
    Let $Z_n\sim Z(\lambda_{n}^{b_n})$ and $Z\sim 
Z(\lambda)$. The Stein equation for compound Poisson random variables \cite{Barbour1992CPStein} implies that  \begin{align}\label{limitlemmastein}\E\Big[ Z_ng(Z_n)\Big]=\E\Big[\sum_{k=1}^{\infty} k\lambda^{b_n}_{n}(k)g(Z_n+k)\Big].\end{align} 
Therefore, using \cref{SteinCPboundTV}, we have
\begin{align*}
\dTV(Z,Z_n) 
& \leq \bigg|\sup_{g\in\mathcal{F}_{\TV}(\lambda)}\E\Big[Z_ng(Z_n)-\sum_{k=1}^{\infty} k \lambda(k) g(Z_n+k) \Big]\bigg|\\
& \overset{(a)}{=} \bigg|\sup_{g\in\mathcal{F}_{\TV}(\lambda)}\E\Big[\sum_{k=1}^{\infty} (k \lambda(k)- k\lambda^{b_n}_{n}(k)) g(Z_n+k)\Big]\bigg|\\
& \leq \sup_{g\in\mathcal{F}_{\TV}(\lambda)}\big\|g\big\|_\infty\sum_{k=1}^{\infty} 
\Big|k(\lambda(k)- k\lambda^{b_n}_{n}(k))\Big|,
\end{align*}where to obtain (a) we used \cref{limitlemmastein}. We proceed to establish a bound for $\sum_{k=1}^{\infty} \Big|k(\lambda(k)- k\lambda^{b_n}_{n}(k))\Big|$.

Note that given the sequence $(b_n)_{n\in\N}$, we can further choose a sequence $(k_n)_{n\in\N}$ that satisfies the following conditions:
\begin{enumerate}[(i)]
\item $k_n\xrightarrow{n\to\infty}\infty$;
\item $k_n\le \Big\lfloor{\mu\sup_{m}\Psi_m(b_n)}^{-1/2}\Big\rfloor$;
\item $\sum_{k=1}^{k_n}
\Big|k(\lambda^{b_n}_n(k)- \lambda^{b_n}(k))\Big|\xrightarrow{n\to\infty}0$.
\end{enumerate}
Then we obtain that the following holds: 
\begin{equation}\label{lasvegas}
\begin{aligned}
& \sum_{k=1}^{\infty} 
 \Big|k(\lambda(k)- \lambda^{b_n}_{n}(k))\Big|\\
= &\sum_{k=1}^{k_n} 
\Big|k(\lambda(k)- \lambda^{b_n}_{n}(k))\Big| + \sum_{k=k_n+1}^{\infty} 
\Big|k(\lambda(k)- k\lambda^{b_n}_{n}(k))\Big|\\
\overset{(a)}{\le}&\sum_{k=1}^{k_n} 
\Big|k(\lambda(k)- \lambda^{b_n}(k))\Big| + \sum_{k=1}^{k_n}\Big|k(\lambda^{b_n}(k)- \lambda^{b_n}_{n}(k))\Big| + \sum_{k=k_n+1}^{\infty} 
\Big|k(\lambda(k)- \lambda^{b_n}_{n}(k))\Big|\\
\overset{(b)}{\le} &\mu\sup_{m}{\Psi_m(b_n)}^{1/2} + \sum_{k=1}^{k_n}\Big|k(\lambda^{b_n}(k)- \lambda^{b_n}_{n}(k))\Big| + \sum_{k\ge k_n} 
\Big|k(\lambda(k)- \lambda^{b_n}_{n}(k))\Big|.
\end{aligned}
\end{equation}
In \cref{lasvegas}, $(a)$ follows from the triangle inequality, while $(b)$ is due to the inequality $\Big|k(\lambda(k)- \lambda^{b_n}_{n}(k))\Big|\le \mu\sup_{m}\Psi_m(b_n)$ that holds for all $k\in\N$, together with the condition $k_n\le \Big\lfloor{\mu\sup_{m}\Psi_m(b_n)}^{-1/2}\Big\rfloor$.
Now our objective is to establish an upper bound for term $(D):=\sum_{k\ge k_n} 
\Big|k(\lambda(k)- \lambda^{b_n}_{n}(k))\Big|$. In this goal, we denote $\bar{\lambda}_n(k) := \inf_{m\ge n}\lambda^{b_m}_{m}(k) \le \lambda^{b_n}_{n}(k)$.  
% Note that for any fixed $k\in\N$, by the triangle inequality, we have $\Big|k(\lambda(k)- k\lambda^{b_n}_{n}(k))\Big|\le \Big|k(\lambda(k)- k\lambda^{b_n}_{n}(k))\Big|+\Big|k(\lambda^{b_n}_n(k)- k\lambda^{b_n}_{n}(k))\Big|$.
% Due to the convergence of $(k\lambda^b(k))_{b\in\N}$ and the fact that $b_n\xrightarrow{n\to\infty}\infty$, it follows that $\Big|k(\lambda(k)- k\lambda^{b_n}_{n}(k))\Big|\xrightarrow{n\to\infty}0$. Moreover, given the specific selection of $(b_n)$ and $(k_n)$, we have $\sum_{l=1}^{k_n}
% \Big|l(\lambda^{b_n,n}_{l}- \lambda^{b_n,n}_l)\Big|\xrightarrow{n\to\infty}0$ which implies $\Big|k(\lambda^{b_n}_n(k)- k\lambda^{b_n}_{n}(k))\Big|\xrightarrow{n\to\infty}0$. Therefore,  $k\lambda^{b_n}_{n}(k)$ converges to $k\lambda(k)$ as $n\to\infty$.
% Note that for any fixed $k\in\N$, by the triangle inequality, we have $\Big|k(\lambda(k)- k\lambda^{b_n}_{n}(k))\Big|\le \Big|k(\lambda(k)- k\lambda^{b_n}_{n}(k))\Big|+\Big|k(\lambda^{b_n}_n(k)- k\lambda^{b_n}_{n}(k))\Big|$.
We note that for any fixed $k\in\N$, the convergence of the sequence $(k\lambda^b(k))_{b\in\N}$, in combination with our selection of the sequence $(b_n)_{n\in\N}$ ensuring $\Big|k(\lambda^{b_n}_n(k)- \lambda^{b_n}(k))\Big|\xrightarrow{n\to\infty}0$, leads to the conclusion that $k\lambda^{b_n}_{n}(k)$ converges to $k\lambda(k)$ as $n\to\infty$. This consequently implies that $$k\bar{\lambda}_n(k)\longrightarrow k\lambda(k) \quad \text{as $n\to\infty$}.$$
Moreover, we remark that $\bar{\lambda}_n(k)$ is a non-decreasing sequence, and hence 
we have that for all $K\in\N$,
$$
\sum_{k\ge K}k\lambda(k) = \sum_{k\ge K}\lim_{n\to\infty}k\bar{\lambda}_n(k) = \sum_{k\ge K}\liminf_{n\to\infty}k\lambda^{b_n}_{n}(k) \overset{(a)}{\le} \liminf_{n\to\infty}\sum_{k\ge K}k\lambda^{b_n}_{n}(k),% = \lim_{n\to\infty} \sum_{k\ge K}k\lambda^{b_n}_n(k).
$$
where $(a)$ is a result of Fatou's lemma. 
This implies that 
\begin{equation*}\label{tailboundbysup}
\sum_{k \ge K}k\lambda(k) \leq \sup_{n} \sum_{k \ge K}k\lambda^{b_n}_n(k).
\end{equation*}
Hence we have 
$$(D)= \sum_{k\ge k_n}
\Big|k(\lambda(k)- k\lambda^{b_n}_{n}(k))\Big| \le 2\sup_{m}\sum_{k\ge k_n}
k\lambda^{b_m}_{m}(k).$$
In the purpose of bounding $(D)$, it would be enough to show that $\sup_{n}\sum_{k\ge K}k\lambda^{b_n}_{n}(k)\xrightarrow{K\rightarrow \infty}0$. By the definition of $\lambda^{b_n}_{n}(k)$, we have
    $$\sum_{k\ge K}k\lambda^{b_n}_{n}(k) = \sum_{k\ge K}\sum_{\phi\in \A_n}\mathbb{E}\Big[f_{n}(\phi X_{n}) \mathbb{I}(\sum_{\substack{\phi^{\prime}\in \G \\ d(\phi, \phi^{\prime}) \leq b_n}} f_{n}(\phi^\prime X_{n})=k)\Big].$$
We remark that $k\lambda^{b_n}_{n}(k) =  0$ for $k>|\B_{b_n}|$, so if $K > |\B_{b_n}|$, we have $\sum_{k\ge K}k\lambda^{b_n}_{n}(k) =0$. If instead $K\le |\B_{b_n}|$, choose $c_K$ 
such that $c_K:=\max\{r: |\B_{r}|<\frac{K}{2}\}$. Then we have
    \begin{align*}
        \sum_{k\ge K}k\lambda^{b_n}_{n}(k) =& \sum_{k\ge K}\sum_{\phi\in \A_n}\mathbb{E}\Big[f_{n}(\phi X_{n}) \mathbb{I}(\sum_{\substack{\phi^{\prime}\in \G \\ d(\phi, \phi^{\prime}) \leq b_n}} f_{n}(\phi^\prime X_{n})=k)\Big]\\
        = & \sum_{\phi\in \A_n}\mathbb{E}\Big[f_{n}(\phi X_{n}) \mathbb{I}(\sum_{\substack{\phi^{\prime}\in \G \\ d(\phi, \phi^{\prime}) \leq b_n}} f_{n}(\phi^\prime X_{n})\ge K)\Big]\\
        \leq & \sum_{\phi\in \A_n}\mathbb{E}\Big[f_{n}(\phi X_{n}) \mathbb{I}(\sum_{\substack{\phi^{\prime}\in \G \\ c_K< d(\phi, \phi^{\prime}) \leq b_n}} f_{n}(\phi^\prime X_{n})\ge K-|\B_{c_K}|)\Big]\\
        \leq & \sum_{\phi\in \A_n}\Big(\mathbb{E}\Big[f_{n}(\phi X_{n}) \mathbb{I}(\sum_{\substack{\phi^{\prime}\in \G \\ c_K< d(\phi, \phi^{\prime}) \leq b_n}} f_{n}(\phi^\prime X_{n})\ge K-|\B_{c_K}|)\Big]
        \\
        & - \mathbb{E}\big[f_{n}(\phi X_{n})\big]\mathbb{E}\Big[\mathbb{I}(\sum_{\substack{\phi^{\prime}\in \G \\ c_K< d(\phi, \phi^{\prime}) \leq b_n}} f_{n}(\phi^\prime X_{n})\ge K-|\B_{c_K}|)\Big]\Big)\\
        & + \sum_{\phi\in \A_n}\mathbb{E}\big[f_{n}(\phi X_{n})\big]\mathbb{E}\Big[\mathbb{I}(\sum_{\substack{\phi^{\prime}\in \G \\ c_K< d(\phi, \phi^{\prime}) \leq b_n}} f_{n}(\phi^\prime X_{n})\ge K-|\B_{c_K}|)\Big]\\ 
        \overset{(a)}{\leq} & \sum_{\phi\in \A_n}\mathbb{E}\big[f_{n}(\phi X_{n})\big]\Psi_n(c_K)\\& + \sum_{\phi\in \A_n}\mathbb{E}\big[f_{n}(\phi X_{n})\big]\mathbb{E}\Big[\mathbb{I}(\sum_{\substack{\phi^{\prime}\in \G \\ c_K< d(\phi, \phi^{\prime}) \leq b_n}} f_{n}(\phi^\prime X_{n})\ge K-|\B_{c_K}|)\Big] \\
        \overset{(b)}{\leq} & \Psi_n(c_K)\sum_{\phi\in \A_n}Q_n(\phi) + \sum_{\phi\in \A_n}Q_n(\phi)\sum_{\substack{\phi^{\prime}\in \G \\ c_K< d(\phi, \phi^{\prime}) \leq b_n}} Q_n(\phi')
        \\
        \overset{}{\leq} & \mu\Big(\Psi_n(c_K) + |\B_{b_n}|\sup_{\phi\in\G}Q_n(\phi)\Big),
    \end{align*}
    where to get (a) we use the definition of $\Psi$-mixing coefficient, and (b) is a result of the union bound: 
    $$
    \begin{aligned}
        \mathbb{E}\Big[\mathbb{I}(\sum_{\substack{\phi^{\prime}\in \G \\ c_K< d(\phi, \phi^{\prime}) \leq b_n}} f_{n}(\phi^\prime X_{n})\ge K-|\B_{c_K}|)\Big] \leq & \mathbb{P}\Big(\sum_{\substack{\phi^{\prime}\in \G \\ c_K< d(\phi, \phi^{\prime}) \leq b_n}} f_{n}(\phi^\prime X_{n})> 0 \Big) \\
        \leq & \sum_{\substack{\phi^{\prime}\in \G \\ c_K< d(\phi, \phi^{\prime}) \leq b_n}}\mathbb{P}\big(f_{n}(\phi^\prime X_{n})= 1\big).
    \end{aligned}
    $$
    % and (c) is because 
    % $$
    % \sum_{\phi\in \A_n}Q_n(\phi)\sum_{\substack{\phi^{\prime}\in \G \\ c_n< d(\phi, \phi^{\prime}) \leq b_n}} Q_n(\phi')\le \sum_{\phi\in \A_n}Q_n(\phi)|\B_{b_n}\setminus\B_{c_n}|\sup_{\phi\in\G}Q_n(\phi)\le \mu|\B_{b_n}|\sup_{\phi\in\G}Q_n(\phi).
    % $$
    Therefore, we have 
\begin{align*}
    \sup_n\sum_{k\ge K}k\lambda^{b_n}_{n}(k)& \le \sup_{n}\ind(|\B_{b_n}|>K)\mu\Big(\Psi_n(c_K) + |\B_{b_n}|\sup_{\phi\in\G}Q_n(\phi)\Big)
    \\&=\sup_{\substack{n ~\textrm{s.t.}~  |\B_{b_n}|\ge K}}\mu\Big(\Psi_n(c_K) + |\B_{b_n}|\sup_{\phi\in\G}Q_n(\phi)\Big)
\end{align*}
%We note that if $K<|B_{b_n}|$ then $\ind(|\B_{b_n}|>K)\mu\Big(\Psi(c_n) + |\B_{b_n}|\sup_{\phi\in\G}Q_n(\phi)\Big)=0$. If $K>|B_{b_n}|$ 
Note that by the definition of $c_K$, we have $c_K\xrightarrow{K\to\infty} \infty$, so by condition (ii), we get $\sup_n\Psi_n(c_K)\xrightarrow{K\to\infty} 0$. 
Moreover, it's important to note that as $K\to\infty$, the $n$ for which $|\B_{b_n}|>K$ also tends to infinity. Given that $|\B_{b_n}|\sup_{\phi\in\G}Q_n(\phi)\xrightarrow{n\to\infty}0$, which is ensured by our selection of $(b_n)_{n\in\N}$, it follows that $|\B_{b_n}|\sup_{\phi\in\G}Q_n(\phi)\longrightarrow0$ as $K\to\infty$.
    Hence, we have
    $$
    \sup_n\sum_{k\ge K}k\lambda^{b_n}_{n}(k) \longrightarrow 0, ~ \textrm{as } K\to \infty.
    $$
    As a result, by taking the limit on the right-hand side of \cref{lasvegas}, we find that
    $$
    \sum_{k=1}^{\infty} 
    \Big|k(\lambda(k)- k\lambda^{b_n}_{n}(k))\Big| \longrightarrow 0 \text{ as $n\to\infty$}.
    $$
Furthermore, by \cref{HTVboundnomono}, we have that 
\begin{equation*}
H_{0}(\lambda)\leq \left(\frac{1}{\lambda({1})} \wedge 1\right) e^{\sum_{k}\lambda(k)}.
\end{equation*}Note that 
$$
\begin{aligned}
\sum_{k\ge1}\lambda(k) \leq \sup_{n} \sum_{k \ge 1}k\lambda^{b_n}_n(k) = & \sup_{n}\sum_{k\ge 1}\sum_{\phi\in \A_n}\mathbb{E}\Big[f_{n}(\phi X_{n}) \mathbb{I}(\sum_{\substack{\phi^{\prime}\in \G \\ d(\phi, \phi^{\prime}) \leq b_n}} f_{n}(\phi^\prime X_{n})=k)\Big]\\
= & \sup_{n}\sum_{\phi\in \A_n}\mathbb{E}\big[f_{n}(\phi X_{n})\big] \\
\le & \mu.
\end{aligned}
$$
Hence $\sup_{g\in\mathcal{F}_{\TV}(\lambda)}\norm{g}_\infty \le H_{0}(\lambda) \le e^{\mu}$. 
    Therefore, combining the results, we have 
    \begin{equation*}
\dTV(Z(\lambda),Z(\lambda_{n}^{b_n})) 
 \leq e^\mu\sum_{k=1}^{\infty} 
\Big|k(\lambda(k)- k\lambda^{b_n}_{n}(k))\Big|
%&\le \sup_{g\in\mathcal{F}_{\TV}}\norm{g}_\infty\Big\{\sqrt{\epsilon_n} + 2\Big(\Psi(c_n)\sum_{\phi\in \A_n}Q_n(\phi) + \sum_{\phi\in \A_n}Q_n(\phi)\sum_{\substack{\phi^{\prime}\in \G \\ c_n< d(\phi, \phi^{\prime}) \leq b_n}} Q_n(\phi')\Big)\Big\}\\
%&\le \sup_{g\in\mathcal{F}_{\TV}}\norm{g}_\infty\Big\{\sqrt{\epsilon_n} + 2\mu\Big(\Psi(c_n) + |\B_{b_n}|\sup_{\phi\in\G}Q_n(\phi)\Big)\Big\},
\longrightarrow 0, \text{ as $n\to\infty$}.     
    \end{equation*}
\end{proof}

The ensuing lemma enumerates a collection of conditions that ensure the convergence of $W_n$ towards a $Z(\lambda)$ distribution.
\begin{lemma}\label{lighthouse}
    Suppose that the conditions of \cref{chowder} holds.
    Suppose further that the following conditions hold: 
    \begin{enumerate}[(i)]
    \item $\sup_n\sum_{j \ge b}|\B_{j+1}\backslash \B_{j}| \Psi_{n}(j) \longrightarrow 0, \text{ as $b\to\infty$};$
%    \item $\frac{\abs{\B_{2b_n}}}{\abs{\A_n}} \gamma_n(2b_n)  \longrightarrow 0$
    \item $\sup_n\sum_{j\geq 2b}\abs{\B_{j+1}\backslash \B_j}\xi_{n}^{b}(j) \longrightarrow 0,\text{ as $b\to\infty$}.$
    \end{enumerate}
    Then the following holds:
    $$
\dTV(W_n,Z(\lambda)) 
\longrightarrow 0, \text{ as $n\to\infty$.} 
    $$
\end{lemma}
\begin{proof} Let $(b_n)_{n\in\N}$ be a sequence satisfying the conditions of \cref{chowder} as well as  $\frac{|\B_{2b_n}|}{|\A_n|}\xrightarrow{n\to\infty} 0$
 and $\frac{|\A_n\B_{b_n}\triangle\A_n|}{|\A_n|}\xrightarrow{n\to\infty} 0$. By the triangle inequality, we have 
    \begin{equation}\label{dice}
    d_{\TV}(W_n, Z(\lambda))\leq d_{\TV}(W_n, Z(\lambda_{n}^{b_n})) + d_{\TV}(Z(\lambda_{n}^{b_n}), Z(\lambda)). 
    \end{equation}
    For the first term on the right-hand side of \cref{dice}, note that in the proof of \cref{chowder}, we have shown that 
    $H_0(
    \lambda^{b_n}_{n}
    ),H_1(
    \lambda^{b_n}_{n}
    )$ are bounded by $e^\mu$. Furthermore, $\mu_n$ and $\eta_{n,1}$ are also bounded by $\mu$. Hence conditions (i) and (ii) ensures that conditions (i) and (iii) of \cref{maine} are satisfied. Moreover, condition (iv) of \cref{chowder} guarantees that condition (ii) of \cref{maine} is satisfied. This implies that $d_{\TV}(W_n, Z(\lambda_{n}^{b_n}))\xrightarrow{n\to\infty} 0$. For the other term on the right-hand side of \cref{dice}, under the conditions of \cref{chowder}, we know that $d_{\TV}(Z(\lambda_{n}^{b_n}), Z(\lambda))\xrightarrow{n\to\infty}0$. The result follows immediately.     
%    Therefore, we obtain that $d_{\TV}(W_n, Z(\lambda)) \longrightarrow 0$.
\end{proof}

Now we provide the proofs for all the corollaries in \cref{app1,app2,app3,app3.5}. In particular, \cref{cor_rf} and \cref{cor_hmrf} are consequences of \cref{lighthouse}.

\begin{proof}[Proof of \cref{cor_rf}]
Conditions (i), (ii), and (iii) of \cref{cor_rf} ensure that the requirements of \cref{chowder} are met. Additionally, conditions (i) and (ii) of \cref{lighthouse} are satisfied due to condition (iii) of \cref{cor_rf}. Therefore, by invoking \cref{lighthouse}, we arrive at the conclusion that $\dTV(W_n,Z(\lambda))\xrightarrow{n\to\infty} 0$.
\end{proof}

\begin{proof}[Proof of \cref{cor_hmrf}]
    The conditions of \cref{cor_hmrf} imply that all the conditions of \cref{chowder} are satisfied. Moreover, because of the conditions of \cref{cor_hmrf}, it follows that
    $$
    \begin{aligned}
    &\sup_n~\sum_{j\ge b}j^{m-1}\Psi_n(j) \le \sum_{j \ge b}c_1 j^{m-1} e^{-c_2j} \rightarrow 0 \text{ as $b\to\infty$},\\
    &\sup_n~\sum_{j\ge b}j^{m-1}\Psi_n(j) \le \sum_{j \ge b}c_1 j^{m-1} e^{-c_2j} \rightarrow 0 \text{ as $b\to\infty$}.
    \end{aligned}
    $$
    The rest follows from \cref{lighthouse}.
\end{proof}

\begin{proof}[Proof of \cref{cor_exchangeableseq}]
By the triangle inequality, it follows that for any sequence of positive integers $(b_n)_{n\in\N}$ such that $\frac{b_n}{n}\xrightarrow{n\to\infty}0$, the following holds:
\begin{equation}\label{tatte}
\begin{aligned}
&\E\Big[\dTV(W_n, \mathrm{Poisson}(\Theta)|\mathbb{S}_\infty)\Big]\\
\le& \E\Big[\dTV(W_n,Z(\lambda_{n}^{b_n})|\mathbb{S}_\infty)\Big] + \E\Big[\dTV(\mathrm{Poisson}(\Theta),Z(\lambda_{n}^{b_n})|\mathbb{S}_\infty)\Big].  
\end{aligned}
\end{equation}
It is sufficient to demonstrate that both terms on the right hand side of \cref{tatte} converge to 0.

    We start by demonstrating the convergence to 0 for the first term.  By de Finetti's Theorem, the $X_i$'s are conditionally i.i.d. under 
    $\sigma(\mathbb{S}_\infty)$,  and hence we have $\Psi_n(k|\mathbb{S}_\infty)=0$ and $\xi_n(k|\mathbb{S}_\infty)=0$ for all $k\geq1$. 
%    Note that given the condition $\sup_{n}\|n\prob(X^n_0=1| \Z)\|_\infty\le\mu$, we have $\mu_{n,1}$, $\eta_{n,1}$, and $\mu^{b_n}_{n}$ are all uniformly bounded.
%    $$\mu_{n,1} = (2n+1)\|\prob(X^n_0=1| \Z)\|_\infty \le 3\mu$$
%    $$\eta_{n,1}=(2n+1)\|\prob(X^n_0=1| \Z)\|_1\le 3\mu$$
%    $$\mu^{b_n}_{n} = (2n+1)^2\|\prob(X^n_0=1| \Z)\|_1^2\le 9\mu^2$$
    Given the condition, there exists $\mu$ such that $\sup_{n}\|n\prob(X^n_0=1| \mathbb{S}_\infty)\|_\infty\le\mu$, and it is important to note that $\mu_{n,1}$, $\eta_{n,1}$, $\mu^{b_n}_{n}$, $\|H_0(
    \lambda^{b_n}_{n}
    )\|_\infty$, and $\|H_1(
    \lambda^{b_n}_{n}
    )\|_\infty$ are all uniformly bounded.  
    Therefore, all the conditions of \cref{maine} are satisfied. Hence, we have 
    $$
    \E\Big[\dTV(W_n,Z(\lambda_{n}^{b_n})|\mathbb{S}_\infty)\Big]\longrightarrow 0.
    $$
    % For $k>1$, we have
    % $$
    % \begin{aligned}
    % k\lambda^{b_n}(k)&= (2n+1)\prob\Big(X_n(1)=1 \text{ and } \sum_{|j|\le b_n}X_n(j)=k\Big|\Z\Big) \\
    % &= (2n+1)\prob(X_n(1)=1|\Z)\prob\Big(\sum_{j\le b_n}X_n(j)=k-1\Big|\Z\Big) \\
    % &\overset{(a)}{\le} (2n+1)\prob(X_n(1)=1|\Z)2b_n\prob(X_n(1)=1|\Z) \\
    % &\le (2n+1)2\log n \prob(X_n(1)=1|\Z)^2\\
    % &\longrightarrow 0~ a.s.
    % \end{aligned}
    % $$
    % where $(a)$ is a result of the union bound.
    % For $k=1$, we have 
    % $$
    % \begin{aligned}
    % \lambda_1^{b_n}&= (2n+1)\prob\Big(X_n(1)=1 \text{ and } \sum_{|j|\le b_n}X_n(j)=1\Big|\Z\Big) \\
    % &= (2n+1)\prob(X_n(1)=1|\Z)\prob\Big(\sum_{j\le b_n}X_n(j)=0\Big|\Z\Big) \\
    % &\overset{(b)}{=} (2n+1)\prob(X_n(1)=1|\Z)\Big(1-\prob(X_n(1)=1|\Z)\Big)^{2b_n} \\
    % \end{aligned}
    % $$
    % where $(b)$ is a result of the conditional independence. Moreover, we have that 
    % $$
    % \Big(1-\prob(X_n(1)=1|\Z)\Big)^{2b_n} \ge \Big(1-\frac{1}{n}\Big)^{2\log n} \longrightarrow 1 \text{ as $n\to\infty$}
    % $$
    % Therefore, $\lambda_1^{b_n} \overset{a.s.}{\longrightarrow} \lambda$.

Next, we focus on proving the convergence to 0 of the second term on the right-hand side of \cref{tatte}. Denote $\lambda:\N\to[0,\infty)$, where $\lambda(1):=\Theta$ and $\lambda(k):=0$ for $k\ge2$. Note that we have $\mathrm{Poisson}(\Theta)\overset{d}{=}Z(\lambda)$. Using \cref{SteinCPboundTV}, we obtain that 
$$
\begin{aligned}
    \E\Big[\dTV(\mathrm{Poisson}(\Theta),Z(\lambda_{n}^{b_n})|\mathbb{S}_\infty)\Big]\leq \Big\|\sup_{g\in\mathcal{F}_{\TV}(\lambda)}|g|\Big\|_\infty\Big\|\sum_{k=1}^{\infty} 
\Big|k(\lambda(k)- k\lambda^{b_n}_{n}(k))\Big|~\Big\|_1.
\end{aligned}
$$
By the triangle inequality, we have 
\begin{equation}\label{waypoint}
\begin{aligned}
\Big\|\sum_{k=1}^{\infty}\Big|k(\lambda(k)- k\lambda^{b_n}_{n}(k))\Big|~\Big\|_1\le& \Big\|\Theta- \lambda^{b_n}_{n}(1)\Big\|_1 + \Big\|\sum_{k=2}^{\infty}k\lambda^{b_n}_{n}(k)\Big\|_1\\
= & (R_1) + (R_2).
\end{aligned}
\end{equation}
We bound the two terms $(R_1)$ and $(R_2)$ on the right hand side. For the first term $(R_1)$, we can obtain that 

\begin{align*}
(R_1)=&\Big\|\Theta- \lambda^{b_n,n}_1\Big\|_1\\
=& \Big\|n\prob\Big(X_n(1)=1 \text{ and } \sum_{j\le b_n}X_n(j)=1\Big|\mathbb{S}_\infty\Big)-\Theta\Big\|_1\\
\overset{(a)}{=}& \Big\|n\prob(X_n(1)=1|\mathbb{S}_\infty)\prob\Big(\sum_{j\le b_n}X_n(j)=0\Big|\mathbb{S}_\infty\Big)-\Theta\Big\|_1\\
\le& \Big\|n\prob(X_n(1)=1|\mathbb{S}_\infty)\prob\Big(\sum_{j\le b_n}X_n(j)=0\Big|\mathbb{S}_\infty\Big)-n\prob(X_n(1)=1|\mathbb{S}_\infty)\Big\|_1   \\
& + \Big\|n\prob(X_n(1)=1|\mathbb{S}_\infty)-\Theta\Big\|_1\\
=:&~ (L_1)+(L_2).
\end{align*}
where $(a)$ follows from the fact that $X^n_t$'s are conditionally i.i.d..
We obtain that
\begin{align*}
    (L_1)&= \Big\|n\prob(X_n(1)=1|\mathbb{S}_\infty)\prob\Big(\sum_{j\le b_n}X_n(j)=0\Big|\mathbb{S}_\infty\Big)-n\prob(X_n(1)=1|\mathbb{S}_\infty)\Big\|_1 \\
    &= \Big\|n\prob(X_n(1)=1|\mathbb{S}_\infty)\Big(1-\prob(X_n(1)=1|\mathbb{S}_\infty)\Big)^{b_n} -n\prob(X_n(1)=1|\mathbb{S}_\infty)\Big\|_1\\
    &= \Big\|n\prob(X_n(1)=1|\mathbb{S}_\infty)\Big(1-\Big(1-\prob(X_n(1)=1|\mathbb{S}_\infty)\Big)^{b_n}\Big)\Big\|_1\\
    &\le  \Big(1-\Big(1-\frac{\mu}{n}\Big)^{b_n}\Big) \Big\|n\prob(X_n(1)=1|\mathbb{S}_\infty)\Big\|_1\\
    & \longrightarrow 0 \text{ as $n\to\infty$}, 
\end{align*}
where the convergence is because $\Big(1-\frac{\mu}{n}\Big)^{b_n}\longrightarrow 1$ when $\frac{b_n}{n}\to0$, and $\Big\|n\prob(X_n(1)=1|\mathbb{S}_\infty)\Big\|_1$ is bounded. 
Moreover, we have
\begin{align*}
(L_2)&=  \Big\|n\prob(X_n(1)=1|\mathbb{S}_\infty)-\Theta\Big\|_1\\
%&\le  \Big\|\prob(X_n(1)=1|\mathbb{S}_\infty)\Big\|_1 + \Big\|2n\prob(X_n(1)=1|\mathbb{S}_\infty)-\lambda\Big\|_1\\
&\longrightarrow 0 \text{ as $n\to\infty$.}
\end{align*}
Therefore, we conclude that $$(R_1)\longrightarrow 0 \text{ as $n\to\infty$}.$$
To bound the second term $(R_2)$, we note that 
\begin{align*}
   (R_2)= \Big\|\sum_{k=2}^{\infty}k\lambda^{b_n}_{n}(k)\Big\|_1 &= \Big\|\sum_{k=2}^{\infty} (2n+1)\prob\Big(X_n(1)=1 \text{ and } \sum_{|j|\le b_n}X_n(j)=k\Big|\mathbb{S}_\infty\Big) \Big\|_1\\
    &= \Big\|(2n+1)\prob\Big(X_n(1)=1 \text{ and } \sum_{|j|\le b_n}X_n(j)\ge 2\Big|\mathbb{S}_\infty\Big) \Big\|_1\\
    &=  \Big\|(2n+1)\prob(X_n(1)=1|\mathbb{S}_\infty)\prob\Big(\sum_{j\le b_n}X_n(j)>0\Big|\mathbb{S}_\infty\Big)\Big\|_1\\
    & \overset{(a)}{\le} \Big\|(2n+1)\prob(X_n(1)=1|\mathbb{S}_\infty)2b_n\prob(X_n(1)=1|\mathbb{S}_\infty)\Big\|_1\\
    &\longrightarrow 0 \text{ as $n\to\infty$},
\end{align*}
where $(a)$ is a result of the union bound. Therefore, both terms on the ride-hand side of \cref{waypoint} converge to 0, so it follows that
$$
\Big\|\sum_{k=1}^{\infty}\Big|k(\lambda(k)- k\lambda^{b_n}_{n}(k))\Big|~\Big\|_1\longrightarrow 0, \text{ as $n\to\infty$}.
$$
Moreover, since $\|n\prob(X^n_0=1| \mathbb{S}_\infty)\|_\infty$ is uniformly bounded, $\|\Theta\|_\infty$ is also bounded almost surely. Therefore, by \cref{HTVboundnomono}, we conclude that 
$$
\Big\|\sup_{g\in\mathcal{F}_{\TV}(\lambda)}|g|\Big\|_\infty 
\le 
\Big\|H_1(\lambda)\Big\|_\infty
\le e^{\|\Theta\|_\infty} < \infty \text{ a.s.}.
$$
Then we get $$ \E\Big[\dTV(\mathrm{Poisson}(\Theta),Z(\lambda_{n}^{b_n})|\mathbb{S}_\infty)\Big]\leq \Big\|\sup_{g\in\mathcal{F}_{\TV}(\lambda)}|g|\Big\|_\infty\Big\|\sum_{k=1}^{\infty} 
\Big|k(\lambda(k)- k\lambda^{b_n}_{n}(k))\Big|~\Big\|_1\xrightarrow{n\to\infty} 0.$$
Therefore, we conclude that 
\begin{equation*}
\begin{aligned}
&\E\Big[\dTV(W_n, \mathrm{Poisson}(\Theta)|\mathbb{S}_\infty)\Big]\\
\le & \E\Big[\dTV(W_n,Z(\lambda_{n}^{b_n})|\mathbb{S}_\infty)\Big] + \E\Big[\dTV(\mathrm{Poisson}(\Theta),Z(\lambda_{n}^{b_n})|\mathbb{S}_\infty)\Big]\xrightarrow{n\to\infty} 0.
\end{aligned}
\end{equation*}
\end{proof}

\begin{proof}[Proof of \cref{cor_cayley}]%I could complete the subset 
To prove this theorem, we will use \cref{result1}. We will check that the conditions are respected. We do so by choosing $b_n:=d_n+1$. To verify the necessary conditions, we first note that $\mathcal{C'}$ is an ergodic object. Secondly, we remark that the edges are deleted independently of each others. Therefore, as for all $\phi',\phi\in \mathbb{G}$ that are such that $d(\phi',\phi)>2d_n$ we have $\B(\phi,d_n)\bigcap \B(\phi',d_n)=\emptyset$, we directly obtain that $$\Psi_{n,\infty}(i)=\xi_{n,\infty}^{b_n}(i)=0,\qquad 
\forall i>d_n.$$ Hence condition \eqref{A4} holds. Moreover, we remark that if $(\phi,\phi')\in \mathcal{E}$ is an edge of the Cayley graph, it means that there exists a unique $g\in \mathcal{S}$ such that $\phi=g\phi'$. Therefore, we obtain that the number of edges in $\mathcal{G}_n$ satisfies the following inequality: $$
|\mathcal{G}_n|\ge \frac{1}{2}\big[|\mathcal{S}| |\mathbf{B}_{d_n}|-(|\mathcal{S}|-1)|\mathbf{B}_{d_n}\setminus\mathbf{B}_{d_n-1}|\big].
$$ Consequently,
for all $\phi\in \mathbb{G}$, we have $$\prob(Y_e=1)=\prob(Y_\phi=1)=p^{|\mathcal{G}_n|}\le p^{\frac{1}{2}\big[|\mathcal{S}| |\mathbf{B}_{d_n}|-(|\mathcal{S}|-1)|\mathbf{B}_{d_n}\setminus\mathbf{B}_{d_n-1}|\big]}\le\frac{1}{|A_n|}.$$ Hence
condition \eqref{A1}, \eqref{A2}, and \eqref{A3} are satisfied, and thus the conditions of \cref{result1} hold. We will determine that $\sum_{k\ge 2}k\lambda_n(k)\rightarrow0.$ In this goal we first note that for all $\phi\ne e$ we have
\begin{align*}
    \mathbb{E}[Y_\phi Y_e]&=p^{|\mathcal{G}_n|} \prob(Y_\phi =1 |Y_e =1)
    \\& =p^{|\mathcal{G}_n|} p^{|\mathbf{B}(\phi,d_n)\setminus \mathbf{B}(e,d_n)|}
   \\& =\omega\Big(\frac{1}{|\mathbf{A}_n||\mathbf{B}_{d_n+1}|}\Big)
\end{align*}Moreover note that 
\begin{align*}
    \mathbb{P}\Big(Y_e=1~\rm{and}~\sum_{\phi\in \mathbf{B}_{b_n}}Y_\phi\ge 2\Big)&\overset{(a)}{\le}\sum_{\phi\in \mathbf{B}_{d_n+1}\setminus \{e\}}\mathbb{E}[Y_\phi Y_e]=o_n(1)
    % \\&\overset{}{\le} p^{\frac{1}{2}|\mathcal{S}| \times |\mathbf{B}_{d_n}|-|\mathbf{B}_{d_n+1}\setminus \mathbf{B}_{d_n}|}\sum_{\phi\in B_{d_n+1}\setminus \{e\}}p^{|\mathcal{S}||\phi B(\phi\setminus B_{d_n}|}
    % \\&\le  \frac{1}{|A_n|}\sum_{\phi\in B_{d_n+1}\setminus \{e\}}p^{|\mathcal{S}||\phi B_{d_n}\setminus B_{d_n}|}
\end{align*} where to get $(a)$ we used the union-bound.  This directly implies that $\sum_{k\ge 2}k\lambda_n(k)\rightarrow 0$. The desired result is therefore established.
%Now note that for all $\phi\in \mathbb{G}$ we have $|\phi B_{d_n}\setminus B_{d_n}|:=\{\phi'\in\}$
\end{proof}

\subsection{Proof of Theorem \ref{MtKatahdin}}\label{proof_gen_thm}
\subsubsection{Notation}
In the following subsection, we introduce the notations employed consistently within the subsequent proofs.
\begin{itemize}
    \item Denote $Q_n(\phi):=\prob(f_n(\phi X_n)= 1| \phi,\mathbb{G})$.
 %   Note that for all $p>0$, for all $\le$, $\|\prob(f_n(\phi_{ni} X_n)= 1| \phi_{ni},\mathbb{G})\|_p \le \sup_{\phi\in\A_n}\frac{\|\mu^{\phi}_n\|_p}{\abs{\A_n}}$
%    \item Let $J_n$ represent the quantity of randomly selected samples, where $J_n$ is defined as a random variable assuming values in $\N$;
    \item Denote the random samples to be
    $\boldsymbol\phi_n := (\phi_{n1},\cdots,\phi_{nJ_n})$.
    \item Recall that in \cref{sec:mainresult}, we define the residue terms
\begin{equation*}
  \mathcal{R}_\Psi(n,p,b)\;:=\;\msum_{i \ge b}|\mathbf{B}_{i+1}\backslash\mathbf{B}_{i}|\; \Psi_{n,p}(\,i\,| \G);
\end{equation*}
\begin{equation*}
  \mathcal{R}_\xi(n,p,b)\;:=\;\msum_{i \ge 2b}|\mathbf{B}_{i+1}\backslash\mathbf{B}_{i}|\; \xi_{n,p}^b(\,i\,| \G).
\end{equation*}
    \item For all $b>0$ and $i \le J_n$, we define the events $E^{n,k}_{b, i}$ as follows: $$E^{n,k}_{b, i}:=\Big\{\sum_{\substack{i^{\prime}\leq J_n\\ d(\phi_{ni}, \phi_{ni^{\prime}}) \leq b}} f_{n}(\phi_{ni^\prime} X_{n})=k\Big\}.$$
    %$$\tilde E^{n,k}_{b, i}:=\Big\{\sum_{\substack{i^{\prime}\leq j_n \\ d(\phi_{ni}, \phi_{ni^{\prime}}) \leq b}} f_{n}(\phi_{ni^\prime} X_{n})=k\Big\}.$$
    Note that those events are random events dependent on the random variables $J_n$ and $(\phi_{ni})_{i=1,\cdots,J_n}$.
    %$$\tilde E^{n,k}_{b,\phi}:=\Big\{\sum_{\substack{i^{\prime}\leq j_n \\ d(\phi_{n1}, \phi_{ni^{\prime}}) \leq b}} f_{n}(\phi^\prime X_{n})=k\Big\}.$$
    \item 
    For all $b>0$, we write 
    $$ k\lambda^{b}_{n}(k) := \mathbb{E}\Big[\sum_{i \leq J_n} f_{n}(\phi_{ni} X_{n}) \mathbb{I}(E^{n,k}_{b,i})\Big|\G\Big]; $$
    $$k\widetilde{\lambda}^{b}_{n}(k) := \mathbb{E}\Big[\sum_{i\leq J_n}f_{n}(\phi_{ni} X_{n}) \mathbb{I}(E^{n,k}_{b,i})\Big|J_n, \boldsymbol\phi_{n},\G\Big].$$
     We denote $\lambda^{b}_{n}:\N\to[0,\infty),~k\mapsto\lambda^{b}_{n}(k)$ and $\widetilde{\lambda}^{b}_{n}:\N\to[0,\infty),~k\mapsto\widetilde{\lambda}^{b}_{n}(k)$.
     \item For all $j\le J_n$ and all integer $j\in \mathbb{N}$, we denote $$W_{j}^{\phi_{ni}}:=\sum\limits_{\substack{i' \le J_n\\d(\phi_{ni},\phi_{ni'})>j}}f_n(\phi_{ni'} X_n).$$ As a reminder, we also write $W_n:=\sum\limits_{i'\le J_n}f_n(\phi_{ni'} X_n)$.

%    \item Set for all $b\ge 1$ the following quantities $$\mathcal{S}_{b,1}^n:=\Big\|\int_{\B(\phi_{n1},b)}\|\mu_n^\phi \|_2d\nu_n(\phi)\Big\|_2,\qquad \mathcal{S}_{b,2}^n:=\mathbb{E}\Big(\int_{\B(\phi_{n1},b_n)}\|\mu_n^\phi\|^2_2d\nu_n(\phi)\Big).$$
%     \item Define $\mathbf{B}^{j_n}_b(\phi_{nk}):=\{\phi_{ni}: i\leq j_n, d(\phi_{nk},\phi_{ni})\leq b\}$ 
%\item For all $a,b>0$ and $p\ge1$, we define the mixing coefficients remainder terms:
%$$\mathcal{R}_{n,p}^{\Psi}(b):=\sum_{j=b}^{\infty}\big|\B_{j+1}\setminus \B_j\big|\Psi_{n,p}(j|\G),$$
%$$\mathcal{R}_{n,p}^{\xi,a}(b):=\sum_{j=b}^{\infty}\big|\B_{j+1}\setminus \B_j\big|\xi^{a}_{n,p}(j|\G).$$
%     \item 
%     For ease of notations we will write for all $s\ge 1$ and all $b>0$ $$\mu_{n,s}:=\frac{1}{|\A_n|}\int_{\A_n}\|\mu_n^{\phi}\|_{\frac{s}{s-1}}d\nu_n(\phi),\qquad\mu_{n}:=\overline{\mu^{\A_n}_{n,\infty}},\qquad \mu^{b}_n:=\frac{1}{|\A_n||\B_{b}|}\sum_{\substack{i\leq j_n\\\phi_{ni'}\in \B_{b}(\phi_{ni})}} \|\mu_n^{\phi_{ni}}\mu_n^{\phi_{ni'}}\|_1$$$$\eta_{n,s}:=\Big[ \frac{1}{|\A_n|}\sum_{i\leq j_n}\|\mu_n^{\phi_{ni}}\|^s_1\Big]^{1/s}$$
\item For the simplicity of notation, we denote 
$$
\mathcal{J}_n := \{J_n, \boldsymbol\phi_n,\G\},
$$
$$\mathcal{L}_n := \{J_n,\widetilde{{\lambda}}^{b_n}_{n},\G\}.$$

\end{itemize}

\subsubsection{Auxiliary results}
In this section, we present lemmas that serve as instrumental components in the proof of \cref{MtKatahdin}.

First, we reiterate the Efron-Stein inequality, a pivotal tool that offers an upper bound for the variance of a function. 
\begin{theorem}[Efron-Stein Inequality \cite{10.1214/aos/1176345462} ]
Suppose that $(X_{i})_{i\in\N}$ is a sequence of independent random variables taking values in a space $\mathcal{X}$, and let $(X'_i)_{i\in\N}$ be an independent copy of it. Let $n$ be a positive integer. For the ease of notation, we write $X=\left(X_{1}, \ldots, X_{n}\right) \text{ and } X^{(i)}=\left(X_{1}, \ldots, X_{i-1}, X_{i}^{\prime}, X_{i+1}, \ldots, X_{n}\right)$. Let $f:\mathcal{X}^n\rightarrow\mathbb{R}$ be a measurable function. Then the following inequality holds:
$$
\operatorname{Var}(f(X)) \leq \frac{1}{2} \sum_{i=1}^{n} \E \Big[\Big(f(X)-f(X^{(i)})\Big)^{2}\Big].
$$
If, in addition, $f$ is assumed to be symmetric in the entries, then the following holds:
$$
\operatorname{Var}(f(X)) \leq \frac{n}{2} \E\Big[\Big(f(X)-f(X^{(1)})\Big)^{2}\Big].
$$
\end{theorem}
\begin{proof}
    The proof for this statement can be found in the work by Efron and Stein \cite{10.1214/aos/1176345462}.
\end{proof}

The following lemma states that if $\Big\|\mathbb{E}\Big[\sum_{i \leq J_n} f_{n}(\phi_{ni} X_{n}) \Big|\G\Big]\Big\|_\infty$ is uniformly bounded for all $n\in\N$, then $\|H_{0}(\lambda^{b_n}_{n})\|_\infty$ and $ \|H_{1}(\lambda^{b_n}_{n})\|_\infty$ are also uniformly bounded for all $n\in\N$. 
\begin{lemma}\label{style}
Suppose that there exists $\mu\in\mathbb{R}$ such that  $\sup_{n\in\N}\big\|\sum_{i \leq J_n} Q_{n}(\phi_{ni})\big\|_\infty\le \mu.$ %$\sup_{n\in\N}\Big\|\mathbb{E}\Big[\sum_{i \leq J_n} f_{n}(\phi_{ni} X_{n}) \Big|\G\Big]\Big\|_\infty\le \mu.$ 
Then the following holds: 
    $$\|H_{0}(\lambda^{b_n}_{n})\|_\infty, \|H_{1}(\lambda^{b_n}_{n})\|_\infty \le e^\mu, \text{ for all $n\in\N$}.$$
\end{lemma}
\begin{proof}
By \cref{HTVboundnomono}, we have that 
\begin{equation*}
H_{0}(\lambda^{b_n}_{n}), H_{1}(\lambda^{b_n}_{n})\leq \left(\frac{1}{\lambda^{b_n}_{n}(1)} \wedge 1\right) e^{\sum_{k}\lambda^{b_n}_{n}(k)}.
\end{equation*}
We note that 
$$
\begin{aligned}
\sum_{k\ge1}\lambda^{b_n}_{n}(k) \leq \sum_{k \ge 1}k\lambda^{b_n}_{n}(k) = & \sum_{k\ge 1}\mathbb{E}\Big[\sum_{i \leq J_n} f_{n}(\phi_{ni} X_{n}) \mathbb{I}(E^{n,k}_{b_n,i})\Big|\G\Big]\\
\overset{(a)}{=} &~ \mathbb{E}\Big[\sum_{i \leq J_n} f_{n}(\phi_{ni} X_{n}) \Big|\G\Big]
\end{aligned}
$$
where $(a)$ follows from the fact that for any $i\le J_n$, the sets $E^{n,k}_{b_n,i}$ for $ k\in \Z^+$ are disjoint. \\
Since $\sup_{n\in\N}\Big\|\mathbb{E}\Big[\sum_{i \leq J_n} f_{n}(\phi_{ni} X_{n}) \Big|\G\Big]\Big\|_\infty\le \mu$, we get   $\|H_{0}(\lambda^{b_n}_{n})\|_\infty, \|H_{1}(\lambda^{b_n}_{n})\|_\infty \le e^\mu \text{ for all $n\in\N$}.$
\end{proof}

The following lemma shows how $\xi$-mixing  can help control the correlation between key terms.
\begin{lemma}[Conditional $\xi$-mixing bound]\label{lemma26general}
Let $(X_n)_{n\in\N}$ be a sequence of random elements taking values in $\mathcal{X}$. Let $(f_n)_{n\in\N}$ be a sequence of measurable functions, where $f_n: \mathcal{X} \to\{0,1\}$. Suppose the probability measure $\nu_n$ satisfies the well-spread condition with the constant $\mathcal{S}$.
%Define $\nu_n$ as a probability measure on $\A_n$, which adheres to the well-spread condition as specified in \cref{well-spread}. Let the constant in the well-spread condition be denoted as $\mathcal{S}$. Furthermore, $J_n$ is described as a random variable that takes values in $\N$, with $\phi_{n1}, \phi_{n2},\cdots, \phi_{nJ_n}$ being independently and identically distributed according to $\nu_n$.
Let $b>0$ be an integer.
For a fixed $\A_n \subseteq \G$ and for all $i\le J_n$, we write $$h_{n,k}^{i}:= f_n(\phi_{ni} X_n)\ind(E^{n,k}_{b,i}).$$ Denote $\mathcal{F}:=\mathcal{F}_{\TV}({\lambda}^{b}_{n})$. Then for all $j\ge b$,  the following holds for all $p\ge1$: 
$$
\begin{aligned}
& \bigg\|\sup_{g\in \mathcal{F}}\Big|\sum_{i\leq J_n}\sum_{k=1}^{\infty}\E\Big[\big(h_{n,k}^{i}-\E[h_{n,k}^{i}| J_n, \boldsymbol\phi_n,\mathbb{G}]\big)\big(g(W^{\phi_{ni}}_{j}+k)-g(W^{\phi_{ni}}_{j+1}+k)\big)\Big|J_n, \boldsymbol\phi_n,\mathbb{G}\Big]\Big|\bigg\|_1\\
\leq
&2 \Big\|\sup_{g\in \mathcal{F}}|\Delta g|\Big\|_{\infty}~\sup_{\phi\in \A_n}\|Q_n(\phi)\|_{\frac{p}{p-1}}\frac{\|J_n\|^2_2\mathcal{S}}{|\A_n|}|\B_{j+1}\setminus\B_{j}|~\xi_{n,p}^{b}(j|\G).
\end{aligned}
$$
\end{lemma}

\begin{proof}  Before diving into the proof we introduce a few notations. Firstly for all $g\in \mathcal{F}$, we write $\Delta_{j,k}^{i}(g,X_n,\boldsymbol{\phi}_n,J_n):=g(W^{\phi_{ni}}_{j}+k)-g(W^{\phi_{ni}}_{j+1}+k)$ and denote 
$$\|\Delta_{j,k}^{i}(g,~\cdot~,\boldsymbol{\phi}_n,J_n)\|_{\infty}:=\sup_{x\in \mathcal{X}}~\Big|\Delta_{j,k}^{i}(g,x,\boldsymbol{\phi}_n,J_n)\Big|.$$ Whenever there is no ambiguity we shorthand $\Delta_{j,k}^{i}(g):=\Delta_{j,k}^{i}(g,X_n,\boldsymbol{\phi}_n,J_n).$ We write $W^{\phi_{ni}}_{j:j+1}$ to designate the pair $(W^{\phi_{ni}}_{j},W^{\phi_{ni}}_{j+1})$.

\noindent Firstly, it will be useful to remark that as $f_n$ takes value in $\{0,1\}$, this directly implies that $$\sum\limits_{\substack{i'\leq J_n\\d(\phi_{ni},\phi_{ni'})\le b}}f_n(\phi_{ni'}X_n)\overset{a.s.}{\le }J_n.$$ %Therefore, we know that given $J_n$, $h_{n,k}^{i}=0$ for all $k>J_n$. 

\noindent In addition, recall that for the simplicity of notation, we denote $\mathcal{J}_n := \{J_n, \boldsymbol\phi_n,\G\}$.

\noindent We remark that, by the definition of $\Delta_{j,k}^{i}(g)$, we have
$$
\begin{aligned}
&\bigg\|\sup_{g\in \mathcal{F}}\Big|\E\Big[\sum_{i\leq  J_n}\sum_{k=1}^{\infty}\Big(h_{n,k}^i-\E[h_{n,k}^i| \Jn]\Big)~\Big(g(W^{\phi_{ni}}_{j}+k)-g(W^{\phi_{ni}}_{j+1}+k)\Big)\Big|\Jn\Big]\Big|\bigg\|_1\\
= &\bigg\|\sup_{g\in \mathcal{F}}\Big|\E\Big[\sum_{i\leq  J_n}\sum_{k=1}^{\infty}
\Big(h_{n,k}^i-\E[h_{n,k}^i|\Jn]\Big)~\Delta_{j,k}^{i}(g)\Big|\Jn\Big]\Big|\bigg\|_1\\
% \le  &\norm{\sup_{g\in \mathcal{F}}\Big|\E \Big[ \sum_{k=1}^{\infty}
% h_{n,k}^\phi\Delta_j^k|\mathbb{G}\Big]\Big|
% +\Big\sum_{k=1}^{\infty}\E[h_{n,k}^\phi|\mathbb{G}]\E[\Delta_j^k|\mathbb{G}\big]}\\
= &\bigg\|\sup_{g\in \mathcal{F}}\Big|\E\Big[\sum_{i\leq  J_n}\sum_{k=1}^{\infty}h_{n,k}^i\Big(\Delta_{j,k}^{i}(g)-\E[\Delta_{j,k}^{i}(g)|\Jn]\Big)\Big| \Jn\Big]\Big|\bigg\|_1.
\end{aligned}
$$
%For simplicity of notation, denote $\ind(C^k_n) := h_{n,k}^\phi$\\
We note that for all $g\in\mathcal{F}$, we have that the function $x\mapsto \Delta_{j,k}^{i}(g, x,\boldsymbol{\phi}_n,J_n)$ is measurable. Therefore, for a fixed $\boldsymbol{\phi}_n,J_n$, we know that for every $\epsilon>0$ there is an integer $N_{\epsilon}^g$, disjoint measurable sets $(A_{l,k}^{i,g})_{l\le N^g_{\epsilon}}$, and scalars $(a_{l,k}^{i,g})_{l\le N^g_{\epsilon}}$ such that  $$|a_{l,k}^{i,g}|\leq \|\Delta_{j,k}^{i}(g,\cdot,\boldsymbol{\phi}_n,J_n)\|_{\infty},  $$and such that if we define
$$\Delta_{j,k}^{i,*} (g)= \sum_{l=1}^{N^g_{\epsilon}}a_{l,k}^{i,g}\ind\big(W^{\phi_{ni}}_{j:j+1}\in A_{l,k}^{i,g}\big),$$   
then the following holds: $$\Big\|\sup_{g\in \mathcal{F}}|\Delta_{j,k}^{i}(g)-\Delta_{j,k}^{i,*} (g)|\Big\|_1\le \epsilon.$$

\noindent Note that since $h_{n,k}^{i}=0$ for all $k>J_n$ and $\|h_{n,k}^i\|_{\infty}\le 1$, 
we have 
\begin{align*}
    &\bigg\|\sup_{g\in \mathcal{F}}\Big|\sum_{i\leq J_n}\sum_{k=1}^{\infty} h_{n,k}^{i}\Big[\Delta_{j,k}^{i}(g)-\mathbb{E}[\Delta^{i}_{j,k}(g)|\Jn]-\Big(\Delta_{j,k}^{i,*}(g)-\mathbb{E}[\Delta_{j,k}^{i,*}(g)|\Jn]\Big)\Big]\Big|\bigg\|_1\\
    \le& 2\|J_n\|^2_2\epsilon.
\end{align*}

\noindent Finally, we abbreviate
%$$I^k_l:=\ind(C_n)\left(\ind\big((W^{\phi_{n}_j,W^{\phi,n}_{j+1})\in A_{l,g}^k\big)-\prob[(W^{\phi,n}_j,W^{\phi,n}_{j+1})\in A_{l,g}^k | \G]\right)$$
$$I^{i,g}_{l,k}:=h_{n,k}^{i}\left(\ind\big(W^{\phi_{ni}}_{j:j+1}\in A_{l,k}^{i,g}\big)-\prob(W^{\phi_{ni}}_{j:j+1}\in A_{l,k}^{i,g} | \Jn)\right).$$
%and wrote $E^k_i = \E[I^k_i| \mathbb{G}]$. Denote 
Using this, we note that the following holds: 
$$
\begin{aligned}
& \bigg\|\sup_{g\in \mathcal{F}}\Big|\sum_{i\leq J_n}\sum_{k=1}^{\infty}\E\Big[h_{n,k}^i\Big(\Delta^{i}_{j,k}(g)-\E[\Delta^{i}_{j,k}(g)|\Jn]\Big)\Big| \Jn\Big]\Big|\bigg\|_1
\\\le& \bigg\|\sup_{g\in \mathcal{F}}\Big|\sum_{i\leq J_n}\sum_{k=1}^{\infty}\E\Big[h_{n,k}^i\Big(\Delta_{j,k}^{i,*}(g)-\E[\Delta_{j,k}^{i,*}(g)| \Jn]\Big)\Big| \Jn\Big]\Big|\bigg\|_1+2\|J_n\|^2_2\epsilon.
\end{aligned}
$$
To bound the first term, we note that 
\begin{align*}
& \bigg\|\sup_{g\in \mathcal{F}}\Big|\sum_{i\leq J_n}\sum_{k=1}^{\infty}\E\Big[h_{n,k}^i\Big(\Delta_{j,k}^{i,*}(g)-\E[\Delta_{j,k}^{i,*}(g)|\Jn]\Big)\Big| \Jn\Big]\Big|\bigg\|_1\\
= & \Bigg\|\sup_{g\in \mathcal{F}}\bigg|\E\bigg[\sum_{i\leq J_n}\sum_{k=1}^{\infty}h_{n,k}^{i}\bigg(\sum_{l=1}^{N_\epsilon^g} a_{l,k}^{i,g}\ind\big(W^{\phi_{ni}}_{j:j+1}\in A_{l,k}^{i,g}\big)\\
&\qquad\qquad\qquad\qquad\qquad -\E\Big[\sum_{l=1}^{N_{\epsilon}^g} a_{l,k}^{i,g}\ind\big(W^{\phi_{ni}}_{j:j+1}\in A_{l,k}^{i,g}\big)\Big|\Jn\Big]\bigg)\bigg|\Jn\bigg]\bigg|\Bigg\|_1~\\
\leq & \Bigg\|\sup_{g\in \mathcal{F}}\bigg|\sum_{i\leq J_n}\sum_{l=1}^{N_{\epsilon}^g} \sum_{k=1}^{\infty}a_{l,k}^{i,g}\E\Big[h_{n,k}^i\left(\ind(W^{\phi_{ni}}_{j:j+1}\in A_{l,k}^{i,g})-\prob(W^{\phi_{ni}}_{j:j+1}\in A_{l,k}^{i,g} |\Jn)\right)\Big|\Jn\Big]\bigg|\Bigg\|_1\\
%\leq & a \Bigg\|\sup_{g\in \mathcal{F}}\bigg|\sum_{l=1}^{N_{\epsilon}^g} \sum_{k=1}^{\infty}a_{l,k}^g\E\Big[h_{n,k}^1\left(\ind((W^{\phi_{n1}}_j,W^{\phi_{n1}}_{j+1})\in A_{l,g}^k-\prob((W^{\phi_{n1}}_j,W^{\phi_{n1}}_{j+1})\in A_{l,g}^k| \Jn)\right)|\Jn\Big]\bigg|\Bigg\|_1+\E[J_n]\epsilon
= & \Bigg\|\sup_{g\in \mathcal{F}}\bigg|\sum_{i\leq J_n}\sum_{l=1}^{N_{\epsilon}^g} \sum_{k=1}^{\infty}a_{l,k}^{i,g}\E\Big[I_{l,k}^{i,g}\Big|\Jn\Big]\bigg|\Bigg\|_1
\\\le &\Bigg\|\sup_{g\in \mathcal{F}}\sum_{i\leq J_n}\sum_{\substack{l\le N_{\epsilon}^g,~ k\in \mathbb{N}\\\mathbb{E}[I^{i,g}_{l,k}|\Jn]\ge 0}}|a_{l,k}^{i,g}|\E\Big[I^{i,g}_{l,k}\Big| \Jn \Big]-\sum_{\substack{l\le N_{\epsilon}^g,~k\in \mathbb{N}\\\mathbb{E}[I^{i,g}_{l,k}|\Jn]<0}}|a_{l,k}^{i,g}|\E\Big[I^{i,g}_{l,k}\Big|\Jn\Big]\Bigg\|_1
\\\le &\Bigg\|\sup_{g\in \mathcal{F}}\sum_{i\leq J_n}\sup_{l\le N_{\epsilon}^g}|a_{l,k}^{i,g}|\sum_{\substack{l\le N_{\epsilon}^g,~ k\in \mathbb{N}\\\mathbb{E}[I^{i,g}_{l,k}|\Jn]\ge 0}}\E\Big[I^{i,g}_{l,k}\Big| \Jn \Big]-\sup_{l\le N_{\epsilon}^g}|a_{l,k}^{i,g}|\sum_{\substack{l\le N_{\epsilon}^g,~k\in \mathbb{N}\\\mathbb{E}[I^{i,g}_{l,k}|\Jn]<0}}\E\Big[I^{i,g}_{l,k}\Big|\Jn\Big]\Bigg\|_1\\
\leq &\bigg\|\sup_{g\in \mathcal{F}}\sum_{i\leq J_n}\sup_{l\le N_{\epsilon}^g}|a_{l,k}^{i,g}|\sum_{\substack{l\le N_{\epsilon}^g,~ k\in \mathbb{N}\\\mathbb{E}[I^{i,g}_{l,k}|\Jn]\ge 0}}\E\Big[I^{i,g}_{l,k}\Big| \Jn \Big]\bigg\|_1 \\
&+ \bigg\|\sup_{g\in \mathcal{F}}\sum_{i\leq J_n}\sup_{l\le N_{\epsilon}^g}|a_{l,k}^{i,g}|\sum_{\substack{l\le N_{\epsilon}^g,~ k\in \mathbb{N}\\\mathbb{E}[I^{i,g}_{l,k}|\Jn]< 0}}\E\Big[I^{i,g}_{l,k}\Big| \Jn \Big]\bigg\|_1.
\end{align*}
Note that conditioning on $J_n$ and $\boldsymbol{\phi}_n$, we have that for all $i\le J_n$,
$$
\Delta_{j,k}^{i}(g) \overset{a.s.}{\le} \norm{\Delta g}_\infty\Big|W^{\phi_{ni}}_{j+1} -W^{\phi_{ni}}_{j}\Big|  \le \norm{\Delta g}_\infty
\sum\limits_{i' \le J_n}\ind(d(\phi_{ni},\phi_{ni'})\in [j,j+1)).$$
which implies for all $i\le J_n$,
$$ 
\sup_{l\le N_{\epsilon}^g}|a_{l,k}^{i,g}| \le \norm{\Delta g}_\infty
\sum\limits_{i' \le J_n}\ind(d(\phi_{ni},\phi_{ni'})\in [j,j+1)).
$$
We note that this bound does not depend on $X_n$. 
Moreover, for the ease of notation, we denote
% $$
% B_{k,n}^{i,-}:=\Big\{W^{\phi_{ni}}_{j:j+1}\in\bigcup_{\substack{l\le N_{\epsilon}^g,~ k\in \mathbb{N}\\\mathbb{E}[I^{i,g}_{l,k}|\Jn]<0}} ~A_{l,k}^{i,g}\Big\},
% $$
$$
B_{k,n}^{i,+}:=\Big\{W^{\phi_{ni}}_{j:j+1}\in\bigcup_{\substack{l\le N_{\epsilon}^g,~ k\in \mathbb{N}\\\mathbb{E}[I^{i,g}_{l,k}|\Jn]\ge0}} ~A_{l,k}^{i,g}\Big\};
$$
and remark that we have
% $$
% \begin{aligned}
% \sum_{\substack{l\le N_{\epsilon}^g,~ k\in \mathbb{N}\\\mathbb{E}[I^{i,g}_{l,k}|\Jn]<0}}~I^{i,g}_{l,k}
% %=&h_{n,k}^{i} ~\Big(\mathbb{I}\Big(W^{\phi_{ni}}_{j:j+1}\in\bigcup_{\substack{l\le N_{\epsilon}^g,~ k\in \mathbb{N}\\\mathbb{E}[I^{i,g}_{l,k}|\Jn]<0}} ~A_{l,k}^{i,g}\Big) - \mathbb{P}\Big(W^{\phi_{ni}}_{j:j+1}\in\bigcup_{\substack{l\le N_{\epsilon}^g,~ k\in \mathbb{N}\\\mathbb{E}[I^{i,g}_{l,k}|\Jn]<0}} ~A_{l,k}^{i,g}\Big|\Jn\Big)\Big)\\
% =&h_{n,k}^i\Big(\mathbb{I}(B_{k,n}^{i,-})-\mathbb{P}(B_{k,n}^{i,-}|\Jn)\Big),  
% \end{aligned}
% $$
%and  
$$
\begin{aligned}
\sum_{\substack{l\le N_{\epsilon}^g,~ k\in \mathbb{N}\\\mathbb{E}[I^{i,g}_{l,k}|\Jn]\ge0}}~I^{i,g}_{l,k}
%=&h_{n,k}^{i} ~\Big(\mathbb{I}\Big(W^{\phi_{ni}}_{j:j+1}\in\bigcup_{\substack{l\le N_{\epsilon}^g,~ k\in \mathbb{N}\\\mathbb{E}[I^{i,g}_{l,k}|\Jn]\ge0}} ~A_{l,k}^{i,g}\Big) -\mathbb{P}\Big(W^{\phi_{ni}}_{j:j+1}\in\bigcup_{\substack{l\le N_{\epsilon}^g,~ k\in \mathbb{N}\\\mathbb{E}[I^{i,g}_{l,k}|\Jn]\ge0}} ~A_{l,k}^{i,g}\Big|\Jn\Big)\Big) \\
=&h^i_{n,k}\Big(\mathbb{I}(B_{k,n}^{i,+})-\mathbb{P}(B_{k,n}^{i,+}|\Jn)\Big).
\end{aligned}
$$ We note that the sets $(B_{k,n}^{i,+})$ %and $(B_{k,n}^{i,-})$ 
implicitly depend on the choice of $g\in \mathcal{F}$. We obtain that 
$$
\begin{aligned}
\sum_{\substack{l\le N_{\epsilon}^g,~ k\in \mathbb{N}\\\mathbb{E}[I^{i,g}_{l,k}|\Jn]\ge 0}}\E\Big[I^{i,g}_{l,k}\Big| \Jn \Big]
= & \sum_{\substack{ k\in \mathbb{N}}}\E\Big[h_{n,k}^i\Big(\ind(B_{n,k}^{i,+})-\mathbb{E}\big[\ind(B_{n,k}^{i,+})\big|\Jn\big]\Big)\Big|\Jn\Big]\\
\le & \sup_{g\in\mathcal{F}}\sum_{\substack{ k\in \mathbb{N}}}\mathbb{P}(h_{n,k}^i=1|\Jn) ~~\Big[\mathbb{P}\Big(B_{n,k}^{i,+}\Big|\Jn,h_{n,k}^{i}=1\Big)-\mathbb{P}\Big(B_{n,k}^{i,+}|\Jn\Big)\Big].
% \\
% \le & \E\bigg[\sup_{g\in \mathcal{F}}\sum_{i\leq J_n}\norm{\Delta g}_\infty
% \sum\limits_{i' \le J_n}\ind(d(\phi_{ni},\phi_{ni'})\in [j,j+1))\\
% & \quad\quad\quad \times\E\Big[\sup_{g\in\mathcal{F}}\sum_{\substack{ k\in \mathbb{N}}}\mathbb{P}(h_{n,k}^i=1|\Jn) ~~\Big(\mathbb{P}\Big(B_{n,k}^{i,+}\Big|\Jn,h_{n,k}^{i}=1\Big)-\mathbb{P}\Big(B_{n,k}^{i,+}|\Jn\Big)\Big)\Big| J_n,\boldsymbol{\phi}_n\Big]\bigg]
\end{aligned}
$$ Therefore we have 
$$
\begin{aligned}
&\bigg\|\sup_{g\in \mathcal{F}}\sum_{i\leq J_n}\sup_{l\le N_{\epsilon}^g}|a_{l,k}^{i,g}|\sum_{\substack{l\le N_{\epsilon}^g,~ k\in \mathbb{N}\\\mathbb{E}[I^{i,g}_{l,k}|\Jn]\ge 0}}\E\Big[I^{i,g}_{l,k}\Big| \Jn \Big]\bigg\|_1
% \le & \E\bigg[\sup_{g\in \mathcal{F}}\sum_{{i,i'\leq J_n}}\norm{\Delta g}_\infty
% \ind(d(\phi_{ni},\phi_{ni'})\in [j,j+1))\sum_{\substack{l\le N_{\epsilon}^g,~ k\in \mathbb{N}\\\mathbb{E}[I^{i,g}_{l,k}|\Jn]\ge 0}}\E\Big[I^{i,g}_{l,k}\Big| \Jn \Big]\bigg]\\
% = & \E\bigg[\sup_{g\in \mathcal{F}}\sum_{{i,i'\leq J_n}}\norm{\Delta g}_\infty
% \ind(d(\phi_{ni},\phi_{ni'})\in [j,j+1))\sum_{\substack{ k\in \mathbb{N}}}\E\Big[h_{n,k}^i\Big(\ind(B_{n,k}^{i,+})-\mathbb{E}\big[\ind(B_{n,k}^{i,+})\big|\Jn\big]\Big)\Big|\Jn\Big]\bigg]\\
% \le & \E\bigg[\sup_{g\in \mathcal{F}}\sum_{{i,i'\leq J_n}}\norm{\Delta g}_\infty
% \ind(d(\phi_{ni},\phi_{ni'})\in [j,j+1))\\
% & \quad\quad\quad \times\sup_{g\in\mathcal{F}}\sum_{\substack{ k\in \mathbb{N}}}\mathbb{P}(h_{n,k}^i=1|\Jn) ~~\Big(\mathbb{P}\Big(B_{n,k}^{i,+}\Big|\Jn,h_{n,k}^{i}=1\Big)-\mathbb{P}\Big(B_{n,k}^{i,+}|\Jn\Big)\Big)\bigg]\\
\\\le & \E\bigg[\sup_{g\in \mathcal{F}}\sum_{i,i'\leq J_n}\norm{\Delta g}_\infty
\ind(d(\phi_{ni},\phi_{ni'})\in [j,j+1))\\
& \quad\quad\times\E\Big[\sup_{g\in\mathcal{F}}\sum_{\substack{ k\in \mathbb{N}}}\mathbb{P}(h_{n,k}^i=1|\Jn) ~~\Big(\mathbb{P}\Big(B_{n,k}^{i,+}\Big|\Jn,h_{n,k}^{i}=1\Big)-\mathbb{P}\Big(B_{n,k}^{i,+}|\Jn\Big)\Big)\Big| J_n,\boldsymbol{\phi}_n\Big]\bigg].
\end{aligned}
$$
%where the last step is a result of the tower rule combined with the fact that $\sum_{i\leq J_n}
%\sum\limits_{i' \le J_n}\ind(d(\phi_{ni},\phi_{ni'})\in [j,j+1))$ is measurable with respect to the $\sigma$-algebra generated by $J_n,\boldsymbol{\phi}_n$.

Moreover, by utilizing the definition of the conditional $\xi$-mixing coefficient and considering the independence of $J_n$, $\nu_n$, and $X_n$, we establish that the following holds for all $i\leq J_n$: 
\begin{align*}
&\E\Big[  \sup_{g\in \mathcal{F}} \sum_{\substack{ k\in \mathbb{N}}}   \mathbb{P}(h_{n,k}^i=1|\Jn)  \Big(\mathbb{P}\Big(B_{n,k}^{i,+}\Big|\Jn,h_{n,k}^{i}=1\Big)-\mathbb{P}\Big(B_{n,k}^{i,+}\Big|\Jn\Big)\Big)\Big|J_n,\boldsymbol{\phi}_n\Big]\\
{=} & \E\Bigg[\sup_{g\in \mathcal{F}}\frac{\mathbb{P}(f_n(\phi_{ni} X_n)=1|\Jn)}{\mathbb{P}(f_n(\phi_{ni}X_n)=1|\Jn)}\sum_{\substack{ k\in \mathbb{N}}}   \mathbb{P}(h_{n,k}^i=1|\Jn)\\
& \qquad \qquad  \qquad \times 
\Big(\mathbb{P}\Big(B_{n,k}^{i,+}\Big|\Jn,h_{n,k}^{i}=1\Big)-\mathbb{P}\Big(B_{n,k}^{i,+}|\Jn\Big)\Big)\Bigg|J_n, \boldsymbol{\phi}_n \Bigg]
% \\
% \overset{(a)}{\leq} &  \E\bigg[\sup_{g\in \mathcal{F}}\sup_{\substack{\big(\{h_{n,k}^i=1\},B_k\big)\in \mathcal{D}_n^{b}(j)}}\frac{1}{\mathbb{P}(f_n(\phi_{ni} X_n)=1|\Jn)}\Big|\sum_{\substack{ k\in \mathbb{N}}}   \mathbb{P}(h_{n,k}^i=1|\Jn)\Big[\mathbb{P}\Big(B_k\Big|\Jn,h_{n,k}^{i}=1\Big)-\mathbb{P}\Big(B_k|\Jn\Big)\Big]\Big|
\\
% & \qquad \times{\mathbb{P}(f_n(\phi_{ni} X_n)=1|\Jn)}\bigg|J_n, \boldsymbol\phi_n\bigg] \\
\overset{(a)}{\leq} & \xi_{n,p}^{b}(j|\mathbb{G})\E\bigg[\Big|\mathbb{P}(f_n(\phi_{ni} X_n)=1|\Jn)\Big|^{\frac{p}{p-1}}\bigg|J_n, \boldsymbol\phi_n\bigg]^{\frac{p-1}{p}}\\
\le & \xi_{n,p}^{b}(j|\mathbb{G})\sup_{\phi\in\A_n}\|Q_n(\phi)\|_{\frac{p}{p-1}},
\end{align*} 
where %to obtain $(a)$ we exploited the fact that for all $i\le J_n$, $\{h_{n,k}^{i}=1\}$ are disjoints sets such that  $\bigcup_{k\in\N}\{h_{n,k}^{i}=1\} = \{f_n(\phi_{ni} X_n) = 1\}$. 
$(a)$ is a result of the H\"{o}lder's inequality. As a result, we have that
\begin{align*}&\bigg\|\sup_{g\in \mathcal{F}}\sum_{i\leq J_n}\sup_{l\le N_{\epsilon}^g}|a_{l,k}^{i,g}|\sum_{\substack{l\le N_{\epsilon}^g,~ k\in \mathbb{N}\\\mathbb{E}[I^{i,g}_{l,k}|\Jn]\ge 0}}\E\Big[I^{i,g}_{l,k}\Big| \Jn \Big]\bigg\|_1
%& \E\bigg[\sup_{g\in \mathcal{F}}\sum_{i\leq J_n}\norm{\Delta g}_\infty
%\sum\limits_{i' \le J_n}\ind(d(\phi_{ni},\phi_{ni'})\in [j,j+1))\\
% & \quad\quad\quad \times\sup_{g\in\mathcal{F}}\sum_{\substack{ k\in \mathbb{N}}}\mathbb{P}(h_{n,k}^i=1|\Jn) ~~\Big(\mathbb{P}\Big(B_{n,k}^{i,+}\Big|\Jn,h_{n,k}^{i}=1\Big)-\mathbb{P}\Big(B_{n,k}^{i,+}|\Jn\Big)\Big)\bigg]\\
\\\overset{}{\le} &\Big\|\sup_{g\in \mathcal{F}}|\Delta g|\Big\|_\infty\E\bigg[\sum_{i,i'\leq J_n}
\ind(d(\phi_{ni},\phi_{ni'})\in [j,j+1))\xi_{n,p}^{b}(j|\mathbb{G})\sup_{\phi\in\A_n}\|Q_n({\phi})\|_{\frac{p}{p-1}}\bigg]\\
\overset{(b)}{\leq} & \Big\|\sup_{g\in \mathcal{F}}|\Delta g|\Big\|_{\infty}\frac{\|J_n\|^2_2\mathcal{S}}{|\A_n|} |\B_{j+1}\setminus\B_{j}|\xi_{n,p}^{b}(j|\G)\sup_{\phi\in\A_n}\|Q_n(\phi)\|_{\frac{p}{p-1}}.
\end{align*}
To arrive at $(b)$, we use the well-spread condition and the fact that $J_n$ and $\boldsymbol{\phi}_n$ are independent of $X_n$. 
Similarly, we write $$
 B_{k,n}^{i,-}:=\Big\{W^{\phi_{ni}}_{j:j+1}\in\bigcup_{\substack{l\le N_{\epsilon}^g,~ k\in \mathbb{N}\\\mathbb{E}[I^{i,g}_{l,k}|\Jn]<0}} ~A_{l,k}^{i,g}\Big\},
 $$
and obtain that  \begin{align*}
& \bigg\|\sup_{g\in \mathcal{F}}\sum_{i\leq J_n}\sup_{l\le N_{\epsilon}^g}|a_{l,k}^{i,g}|\sum_{\substack{l\le N_{\epsilon}^g,~ k\in \mathbb{N}\\\mathbb{E}[I^{i,g}_{l,k}|\Jn]< 0}}\E\Big[I^{i,g}_{l,k}\Big| \Jn \Big]\bigg\|_1\\
\overset{}{\leq} & \Big\|\sup_{g\in \mathcal{F}}|\Delta g|\Big\|_{\infty}\frac{\|J_n\|^2_2\mathcal{S}}{|\A_n|} |\B_{j+1}\setminus\B_{j}|\xi_{n,p}^{b}(j|\G)\sup_{\phi\in\A_n}\|Q_n(\phi)\|_{\frac{p}{p-1}}.
\end{align*}
Therefore, this implies that 
\begin{equation*}\label{lemma26equigeneraleqn_1}
\begin{aligned}
& \bigg\|\sup_{g\in \mathcal{F}}\Big|\sum_{i\leq J_n}\sum_{k=1}^{\infty}\E\Big[h_{n,k}^i\Big(\Delta^{i}_{j,k}(g)-\E[\Delta^{i}_{j,k}(g)|\Jn]\Big)\Big|\Jn\Big]\Big|\bigg\|_1
\\&\le 2\Big\|\sup_{g\in \mathcal{F}}|\Delta g|\Big\|_{\infty}\frac{\|J_n\|^2_2\mathcal{S}}{|\A_n|} |\B_{j+1}\setminus\B_{j}|\xi_{n,p}^{b}(j|\G)\sup_{\phi\in\A_n}\|Q_n(\phi)\|_{\frac{p}{p-1}} +2\|J_n\|^2_2\epsilon.
\end{aligned}
\end{equation*}
As this result holds for arbitrary $\epsilon>0$, we obtain the desired result:
\begin{equation}\label{lemma26equigeneraleqn}
\begin{aligned}
& \bigg\|\sup_{g\in \mathcal{F}}\Big|\sum_{i\leq J_n}\sum_{k=1}^{\infty}\E\Big[h_{n,k}^i\Big(\Delta^{i}_{j,k}(g)-\E[\Delta^{i}_{j,k}(g)|\Jn]\Big)\Big|\Jn\Big]\Big|\bigg\|_1
\\&\le 2\Big\|\sup_{g\in \mathcal{F}}|\Delta g|\Big\|_{\infty}\frac{\|J_n\|^2_2\mathcal{S}}{|\A_n|} |\B_{j+1}\setminus\B_{j}|\xi_{n,p}^{b}(j|\G)\sup_{\phi\in\A_n}\|Q_n(\phi)\|_{\frac{p}{p-1}}.
\end{aligned}
\end{equation}

\end{proof}

{
Recall that we define for all $b>0$, $ k\lambda^{b}_{n}(k) := \mathbb{E}\Big[\sum_{i \leq J_n} f_{n}(\phi_{ni} X_{n}) \mathbb{I}(E^{n,k}_{b,i})\Big|\G\Big]$. We further define 
    $k\bar{\lambda}^{b}_{n}(k) := \mathbb{E}\Big[\sum_{i\leq J_n}f_{n}(\phi_{ni} X_{n}) \mathbb{I}(E^{n,k}_{b,i})\Big|J_n,\G\Big]$. The next lemma provides a bound to control $\Big\|\sum_{k=1}^{\infty}\Big|k(\lambda^{b_n}_{n}(k)-\bar{\lambda}^{b_n}_{n}(k))\Big|~\Big\|_1$. } This established result is used in the proof of \cref{cpconvergence_randomization}. 
\begin{lemma}\label{EScorrectionterm} Suppose that the conditions of \cref{lemma26general} holds. Then the following inequality holds: 
$$
   \begin{aligned}
    & \bigg\|\sum_{k=1}^{\infty}\bigg|\mathbb{E}\Big[\sum_{i\leq J_n}f_{n}(\phi_{ni} X_{n}) \mathbb{I}(E^{n,k}_{b_n,i})\Big| J_n, \G\Big]-\mathbb{E}\Big[\sum_{i\leq J_n}f_{n}(\phi_{ni} X_{n}) \mathbb{I}(E^{n,k}_{b_n,i})\Big|\G\Big]\bigg|~\bigg\|_{1}\\  
    \le & \sup_{\phi\in\A_n}\|Q_n(\phi)\|_1\sqrt{2\mathrm{Var}(J_n)}\Big( \frac{\mathcal{S}\abs{\B_{b_n}}}{\abs{\A_n}}\|J_n\|_2 + 1\Big).
    \end{aligned}
$$
\end{lemma}
\begin{proof}
In this proof, for the ease of notation, we define the event $$E^{n}_{j_n,i}:=\Big\{\sum_{\substack{i^{\prime}\leq j_n \\ d(\phi_{ni}, \phi_{ni^{\prime}}) \leq b_n}} f_{n}(\phi_{ni^\prime} X_{n})=k\Big\},$$ where for simplicity, we made the dependence on $k$ and $b_n$ implicit. Moreover we remark that $E^{n}_{J_n,i}=E^{n,k}_{b_n, i}$. 
We note that
    \begin{align*}
    & \bigg\|\sum_{k=1}^{\infty}\bigg|\mathbb{E}\Big[\sum_{i\leq J_n}f_{n}(\phi_{ni} X_{n}) \mathbb{I}(E^{n,k}_{b_n,i})\Big| J_n, \G\Big]-\mathbb{E}\Big[\sum_{i\leq J_n}f_{n}(\phi_{ni} X_{n}) \mathbb{I}(E^{n,k}_{b_n,i})\Big|\G\Big]\bigg|~\bigg\|_{1}\\  
    = & \sum_{j_n\in\N}\prob(J_n = j_n)\bigg\| \sum_{k=1}^{\infty}\bigg|\mathbb{E}\Big[\sum_{i\leq J_n}f_{n}(\phi_{ni} X_{n}) \mathbb{I}(E^{n,k}_{b_n,i})\Big| J_n = j_n, \G\Big]\\
    &\qquad\qquad\qquad\qquad\qquad\qquad\qquad\qquad\qquad-\mathbb{E}\Big[\sum_{i\leq J_n}f_{n}(\phi_{ni} X_{n}) \mathbb{I}(E^{n,k}_{b_n,i})\Big|\G\Big]\bigg|~
    \bigg\|_1\\
   \overset{(a)}{=} & \sum_{j_n\in\N}\prob(J_n = j_n)\bigg\| \sum_{k=1}^{\infty}\bigg|\mathbb{E}\Big[\sum_{i\leq j_n}f_{n}(\phi_{ni} X_{n}) \mathbb{I}\Big(E^{n}_{j_n,i}\Big)\Big|\G\Big] \\
    &\qquad\qquad\qquad\qquad\qquad\qquad\qquad\qquad\qquad-\mathbb{E}\Big[\sum_{i\leq J_n}f_{n}(\phi_{ni} X_{n}) \mathbb{I}\Big(E^{n}_{J_n,i}\Big)\Big|\G\Big]\bigg|~
    \bigg\|_1
    \end{align*}
    where to get (a) we used the independence of $J_n$ with $\phi_{ni}$'s and $X_n$. For each $j_n \in \N$, we have 
\begin{align*}
& \bigg\| \sum_{k=1}^{\infty}\bigg|\mathbb{E}\Big[\sum_{i\leq j_n}f_{n}(\phi_{ni} X_{n}) \mathbb{I}\big(E^{n}_{j_n,i}\big)\Big|\G\Big]-\mathbb{E}\Big[\sum_{i\leq J_n}f_{n}(\phi_{ni} X_{n}) \mathbb{I}\big(E^{n}_{J_n,i}\big)\Big|\G\Big]\bigg|~\bigg\|_1\\
= & \mathbb{E}\bigg[\sum_{k=1}^{\infty}\bigg|\mathbb{E}\Big[\sum_{i\leq j_n}f_{n}(\phi_{ni} X_{n}) \mathbb{I}\big(E^{n}_{j_n,i}\big)-\sum_{i\leq J_n}f_{n}(\phi_{ni} X_{n}) \mathbb{I}\big(E^{n}_{J_n,i}\big)\Big|\G\Big]\bigg|\bigg]\\
\leq & \sum_{k=1}^{\infty}\bigg\|\sum_{i\leq j_n}f_{n}(\phi_{ni} X_{n}) \mathbb{I}\big(E^{n}_{j_n,i}\big)-\sum_{i\leq J_n}f_{n}(\phi_{ni} X_{n}) \mathbb{I}\big(E^{n}_{J_n,i}\big)\bigg\|_1\\\overset{(b)}{\le} &
%{=} &\Bigg\|
%\sum_{k=1}^{\infty}\bigg|\sum_{i\leq j_n} f_{n}(\phi_{ni} X_{n}) \Big(\mathbb{I}(\sum_{\substack{i^{\prime}\leq J_n \\ d(\phi_{ni}, \phi_{ni^{\prime}}) \leq b_n}} f_{n}(\phi_{ni^{\prime}} X_{n})=k)-\mathbb{I}(\sum_{\substack{i^{\prime}\leq j_n \\ d(\phi_{ni}, \phi_{ni^{\prime}}) \leq b_n}} f_{n}(\phi_{ni^{\prime}} X_{n})=k)\Big)\\
%&\pm \sum_{i = \min\{j_n,J_n\}+1}^{\max\{j_n,J_n\}} f_{n}(\phi_{ni} X_{n}) \mathbb{I}(\sum_{\substack{i^{\prime}\leq J_n \\ d(\phi_{ni}, \phi_{ni^{\prime}}) \leq b_n}} f_{n}(\phi_{ni^{\prime}} X_{n})=k) \bigg|\Bigg\|_1\\
%{\leq} &\Bigg\|
%\sum_{k=1}^{\infty}\bigg|\sum_{i\leq j_n} f_{n}(\phi_{ni} X_{n}) \Big(\mathbb{I}(\sum_{\substack{i^{\prime}\leq J_n \\ d(\phi_{ni}, \phi_{ni^{\prime}}) \leq b_n}} f_{n}(\phi_{ni^{\prime}} X_{n})=k)-\mathbb{I}(\sum_{\substack{i^{\prime}\leq j_n \\ d(\phi_{ni}, \phi_{ni^{\prime}}) \leq b_n}} f_{n}(\phi_{ni^{\prime}} X_{n})=k)\Big)\bigg|\\
%&+ \sum_{k=1}^{\infty}\bigg| \sum_{i = \min\{j_n,J_n\}+1}^{\max\{j_n,J_n\}} f_{n}(\phi_{ni} X_{n}) \mathbb{I}(\sum_{\substack{i^{\prime}\leq J_n \\ d(\phi_{ni}, \phi_{ni^{\prime}}) \leq b_n}} f_{n}(\phi_{ni^{\prime}} X_{n})=k) \bigg|\Bigg\|_1\\
\sum_{k=1}^{\infty}\bigg\|
\sum_{i\leq j_n} f_{n}(\phi_{ni} X_{n}) \Big(\mathbb{I}(E^{n}_{J_n,i})-\mathbb{I}(E^{n}_{j_n,i})\Big)\bigg\|_1\\
&+ \sum_{k=1}^{\infty}\Bigg\| \sum_{i = \min\{j_n,J_n\}+1}^{\max\{j_n,J_n\}} f_{n}(\phi_{ni} X_{n}) \mathbb{I}(E^{n}_{J_n,i})\Bigg\|_1\\
:=& (A)+(B).
\end{align*}
where $(b)$ is established through the comparison of $j_n$ and $J_n$, determining which has the greater value, and this is followed by the application of the triangle inequality.

Next, we provide bounds for the two terms $(A)$ and $(B)$ respectively. We start with term $(A)$. We have
\begin{align*}
(A) \le
& \sum_{i\leq j_n}\sum_{k=1}^{\infty}\mathbb{E}\Big[
 f_{n}(\phi_{ni} X_{n}) \Big|\mathbb{I}(E^{n}_{J_n,i})-\mathbb{I}(E^{n}_{j_n,i})\Big|\Big]\\
\overset{}{=}& 
\sum_{i\leq j_n}\mathbb{E}\Big[ f_{n}(\phi_{ni} X_{n}) \mathbb{I}\big(\sum_{\substack{{\min\{j_n,J_n\} < i^{\prime}\leq \max\{j_n,J_n\}} \\ d(\phi_{ni}, \phi_{ni^{\prime}}) \leq b_n}} f_{n}(\phi_{ni^{\prime}} X_{n})\neq 0\Big)\Big] \\
\leq & \sum_{i\leq j_n} \mathbb{E}\Big[\sum_{\substack{ \min\{j_n,J_n\} < i^{\prime}\leq \max\{j_n,J_n\}\\d(\phi_{ni}, \phi_{ni^{\prime}}) \leq b_n} }f_{n}(\phi_{ni} X_{n}) \mathbb{I}( f_{n}(\phi_{ni^{\prime}} X_{n})\neq 0)\Big]\\
\leq & \sum_{i\leq j_n} \mathbb{E}\Big[\sum_{\min\{j_n,J_n\} < i^{\prime}\leq \max\{j_n,J_n\}}f_{n}(\phi_{ni} X_{n}) \mathbb{I}(d(\phi_{ni}, \phi_{ni^{\prime}}) \leq b_n)\Big]
\\
\overset{(c)}{\leq} & j_n \|J_n-j_n\|_1 \E\Big[f_{n}(\phi_{n1} X_{n}) \mathbb{I}(d(\phi_{n1}, \phi_{n2}) \leq b_n)\Big]\\
\overset{(d)}{\leq} & j_n \|J_n-j_n\|_1 \mathbb{E}\Big[f_{n}(\phi_{n1} X_{n})\Big]\mathbb{E}\Big[\mathbb{I}(d(\phi_{n1}, \phi_{n2}) \leq b_n)\Big] \\
\leq & j_n \|J_n-j_n\|_1 \frac{\mathcal{S}\abs{\B_{b_n}}}{\abs{\A_n}}\sup_{\phi\in\A_n}\|Q_n(\phi)\|_1.
\end{align*} 
where $(c)$ follows from the fact that $J_n$ is independent of $\nu_n$ and $X_n$, and $(d)$ is due to the independence of the $\phi_{ni}$'s from $X_n$.
% where $(c)$ is because of the fact that 
% $$
% \begin{aligned}
% &\sum_{k=1}^{\infty}\Big|\mathbb{I}(\sum_{\substack{i^{\prime}\leq J_n \\ d(\phi_{ni}, \phi_{ni^{\prime}}) \leq b_n}} f_{n}(\phi_{ni^\prime} X_{n})=k)-\mathbb{I}(\sum_{\substack{i^{\prime}\leq j_n \\ d(\phi_{ni}, \phi_{ni^{\prime}}) \leq b_n}} f_{n}(\phi_{ni^\prime} X_{n})=k)\Big|\\
% =&\mathbb{I}\big(\sum_{\substack{{\min\{j_n,J_n\} < i^{\prime}\leq \max\{j_n,J_n\}} \\ d(\phi_{ni}, \phi_{ni^{\prime}}) \leq b_n}} f_{n}(\phi_{ni^{\prime}} X_{n})\neq 0\Big).
% \end{aligned}
% $$
Next, we bound $(B)$. We have 
\begin{align*}
(B) =&\sum_{k=1}^{\infty} \E\bigg[\sum_{i = \min\{j_n,J_n\}+1}^{\max\{j_n,J_n\}} f_{n}(\phi_{ni} X_{n}) \mathbb{I}\big(E^{n}_{J_n,i}\big) \bigg]\\
%=& \Bigg\|\sum_{k=1}^{\infty} \sum_{i = \min\{j_n,J_n\}+1}^{\max\{j_n,J_n\}} f_{n}(\phi_{ni} X_{n}) \mathbb{I}(E^{n}_{J_n,i})\Bigg\|_1\\
\leq& \E\bigg[\sum_{i = \min\{j_n,J_n\}+1}^{\max\{j_n,J_n\}} f_{n}(\phi_{ni} X_{n}) \bigg]\\
\overset{(e)}{\leq}& \|J_n-j_n\|_1\E\Big[f_{n}(\phi_{n1} X_{n})\Big] \\
\le& {\|J_n-j_n\|_1 }\sup_{\phi\in\A_n}\|Q_n(\phi)\|_1,
\end{align*}
where $(e)$ follows from the fact that $J_n$ is independent of $\mu_{n}$ and $X_n$.
Therefore, combining the results above, we have
$$
\begin{aligned}
& \bigg\| \sum_{k=1}^{\infty}\bigg|\mathbb{E}\Big[\sum_{i\leq j_n}f_{n}(\phi_{ni} X_{n}) \mathbb{I}\big(E^{n}_{j_n,i}\big)\Big|\G\Big]-\mathbb{E}\Big[\sum_{i\leq J_n}f_{n}(\phi_{ni} X_{n}) \mathbb{I}\big(E^{n}_{J_n,i}\big)\Big|\G\Big]\bigg|~\bigg\|_1\\
\leq&{\|J_n-j_n\|_1 }\sup_{\phi\in\A_n}\|Q_{n}(\phi)\|_1\Big(j_n  \frac{\mathcal{S}\abs{\B_{b_n}}}{\abs{\A_n}}+1\Big).
\end{aligned}
$$
 Let $J_n^\prime$ be an i.i.d. copy of $J_n$. The above inequality implies that
$$
   \begin{aligned}
    & \bigg\|\sum_{k=1}^{\infty}\bigg|\mathbb{E}\Big[\sum_{i\leq J_n}f_{n}(\phi_{ni} X_{n}) \mathbb{I}(E^{n,k}_{b_n,i})\Big| J_n, \G\Big]-\mathbb{E}\Big[\sum_{i\leq J_n}f_{n}(\phi_{ni} X_{n}) \mathbb{I}(E^{n,k}_{b_n,i})\Big|\G\Big]\bigg|~\bigg\|_{1}\\  
    \le & \sup_{\phi\in\A_n}\|Q_{n}(\phi)\|_1\sum_{j_n\in\N}\prob(J_n = j_n)\|J_n-j_n\|_1 \Big(j_n  \frac{\mathcal{S}\abs{\B_{b_n}}}{\abs{\A_n}}+1\Big)\\
    = & \sup_{\phi\in\A_n}\|Q_{n}(\phi)\|_1\Big( \frac{\mathcal{S}\abs{\B_{b_n}}}{\abs{\A_n}}\|(J_n-J_n^{\prime})J_n^{\prime}\|_1+\|J_n-J_n^{\prime}\|_1\Big).
    \end{aligned}
$$
We remark that  $$\|J_n-J_n^{\prime}\|_1 \leq \sqrt{2\mathrm{Var}(J_n)}$$
and that by H\"{o}lder's inequality,
$$\|(J_n-J_n^{\prime})J_n^{\prime}\|_1
\le \|J_n-J_n^{\prime}\|_2\|J_n^{\prime}\|_2\le \sqrt{2\mathrm{Var}(J_n)}\|J_n\|_2.$$
Hence we have 
$$
   \begin{aligned}
    & \bigg\|\sum_{k=1}^{\infty}\bigg|\mathbb{E}\Big[\sum_{i\leq J_n}f_{n}(\phi_{ni} X_{n}) \mathbb{I}(E^{n,k}_{b_n,i})\Big| J_n, \G\Big]-\mathbb{E}\Big[\sum_{i\leq J_n}f_{n}(\phi_{ni} X_{n}) \mathbb{I}(E^{n,k}_{b_n,i})\Big|\G\Big]\bigg|~\bigg\|_{1}\\  \le&\sup_{\phi\in\A_n}\|Q_n(\phi)\|_1\sqrt{2\mathrm{Var}(J_n)}\Big( \frac{\mathcal{S}\abs{\B_{b_n}}}{\abs{\A_n}}\|J_n\|_2 + 1\Big).
    \end{aligned}
$$
\end{proof}

\subsubsection{Proof of Theorem \ref{MtKatahdin}}
The subsequent lemma divides the task of establishing an upper bound for $\mathbb{E}\Big[\dTV(Z(\lambda^{b_n}_{n}),~W_n|\G)\Big]$ into two distinct components.
\begin{lemma}\label{boundMH}
Assuming the conditions outlined in \cref{lemma26general}. Let $(b_n)_{n\in\N}$ be a sequence of positive integers.  Then the following holds: 
\begin{equation}\label{pistachioicecream}
\begin{aligned}
&\mathbb{E}\Big[\dTV(Z(\lambda^{b_n}_{n}),~W_n|\G)\Big]\\
\le& \bigg\| \sup_{g\in \mathcal{F}_{\TV}(\lambda^{b_n}_{n})}\Big|\E\Big[W_n g(W_n)-\sum_{k=1}^{\infty} k \widetilde{\lambda}^{b_n}_{n}(k) g(W_n+k)\Big|J_n, \widetilde{{\lambda}}^{b_n}_{n}, \mathbb{G}\Big]\Big|~\bigg\|_1\\
&+ \bigg\| \sup_{g\in \mathcal{F}_{\TV}(\lambda^{b_n}_{n})}\Big|\E\Big[\sum_{k=1}^{\infty} k (\widetilde{\lambda}^{b_n}_{n}(k)-\lambda^{b_n}_{n}(k)) g(W_n+k)\Big|J_n, \widetilde{{\lambda}}^{b_n}_{n}, \mathbb{G}\Big]\Big|~\bigg\|_1\\
=: & (M) + (H).
\end{aligned} 
\end{equation}
\end{lemma}
\begin{proof}
Applying the result from \cref{SteinCPboundTV} and the triangle inequality, we deduce that

\begin{align*}
&\mathbb{E}\Big[\dTV(Z(\lambda^{b_n}_{n}),~W_n|\G)\Big]\\
\le&\bigg\| \sup_{g\in \mathcal{F}_{\TV}(\lambda^{b_n}_{n})}\Big|\E\Big[W_n g(W_n)-\sum_{k=1}^{\infty} k \lambda^{b_n}_{n}(k) g(W_n+k)\Big| \mathbb{G}\Big]\Big|~\bigg\|_1\\
\le& \bigg\| \sup_{g\in \mathcal{F}_{\TV}(\lambda^{b_n}_{n})}\Big|\E\Big[W_n g(W_n)-\sum_{k=1}^{\infty} k \widetilde{\lambda}^{b_n}_{n}(k) g(W_n+k)\Big| \mathbb{G}\Big]\Big|~\bigg\|_1\\
&+ \bigg\| \sup_{g\in \mathcal{F}_{\TV}(\lambda^{b_n}_{n})}\Big|\E\Big[\sum_{k=1}^{\infty} k (\widetilde{\lambda}^{b_n}_{n}(k)-\lambda^{b_n}_{n}(k)) g(W_n+k)\Big| \mathbb{G}\Big]\Big|~\bigg\|_1\\
\overset{(a)}{\le}& \bigg\| \sup_{g\in \mathcal{F}_{\TV}(\lambda^{b_n}_{n})}\Big|\E\Big[W_n g(W_n)-\sum_{k=1}^{\infty} k \widetilde{\lambda}^{b_n}_{n}(k) g(W_n+k)\Big|J_n, \widetilde{{\lambda}}^{b_n}_{n}, \mathbb{G}\Big]\Big|~\bigg\|_1\\
&+ \bigg\| \sup_{g\in \mathcal{F}_{\TV}(\lambda^{b_n}_{n})}\Big|\E\Big[\sum_{k=1}^{\infty} k (\widetilde{\lambda}^{b_n}_{n}(k)-\lambda^{b_n}_{n}(k)) g(W_n+k)\Big|J_n, \widetilde{{\lambda}}^{b_n}_{n}, \mathbb{G}\Big]\Big|~\bigg\|_1,
\end{align*}
where $(a)$ is a result of Jensen's inequality. 
\end{proof}

In the two subsequent subsections, we will successively establish upper bounds for each term appearing on the right-hand side of \cref{pistachioicecream}.
%%%%%%%%%%%%%%%%%%%%%%%%%%%%%%%%%%%%%%%%%%%%%%%%%%%%%%%%%%%%%%%%%%%%%%%%%%%%%%%%%%%%%%%%%%%

\subsubsection{Bounding $(M)$}

For the ease of notation, we adopt $\mathcal{F}{\TV}:=\mathcal{F}{\TV}(\lambda^{b_n}_{n})$ as a standing notation in the proofs presented in the following subsection. 

\begin{lemma}\label{sun}
Assume that the conditions of \cref{lemma26general} are satisfied. Let $(b_n)_{n\in\N}$ be a sequence of positive integers. Then the following holds:  
\begin{equation}\label{Miami}
\begin{aligned}
    &(M)\\
    \leq &\Bigg\|\sup_{g\in\mathcal{F}_{\TV}}\bigg|\mathbb{E}\bigg[\sum_{i\leq J_n}f_{n}(\phi_{ni} X_{n}) \sum_{k=1}^{\infty} \mathbb{I}(E^{n,k}_{b_n,{i}})\Big(g({W}^{\phi_{ni}}_{b_n}+k)-g({W}^{\phi_{ni}}_{2 b_n}+k)\Big)\bigg|J_n,  \boldsymbol\phi_n,\G \bigg]\bigg|\Bigg\|_1\\ &+ \Bigg\|\sup_{g\in\mathcal{F}_{\TV}}\bigg|\mathbb{E}\bigg[\sum_{i\leq J_n}f_{n}(\phi_{ni} X_{n}) \sum_{k=1}^{\infty} \mathbb{I}(E^{n,k}_{b_n,{i}})g({W}^{\phi_{ni}}_{2 b_n}+k)\\
    &\qquad\qquad\qquad-\sum_{k=1}^{\infty} k \widetilde\lambda^{b_n}_{n}(k) g(W_n+k)\bigg|J_n, \boldsymbol\phi_n,\G \bigg]\bigg|\Bigg\|_1\\
    =:&(A)+(B)
\end{aligned}
\end{equation}
\end{lemma}
\begin{proof}
Recall that 
$$(M) = 
    \bigg\| \sup_{g\in \mathcal{F}_{\TV}}\Big|\E\Big[W_n g(W_n)-\sum_{k=1}^{\infty} k \widetilde{\lambda}^{b_n}_{n}(k) g(W_n+k)\Big|J_n, \widetilde{{\lambda}}^{b_n}_{n}, \mathbb{G}\Big]\Big|~\bigg\|_1.$$
Using the same derivation as in \cref{snow}, 
we obtain
\begin{equation*}
\begin{aligned}
    (M)
    \leq &\Bigg\|\sup_{g\in\mathcal{F}_{\TV}}\bigg|\mathbb{E}\bigg[\sum_{i\leq J_n}f_{n}(\phi_{ni} X_{n}) \sum_{k=1}^{\infty} \mathbb{I}(E^{n,k}_{b_n,{i}})\Big(g({W}^{\phi_{ni}}_{b_n}+k)-g({W}^{\phi_{ni}}_{2 b_n}+k)\Big)\bigg| J_n, \widetilde{\lambda}^{b_n}_{n},\G\bigg]\bigg|\Bigg\|_1\\ &+ \Bigg\|\sup_{g\in\mathcal{F}_{\TV}}\bigg|\mathbb{E}\bigg[\sum_{i\leq J_n}f_{n}(\phi_{ni} X_{n}) \sum_{k=1}^{\infty} \mathbb{I}(E^{n,k}_{b_n,{i}})g({W}^{\phi_{ni}}_{2 b_n}+k)\\
    &\qquad\qquad\qquad-\sum_{k=1}^{\infty} k \widetilde{\lambda}^{b_n}_{n}(k) g(W_n+k)\bigg| J_n, \widetilde{\lambda}^{b_n}_{n},\G\bigg]\bigg|\Bigg\|_1\\
    \overset{(a)}{\leq} &\Bigg\|\sup_{g\in\mathcal{F}_{\TV}}\bigg|\mathbb{E}\bigg[\sum_{i\leq J_n}f_{n}(\phi_{ni} X_{n}) \sum_{k=1}^{\infty} \mathbb{I}(E^{n,k}_{b_n,{i}})\Big(g({W}^{\phi_{ni}}_{b_n}+k)-g({W}^{\phi_{ni}}_{2 b_n}+k)\Big)\bigg| J_n, \boldsymbol\phi_n,\G \bigg]\bigg|\Bigg\|_1\\ &+ \Bigg\|\sup_{g\in\mathcal{F}_{\TV}}\bigg|\mathbb{E}\bigg[\sum_{i\leq J_n}f_{n}(\phi_{ni} X_{n}) \sum_{k=1}^{\infty} \mathbb{I}(E^{n,k}_{b_n,{i}})g({W}^{\phi_{ni}}_{2 b_n}+k)\\
    &\qquad\qquad\qquad-\sum_{k=1}^{\infty} k \widetilde\lambda^{b_n}_{n}(k) g(W_n+k)\bigg| J_n, \boldsymbol\phi_n,\G \bigg]\bigg|\Bigg\|_1,
\end{aligned}
\end{equation*}
where $(a)$ follows from the tower property and Jensen's inequality.
\end{proof}
In the next three lemmas, we successively bound the two terms on the right-hand side of \cref{Miami}. We commence by bounding the term $(A)$.
\begin{lemma}\label{ts131}
Suppose that all the conditions of \cref{sun} are satisfied. Then 
%for all $s\geq 1$ 
the
following holds:
$$
\begin{aligned}
(A)=
&\Bigg\|\sup_{g\in\mathcal{F}_{\TV}}\bigg|\mathbb{E}\bigg[\sum_{i\leq J_n}f_{n}(\phi_{ni} X_{n}) \sum_{k=1}^{\infty} \mathbb{I}(E^{n,k}_{b_n,{i}})\Big(g({W}^{\phi_{ni}}_{b_n}+k)-g({W}^{\phi_{ni}}_{2 b_n}+k)\Big)\bigg| J_n, \boldsymbol\phi_n,\G \bigg]\bigg|\Bigg\|_1\\ \leq&\Big\|\sup_{g\in \mathcal{F}_{\TV}}|\Delta g|\Big\|_{\infty}\mathcal{S}\frac{\|J_n\|^2_2}{|\A_n|}\Big(\sup_{\phi\in \A_n}\|Q_n(\phi)\|_{\frac{p}{p-1}}\mathcal{R}_\Psi(n,p,b_n)
+{|\B_{2b_n}\setminus\B_{b_n}|}\sup_{\phi\in \A_n}\|Q_n(\phi)\|_2^2\Big).
%\leq & \sup_{g\in \mathcal{F}}\|\Delta g\|_{\infty}\frac{j_n^2}{|\A_n|^2} \Big[\frac{1}{|\A_n|}\sum_{j=b_n}^{2b_n-1}\int_{\A_n}\int_{\mathbf{B}^{\A_n}_{j+1}(\phi_{ni})\setminus\mathbf{B}^{\A_n}_{j}(\phi_{ni})} \norm{\mu^{\phi}_n}_1\Psi(j)d\nu_n(\phi')d\nu_n(\phi)\\
%&\quad\quad\quad\quad\quad\quad\quad\quad\quad+{j_n}\int_{\A_n}\int_{\B^{\A_n}_{2b_n}(\phi_{ni})\setminus \B^{\A_n}_{b_n}(\phi_{ni})}\norm{\mu^{\phi}_n \mu^{\phi'}_n}_1 d\nu_n(\phi')d\nu_n(\phi)\Big]
\end{aligned}
$$
\end{lemma}
\begin{proof}
For all $g\in\mathcal{F}_{\TV}$, we have
\begin{equation*}
\begin{aligned}
&\mathbb{E}\bigg[\sum_{i \leq J_n}f_{n}(\phi_{ni} X_{n}) \sum_{k=1}^{\infty} \mathbb{I}(E^{n,k}_{b_n,i})\Big(g(W^{\phi_{ni}}_{b_n}+k)-g(W^{\phi_{ni}}_{2b_n}+k)\Big)\bigg|J_n, \boldsymbol\phi_n, \mathbb{G}\bigg]\\
\overset{(a)}{\leq}&\sum_{i \leq J_n}\mathbb{E}\Big[f_{n}(\phi_{ni} X_{n}) \sup _{k\in \mathbb{N}} ~~\Big|g(W^{\phi_{ni}}_{b_n}+k)-g(W^{\phi_{ni}}_{ 2 b_n}+k)\Big|~\Big|J_n, \boldsymbol\phi_n,\mathbb{G}\Big] \\
\leq &\sum_{i \leq J_n}\|\Delta g\|_{\infty} \mathbb{E}\Big[f_n(\phi_{ni} X_n) \sum_{i'\le J_n}\mathbb{I}\big(d(\phi_{ni},\phi_{ni'})\in (b_n,2b_n]\big)~ f_{n}(\phi_{ni^\prime }X_{n})\Big|J_n, \boldsymbol\phi_n,\mathbb{G}\Big],
\end{aligned}
\end{equation*} where to get $(a)$ we used the fact that $\bigcup_{k=1}^{\infty}E_{b_n,i}^{n,k}\supset\{f_{n}(\phi_{ni} X_{n}) \neq 0\}.$~Recall that we denote $\Jn:=\{J_n, \boldsymbol\phi_n,\mathbb{G}\}$. Therefore, this implies that 
\begin{align*}
&\Bigg\|\sup_{g\in\mathcal{F}_{\TV}}\bigg|\mathbb{E}\bigg[\sum_{i\leq J_n}f_{n}(\phi_{ni} X_{n}) \sum_{k=1}^{\infty} \mathbb{I}(E^{n,k}_{b_n,{i}})\Big(g({W}^{\phi_{ni}}_{b_n}+k)-g({W}^{\phi_{ni}}_{2 b_n}+k)\Big)\bigg| J_n, \boldsymbol\phi_n,\G \bigg]\bigg|\Bigg\|_1
\\\le & \Big\|\sup_{g\in \mathcal{F}_{\TV}}|\Delta g|\Big\|_{\infty} \bigg\|\sum_{i \leq J_n}\mathbb{E}\Big[f_n(\phi_{ni} X_n) \\
&\qquad\qquad\qquad\qquad\times\sum_{j=b_n}^{2b_n-1}\sum_{i'\le J_n}\mathbb{I}\big(d(\phi_{ni},\phi_{ni'})\in (j,j+1]\big)~f_{n}(\phi_{ni^\prime} X_{n}) \Big|\Jn\Big]\bigg\|_1\\
%= & \abs{\A_n}\|\Delta g\|_{\infty} \sum_{j=b_n}^{2b_n-1}\sum_{\phi^{\prime} \in B(e, j+1) \bac\TVlash B(e, j)}\mathbb{E}[f_n(\phi X_n) f_{n}(\phi^\prime X_{n}) ]aaa\\
\overset{(a)}{\le} & \Big\|\sup_{g\in \mathcal{F}_{\TV}}|\Delta g|\Big\|_{\infty}\sum_{j=b_n}^{2b_n-1}\Big(\Big\|\sum_{i,i' \leq J_n} ~\mathbb{I}\big(d(\phi_{ni},\phi_{ni'})\in (j,j+1]\big)~~\\
&\qquad \qquad \times\Big(\mathbb{E}[f_{n}(\phi_{ni} X_{n}) f_{n}(\phi_{ni^\prime} X_{n})| \Jn ]-\mathbb{E}[f_{n}(\phi_{ni} X_{n}) | \Jn]~\E[f_{n}(\phi_{ni^\prime} X_{n}) | \Jn]\Big)\Big\|_1\\&\qquad   +\Big\|\sum_{i,i' \leq J_n} \mathbb{I}\big(d(\phi_{ni},\phi_{ni'})\in (j,j+1]\big)~~\mathbb{E}[f_{n}(\phi_{ni} X_{n}) | \Jn]~\E[f_{n}(\phi_{ni^\prime} X_{n}) | \Jn]\Big\|_1\Big),
\end{align*}
where to get $(a)$ we applied the triangle equality. Observe that by the definition of $\mu^{\phi}_n$, and given the independence of $\boldsymbol \phi_n$ and $X_n$, as well as the independence of $J_n$ and $X_n$, it follows that

$$
\begin{aligned}
    &\Big\|\sum_{i,i'\le J_n}~\mathbb{I}\big(d(\phi_{ni},\phi_{ni'})\in (j,j+1]\big)~~\mathbb{E}[f_{n}(\phi_{ni} X_{n}) | \Jn]~\E[f_{n}(\phi_{ni^\prime} X_{n}) | \Jn]\Big\|_1\\   
    %=& \Big\|\sum_{i^\prime\leq j_n}\mathbb{I}(\phi_{ni^\prime} \in \mathbf{B}^{\A_n}_{j+1}(\phi_{ni})\setminus\mathbf{B}^{\A_n}_{j}(\phi_{ni}))\mathbb{E}[f_{n}(\phi_{ni} X_{n}) | \boldsymbol\phi_n,\mathbb{G}]\E[f_{n}(\phi_{ni^\prime} X_{n}) | \boldsymbol\phi_n,\mathbb{G}]\Big\|_1\\
    %\leq & \sum_{i^\prime\leq j_n}\E\Big[\mathbb{I}(\phi_{ni^\prime} \in \mathbf{B}^{\A_n}_{j+1}(\phi_{ni})\setminus\mathbf{B}^{\A_n}_{j}(\phi_{ni}))\mathbb{E}[f_{n}(\phi_{ni} X_{n}) | \boldsymbol\phi_n,\mathbb{G}]\E[f_{n}(\phi_{ni^\prime} X_{n}) | \boldsymbol\phi_n,\mathbb{G}]\Big]\\
    = & {\|J_n\|^2_2}\int_{\A_n}\int_{\B^{\A_n}_{j+1}(\phi)\setminus \B^{\A_n}_{j}(\phi)}\norm{Q_n(\phi)Q_n(\phi')}_1 d\nu_n(\phi')d\nu_n(\phi)
    \\ \le& \|J_n\|^2_2\int_{\A_n}\|Q_n(\phi)\|_2\int_{\B^{\A_n}_{j+1}(\phi)\setminus \B^{\A_n}_{j}(\phi)}\norm{{Q_n(\phi')}}_2 d\nu_n(\phi')d\nu_n(\phi)   
    \\ \le& \frac{\mathcal{S}\|J_n\|^2_2|\B_{j+1}\setminus\B_{j}|}{|\A_n|}\sup_{\phi\in \A_n}\|Q_n(\phi)\|_2^2.
 %   \\&\le \frac{j_n}{|\A_n|^2}\mathcal{S}^n_{b_n,1}\sqrt{\mathcal{S}^n_{b_n,2}}
\end{aligned}
$$
Morever, by using the definition of the $\Psi$-mixing coefficients and the H\"{o}lder's inequality, we obtain that  

\begin{align*}
    &\bigg\|\sum_{i,i'\le J_n}\mathbb{I}\big(d(\phi_{ni},\phi_{ni'})\in (j,j+1]\big)\Big(\mathbb{E}[f_{n}(\phi_{ni} X_{n}) f_{n}(\phi_{ni^\prime} X_{n})| \Jn ]\\&\qquad \qquad \qquad \qquad \qquad \qquad-\mathbb{E}[f_{n}(\phi_{ni} X_{n}) | \Jn]\E[f_{n}(\phi_{ni^\prime} X_{n}) | \Jn]\Big)\bigg\|_1\\
    =& \bigg\|\sum_{i,i^\prime\leq J_n}\mathbb{I}(\phi_{ni^\prime} \in \mathbf{B}^{\A_n}_{j+1}(\phi_{ni})\setminus\mathbf{B}^{\A_n}_{j}(\phi_{ni}))\Big(\mathbb{P}(f_{n}(\phi_{ni^{\prime}} X_{n}) = 1, f_{n}(\phi_{ni} X_{n}) = 1| \Jn )\\
    &\quad \quad \quad-\prob(f_{n}(\phi_{ni^{\prime}} X_{n}) = 1| \Jn)\prob(f_{n}(\phi_{ni} X_{n}) = 1| \Jn )\Big)\bigg\|_1\\
    =& \bigg\|\sum_{i, i^\prime\leq J_n}\mathbb{I}(\phi_{ni^\prime} \in \mathbf{B}^{\A_n}_{j+1}(\phi_{ni})\setminus\mathbf{B}^{\A_n}_{j}(\phi_{ni}))\prob(f_{n}(\phi_{ni} X_{n}) = 1| \Jn)\\
    &\quad \quad \quad \Big(\mathbb{P}(f_{n}(\phi_{ni^{\prime}} X_{n}) = 1 | f_{n}(\phi_{ni} X_{n}) = 1, \Jn )
    -\prob(f_{n}(\phi_{ni^{\prime}} X_{n}) = 1| \Jn)\Big)\bigg\|_1\\
    \leq &  \|J_n\|^2_2\int_{\A_n}\int_{\mathbf{B}^{\A_n}_{j+1}(\phi)\setminus\mathbf{B}^{\A_n}_{j}(\phi)} \norm{Q_n(\phi)}_{\frac{p}{p-1}}\Psi_{n,p}(d(\phi,\phi')|\G)~d\nu_n(\phi')d\nu_n(\phi)
    \\\le& \sup_{\phi\in \A_n}\|Q_n(\phi)\|_{\frac{p}{p-1}}\frac{\|J_n\|^2_2\mathcal{S}}{|\A_n|}\big|\B_{j+1}\setminus \B_j\big|\Psi_{n,p}(j|\G).
\end{align*}
Therefore, we obtain that 
\begin{equation*}
\begin{aligned}
&\bigg\|\sup_{g\in \mathcal{F}_{\TV}}\Big|\mathbb{E}\Big[\sum_{i \leq j_n}f_{n}(\phi_{ni} X_{n}) \sum_{k=1}^{\infty} \mathbb{I}(E^{n,k}_{b_n,i})(g(W^{\phi_{ni}}_{b_n}+k)-g(W^{\phi_{ni}}_{2b_n}+k))\Big| J_n, \boldsymbol\phi_n,\mathbb{G}\Big]\Big|\bigg\|_1
\\\le & \Big\|\sup_{g\in \mathcal{F}_{\TV}}|\Delta g|\Big\|_{\infty}\mathcal{S}\frac{\|J_n\|^2_2}{|\A_n|}\Big(\sup_{\phi\in \A_n}\|Q_n(\phi)\|_{\frac{p}{p-1}}\sum_{j=b_n}^{2b_n}\big|\B_{j+1}\setminus \B_j\big|\Psi_n(j|\G)\\
& \qquad\qquad\qquad\qquad\qquad\qquad\qquad\qquad\qquad\qquad +{|\B_{2b_n}\setminus\B_{b_n}|}\sup_{\phi\in \A_n}\|Q_n(\phi)\|_2^2\Big)\\
%= & \sup_{g\in \mathcal{F}_{\TV}}\|\Delta g\|_{\infty}\frac{j_n^2}{|\A_n|^2} \Big(\frac{1}{|\A_n|}\sum_{j=b_n}^{2b_n-1}\int_{\A_n}\int_{\mathbf{B}^{\A_n}_{j+1}(\phi_{ni})\setminus\mathbf{B}^{\A_n}_{j}(\phi_{ni})} \norm{\mu^{\phi}_n}_1\Psi(j)d\nu_n(\phi')d\nu_n(\phi)\\
%&\quad\quad\quad\quad\quad\quad\quad\quad\quad+{j_n}\int_{\A_n}\int_{\B^{\A_n}_{2b_n}(\phi_{ni})\setminus \B^{\A_n}_{b_n}(\phi_{ni})}\norm{\mu^{\phi}_n \mu^{\phi'}_n}_1 d\nu_n(\phi')d\nu_n(\phi)\Big)
\le&\Big\|\sup_{g\in \mathcal{F}_{\TV}}|\Delta g|\Big\|_{\infty}\mathcal{S}\frac{\|J_n\|^2_2}{|\A_n|}\Big(\sup_{\phi\in \A_n}\|Q_n(\phi)\|_{\frac{p}{p-1}}\mathcal{R}_\Psi(n,p,b_n)
+|\B_{2b_n}\setminus\B_{b_n}|\sup_{\phi\in \A_n}\|Q_n(\phi)\|_2^2\Big).
\end{aligned}
\end{equation*}

\end{proof}

The following two lemmas provide bounds for the term $(B)$ on the right-hand side of \cref{Miami}.
\begin{lemma}
Suppose that all the conditions of \cref{sun} are satisfied. Then the following holds
\begin{equation}
\label{florida}
\begin{aligned}
(B)= &\Bigg\|\sup_{g\in\mathcal{F}_{\TV}}\bigg|\sum_{i\leq J_n}\sum_{k=1}^{\infty}\E\Big[f_n(\phi_{ni} X_n)\ind(E^{n,k}_{b_n,i})g(W^{\phi_n}_{ 2 b_n}+k)\Big| J_n, \boldsymbol\phi_n,\G \Big]\\
& - \sum_{i\leq J_n}\sum_{k=1}^{\infty}\E\Big[f_n(\phi_{ni} X_n)\ind(E^{n,k}_{b_n,i})\Big|J_n, \boldsymbol\phi_n,\G \Big]\mathbb{E}\Big[g(W_n+k)\Big| J_n, \boldsymbol\phi_n, \G \Big]\bigg|\Bigg\|_1.\\
\end{aligned}
\end{equation}
\end{lemma}
\begin{proof}
Recall that 
$$(B) =\Bigg\|\sup_{g\in\mathcal{F}_{\TV}}\bigg|\mathbb{E}\bigg[\sum_{i\leq J_n}f_{n}(\phi_{ni} X_{n}) \sum_{k=1}^{\infty} \mathbb{I}(E^{n,k}_{b_n,{i}})g({W}^{\phi_{ni}}_{2 b_n}+k)-\sum_{k=1}^{\infty} k \widetilde\lambda^{b_n}_{n}(k) g(W_n+k)\bigg|\Jn\bigg]\bigg|\Bigg\|_1.$$
Plugging in the definition
$$k\widetilde{\lambda}_k^{b_n} := \sum_{i\leq J_n}\mathbb{E}\Big[f_{n}(\phi_{ni} X_{n}) \mathbb{I}(E^{n,k}_{b_n,i})\Big|\Jn\Big],$$ we obtain
$$
\begin{aligned}
&\Bigg\|\sup_{g\in\mathcal{F}_{\TV}}\bigg|\mathbb{E}\bigg[\sum_{i\leq J_n}f_{n}(\phi_{ni} X_{n}) \sum_{k=1}^{\infty} \mathbb{I}(E^{n,k}_{b_n,{i}})g({W}^{\phi_{ni}}_{2 b_n}+k)-\sum_{k=1}^{\infty} k \widetilde\lambda^{b_n}_{n}(k) g(W_n+k)\bigg| \Jn\bigg]\bigg|\Bigg\|_1\\
= &\Bigg\|\sup_{g\in\mathcal{F}_{\TV}}\bigg|\sum_{i\leq J_n}\sum_{k=1}^{\infty}\E\Big[f_n(\phi_{ni} X_n)\ind(E^{n,k}_{b_n,i})g({W}^{\phi_{ni}}_{ 2 b_n}+k)\Big| \Jn \Big]\\
& - \sum_{i\leq J_n}\sum_{k=1}^{\infty}\E\Big[f_n(\phi_{ni} X_n)\ind(E^{n,k}_{b_n,i})\Big|\Jn \Big]\mathbb{E}\Big[g(W_n+k)\Big|\Jn\Big]\bigg|\Bigg\|_1.
\end{aligned}
$$
\end{proof}
We will bound the right-hand side of \cref{florida}.

\begin{lemma}\label{ts132}Suppose that all the conditions of \cref{sun} are satisfied. Then the following holds  for all $q\ge 1$: 
$$
\begin{aligned}
(B)
&\le \Big\|\sup_{g\in \mathcal{F}_{\TV}}|\Delta g|\Big\|_\infty\frac{\mathcal{S}\|J_n\|^2_2}{|\A_n|}\Big(|\B_{2b_n}|\sup_{\phi\in \A_n}\|Q_n(\phi)\|_2^2 + 2\sup_{\phi\in \A_n}\|Q_n(\phi)\|_{\frac{q}{q-1}}\mathcal{R}_\xi(n,q,b_n)\Big).
\end{aligned}
$$
\end{lemma}

\begin{proof}
Recall that in the previous lemma, we simplified $(B)$ to
$$
\begin{aligned}
(B)=&\Bigg\|\sup_{g\in\mathcal{F}_{\TV}}\bigg|\sum_{k=1}^{\infty}\sum_{i\leq J_n}\E\Big[f_n(\phi_{ni} X_n)\ind(E^{n,k}_{b_n,i})g({W}^{\phi_{ni}}_{ 2 b_n}+k)\Big| \Jn \Big]\\
& - \sum_{k=1}^{\infty}\sum_{i\leq J_n}\E\Big[f_n(\phi_{ni} X_n)\ind(E^{n,k}_{b_n,i})\Big|\Jn \Big]~\mathbb{E}\Big[g(W_n+k)\Big| \Jn\Big]\bigg|\Bigg\|_1.
\end{aligned}
$$
For the ease of notation, for all $k\in \mathbb{N}$ and $i\le J_n$ given $J_n$, we write $ h_{n,k}^i= f_n(\phi_{ni} X_n)\ind(E^{n,k}_{b_n,i}).$ %Moreover we also denote for all $g\in \mathcal{F}$
% $$D^g_k  :=  \sum_{i \le j_n}\Big(\E\Big[f_n(\phi_{ni} X_n)\ind(E^{n,k}_{b_n,i})g(W^{\phi_{ni}}_{ 2 b_n}+k)\Big|\boldsymbol\phi_n, \mathbb{G}\Big]
% - \E\Big[f_n(\phi_{ni} X_n)\ind(E_{b_n,\phi}^{n,k})\Big| \boldsymbol\phi_n,\mathbb{G}\Big]\mathbb{E}\Big[g(W_n+k)\Big|\boldsymbol\phi_n, \mathbb{G}\Big]\Big)$$
\\We remark that for all $k\in \mathbb{N}$, we have 
$$
\begin{aligned}
& \sum_{i \le J_n}\Big(\E\Big[f_n(\phi_{ni} X_n)\ind( E^{n,k}_{b_n,i})g(W^{\phi_{ni}}_{ 2 b_n}+k)\Big|\Jn\Big]\\
&\qquad\qquad- \E\Big[f_n(\phi_{ni} X_n)\ind(E_{b_n,\phi}^{n,k})\Big| \Jn\Big]\mathbb{E}\Big[g(W_n+k)\Big|\Jn\Big]\Big)\\
= & \sum_{i \le J_n}\E\Big[h_{n,k}^i~g(W^{\phi_{ni}}_{ 2 b_n}+k)\Big|\Jn\Big]
- \sum_{i \le J_n}\E\Big[h_{n,k}^i\Big| \Jn\Big]\mathbb{E}\Big[g(W_n+k)\Big|\Jn\Big]\\
= & \sum_{i \le J_n}\E\Big[h_{n,k}^i~g(W^{\phi_{ni}}_{ 2 b_n}+k)\Big|\Jn\Big]
- \sum_{i \le J_n}\E\Big[h_{n,k}^i\Big| \Jn\Big]\mathbb{E}\Big[g(W^{\phi_{ni}}_{ 2 b_n}+k)\Big|\Jn\Big] \\
 & + \sum_{i \le J_n}\E\Big[h_{n,k}^i\Big|\Jn\Big]~\Big(\mathbb{E}\Big[g(W^{\phi_{ni}}_{ 2 b_n}+k)\Big|\Jn\Big]-\mathbb{E}\Big[g(W_n+k)\Big|\Jn\Big]\Big).
\end{aligned}
$$

Therefore, using the triangle inequality, we obtain that 
\begin{equation*}
\begin{aligned}
(B)  \leq &\norm{\sup_{g\in \mathcal{F}_{\TV}}\bigg|\sum_{i \le J_n}\Big(\sum_{k=1}^{\infty}\E\Big[h_{n,k}^ig(W^{\phi_{ni}}_{ 2 b_n}+k)\Big|\Jn\Big]
- \E\Big[h_{n,k}^i\Big| \Jn\Big]\mathbb{E}\Big[g(W^{\phi_{ni}}_{ 2 b_n}+k)\Big|\Jn\Big]\Big)\bigg|}_1\\
 + & \norm{\sup_{g\in \mathcal{F}_{\TV}}\bigg|\sum_{i \le J_n}\sum_{k=1}^{\infty}\E\Big[h_{n,k}^i\Big|\Jn\Big]~\Big(\mathbb{E}\Big[g(W^{\phi_{ni}}_{ 2 b_n}+k)\Big|\Jn\Big]-\mathbb{E}\Big[g(W_n+k)\Big|\Jn\Big]\Big)\bigg|}_1\\
  =:& (C_1) + (C_2)
\end{aligned}
\end{equation*} We will proceed to bound each term successively. Firstly we remark that the term $(C_2)$ can be bounded by using the fact that, for all functions $g\in \mathcal{F}_{\TV}$, we know that $g$ is $\|\Delta g\|_{\infty}$-Lipchitz. Indeed, by exploiting H\"{o}lder's inequality, we obtain that 
\begin{align*}
(C_2) & = \norm{\sup_{g\in \mathcal{F}_{\TV}}\bigg|\sum_{i \le J_n}\sum_{k=1}^{\infty}\E\Big[h_{n,k}^i\Big|\Jn\Big]~\Big(\mathbb{E}\Big[g(W^{\phi_{ni}}_{ 2 b_n}+k)\Big|\Jn\Big]-\mathbb{E}\Big[g(W_n+k)\Big|\Jn\Big]\Big)\bigg|}_1\\
%& = \norm{\sup_{g\in \mathcal{F}_{\TV}}\bigg|\sum_{i \le J_n}\sum_{k=1}^{\infty}\E\Big[h_{n,k}^i\Big|\Jn\Big]\mathbb{E}\Big[g(W^{\phi_{ni}}_{ 2 b_n}+k)-g(W_n+k)\Big|\Jn\Big]\bigg|}_1\\
%& = \sum_{i \le j_n}\norm{\sup_{g\in \mathcal{F}_{\TV}}\bigg|\sum_{k=1}^{\infty}\E\Big[h_{n,k}^i\Big|\Jn\Big]\Big(\mathbb{E}\Big[g(W^{\phi_{ni}}_{ 2 b_n}+k)\Big|\Jn\Big]-\mathbb{E}\Big[g(W_n+k)\Big|\Jn\Big]\Big)\bigg|}_1\\
& \leq \Big\|\sup_{g\in \mathcal{F}_{\TV}}|\Delta g|\Big\|_\infty\norm{\sum_{i \le J_n}\sum_{k=1}^{\infty}\E[h_{n,k}^{i}| \Jn]\Big|\mathbb{E}\Big[{W^{\phi_{ni}}_{ 2 b_n}-W_n}\Big|\Jn\Big]\Big|}_1\\
& \le \Big\|\sup_{g\in \mathcal{F}_{\TV}}|\Delta g|\Big\|_\infty\bigg\|\sum_{i \le J_n}\sum_{k=1}^{\infty}\E[h_{n,k}^i| \Jn]\sum_{\substack{i^{\prime} \le J_n\\ d(\phi_{ni}, \phi_{ni\prime})<2b_n}}\mathbb{E}[f_n(\phi_{ni^\prime} X)|\Jn]\bigg\|_1\\
& \overset{(a)}{=} \Big\|\sup_{g\in \mathcal{F}_{\TV}}|\Delta g|\Big\|_\infty\bigg\|\sum_{i \le J_n}\sum_{\substack{i^{\prime} \le J_n\\ d(\phi_{ni}, \phi_{ni\prime})<2b_n}}\E[f_n(\phi_{ni} X_n)| \Jn]~\mathbb{E}[f_n(\phi_{ni^\prime} X)|\Jn]\bigg\|_1\\
& \overset{(b)}{\le} \Big\|\sup_{g\in \mathcal{F}_{\TV}}|\Delta g|\Big\|_\infty \|J_n\|^2_2\int_{\A_n}\int_{\B^{\A_n}_{2b_n}(\phi)}\norm{Q_n(\phi)Q_n(\phi')}_1 d\nu_n(\phi')d\nu_n(\phi)\\
&\le \Big\|\sup_{g\in \mathcal{F}_{\TV}}|\Delta g|\Big\|_\infty{\|J_n\|^2_2}\int_{\A_n}\|Q_n(\phi)\|_2\int_{\B^{\A_n}_{2b_n}(\phi)}\|{Q_n(\phi')}\|_2 d\nu_n(\phi')d\nu_n(\phi)\\
&\le \Big\|\sup_{g\in \mathcal{F}_{\TV}}|\Delta g|\Big\|_\infty\frac{\mathcal{S}\|J_n\|^2_2|\B_{2b_n}|}{|\A_n|}\sup_{\phi\in \A_n}\|Q_n(\phi)\|_2^2,
\end{align*} where to obtain (a) we used the fact that $\sum_{k=1}^{\infty}h_{n,k}^{i}=f_n(\phi_{ni} X_n),$ and to obtain (b) we exploit the independence of $J_n$ and $\nu_n$.

\noindent We now bound $(C_1)$. In this goal, remark that by
\cref{lemma26general} and a telescoping sum argument, we obtain that 

\begin{align*}
 (C_1)& = \bigg\|\sup_{g\in \mathcal{F}_{\TV}}\bigg|\sum_{i \le J_n}\Big(\sum_{k=1}^{\infty}\E\Big[h_{n,k}^i~g(W^{\phi_{ni}}_{ 2 b_n}+k)\Big|\boldsymbol\phi_n, \mathbb{G}\Big]
\\& \qquad\qquad\qquad\qquad\qquad\qquad\qquad\qquad- \E\Big[h_{n,k}^i\Big| \Jn\Big]\mathbb{E}\Big[g(W^{\phi_{ni}}_{ 2 b_n}+k)\Big|\Jn\Big]\Big)\bigg|\bigg\|_1\\
& = \bigg\|\sup_{g\in \mathcal{F}_{\TV}}\Big|\sum_{i \le J_n}\sum_{k=1}^{\infty}\sum_{j\geq 2b_n}\E\Big[h_{n,k}^i\Big(g(W^{\phi_{ni}}_{j}+k)-g(W^{\phi_{ni}}_{j+1}+k)\Big)\Big| \Jn\Big]\\&\qquad \qquad \qquad \qquad \qquad- \E\Big[h_{n,k}^i\Big|\Jn\Big]~\mathbb{E}\Big[g(W^{\phi_{ni}}_{j}+k)-g(W^{\phi_{ni}}_{j+1}+k)\Big| \Jn\Big]\Big|\bigg\|_1\\
& = \sum_{j\geq 2b_n}\bigg\|\sup_{g\in \mathcal{F}_{\TV}}\Big|\sum_{i \le J_n}\sum_{k=1}^{\infty}\E\Big[\Big(h_{n,k}^i-\E[h_{n,k}^i| \Jn]\Big)\\& \qquad\qquad\qquad\qquad\qquad\qquad\qquad\qquad\times\Big(g(W^{\phi_{ni}}_{j}+k)-g(W^{\phi_{ni}}_{j+1}+k)\Big)\Big| \Jn\Big]\Big|\bigg\|_1\\
%& \overset{(a)}{\leq} 2\sup_{g\in \mathcal{F}_{\TV}}\norm{\Delta g}_\infty\sum_{\phi \in \A_n}\norm{\E[f_n(\phi X_n)|\mathbb{G}]}_{\frac{p}{p-1}}\sum_{j\geq 2b_n}|\B_{j+1}\bac\TVlash \B_j|\xi_{n,p}^{b_n}(j-b_n|\mathbb{G})\\
&\overset{(a)}{ \leq} 2 \frac{\|J_n\|^2_2\mathcal{S}}{|\A_n|}\Big\|\sup_{g\in \mathcal{F}}|\Delta g|\Big\|_{\infty}~\sup_{\phi\in \A_n}\|{Q_n(\phi)}\|_{\frac{q}{q-1}}\sum_{j\ge 2b_n}|\B_{j+1}\setminus\B_{j}|\xi_{n,q}^{b_n}(j|\G)\\
&= 2 \frac{\|J_n\|^2_2\mathcal{S}}{|\A_n|}\Big\|\sup_{g\in \mathcal{F}}|\Delta g|\Big\|_{\infty}~\sup_{\phi\in \A_n}\|{Q_n({\phi})}\|_{\frac{q}{q-1}}\mathcal{R}_\xi(n,q,b_n),
% & \leq 2\sup_{g\in \mathcal{F}_{\TV}}\norm{\Delta g}_\infty\sum_{\phi \in \A_n}\norm{Q_n(\phi)}_{\frac{q}{q-1}}\sum_{j\geq 2b_n}|\B^{\A_n}_{j+1}(\phi)\setminus\B^{\A_n}_{j}(\phi)|\xi_{n,q}^{b_n}(j-b_n|\mathbb{G})
% \\
% & \leq 2\sup_{g\in \mathcal{F}_{\TV}}\norm{\Delta g}_\infty\sum_{\phi \in \A_n}\norm{Q_n(\phi)}_{\frac{q}{q-1}}\sum_{j\geq 2b_n}|\B_{j+1}\setminus\B_{j}|\xi_{n,q}^{b_n}(j|\mathbb{G})
\end{align*} where to get (a) we used \cref{lemma26general}. 
This directly implies the desired result.
\end{proof}

With the previously established lemmas, we are now prepared to provide a bound for the term $(M)$.
\begin{lemma}\label{ts133}
Given all the conditions of \cref{sun}, the following inequality holds:
\begin{align*}
(M)
\leq&\mathcal{S}\frac{\|J_n\|^2_2}{|\A_n|} \Big\|\sup_{g\in \mathcal{F}_{\TV}}|\Delta g|\Big\|_{\infty}\Big(\sup_{\phi\in \A_n}\|Q_n(\phi)\|_{\frac{p}{p-1}}\mathcal{R}_\Psi(n,p,b_n)\\&\qquad+2\sup_{\phi\in \A_n}\|Q_n(\phi)\|_{\frac{q}{q-1}}\mathcal{R}_\xi(n,q,b_n)
+\sup_{\phi\in \A_n}\|Q_n(\phi)\|_2^2
\Big( {|\B_{2b_n}|}+{|\B_{2b_n}\setminus\B_{b_n}|}\Big)\Big).\end{align*}
\end{lemma}
\begin{proof}
This is a direct result of combining \cref{sun,ts131,ts132}.
\end{proof}

\subsubsection{Bounding $(H)$ }
\begin{lemma}\label{cpconvergence_randomization}Suppose that the conditions of \cref{lemma26general} hold. 
Let $(b_n)$ and $(c_n)$ be increasing sequences of positive integers. Then the following holds for all $0<\alpha < 1$: 
\begin{align*}
(H) \le& \Big\|\sup_{g\in \mathcal{F}_{\TV}}|g|\Big\|_{\infty}\bigg({\epsilon_n}^{\alpha} + 2\frac{\mathcal{S}\|J_n\|^2_2|\B_{b_n}\setminus \B_{c_n}|}{|\A_n|}\sup_{\phi\in \A_n}\|Q_n(\phi)\|_2^2 \\&\qquad\qquad\qquad + 2{\E[J_n]}\Psi_{n,s}(c_n|\G)\sup_{\phi\in \A_n}\|Q_n(\phi)\|_{\frac{s}{s-1}}\\ 
&\qquad\qquad\qquad  + \frac{1}{|\A_n|}\sup_{\phi\in \A_n}\|Q_n(\phi)\|_1\Big(2\mathcal{S}\abs{\B_{c_n}}\frac{1}{k_n}\|J_n\|^2_2\\
&\qquad\qquad\qquad +  \mathcal{S}\abs{\B_{b_n}}\sqrt{2\mathrm{Var}(J_n)}\|J_n\|_2 + |\A_n|\sqrt{2\mathrm{Var}(J_n)}~\Big)\bigg),
\end{align*}
where $
\resizebox{0.95\textwidth}{!}{$
\epsilon_n:=2\Big(\frac{2\mathcal{S}^2\|J_n\|_3^3\,|\B_{b_n}|^2}{|\A_n|^2}\sup_{\phi\in \A_n}\|\Q_n(\phi)\|_1^2+\big(\|J_n\|_1+\frac{4\mathcal{S}\|J_n\|^2_2\,|\B_{b_n}|}{|\A_n|}\big)\sup_{\phi\in\A_n}\|\Q_n(\phi)\|_2^2\Big)^{\frac{1}{2}}$}
$ and $k_n:=\lfloor {\epsilon_n}^{\alpha - 1}\rfloor.$
\end{lemma}

\begin{proof}
Recall that we have 
$$
(H) = \bigg\| \sup_{g\in \mathcal{F}_{\TV}}\Big|\E\Big[\sum_{k=1}^{\infty} k (\widetilde{\lambda}^{b_n}_{n}(k)-\lambda^{b_n}_{n}(k)) g(W_n+k)\Big|J_n, \widetilde{{\lambda}}^{b_n}_{n}, \mathbb{G}\Big]\Big|~\bigg\|_1.
$$
Also recall that for the simplicity of notations, we denote $\mathcal{L}_n := \{J_n,\widetilde{{\lambda}}^{b_n}_{n},\G\}$. Then we have 
\begin{equation}\label{milano}
\begin{aligned}
(H)
& \overset{}{=} \bigg\|\sup_{g\in\mathcal{F}_{\TV}}\E\Big[\sum_{k=1}^{\infty} (k \widetilde\lambda^{b_n}_{n}(k)- k \lambda^{b_n}_{n}(k)) g(W_n+k)\Big| \Ln \Big]\bigg\|_1\\
& \overset{(a)}{\leq} \Big\|\sup_{g\in\mathcal{F}_{\TV}}|g|\Big\|_\infty\E\Big[\sum_{k=1}^{\infty} 
\Big|k(\tilde\lambda^{b_n}_{n}(k)- \lambda^{b_n}_{n}(k))\Big|\Big], %\\
%& \leq \Big\|\sup_{g\in\mathcal{F}_{\TV}}g\Big\|_\infty\sum_{k=1}^{\infty} 
%k\bigg\|\tilde\lambda^{b_n}_{n}(k)- \lambda^{b_n}_{n}(k)\bigg\|_1 
\end{aligned}
\end{equation}
where $(a)$ is a consequence of Jensen's inequality. \\
We know that $\Big\|\sup_{g\in\mathcal{F}_{\TV}}|g|\Big\|_\infty$ is bounded by $\Big\|H_0(
    \lambda^{b_n}_{n}
    )\Big\|_\infty$. 
    Therefore, we only need to establish a bound for $\E\Big[\sum_{k=1}^{\infty} 
\Big|k(\tilde\lambda^{b_n}_{n}(k)- \lambda^{b_n}_{n}(k))\Big|\Big]$. %$ \sum_{k=1}^{\infty} k\norm{\widetilde\lambda^{b_n}_{n}(k)- \lambda^{b_n}_{n}(k)}_1$.\\
%In this goal we first bound $\norm{k(\lambda^{b_n}_{n}(k)- \widetilde\lambda^{b_n}_{n}(k))}_1$ for all $k\in \mathbb{N}$. 
To do so, recall that we defined $\Jn=\{J_n, \boldsymbol{\phi}_n,\G\}$. We note that the following holds:
\begin{equation}
\begin{aligned}\label{trav}
&\E\Big[\sum_{k=1}^{\infty} 
\Big|k(\tilde\lambda^{b_n}_{n}(k)- \lambda^{b_n}_{n}(k))\Big|\Big]\\
% \sum_{k=1}^{\infty} k\norm{\widetilde\lambda^{b_n}_{n}(k)- \lambda^{b_n}_{n}(k)}_1 
%    =&\Bigg\|\sum_{i\leq j_n}\mathbb{E}\bigg[f_{n}(\phi_{ni} X_{n}) \mathbb{I}(\sum_{\substack{i^{\prime}\leq j_n \\ d(\phi_{ni}, \phi_{ni^{\prime}}) \leq b_n}} f_{n}(\phi_{ni^{\prime}} X_{n})=k)\bigg| \boldsymbol\phi_n \bigg]-j_n\mathbb{E}\bigg[f_{n}(\phi_{n1} X_{n}) \mathbb{I}(\sum_{\substack{i^{\prime}\leq j_n \\ d(\phi_{n1}, \phi_{ni^{\prime}}) \leq b_n}} f_{n}(\phi_{ni^{\prime}} X_{n})=k)\bigg]\Bigg\|_{2}\\
    = &\bigg\|\sum_{k=1}^{\infty}\Big|\sum_{i\leq J_n}\mathbb{E}\Big[f_{n}(\phi_{ni} X_{n}) \mathbb{I}(E^{n,k}_{b_n,i})\Big| \Jn\Big]-\mathbb{E}\Big[\sum_{i\leq J_n}f_{n}(\phi_{ni} X_{n}) \mathbb{I}(E^{n,k}_{b_n,i})\Big|\G\Big]\Big|\bigg\|_{1}\\
    \leq &\sum_{k=1}^{\infty}\bigg\|\sum_{i\leq J_n}\mathbb{E}\Big[f_{n}(\phi_{ni} X_{n}) \mathbb{I}(E^{n,k}_{b_n,i})\Big| \Jn\Big]-\mathbb{E}\Big[\sum_{i\leq J_n}f_{n}(\phi_{ni} X_{n}) \mathbb{I}(E^{n,k}_{b_n,i})\Big|J_n, \G\Big]\bigg\|_{1}\\
    &+ \bigg\|\sum_{k=1}^{\infty}\Big|\sum_{i\leq J_n}\mathbb{E}\Big[f_{n}(\phi_{ni} X_{n}) \mathbb{I}(E^{n,k}_{b_n,i})\Big| J_n, \G\Big]-\mathbb{E}\Big[\sum_{i\leq J_n}f_{n}(\phi_{ni} X_{n}) \mathbb{I}(E^{n,k}_{b_n,i})\Big|\G\Big]\Big|\bigg\|_{1}\\
    %&\sum_{k=1}^{\infty}\Big\|\Big(\sum_{i\leq J_n}\mathbb{E}\Big[f_{n}(\phi_{ni} X_{n}) \mathbb{I}(E^{n,k}_{b_n,i})\Big| \Jn\Big]-\mathbb{E}\Big[\sum_{i\leq J_n}f_{n}(\phi_{ni} X_{n}) \mathbb{I}(E^{n,k}_{b_n,i})\Big|J_n, \G\Big]\Big)\Big\|_{2}\\
%    &+ \Big\|\sum_{k=1}^{\infty}\Big(\sum_{i\leq J_n}\mathbb{E}\Big[f_{n}(\phi_{ni} X_{n}) \mathbb{I}(E^{n,k}_{b_n,i})\Big| J_n, \G\Big]-\mathbb{E}\Big[\sum_{i\leq J_n}f_{n}(\phi_{ni} X_{n}) \mathbb{I}(E^{n,k}_{b_n,i})\Big|\G\Big]\Big)\Big\|_{1}\\
    = :& (S) + (T).
\end{aligned}
\end{equation}
By \cref{EScorrectionterm}, we have 
$$
   \begin{aligned}
    (T) = & \bigg\|\sum_{k=1}^{\infty}\Big|\sum_{i\leq J_n}\mathbb{E}\Big[f_{n}(\phi_{ni} X_{n}) \mathbb{I}(E^{n,k}_{b_n,i})\Big| J_n, \G\Big]-\mathbb{E}\Big[\sum_{i\leq J_n}f_{n}(\phi_{ni} X_{n}) \mathbb{I}(E^{n,k}_{b_n,i})\Big|\G\Big]\Big|\bigg\|_{1}\\  
    \le  & \sup_{\phi\in\A_n}\|Q_{n}(\phi)\|_1\sqrt{2\mathrm{Var}(J_n)}\Big( \frac{\mathcal{S}\abs{\B_{b_n}}}{\abs{\A_n}}\|J_n\|_2 + 1\Big).
    \end{aligned}
$$
We then bound $(S)$. We remark that 
$$ 
\begin{aligned}
    (S) = & \sum_{k=1}^{\infty}\bigg\|\sum_{i\leq J_n}\mathbb{E}\Big[f_{n}(\phi_{ni} X_{n}) \mathbb{I}(E^{n,k}_{b_n,i})\Big| \Jn\Big]-\mathbb{E}\Big[\sum_{i\leq J_n}f_{n}(\phi_{ni} X_{n}) \mathbb{I}(E^{n,k}_{b_n,i})\Big|J_n, \G\Big]\bigg\|_{1}\\
    =:& \sum_{k=1}^{\infty}(S_k),
\end{aligned}
$$
where we denoted 
$$
(S_k) := \bigg\|\sum_{i\leq J_n}\mathbb{E}\Big[f_{n}(\phi_{ni} X_{n}) \mathbb{I}(E^{n,k}_{b_n,i})\Big| \Jn\Big]-\mathbb{E}\Big[\sum_{i\leq J_n}f_{n}(\phi_{ni} X_{n}) \mathbb{I}(E^{n,k}_{b_n,i})\Big|J_n, \G\Big]\bigg\|_{1}.
$$
We remark that 
$$(S_k) \le \mathbb{E}\Big[\mathrm{Var}\Big(\sum_{i\leq J_n}\mathbb{E}\Big[f_{n}(\phi_{ni} X_{n}) \mathbb{I}(E^{n,k}_{b_n,i})\Big|\Jn\Big]\Big|J_n,\G\Big)\Big]^{\frac{1}{2}}.$$
To bound $(S)$, we bound each $(S_k)$. In this goal, we bound \\$\mathbb{E}\Big[\mathrm{Var}\Big(\sum_{i\leq J_n}\mathbb{E}\Big[f_{n}(\phi_{ni} X_{n}) \mathbb{I}(E^{n,k}_{b_n,i})\Big|\Jn\Big]\Big|J_n,\G\Big)\Big]$ for all $k\in \mathbb{N}$. Given $J_n$, for all $i\le J_n$, denote $\boldsymbol\phi^{(i)}_n := (\phi_{n1},\dots,\phi_{ni-1},\phi^\prime_{ni}, \phi_{ni+1}, \cdots, \phi_{nJ_n})$ the process where we replaced $\phi_{ni}$ by an independent copy $\phi^\prime_{ni} \sim \nu_n$. 
For the ease of notation, we denote for all $b>0$,
$$E^{(j),n,k}_{b,{i}}:= \Big\{\sum_{\substack{i^{\prime}\leq J_n \\ d(\phi^{(j)}_{ni}, \phi^{(j)}_{ni^{\prime}}) \leq b}} f_{n}(\phi^{(j)}_{ni^\prime} X_{n})=k\Big\}.$$ 
Recall that we used the notation $\mathcal{J}_n := \{J_n,\boldsymbol{\phi}_n,\G\}$. We further denote $\mathcal{J}_n^{(j)} := \{J_n,\boldsymbol{\phi}^{(j)}_n,\G\}$.
By the Efron-Stein inequality, we have 
\begin{equation}
\begin{aligned}
    &\mathbb{E}\Big[\mathrm{Var}\Big(\sum_{i\leq J_n}\mathbb{E}\Big[f_{n}(\phi_{ni} X_{n}) \mathbb{I}(E^{n,k}_{b_n,i})\Big| \Jn \Big]\Big|J_n, \G\Big)\Big]\\ 
    \leq & \frac{1}{2}\E\bigg[\sum_{j\le J_n}\E\Big[\Big(\sum_{i\leq J_n}\E\Big[f_{n}(\phi_{ni} X_{n}) \mathbb{I}(E^{n,k}_{b_n,{i}})\Big| \Jn \Big]\\&\qquad\qquad-\E\Big[f_{n}(\phi^{(j)}_{ni} X_{n}) \mathbb{I}(E^{(j),n,k}_{b_n,{i}})\Big| \Jjn\Big]\Big)^2\Big|J_n, \G\Big]\bigg]\\
    \leq & \frac{1}{2}\E\Big[\sum_{j\leq J_n}\Big(\sum_{i\leq J_n}\E\Big[f_{n}(\phi_{ni} X_{n}) \mathbb{I}(E^{n,k}_{b_n,{i}})\Big| \Jn\Big]-\E\Big[f_{n}(\phi^{(j)}_{ni} X_{n}) \mathbb{I}(E^{(j),n,k}_{b_n,{i}})\Big|\Jn^{(j)}\Big]\Big)^2\Big]\\
    \overset{(a)}{\leq}& \E\bigg[\sum_{j\leq J_n}\Big(\E\Big[f_{n}(\phi_{nj} X_{n}) \mathbb{I}(E^{n,k}_{b_n,{j}})\Big|\Jn \Big]-\E\Big[f_{n}(\phi^{(j)}_{nj} X_{n}) \mathbb{I}(E^{(j),n,k}_{b_n,{j}})\Big|\Jn^{(j)}\Big]\Big)^2\Big]\\
    &+\mathbb{E}\Big[\sum_{j\leq J_n}\Big(\sum_{i\leq J_n, i\neq j}\mathbb{E}\Big[f_{n}(\phi_{ni} X_{n}) \mathbb{I}(E^{n,k}_{b_n,{i}})\Big|\Jn\Big]\ind(d(\phi_{ni},\phi_{nj})\leq b_n)\\
    &~~-\sum_{i\leq J_n, i\neq j}\mathbb{E}\Big[f_{n}(\phi_{ni} X_{n}) \mathbb{I}(E^{(j),n,k}_{b_n,{i}})\Big| \Jn^{(j)}\Big]\ind(d(\phi^{}_{ni},\phi^{(j)}_{nj})\leq b_n)\Bigr)^2\biggl],
\end{aligned}
\end{equation}
where to get $(a)$ we used the classical inequality $(a+b)^2\le 2(a^2+b^2)$ 
and the fact that if $d(\phi_{ni},\phi_{nj}^{(j)})>b_n$ and  $d(\phi_{ni},\phi_{nj})>b_n$ then $$f_n(\phi_{ni}X_n)\mathbb{I}(E^{n,k}_{b_n,i})\Big|\mathcal{J}_n=f_n(\phi_{ni}^{(j)}X_n)\mathbb{I}(E^{(j),n,k}_{b_n,i})\Big|\mathcal{J}_n^{(j)}.$$ 
This further implies that \begin{equation}
\begin{aligned}
    &\mathbb{E}\Big[\mathrm{Var}\Big(\sum_{j\leq j_n}\mathbb{E}\big[f_{n}(\phi_{ni} X_{n}) \mathbb{I}(E^{n,k}_{b_n,i})\big| \Jn\big]\Big|J_n,\G\Big)\Big]\\
    \overset{(a)}{\leq}& 4 \E\bigg[\sum_{j\leq J_n}\E\Big[f_{n}(\phi_{nj} X_{n}) \mathbb{I}(E^{n,k}_{b_n,{j}})\Big|\Jn \Big]^2\bigg]\\
    +& 4 \mathbb{E}\Big[\sum_{j\leq J_n}\Big(\sum_{i\leq J_n, i\neq j}\mathbb{E}[f_{n}(\phi_{ni} X_{n}) \mathbb{I}(E^{n,k}_{b_n,{i}})|\Jn ]\ind(d(\phi_{ni},\phi_{nj})\leq b_n)\Bigr)^2\biggl]
    \\ \le &  4\Big((A)+(B)\Big).
\end{aligned}
\end{equation}
% After that, using the Cauchy-Schwarz inequality again, we get
% \begin{equation*}
% \begin{aligned}
%     &\mathrm{Var}\Big(\sum_{i\leq j_n}\mathbb{E}[f_{n}(\phi_{ni} X_{n}) \mathbb{I}(S^{b_n}_{\phi_n}(\phi_{ni})=k)|\boldsymbol\phi_n ]\Big)\\
%     \leq & 2j_n\E\biggl[\mathbb{E}[f_{n}(\phi_{n1} X_{n}) \mathbb{I}(S^{b_n}_{\phi_n}(\phi_{n1})=k)| \boldsymbol\phi_n ]^2+\mathbb{E}[f_{n}(\phi^{(1)}_{n1} X_{n}) \mathbb{I}(S^{b_n}_{\phi^{(1)}}(\phi^{(1)}_{n1})=k)| \boldsymbol\phi^{(1)}_n ]^2\\
%     &+ \left(\sum_{1< i\leq j_n}\mathbb{E}[f_{n}(\phi_{ni} X_{n}) \mathbb{I}(S^{b_n}_{\phi_n}(\phi_{ni})=k)| \boldsymbol\phi_n ]\ind(d(\phi_{ni},\phi_{n1})\leq b_n)\right)^2\\
%     &+\left(\sum_{1< i\leq j_n}\mathbb{E}[f_{n}(\phi^{(1)}_{ni} X_{n}) \mathbb{I}(S^{b_n}_{\phi^{(1)}}(\phi^{(1)}_{ni})=k)| \boldsymbol\phi^{(1)}_n ]\ind(d(\phi^{(1)}_{ni},\phi^{(1)}_{n1})\leq b_n)\right)^2\biggl]\\
%     =&4j_n\E\left[\mathbb{E}[f_{n}(\phi_{n1} X_{n}) \mathbb{I}(S^{b_n}_{\phi_n}(\phi_{n1})=k)| \boldsymbol\phi_n ]^2\right]\\
%     &+ 4j_n\E[\left(\sum_{1< i\leq j_n}\mathbb{E}[f_{n}(\phi_{ni} X_{n}) \mathbb{I}(S^{b_n}_{\phi_n}(\phi_{ni})=k)| \boldsymbol\phi_n ]\ind(d(\phi_{ni},\phi_{n1})\leq b_n)\right)^2]\\
%     =&: 4j_n\Bigl((A)+(B)\Bigl)
% \end{aligned}
% \end{equation*}
where to get $(a)$ we use Jensen inequality and the fact that
$$\E\Big[f_{n}(\phi_{nj} X_{n}) \mathbb{I}(E^{n,k}_{b_n,{j}})\Big|\Jn \Big]\overset{d}{=}\E\Big[f_{n}(\phi^{(j)}_{nj} X_{n}) \mathbb{I}(E^{(j),n,k}_{b_n,{j}})\Big|\Jn^{(j)} \Big]$$
and the fact that
$$
\begin{aligned}
&\sum_{i\leq J_n, i\neq j}\mathbb{E}\Big[f_{n}(\phi_{ni} X_{n}) \mathbb{I}(E^{n,k}_{b_n,{i}})\Big|\Jn\Big]\ind(d(\phi_{ni},\phi_{nj})\leq b_n)\\
\overset{d}{=}&\sum_{i\leq J_n, i\neq j}\mathbb{E}\Big[f_{n}(\phi_{ni} X_{n}) \mathbb{I}(E^{(j),n,k}_{b_n,{i}})\Big| \Jn^{(j)}\Big]\ind(d(\phi^{}_{ni},\phi^{(j)}_{nj})\leq b_n).
\end{aligned}
$$
We will now proceed to provide bounds for terms $(A)$ and $(B)$ in succession.
Firstly we remark that the following holds
\begin{align*}
    (A)=&\E\bigg[\sum_{j\leq J_n}\E\Big[f_{n}(\phi_{nj} X_{n}) \mathbb{I}(E^{n,k}_{b_n,{j}})\Big|\Jn \Big]^2\bigg]\\
    \leq &\E\bigg[\sum_{j\leq J_n}\E\Big[f_{n}(\phi_{nj} X_{n})\Big|\Jn \Big]^2\bigg]\\
    \overset{(a)}{=}&{\E[J_n]}\int_{\A_n}\|Q_n(\phi)\|_2^2d\nu_n(\phi)
    \\\le&{\E[J_n]} \sup_{\phi\in\A_n}\|Q_n(\phi)\|_2^2
  %  = & \E\left[\mathbb{E}[f_{n}( X_{n})]^2\right] \tag{by invariance}\\
    %=& (\frac{\mu}{\abs{\A_n}})^2.
\end{align*} where to obtain (a) we exploited the independence of $J_n$ and $\phi_{ni}$'s and $X_n$. We now seek to upper bound $(B)$. For the simplicity of notation, denote $$F^{b_n}_{i,j} = 
\ind(d(\phi_{ni},\phi_{nj})\leq b_n).$$
We remark that 
\begin{align*}
    (B)=& \mathbb{E}\Big[\sum_{j\leq J_n}\Big(\sum_{i\leq J_n, i\neq j}\mathbb{E}[f_{n}(\phi_{ni} X_{n}) \mathbb{I}(E^{n,k}_{b_n,{i}})|\Jn ]F^{b_n}_{i,j}\Bigr)^2\Bigr]\\
    \leq & \E\Bigr[\sum_{j\leq J_n}\Bigl(\sum_{i\leq J_n, i\neq j}\mathbb{E}\big[f_{n}(\phi_{ni} X_{n}) | \Jn\big]F^{b_n}_{i,j}\Bigr)^2\Bigr] \\
    \le & 2\E\Bigr[\sum_{j\leq J_n}\Bigl(\sum_{i\leq J_n, i\neq j}\mathbb{E}\big[f_{n}(\phi_{ni} X_{n}) | \Jn \big]F^{b_n}_{i,j}-\mathbb{E}\big[f_{n}(\phi_{ni} X_{n})F^{b_n}_{i,j}\big|\phi_{nj}, J_n\big]~\Bigr)^2\Bigr]
    \\&+ 2\E\Bigr[\sum_{j\leq J_n}\Bigl(\sum_{i\leq J_n, i\neq j}\mathbb{E}\Big[f_{n}(\phi_{ni} X_{n}) F^{b_n}_{i,j}\Big|\phi_{nj},J_n\Big]\Bigr)^2\Bigr]
    \\\le & (B_1)+(B_2)
    % = & \E\Bigr[\Bigl(\sum_{1< i\leq j_n}\mathbb{E}[f_{n}(\phi_{ni} X_{n}) | \phi_{ni}]F^{b_n}_{i}\Bigl)^2\Bigr]\\
  %   =&\frac{1}{|\A_n|^2}\mathbb{E}\Big(\int_{\B(\phi_{n1},b_n)}\mu_n^{\phi}d\nu_n(\phi)\Big)\\
  % %  = & \E\Bigl[\Bigl(\sum_{1< i\leq j_n}\mathbb{E}[f_{n}(X_{n})]F^{b_n}_{i}\Bigl)^2\Bigl] \tag{by invariance and the independence of and $X_n$}\\
  %   = & (\frac{\mu}{\abs{\A_n}})^2\E\Bigr[\Bigl(\sum_{1< i\leq j_n}F^{b_n}_{i}\Bigl)^2\Bigr]
\end{align*}
We can upper bound $(B_1)$ and $(B_2)$ separately. We note that 
\begin{align*}
   (B_2)&= 2\E\Bigr[\sum_{j\leq J_n}\Bigl(\sum_{i\leq J_n, i \neq j}\mathbb{E}\Big[f_{n}(\phi_{ni} X_{n}) \times F^{b_n}_{i,j}\Big|\phi_{nj},J_n\Big]\Bigl)^2\Bigr]
    \\&
    \overset{(a)}{\le} 2\E\Bigr[\sum_{j\leq J_n}(J_n-1)\sum_{i\leq J_n, i \neq j}\mathbb{E}\Big[f_{n}(\phi_{ni} X_{n}) \times F^{b_n}_{i,j}\Big|\phi_{nj}, J_n\Big]^2\Bigr]\\  
    &
    \overset{(b)}{\le} 2\|J_n\|_3^3\E\bigg[\mathbb{E}\Big[\mathbb{E}\big[f_{n}(\phi_{n2} X_{n}) | \phi_{n2},\G\big]\times F^{b_n}_{2,1}\Big|\phi_{n1}, J_n\Big]^2\bigg]  
    \\&\le{2\|J_n\|_3^3}\mathbb{E}\Big[\Big(\int_{\B_{b_n}(\phi_{n1})}\|Q_n(\phi) \|_1d\nu_n(\phi)\Big)^2\Big] 
    %\\&\le\frac{2j_n^2}{|\A_n|^2}\mathbb{E}\Big(\Big[\int_{\B(\phi_{n1},b_n)}\|\mu_n^\phi \|_2d\nu_n(\phi)\Big]^2\Big) 
    \\&\le \frac{2\|J_n\|_3^3\mathcal{S}^2|\B_{b_n}|^2}{|\A_n|^2}\sup_{\phi\in \A_n}\|Q_n(\phi)\|_1^2,%\mathbb{E}\Big(\Big[\int_{\B(\phi_{n1},b_n)}\|\mu_n^\phi \|_2d\nu_n(\phi)\Big]^2\Big)
\end{align*}
 where (a) is a result of the Cauchy-Schwarz inequality, and (b) is obtained by the independence of $J_n$ and $\nu_n$; the last inequality is due to the well-spread condition. 
 Moreover as the observations $\Big(\mathbb{E}\big[f_{n}(\phi_{ni} X_{n}) | \Jn\big]~F^{b_n}_{i,j}\Big)$ are conditionally i.i.d., we obtain that 
\begin{align*}
    (B_1)&\le 4\E\Big[\sum_{j\leq J_n}J_n\E\Big[\Big(\mathbb{E}\big[f_{n}(\phi_{n2} X_{n}) | \phi_{n2},\G\big]~F^{b_n}_{2,1}\Big)^2\Big| J_n\Big]\Big] 
    \\&\le{4\|J_n\|^2_2}\mathbb{E}\Big[\int_{\B_{b_n}(\phi_{n1})}\|Q_n(\phi)\|_2^2d\nu_n(\phi)\Big]
 \\ &\le \frac{4\|J_n\|^2_2|\B_{b_n}|\mathcal{S}}{|\A_n|}\sup_{\phi\in\A_n}\|Q_n(\phi)\|_2^2.
\end{align*}
Therefore, we obtain that 
$$
\begin{aligned}&
    \mathbb{E}\Big[\mathrm{Var}\Big(\sum_{i\leq j_n}\mathbb{E}\Big[f_{n}(\phi_{ni} X_{n}) \mathbb{I}(E^{n,k}_{b_n,i})\Big|\Jn \Big]\Big|J_n,\G\Big)\Big]
    \\\le&{4}\Big(\E[J_n]\sup_{\phi\in \A_n}\|Q_n(\phi)\|_2^2+{2\|J_n\|_3^3\mathcal{S}^2}\frac{|\B_{b_n}|^2}{|\A_n|^2}\sup_{\phi\in \A_n}\|Q_n(\phi)\|_1^2\\
    &\qquad+\frac{4\|J_n\|^2_2\mathcal{S}|\B_{b_n}|}{|\A_n|}\sup_{\phi\in\A_n}\|Q_n(\phi)\|_2^2\Big).
%\\&\le\frac{1}{|\A_n|^2}\Big\{\int_{\A_n}\|\mu_n^{\phi}\|_2^2d\nu_n(\phi)+{2j_n^2}\mathbb{E}\Big(\Big[\int_{\B(\phi_{n1},b_n)}\|\mu_n^\phi \|_1d\nu_n(\phi)\Big]^2\Big)+{4j_n}\mathbb{E}\Big(\int_{\B(\phi_{n1},b_n)}\|\mu_n^\phi\|^2_2d\nu_n(\phi)\Big)\Big\}
\end{aligned}
$$
Using \cref{trav} we have 
\begin{equation}
\begin{aligned}&\label{summandbound} (S_k)\le{2}\Big(\E[J_n]\sup_{\phi\in \A_n}\|Q_n(\phi)\|_2^2+{2\|J_n\|_3^3\mathcal{S}^2}\frac{|\B_{b_n}|^2}{|\A_n|^2}\sup_{\phi\in \A_n}\|Q_n(\phi)\|_1^2\\
&\qquad\qquad+\frac{4\|J_n\|^2_2\mathcal{S}|\B_{b_n}|}{|\A_n|}\sup_{\phi\in\A_n}\|Q_n(\phi)\|_2^2\Big)^{1/2}
%\frac{1}{|\A_n|}\Big\{\int_{\A_n}\|\mu_n^{\phi}\|_2^2d\nu_n(\phi)+{2j_n^2}\mathbb{E}\Big(\Big[\int_{\B(\phi_{n1},b_n)}\|\mu_n^\phi \|_1d\nu_n(\phi)\Big]^2\Big)+{4j_n}\mathbb{E}\Big(\int_{\B(\phi_{n1},b_n)}\|\mu_n^\phi\|^2_2d\nu_n(\phi)\Big)\Big\}^{1/2}
\end{aligned}
\end{equation}We remark that the upperbound obtained here does not depend on $k$. This will be useful in upper-bounding $\sum_k  (S_k) $. For the ease of notation, we write $$
\begin{aligned}
\epsilon_n:=&{2}\Big(\E[J_n]\sup_{\phi\in \A_n}\|Q_n(\phi)\|_2^2+{2\|J_n\|_3^3\mathcal{S}^2}\frac{|\B_{b_n}|^2}{|\A_n|^2}\sup_{\phi\in \A_n}\|Q_n(\phi)\|_1^2\\
&\qquad\qquad +\frac{4\|J_n\|^2_2\mathcal{S}|\B_{b_n}|}{|\A_n|}\sup_{\phi\in\A_n}\|Q_n(\phi)\|_2^2\Big)^{1/2}
\end{aligned}
$$
We remark that \cref{summandbound} only bounds $(S_k)$, while our goal is to bound the sum $\sum_{k=1}^{\infty} (S_k)$. In this goal, we set $k_n:=\lfloor{\epsilon_n}^{\alpha-1}\rfloor$ for some $0<\alpha<1$. Note that for all $0<c_n<b_n$, the following holds:
\begin{align}\label{shakshuka}
    (S) = &\sum_{k=1}^{k_n} (S_k)+\sum_{k=k_n+1}^{\infty} \Big\|\sum_{i\leq J_n}\mathbb{E}\Big[f_{n}(\phi_{ni} X_{n}) \mathbb{I}(E^{n,k}_{b_n,i})\Big| \Jn\Big]\nonumber\\
    &\qquad\qquad\qquad\qquad-\mathbb{E}\Big[\sum_{i\leq J_n}f_{n}(\phi_{ni} X_{n}) \mathbb{I}(E^{n,k}_{b_n,i})\Big|J_n, \G\Big]\Big\|_{1}\nonumber\\
    \overset{(a)}{\leq}& \sum_{k=1}^{k_n} (S_k)+ \sum_{k=k_n+1}^{\infty} 2\norm{\sum_{i\leq J_n}\mathbb{E}\Bigl[f_{n}(\phi_{ni} X_{n}) \mathbb{I}(E^{n,k}_{b_n,i})\Big| \Jn\Bigl]}_1\nonumber\\
    \leq&~{\epsilon_n}^\alpha +\sum_{k=k_n+1}^{\infty} 2\norm{\sum_{i\leq J_n}\mathbb{E}\Bigl[f_{n}(\phi_{ni} X_{n}) \mathbb{I}(E^{n,k}_{b_n,i})\Big| \Jn\Bigl]}_1\\
    \overset{(b)}{\leq}&~{\epsilon_n}^{\alpha}+\sum_{k=k_n+1}^{\infty} 2\norm{\sum_{i\leq J_n}\mathbb{E}\Bigl[f_{n}(\phi_{ni} X_{n}) \mathbb{I}(E_{c_n,i}^{n,k})\Big| \Jn \Bigl]}_1\nonumber\\
    & +\sum_{k=k_n+1}^{\infty} 2\norm{\sum_{i\leq J_n}\mathbb{E}\bigg[f_{n}(\phi_{ni} X_{n}) \Bigl(\mathbb{I}(E_{c_n,i}^{n,k})-\mathbb{I}(E_{b_n,i}^{n,k})\Bigr) \bigg| \Jn \bigg]}_1\nonumber\\
    =:&~ {\epsilon_n}^{\alpha}+(C)+(D).\nonumber
\end{align}
To obtain $(a)$, we apply the triangle inequality, followed by Jensen's inequality, and to reach $(b)$, we again employ the triangle inequality. We bound $(C)$ and $(D)$ separately.
\\
 In this goal, we observe that 
 \begin{align}\label{boum}\sum_{k>k_n}\mathbb{I}\big(E_{c_n,i}^{n,k}\big)=\mathbb{I}\Big(\sum_{\substack{i^{\prime}\leq J_n \\ d(\phi_{ni}, \phi_{ni^{\prime}}) \leq c_n}} f_{n}(\phi_{ni^{\prime}} X_{n})> k_n\Big).\end{align} Therefore we obtain that 
\begin{align*}
& \sum_{k=k_n+1}^{\infty}\norm{\sum_{i\leq J_n}\mathbb{E}\Big[f_{n}(\phi_{ni} X_{n}) \mathbb{I}(E_{c_n,i}^{n,k})\Big| \Jn \Big]}_1\\
%\leq &\sum_{k>k_n}\norm{f_{n}(\phi_{ni}X_{n}) \mathbb{I}(E_{c_n,i}^{n,k})}_1\\
= &\sum_{k>k_n}\E\Big[\sum_{i\leq J_n} f_{n}(\phi_{ni}X_{n}) \mathbb{I}(E_{c_n,i}^{n,k})\Big]\\
\overset{(a)}{=} & \E\Big[\sum_{i\leq J_n} f_{n}(\phi_{ni}X_{n}) \mathbb{I}(\sum_{\substack{i^{\prime}\leq J_n \\ d(\phi_{ni}, \phi_{ni^{\prime}}) \leq c_n}} f_{n}(\phi_{ni^{\prime}} X_{n})> k_n)\Big]\\
\leq & \E\Big[\sum_{i\leq J_n} f_{n}(\phi_{ni}X_{n}) \mathbb{I}\Big(\sum_{i^{\prime}\leq J_n }\ind(d(\phi_{ni}, \phi_{ni^{\prime}}) \leq c_n)>k_n\Big)\Big]\\
= & \E\Big[\sum_{i\leq J_n} f_{n}(\phi_{ni}X_{n}) \mathbb{I}\Big(\sum_{i^{\prime}\leq J_n }F^{c_n}_{i^\prime,i}>k_n\Big)\Big]\\
%= & \E\Big[\E[f_{n}(\phi_{n1}X_{n})\mathbb{I}\left(\sum_{i^{\prime}\leq j_n }F^{c_n}_{i^\prime}>k_n\right) | \phi_{n1}]\Big]\\
\overset{(b)}{=} & \E\Big[\sum_{i\leq J_n} \prob\Big(\sum_{\substack{i^{\prime}\neq i\\i^{\prime}\leq J_n }}F^{c_n}_{i^\prime,i}\ge k_n \Big| J_n, \phi_{ni}\Big)\prob(f_{n}(\phi_{ni}X_{n})=1| J_n, \phi_{ni},\G)\Big]
\\= &\sum_{j_n\in\N}\prob(J_n=j_n)j_n\int_{\A_n}\prob\Big(\sum_{1<i^{\prime}\leq j_n }\mathbb{I}\big(d(\phi_{ni'},\phi)\le c_n\big)\ge k_n \Big)\|Q_n(\phi)\|_1d\nu_n(\phi)
\\\le& \sup_{\phi\in \A_n}\|Q_n(\phi)\|_1\sum_{j_n\in\N}\prob(J_n=j_n)j_n\sup_{\phi\in \A_n}\prob\Big(\sum_{1<i^{\prime}\leq j_n }\mathbb{I}\big(d(\phi_{ni'},\phi)\le c_n\big)\ge k_n \Big),
% = & \norm{\prob\Big[\sum_{i^{\prime}\leq j_n }F^{c_n}_{i^\prime}>k_n | \phi_{n1},f_{n}(\phi_{n1}X_{n})=1\Big]}_\infty\E[\prob[f_{n}(\phi_{n1}X_{n})=1| \phi_{n1}]] \\
% = & \norm{\prob\Big[\sum_{i^{\prime}\leq j_n }F^{c_n}_{i^\prime}>k_n | \phi_{n1},f_{n}(\phi_{n1}X_{n})=1\Big]}_\infty\E[\prob[f_{n}(X_{n})=1]] \tag{by invariance}\\
% = & \frac{\mu}{\abs{\A_n}}\norm{\prob\Big[\sum_{1<i^{\prime}\leq j_n }F^{c_n}_{i^\prime}\geq k_n | \phi_{n1},f_{n}(\phi_{n1}X_{n})=1\Big]}_\infty\\
% = & \frac{\mu}{\abs{\A_n}}\norm{\prob\Big[\sum_{1<i^{\prime}\leq j_n }F^{c_n}_{i^\prime} \geq k_n | \phi_{n1}\Big]}_\infty \tag{by tower law}
\end{align*}
where $(a)$ is obtained using \cref{boum}, and $(b)$ follows from the tower property along with the independence of $\boldsymbol\phi_n$ and $X_n$.
Note that for all $\phi\in \G$, the random variables 
$\Big(\mathbb{I}\big(d(\phi_{ni^\prime},\phi)\le c_n\big)\Big)$ are i.i.d Bernoulli with parameter $q_n(\phi) := \prob(d(\phi, \phi_{n1}) \leq c_n)$. 
Note that by the well-spread condition, we have assumed that there is a constant $\mathcal{S}>0$ such that
$\forall c\geq 1$ and that $\forall \phi\in A_n$,
$$q_n(\phi)=\prob_{\phi_{n1} \sim \nu_n}(d(\phi,\phi_{n1})\leq c)\leq \mathcal{S}\frac{\abs{\B_{c}}}{\abs{\A_n}}.$$
Moreover, according to Markov's inequality, we have  
\begin{align*}
    \prob\Big(\sum_{1<i^{\prime}\leq j_n }\mathbb{I}\big(d(\phi_{ni'},\phi)\le c_n\big)\ge k_n \Big)\le \frac{(j_n-1)}{k_n}\mathcal{S}\frac{\abs{\B_{c_n}}}{\abs{\A_n}}.
\end{align*}
This directly implies that 
$$
\begin{aligned}
(C)&=2\sum_{k=k_n+1}^{\infty} \norm{\sum_{i\leq J_n}\mathbb{E}\Big[f_{n}(\phi_{ni} X_{n}) \mathbb{I}(E_{c_n,i}^{n,k})| \Jn \Big]}\\
&\leq2~\sup_{\phi\in \A_n}\|Q_n(\phi)\|_1\sum_{j_n\in\N}\prob(J_n=j_n)j_n\frac{(j_n-1)}{k_n}\mathcal{S}\frac{\abs{\B_{c_n}}}{\abs{\A_n}}\\
& \leq 2~\frac{1}{|\A_n|}\sup_{\phi\in \A_n}\|Q_n(\phi)\|_1{\mathcal{S}\abs{\B_{c_n}}}\frac{1}{k_n}\|J_n\|^2_2.
\end{aligned}
$$
We now move to bound the  term $(D)$. By Jensen's inequality, we remark that 
\begin{align*}
&\frac{1}{2}(D)\\=&\sum_{k=k_n+1}^{\infty} \bigg\|\sum_{i\leq J_n}\mathbb{E}\Big[f_{n}(\phi_{ni} X_{n}) \Big(\mathbb{I}(E^{n,k}_{b_n,{i}})-\mathbb{I}(E^{n,k}_{c_n,{i}})\Big)\Big| \Jn \Big]\bigg\|_1
\\
% \\
% \leq &\sum_{k=k_n+1}^{\infty} \Bigg\|\sum_{i\leq j_n}f_{n}(\phi_{ni} X_{n}) \Big(\mathbb{I}(E^{n,k}_{b_n,{i}})-\mathbb{I}(E^{n,k}_{c_n,{i}})\Big)\Bigg\|_1\\
% \leq &\sum_{i\leq j_n}\sum_{k=k_n+1}^{\infty}\Bigg\|f_{n}(\phi_{ni} X_{n})  \Big(\mathbb{I}(E^{n,k}_{b_n,{i}})-\mathbb{I}(E^{n,k}_{c_n,{i}})\Big)\Bigg\|_1\\
\leq& \mathbb{E}\bigg[\sum_{i\leq J_n}f_{n}(\phi_{ni} X_{n})\sum_{k=1}^{\infty}\Big|\mathbb{I}(E^{n,k}_{b_n,{i}})-\mathbb{I}(E^{n,k}_{c_n,{i}})\Big|\bigg]\\
\overset{}{=}& \mathbb{E}\bigg[\sum_{i\leq J_n} f_{n}(\phi_{ni} X_{n})~\mathbb{I}\Big(\sum_{\substack{i^{\prime}\leq J_n\\ c_n< d(\phi_{ni}, \phi_{ni^{\prime}}) \leq b_n}} f_{n}(\phi_{ni^{\prime}} X_{n})\neq 0\Big)\bigg]\\
=& \Big(\mathbb{E}\Big[\sum_{i\leq J_n} f_{n}(\phi_{ni} X_{n})~\mathbb{I}\Big(\sum_{\substack{i^{\prime}\leq J_n\\ c_n< d(\phi_{n1}, \phi_{ni^{\prime}}) \leq b_n}} f_{n}(\phi_{ni^{\prime}} X_{n})\neq 0\Big)\Big]\\&-\mathbb{E}\Big[\sum_{i\leq J_n}\mathbb{E}\big[f_n(\phi_{ni} X_n)\big|J_n, \G,\phi_{ni}\big]~\mathbb{P}\Big(\sum_{\substack{i^{\prime}\leq J_n\\ c_n< d(\phi_{ni}, \phi_{ni^{\prime}}) \leq b_n}} f_{n}(\phi_{ni^{\prime}} X_{n})\neq 0\big|J_n, \G,\phi_{ni}\Big)\Big]\Big)\\
&+\mathbb{E}\Big[\sum_{i\leq J_n}\mathbb{E}\big[f_n(\phi_{ni} X_n)\big|J_n, \G,\phi_{ni}\big]~\mathbb{P}\Big(\sum_{\substack{i^{\prime}\leq J_n\\ c_n< d(\phi_{ni}, \phi_{ni^{\prime}}) \leq b_n}} f_{n}(\phi_{ni^{\prime}} X_{n})\neq 0\big|J_n, \G,\phi_{ni}\Big)\Big] \\
=&:(D_1)+(D_2).
%\overset{(5)}{\leq} & \sum_{i\leq j_n}(\frac{\mu}{\abs{\A_n}}\Psi(c_n)+\frac{\mu}{\abs{\A_n}}C\frac{\mu\abs{\B_{b_n}}}{\abs{\A_n}})\\
%\leq & \mu\Psi(c_n)+C\frac{\mu^2\abs{\B_{b_n}}}{\abs{\A_n}}\overset{(6)}{\longrightarrow}0\\
\end{align*} 
For the simplicity of notation, denote $\mathcal{J}_{ni} := \{J_n, \G, \phi_{ni}\}$.
We bound $(D_1)$ and $(D_2)$ successively. Firstly, to bound $(D_2)$, we remark that 
\begin{align*}
(D_2)=&\mathbb{E}\Big[\sum_{i\leq J_n}\mathbb{E}\big[f_n(\phi_{ni} X_n)\big|\mathcal{J}_{ni}\big]~\mathbb{P}\Big(\sum_{\substack{i^{\prime}\leq J_n\\ c_n< d(\phi_{n1}, \phi_{ni^{\prime}}) \leq b_n}} f_{n}(\phi_{ni^{\prime}} X_{n})\neq 0\big|\mathcal{J}_{ni}\Big)\Big]\\
\overset{(a)}{\leq} & \mathbb{E}\Big[\sum_{i\leq J_n}Q_{n}(\phi_{ni})\sum_{i^{\prime}\leq J_n}\prob\Big(f_{n}(\phi_{ni^{\prime}} X_{n})= 1\quad \textrm{and}\quad c_n< d(\phi_{ni}, \phi_{ni^{\prime}}) \leq b_n\big|\mathcal{J}_{ni}\Big)\Big]\\
\overset{(b)}{\le} &\|J_n\|^2_2\mathbb{E}\Big[Q_{n}(\phi_{n1})\prob\Big(f_{n}(\phi_{n2} X_{n})= 1\quad \textrm{and}\quad c_n< d(\phi_{n1}, \phi_{n2}) \leq b_n\big|\Jsn\Big)\Big]
\\= &{\|J_n\|^2_2}\E\Big[Q_{n}(\phi_{n1})\int_{\B_{b_n}(\phi_{n1})\setminus \B_{c_n}(\phi_{n1})}\prob(f_{n}\big(\phi X_{n})= 1|J_n, \G\big)d\nu_n(\phi)\Big]
\\= &{\|J_n\|^2_2}\E\Big[Q_{n}(\phi_{n1})\int_{\B_{b_n}(\phi_{n1})\setminus \B_{c_n}(\phi_{n1})}Q_{n}(\phi) d\nu_n(\phi)\Big]\\
\le  &{\|J_n\|^2_2}\E\Big[\|Q_{n}(\phi_{n1})\|_2\int_{\B_{b_n}(\phi_{n1})\setminus \B_{c_n}(\phi_{n1})}\|Q_{n}(\phi)\|_2 d\nu_n(\phi)\Big]
\\\le& \frac{\mathcal{S}\|J_n\|^2_2|\B_{b_n}\setminus \B_{c_n}|}{|\A_n|}\sup_{\phi\in \A_n}\|Q_{n}(\phi)\|_2^2.
\end{align*} 
where (a) is a consequence of the union bound and (b) is because of the independence of $J_n$ and $X_n$ as well as $\nu_n$. Therefore, we obtain that 
\begin{align*}
    (D_2)&\le\frac{\mathcal{S}\|J_n\|^2_2|\B_{b_n}\setminus \B_{c_n}|}{|\A_n|}\sup_{\phi\in \A_n}\|Q_{n}(\phi)\|_2^2.% \frac{j_n}{|\A_n|^2}\sqrt{\int_{\A_n}\|\mu_n^{\phi}\|_2^2d\nu_n(\phi)}\sqrt{\mathbb{E}\Big(\Big[\int_{\B(\phi_{n1},b_n)}\|\mu_n^\phi\|_2d\nu_n(\phi)\Big]^2\Big)}
\end{align*} To bound $(D_1)$, we exploit the definition of mixing coefficients:
$$
\begin{aligned}
(D_1)\le& \bigg|\mathbb{E}\Big[\sum_{j\leq J_n}f_{n}(\phi_{ni} X_{n})~\mathbb{I}\Big(\sum_{\substack{i^{\prime}\leq J_n\\ c_n< d(\phi_{n1}, \phi_{ni^{\prime}}) \leq b_n}} f_{n}(\phi_{ni^{\prime}} X_{n})\neq 0\Big)\Big]\\&-\mathbb{E}\Big[\sum_{j\leq J_n}\mathbb{E}\big[f_n(\phi_{ni} X_n)\big|\mathcal{J}_{ni}\big]~\mathbb{P}\Big(\sum_{\substack{i^{\prime}\leq J_n\\ c_n< d(\phi_{ni}, \phi_{ni^{\prime}}) \leq b_n}} f_{n}(\phi_{ni^{\prime}} X_{n})\neq 0\big|\mathcal{J}_{ni}\Big)\Big]\bigg|\\
% = & \E\Bigg[\sum_{j\leq J_n}\E\bigg[\mathbb{E}\Big[f_{n}(\phi_{ni} X_{n})~\mathbb{I}\Big(\sum_{\substack{i^{\prime}\leq J_n\\ c_n< d(\phi_{ni}, \phi_{ni^{\prime}}) \leq b_n}} f_{n}(\phi_{ni^{\prime}} X_{n})\neq 0\Big)\Big|\mathcal{J}_{ni}\Big]\\&-\mathbb{E}\Big[\mathbb{E}\big[f_n(\phi_{ni} X_n)\big|\mathcal{J}_{ni}\big]~\mathbb{P}\Big(\sum_{\substack{i^{\prime}\leq J_n\\ c_n< d(\phi_{ni}, \phi_{ni^{\prime}}) \leq b_n}} f_{n}(\phi_{ni^{\prime}} X_{n})\neq 0\big|\mathcal{J}_{ni}\Big)\Big]\bigg| J_n, \phi_{ni}\bigg]\Bigg]\\
\overset{(a)}{\leq} & \E\bigg[\sum_{j\leq J_n}\Psi_{n,s}(c_n|\G)\E\Big[\mathbb{P}(f_n(\phi_{ni} X_n)=1|\mathcal{J}_{ni})^{\frac{s}{s-1}}\Big|J_n, \phi_{ni}\Big]^{\frac{s-1}{s}}\bigg]\\
\overset{(b)}{\leq} & \E[J_n]\Psi_{n,s}(c_n|\G) \Big\|\mathbb{P}(f_n(\phi_{n1} X_n)=1|\Jsn)\Big\|_{\frac{s}{s-1}}.
\end{aligned}
$$
where (a) is by H\"{o}lder's inequality and (b) is a result of Jensen's inequality and the independence of $J_n$ and $X_n$.
Therefore, we have
$$
    (D_1)\le {\E[J_n]}\Psi_{n,s}(c_n|\G)\sup_{\phi\in \A_n}\|Q_{n}(\phi)\|_{\frac{s}{s-1}}.
$$
% and the convergence (6) is because we have the assumption that
% $$
% \frac{\abs{\B_{b_n}}}{\abs{\A_n}}\longrightarrow 0
% $$
% and because
% $$
% c_n\longrightarrow\infty \Rightarrow \Psi(c_n) \longrightarrow 0
% $$
%Hence we proved that $(D)\longrightarrow 0$.
% Therefore,$$\sum_{k=1}^{\infty} \norm{k (\lambda^{b_n}_{n}(k)- \widetilde\lambda^{b_n}_{n}(k))}_1\leq \sqrt{\epsilon_n} +(C)+(D)\longrightarrow0$$
% Thus, $$\E[\dTV(Z(\lambda),Z(\widetilde\lambda^{b_n,n})| \widetilde\lambda^{b_n,n})]\longrightarrow 0$$
With all the derived bounds plugged into \cref{shakshuka}, we obtain:
$$
\begin{aligned}
(S)\le& {\epsilon_n}^\alpha+({C})+(D)\\
\le & {\epsilon_n}^\alpha + 2\frac{\mathcal{S}\|J_n\|^2_2|\B_{b_n}\setminus \B_{c_n}|}{|\A_n|}\sup_{\phi\in \A_n}\|Q_{n}(\phi)\|_2^2 + 2{\E[J_n]}\Psi_{n,s}(c_n|\G)\sup_{\phi\in \A_n}\|Q_{n}(\phi)\|_{\frac{s}{s-1}}\\ 
     & +2\frac{1}{|\A_n|} \sup_{\phi\in \A_n}\|Q_{n}(\phi)\|_1\mathcal{S}\abs{\B_{c_n}}\frac{1}{k_n}\|J_n\|^2_2.
\end{aligned}
$$
Finally, integrating all the obtained bounds into \cref{trav}, we have
$$
\begin{aligned}
    &\sum_{k=1}^{\infty} \norm{k (\lambda^{b_n}_{n}(k)- \widetilde\lambda^{b_n}_{n}(k))}_1\\  \leq& ~{\epsilon_n}^{\alpha} + 2\frac{\mathcal{S}\|J_n\|^2_2|\B_{b_n}\setminus \B_{c_n}|}{|\A_n|}\sup_{\phi\in \A_n}\|Q_{n}(\phi)\|_2^2 + 2{\E[J_n]}\Psi_{n,s}(c_n|\G)\sup_{\phi\in \A_n}\|Q_{n}(\phi)\|_{\frac{s}{s-1}}\\ 
     & + \frac{1}{|\A_n|}\sup_{\phi\in \A_n}\|Q_{n}(\phi)\|_1\Big[2\mathcal{S}\abs{\B_{c_n}}\frac{1}{k_n}\|J_n\|^2_2\\
     &+  \mathcal{S}\abs{\B_{b_n}}\sqrt{2\mathrm{Var}(J_n)}\|J_n\|_2 + |\A_n|\sqrt{2\mathrm{Var}(J_n)}~\Big],
     % \frac{j_n}{|\A_n|^2}\sqrt{\int_{\A_n}\|\mu_n^{\phi}\|_2^2d\nu_n(\phi)}\sqrt{\mathbb{E}\Big(\Big[\int_{\B(\phi_{n1},b_n)}\|\mu_n^\phi\|_2d\nu_n(\phi)\Big]^2\Big)}
\end{aligned}
$$
where $
\resizebox{0.95\textwidth}{!}{$\epsilon_n:=2\Big(\frac{2\mathcal{S}^2\|J_n\|_3^3\,|\B_{b_n}|^2}{|\A_n|^2}\sup_{\phi\in \A_n}\|\Q_n(\phi)\|_1^2+\big(\|J_n\|_1+\frac{4\mathcal{S}\|J_n\|^2_2\,|\B_{b_n}|}{|\A_n|}\big)\sup_{\phi\in\A_n}\|\Q_n(\phi)\|_2^2\Big)^{\frac{1}{2}}
$}$  and $k_n=\lfloor {\epsilon_n}^{\alpha - 1}\rfloor.$ The final result is a direct consequence of \cref{milano}.
\end{proof}

%Now we need to prove step 2: $W\longrightarrow Z(\widetilde{\lambda})$. Before we begin the proof, note that the following lemma allows us to change the problem of conditioning on $\widetilde\lambda^{b_n,n}$ to conditioning on $\boldsymbol\phi_n $
%We can formulate the bound in terms of the norm of the conditional expectation on $\boldsymbol\phi_n $. To see this, denote $\mathcal{W}^{\phi_{ni}}_{ b} = \sum_{\substack{i^\prime \leq j_n\\d(\phi_{ni},\phi_{ni^\prime})\leq b}}f_n(\phi_{ni^\prime} X_n)=\sum_{\substack{i^\prime \leq j_n\\\phi_{ni^\prime}\in \mathcal{B}^n_b(\phi_{ni})}}f_n(\phi^\prime X_n)$

%%%%%%%%%%%%%%%%%%%%%%%%%%%%%%%%%%%%%%%%%%%%%

%%%%%%%%%%%%%%%%%%%%%%%%%%%%%%%%%%%%%%%%%%%%5

\subsubsection{Obtaining the final bound}
Recall that we denote $\mathcal{F}{\TV}:=\mathcal{F}{\TV}(\lambda^{b_n}_{n})$.
\begin{lemma}\label{Jnbound}
Suppose all the conditions of \cref{lemma26general} holds. Then 
    we obtain that 
    \begin{align*}&
        \mathbb{E}\Big[d_{\TV}(W_n,Z({\lambda}^{b_n}_{n})|\G)\Big]\\\leq &\Big\|\sup_{g\in \mathcal{F}_{\TV}}|\Delta g|\Big\|_{\infty}\mathcal{S}\frac{\|J_n\|^2_2}{|\A_n|} \Big(\sup_{\phi\in \A_n}\|Q_n(\phi)\|_{\frac{p}{p-1}}\mathcal{R}_\Psi(n,p,b_n)+2\sup_{\phi\in \A_n}\|Q_n(\phi)\|_{\frac{q}{q-1}}\mathcal{R}_\xi(n,q,b_n)
        \\&\qquad \qquad \qquad\qquad\qquad+\sup_{\phi\in \A_n}\|Q_n(\phi)\|_2^2
        \Big( {|\B_{2b_n}|}+{|\B_{2b_n}\setminus\B_{b_n}|}\Big)\Big)\\
        & +\Big\|\sup_{g\in \mathcal{F}_{\TV}}|g|\Big\|_{\infty}\bigg({\epsilon_n}^{\alpha} + 2\frac{\mathcal{S}\|J_n\|^2_2|\B_{b_n}\setminus \B_{c_n}|}{|\A_n|}\sup_{\phi\in \A_n}\|Q_n(\phi)\|_2^2 \\&\qquad \qquad \qquad\qquad\qquad+ 2{\E[J_n]}\Psi_{n,s}(c_n|\G)\sup_{\phi\in \A_n}\|Q_n(\phi)\|_{\frac{s}{s-1}}\\ 
&\qquad\qquad\qquad  + \frac{1}{|\A_n|}\sup_{\phi\in \A_n}\|Q_n(\phi)\|_1\Big(2\mathcal{S}\abs{\B_{c_n}}\frac{1}{k_n}\|J_n\|^2_2\\
&\qquad\qquad\qquad +  \mathcal{S}\abs{\B_{b_n}}\sqrt{2\mathrm{Var}(J_n)}\|J_n\|_2 + |\A_n|\sqrt{2\mathrm{Var}(J_n)}~\Big)\bigg),
    \end{align*}
    where $\resizebox{0.95\textwidth}{!}{$
\epsilon_n:=2\Big(\frac{2\mathcal{S}^2\|J_n\|_3^3\,|\B_{b_n}|^2}{|\A_n|^2}\sup_{\phi\in \A_n}\|\Q_n(\phi)\|_1^2+\big(\|J_n\|_1+\frac{4\mathcal{S}\|J_n\|^2_2\,|\B_{b_n}|}{|\A_n|}\big)\sup_{\phi\in\A_n}\|\Q_n(\phi)\|_2^2\Big)^{\frac{1}{2}}$}
$ and $k_n:=\lfloor {\epsilon_n}^{\alpha -1}\rfloor.$
\end{lemma}
\begin{proof}Firstly, we observe that by using \cref{boundMH}, we have $$\mathbb{E}\Big[\dTV(Z({\lambda}^{b_n}_{n}),~W_n|\G)\Big]\le(M)+(H).$$
In order to upper bound the first term, we can utilize \cref{ts133}. The control over the second term is achieved through the application of \cref{cpconvergence_randomization}.
\end{proof}

\begin{lemma}[Simplified bound]\label{simplifiedgenbound}
Suppose all the conditions of \cref{sun} holds. Note that $\mathcal{S}$ is the constant in the well-spread condition. 
Assume that $\sup_{n}\frac{\|J_n\|_3}{|\A_n|}$ is bounded and that $\sup_{n}\sup_{\phi\in \A_n}|A_n|\|Q_n(\phi)\|_2$ is bounded. Denote $\tau_1 :=\max\{1,\sup_{n}\frac{\|J_n\|_3}{|\A_n|}\} $ and $\tau_2:=\max\{1, \sup_{n}\sup_{\phi\in \A_n}|A_n|\|Q_n(\phi)\|_2\}$. Then 
    we obtain that, for all $0<\alpha<1$ and $0<\beta<1$,
    \begin{equation}\label{calamari}
    \begin{aligned}&
        \mathbb{E}\Big[d_{\TV}(W_n,Z({\lambda}^{b_n}_{n})|\G)\Big]
        \leq \;
      \kappa_1\,\|H_1^{\TV}(\lambda_n)\|_\infty
      \,\Big(\mathcal{R}_{\Psi}(n,2,b_n)+ \mathcal{R}_{\xi}(n,2,b_n)+
        \frac{|\B_{2b_n}|}{|\A_n|}\Big)\\
        &\qquad+\; \kappa_2\,\|H_0^{\TV}(\lambda_n)\|_\infty\,
        \Big(\Psi_{n,2}(c_n|\G)
        +
        \frac{\abs{\B_{b_n}}}
             {|\A_n|}\sqrt{\mathrm{Var}(J_n)}
             +
        \Big(\frac{|\B_{b_n}|}{\sqrt{|\A_n|}}\Big)^{\min\{\alpha,(1-\alpha)\beta\}}
        \Big),
    \end{aligned}
    \end{equation}
    where $c_n:=\max\{r:|\B_{r}| \le {k_n}^{1-\beta}\}$ with $k_n$ as defined in \cref{Jnbound}, and  $\kappa_1,\kappa_2$ are positive constants such that 
    \begin{itemize}
        \item 
    $\kappa_1: = O( \mathcal{S}\tau_1^2\tau_2^2)$, 
    \item $\kappa_2 := O(\max\{\sigma_1^{\alpha},2\mathcal{S}{\tau_1}^2{\tau_2}^2,2\tau_1\tau_2, 2^{\beta+1}{\tau_1}^2\tau_2\mathcal{S}\min\{\sigma_1^{(1-\alpha)\beta},1\},\sqrt{2}\tau_1\tau_2\mathcal{S},\sqrt{2}\tau_2\})$,
    \end{itemize}with $\sigma_1 := 4{\tau_1}^{1/2}\tau_2\max\{1,2\tau_1{\mathcal{S}}\}$.
\end{lemma}
\begin{proof}
The simplified bound is obtained by leveraging the bound provided in \cref{Jnbound}.
Denote the two terms in the bound in \cref{Jnbound} by:
\begin{align*}
    (P_1):=&\Big\|\sup_{g\in \mathcal{F}_{\TV}}|\Delta g|\Big\|_{\infty}\mathcal{S}\frac{\|J_n\|^2_2}{|\A_n|} \Big(\sup_{\phi\in \A_n}\|Q_n(\phi)\|_{\frac{p}{p-1}}\mathcal{R}_\Psi(n,p,b_n)\\&+2\sup_{\phi\in \A_n}\|Q_n(\phi)\|_{\frac{q}{q-1}}\mathcal{R}_\xi(n,q,b_n)
        +\sup_{\phi\in \A_n}\|Q_n(\phi)\|_2^2
        \Big( {|\B_{2b_n}|}+{|\B_{2b_n}\setminus\B_{b_n}|}\Big)\Big),
\end{align*}
\begin{align*}
    (P_2):=         & \Big\|\sup_{g\in \mathcal{F}_{\TV}}|g|\Big\|_{\infty}\bigg({\epsilon_n}^{\alpha} + 2\frac{\mathcal{S}\|J_n\|^2_2|\B_{b_n}\setminus \B_{c_n}|}{|\A_n|}\sup_{\phi\in \A_n}\|Q_n(\phi)\|_2^2 \\
    &\qquad\qquad\qquad + 2{\E[J_n]}\Psi_{n,s}(c_n|\G)\sup_{\phi\in \A_n}\|Q_n(\phi)\|_{\frac{s}{s-1}}\\ 
&\qquad\qquad\qquad  + \frac{1}{|\A_n|}\sup_{\phi\in \A_n}\|Q_n(\phi)\|_1\Big(2\mathcal{S}\abs{\B_{c_n}}\frac{1}{k_n}\|J_n\|^2_2\\
&\qquad\qquad\qquad +  \mathcal{S}\abs{\B_{b_n}}\sqrt{2\mathrm{Var}(J_n)}\|J_n\|_2 + |\A_n|\sqrt{2\mathrm{Var}(J_n)}~\Big)\bigg).
\end{align*}
We will simplify each term. 

First we simplify $(P_1)$. We choose $p=2$ and $q=2$, then we have
\begin{align*}
    (P_1)\le&\Big\|\sup_{g\in \mathcal{F}_{\TV}}|\Delta g|\Big\|_{\infty}\mathcal{S}\frac{\|J_n\|^2_2}{|\A_n|^2} \max\{1,\sup_{\phi\in \A_n}|\A_n|^2\|Q_n(\phi)\|^2_{2}\}\\&\qquad\qquad\qquad\times\bigg(\mathcal{R}_{\Psi}(n,2,b_n)+2 \mathcal{R}_{\xi}(n,2,b_n)+2
        \frac{|\B_{2b_n}|}{|\A_n|}\bigg)\\
    \le & ~2\Big\|\sup_{g\in \mathcal{F}_{\TV}}|\Delta g|\Big\|_{\infty}\mathcal{S}\tau_1^2 \tau_2^2\bigg(\mathcal{R}_{\Psi}(n,2,b_n)+ \mathcal{R}_{\xi}(n,2,b_n)+
        \frac{|\B_{2b_n}|}{|\A_n|}\bigg)\\
        \le &~\Big\|H_1(
    \lambda^{b_n}_{n}
    )\Big\|_\infty\kappa_1\bigg(\mathcal{R}_{\Psi}(n,2,b_n)+ \mathcal{R}_{\xi}(n,2,b_n)+
        \frac{|\B_{2b_n}|}{|\A_n|}\bigg),
\end{align*}
where $\kappa_1: = O( \mathcal{S}\tau_1^2\tau_2^2)$.

Next we bound $(P_2)$.
Choose $c_n$ such that $c_n:=\max\{r:|\B_{r}| \le {k_n}^{1-\beta}\}$ for some $0<\beta<1$. Then we have that $\frac{|\B_{c_n}|}{k_n} \le {k_n}^{-\beta}$. 
Recall that $k_n:=\lfloor {\epsilon_n}^{\alpha-1}\rfloor$. 
We can bound $\epsilon_n$ by: 
    $$
    \begin{aligned}
    \epsilon_n & = 2\Big(\frac{2\mathcal{S}^2\|J_n\|_3^3\,|\B_{b_n}|^2}{|\A_n|^2}\sup_{\phi\in \A_n}\|\Q_n(\phi)\|_1^2+\big(\|J_n\|_1+\frac{4\mathcal{S}\|J_n\|^2_2\,|\B_{b_n}|}{|\A_n|}\big)\sup_{\phi\in\A_n}\|\Q_n(\phi)\|_2^2\Big)^{\frac{1}{2}}\\
   % & \le \frac{2}{|\A_n|}\Big\{\|J_n\|_3\sup_{\phi\in \A_n}\|\mu_n^\phi\|_2^2+{2\mathcal{S}^2\|J_n\|_3^3}\frac{|\B_{b_n}|^2}{|\A_n|^2}\sup_{\phi\in \A_n}\|\mu_n^\phi\|_2^2+4\mathcal{S}\|J_n\|^2_3\frac{|\B_{b_n}|}{|\A_n|}\sup_{\phi\in\A_n}\|\mu_n^\phi\|_2^2\Big\}^{1/2}\\
    & \le\frac{2}{|\A_n|}\max\{1,\sup_{\phi\in \A_n}|\A_n|\|Q_n(\phi)\|_2\}\|J_n\|_3^{1/2}\\
    &\qquad\qquad\qquad \times \Big(1+{2\mathcal{S}^2}\|J_n\|^2_3\frac{|\B_{b_n}|^2}{|\A_n|^2}+4\mathcal{S}\|J_n\|_3\frac{|\B_{b_n}|}{|\A_n|}\Big)^{1/2}\\
    & \le {\tau_1}^{1/2}\tau_2 \frac{2}{|\A_n|^{1/2}}\Big\{1+2\tau_1^2{\mathcal{S}^2}|\B_{b_n}|^2+4\tau_1\mathcal{S}|\B_{b_n}|\Big\}^{1/2}\\
    & \le {\tau_1}^{1/2}\tau_2 \frac{2}{|\A_n|^{1/2}}\Big(1+2\tau_1{\mathcal{S}}|\B_{b_n}|\Big)\\
    & \le \sigma_1\frac{|\B_{b_n}|}{|\A_n|^{1/2}},
    \end{aligned}
    $$
    where $\sigma_1 := 4{\tau_1}^{1/2}\tau_2\max\{1,2\tau_1{\mathcal{S}}\}$. % and $\gamma_n:=\frac{|\B_{b_n}|}{|\A_n|^{1/2}}$.
%    $$\epsilon_n\ge\frac{2}{|\A_n|}\sup_{\phi\in \G}\|\mu_n^\phi\|_1\E[J_n]^{1/2}\Big\{1+{2C^2}\frac{|\B_{b_n}|^2}{|\A_n|^2}+\frac{4|\B_{b_n}|C}{|\A_n|}\Big\}^{1/2}.$$
Hence ${k_n}^{-\beta}$ is bounded by 
$$
\begin{aligned}
{k_n}^{-\beta} \le & \Big(\sigma_1^{\alpha-1}\Big(\frac{|\B_{b_n}|}{|\A_n|^{1/2}}\Big)^{\alpha-1} - 1\Big)^{-\beta}\\
\le& \min\{\sigma_1^{(1-\alpha)\beta},1\}\Big(\Big(\frac{|\B_{b_n}|}{|\A_n|^{1/2}}\Big)^{\alpha-1} - 1\Big)^{-\beta}\\
\le& 2^{\beta}\min\{\sigma_1^{(1-\alpha)\beta},1\}\Big(\frac{|\B_{b_n}|}{|\A_n|^{1/2}}\Big)^{(1-\alpha)\beta},
\end{aligned}
$$  
where the last inequality is because 
$\Big(\frac{|\B_{b_n}|}{|\A_n|^{1/2}}\Big)^{\alpha-1} - 1 > \frac{1}{2}\Big(\frac{|\B_{b_n}|}{|\A_n|^{1/2}}\Big)^{\alpha-1}$. 
Then we obtain that, by taking $s = 2$, 
\begin{align*}
    (P_2)\le & \Big\|\sup_{g\in \mathcal{F}_{\TV}}|g|\Big\|_{\infty}\bigg\{{\epsilon_n}^{\alpha} + 2\frac{\mathcal{S}\|J_n\|^2_2|\B_{b_n}|}{|\A_n|^3}\sup_{\phi\in \A_n}|\A_n|^2\|Q_n(\phi)\|_2^2\\ 
    & \qquad \qquad \qquad + 2\frac{\E[J_n]}{|\A_n|}\Psi_{n,2}(c_n|\G)\sup_{\phi\in \A_n}|\A_n|\|Q_n(\phi)\|_{2}\\ 
    & \qquad \qquad \qquad+ \sup_{\phi\in \A_n}|\A_n|\|Q_n(\phi)\|_2\Big[2\mathcal{S}{k_n}^{-\beta}\frac{\|J_n\|^2_2}{|\A_n|^2}\\
    &\qquad \qquad  \qquad+  \mathcal{S}\frac{\abs{\B_{b_n}}}{|\A_n|}\sqrt{2\mathrm{Var}(J_n)}\frac{\|J_n\|_2}{|\A_n|} + \frac{1}{|\A_n|}\sqrt{2\mathrm{Var}(J_n)}~\Big]\bigg\}\\
    \le & \Big\|\sup_{g\in \mathcal{F}_{\TV}}|g|\Big\|_{\infty}\bigg\{\sigma_1^{\alpha}\Big(\frac{|\B_{b_n}|}{|\A_n|^{1/2}}\Big)^{\alpha} + 2\mathcal{S}{\tau_1}^2{\tau_2}^2\frac{|\B_{b_n}|}{|\A_n|} + 2\tau_1\tau_2\Psi_{n,2}(c_n|\G)\\ 
    & \qquad \qquad \qquad+ 2^{\beta+1}{\tau_1}^2\tau_2\mathcal{S}\min\{\sigma_1^{(1-\alpha)\beta},1\}\Big(\frac{|\B_{b_n}|}{|\A_n|^{1/2}}\Big)^{(1-\alpha)\beta} \\
    & \qquad \qquad \qquad\qquad+ \tau_1\tau_2\mathcal{S}\frac{\abs{\B_{b_n}}}{|\A_n|}\sqrt{2\mathrm{Var}(J_n)} + \tau_2\frac{1}{|\A_n|}\sqrt{2\mathrm{Var}(J_n)}~\bigg\}\\
    \le & \Big\|H_0(\lambda^{b_n}_{n}
    )\Big\|_\infty\kappa_2\bigg\{\Big(\frac{|\B_{b_n}|}{|\A_n|^{1/2}}\Big)^{\min\{\alpha,(1-\alpha)\beta\}} + \Psi_{n,2}(c_n|\G) + \frac{\abs{\B_{b_n}}}{|\A_n|}\sqrt{\mathrm{Var}(J_n)}~\bigg\},
\end{align*}
where $\kappa_2 := O(\max\{\sigma_1^{\alpha},2\mathcal{S}{\tau_1}^2{\tau_2}^2,2\tau_1\tau_2, 2^{\beta+1}{\tau_1}^2\tau_2\mathcal{S}\min\{\sigma_1^{(1-\alpha)\beta},1\},\sqrt{2}\tau_1\tau_2\mathcal{S},\sqrt{2}\tau_2\}).$ 
%By taking $\kappa:=\max\{\kappa_1,\kappa_2\}$, we arrive at the desired result.
\end{proof}

\begin{theorem}\label{roadtrip}
Assume all the conditions of \cref{simplifiedgenbound} holds. Suppose that the sequence $(b_n)_{n\in\N}$ satisfies the following conditions: 
\begin{enumerate}[(i)]
    \item $\sup_{n}\Big\|\mathbb{E}\Big[\sum_{i \leq J_n} f_{n}(\phi_{ni} X_{n}) \Big|\G\Big]\Big\|_\infty$ is bounded; 
    \item $\mathcal{R}_{\Psi}(n,2,b_n)\xrightarrow{n\to\infty} 0$, $\mathcal{R}_{\xi}(n,2,b_n) \xrightarrow{n\to\infty} 0$;
    \item $\frac{|\B_{2b_n}|}{|\A_n|}\xrightarrow{n\to\infty} 0$;
    \item $\frac{|\B_{b_n}|}{|\A_n|^{1/2}}\xrightarrow{n\to\infty} 0$; 
    \item $\Psi_{n,2}(c_n|\G)\xrightarrow{n\to\infty} 0$;
    \item $\frac{\abs{\B_{b_n}}}
    {|\A_n|}\sqrt{\mathrm{Var}(J_n)}\xrightarrow{n\to\infty} 0$.
\end{enumerate}
Then the following holds: $$\E\Big[\dTV(W_n,Z({\lambda}^{b_n}_{n})|\G)\Big]\xrightarrow{n\to\infty} 0.$$  
\end{theorem}
\begin{proof} Provided condition (i) holds, we can establish that $\|H_{0}(\lambda^{b_n}_{n})\|_\infty, \|H_{1}(\lambda^{b_n}_{n})\|_\infty$ are also bounded, as a result of \cref{style}.   Under the conditions we have, we can obtain that all the terms in \cref{calamari} converge to 0. Hence we have $\E\Big[\dTV(W_n,Z({\lambda}^{b_n}_{n})|\G)\Big]\xrightarrow{n\to\infty} 0$.
\end{proof}

\subsubsection{Proofs of corollaries}

\begin{proof}[Proof of \cref{lobsterroll}(i)]
Since $j_n$ is not random and $j_n=O(|A_n|)$, we have that $\mathrm{Var}(j_n)=0$, and $\frac{\|j_n\|_3}{|\A_n|}$ is bounded. The conclusion can be directly drawn from \cref{roadtrip}.
\end{proof}

\begin{proof}[Proof of \cref{lobsterroll}(ii)] Denote $$
k\bar\lambda_k^{b_n,n} :=\mathbb{E}\Big[\sum_{i\leq J_n}f_{n}(\phi_{ni} X_{n}) \mathbb{I}(\sum_{\substack{i^{\prime}\leq J_n \\ d(\phi_{ni}, \phi_{ni^{\prime}}) \leq b_n}} f_{n}(\phi_{ni^{\prime}} X_{n})=k)\Big|\G\Big].
$$
By the triangle inequality, we have
\begin{equation}\label{ricecrackers}
\E\Big[\dTV(W_n,Z({\lambda}^{b_n}_{n})|\G)\Big]\le\E\Big[\dTV(W_n,Z(\bar{{\lambda}}^{b_n}_{n})|\G)\Big]+\E\Big[\dTV(Z({\lambda}^{b_n}_{n}),Z(\bar{{\lambda}}^{b_n}_{n})|\G)\Big].
\end{equation}

We aim to demonstrate that both terms on the right-hand side of \cref{ricecrackers} converge to 0. First, we focus on showing the convergence of the first term. Since $J_n\sim \textrm{Poisson}(j_n)$, we have $\mathrm{Var}(J_n)=j_n$ and $\E[J_n^3]=j_n^3+3j_n^2+j_n$. Moreover, we have $j_n=O(|A_n|)$, so $\frac{\|J_n\|_3}{|\A_n|}$ is bounded. Plugging $\mathrm{Var}(J_n)=j_n$ into condition (vii) of \cref{roadtrip}, we get that the condition becomes $\frac{\abs{\B_{b_n}}}
    {|\A_n|^{\frac{1}{2}}}\xrightarrow{n\to\infty} 0$, which is satisfied as long as \eqref{C2} holds. Hence, by \cref{roadtrip}, it follows that 
    $$
\E\Big[\dTV(W_n,Z(\bar{{\lambda}}^{b_n}_{n})|\G)\Big]\xrightarrow{n\to\infty}0.
$$ 

For the second term on the right-hand side of \cref{ricecrackers}, note that by \cref{SteinCPboundTV}, we obtain that 
$$
\E\Big[\dTV(Z({\lambda}^{b_n}_{n}),Z(\bar{{\lambda}}^{b_n}_{n})|\G)\Big]\leq \Big\|\sup_{g\in\mathcal{F}_{\TV}({\lambda}^{b_n}_{n})}|g|\Big\|_\infty\Big\|\sum_{k=1}^{\infty} 
\Big|k(\lambda^{b_n}_{n}(k)- \bar\lambda^{b_n}_{n}(k))\Big|~\Big\|_1.
$$
Recall that a result we have proved in \cref{EScorrectionterm} is that
    $$
\Big\|\sum_{k=1}^{\infty} 
\Big|k(\lambda^{b_n}_{n}(k)- \bar\lambda^{b_n}_{n}(k))\Big|~\Big\|_1
\leq\|J_n-j_n\|_1 \sup_{\phi\in\A_n}\|Q_{n}(\phi)\|_1\Big(j_n  \frac{\mathcal{S}\abs{\B_{b_n}}}{\abs{\A_n}}+1\Big).
$$
Note that $\|J_n-j\|_1\leq \|J_n-j_n\|_2 = \mathrm{Var}(J_n)^{\frac{1}{2}} = j_n^{\frac{1}{2}} = o(|\A_n|^{\frac{1}{2}})$. Plugging in the conditions that $(b_n)_{n\in\N}$ satisfies, we obtain $\Big\|\sum_{k=1}^{\infty} 
\Big|k(\lambda^{b_n}_{n}(k)- \bar\lambda^{b_n}_{n}(k))\Big|~\Big\|_1\xrightarrow{n\to\infty}0$. Moreover, as a result of \eqref{C0}, $\Big\|\sup_{g\in\mathcal{F}_{\TV}({\lambda}^{b_n}_{n})}|g|\Big\|_\infty$ is bounded. Therefore, we have
$$
\E\Big[\dTV(Z({\lambda}^{b_n}_{n}),Z(\bar{{\lambda}}^{b_n}_{n})|\G)\Big]\xrightarrow{n\to\infty}0.
$$
The conclusion follows from \cref{ricecrackers}.
\end{proof}

\begin{proof}[Proof of \cref{acadia}]
    The result directly follows from \cref{lobsterroll} (ii).
\end{proof}

\end{appendix}

\end{document}